\documentclass[11pt,a4paper]{article}

\usepackage[OT1]{fontenc}
\usepackage{lmodern}
\usepackage[english]{babel}
\usepackage[utf8x]{inputenc}
\usepackage[sc]{mathpazo}
\usepackage{amsmath,amssymb,amsfonts,mathrsfs,amstext}
\usepackage[amsmath,thmmarks]{ntheorem}
\usepackage{graphicx}
\usepackage{soul}
\usepackage{pdfpages}

\usepackage{varioref}

\usepackage{datetime}

\usepackage{mathtools}

\usepackage{array}

\usepackage{listings}
\numberwithin{equation}{section}

\newtheorem{theorem}{Theorem}[section]

\newtheorem{remark}[theorem]{Remark}

\newtheorem{definition}[theorem]{Definition}
\newtheorem{lemma}[theorem]{Lemma}
\newtheorem{proposition}[theorem]{Proposition}

\theoremstyle{nonumberplain}
\theorembodyfont{\normalfont}
\theoremsymbol{\ensuremath{\blacksquare}}
\newtheorem{proof}{Proof}
\DeclarePairedDelimiter\norm{\lVert}{\rVert}

\renewcommand{\epsilon}{\ensuremath\varepsilon}

\renewcommand{\phi}{\ensuremath{\varphi}}

\usepackage[linkcolor=black,colorlinks=true,citecolor=black,filecolor=black]{hyperref}

\usepackage[]{algorithmicx}
\usepackage{algpseudocode}
\usepackage[a4paper]{geometry} %For the Code part; to make boudaries larger.
\usepackage{color}
\usepackage{pstricks}    %for embedding pspicture.
\usepackage{graphicx}    %for figure environment.
\usepackage{wrapfig}
\usepackage{caption}
\usepackage{subcaption}
\usepackage[noabbrev]{cleveref}
\usepackage{csquotes}
\usepackage{bm}
\usepackage{mathtools}

\newcommand{\ki}[1]{{\color{blue}{#1}}}

\newcommand\mathmidscript[1]{\vcenter{\hbox{$\scriptstyle #1$}}}

\newcommand{\pvint}{\ensuremath{\;\mathmidscript{\raisebox{-0.9ex}{\tiny{p.\,v.\;}}}\!\!\!\!\!\!\!\!\int}}
\newcommand{\pvintnl}{{\;\mathmidscript{\raisebox{-0.9ex}{\tiny{p.\,v.\;}}}\!\!\!\!\!\!\!\!\int}}
\newcommand{\pvintepsnl}{\pvintnl^{\epsilon}_{\!\!\!-\epsilon}}
\newcommand{\Hfpint}{\ensuremath{\;\mathmidscript{\raisebox{-0.9ex}{\tiny{H.\,f.\,p.\;}}}\!\!\!\!\!\!\!\!\!\!\!\!\int}}
\newcommand{\Hfpinteps}{\Hfpint\limits_{-\eps}^\eps\;\;}
\DeclareMathAlphabet{\mathpzc}{OT1}{pzc}{m}{it}

\DeclareMathOperator*{\sinc}{sinc}

\newcommand*\Laplace{\mathop{}\!\mathbin\bigtriangleup}

\newcommand{\NORM}[1]{\left\lVert#1\right\rVert} %already defined somewhere

\newcommand{\DEF}{\coloneqq}
\newcommand{\FED}{\eqqcolon}
\newcommand{\RR}{\mathbb{R}}
\newcommand{\CC}{\mathbb{C}}
\newcommand{\NN}{\mathbb{N}}
\newcommand{\ZZ}{\mathbb{Z}}
\newcommand{\Om}{\Omega}

\newcommand{\del}{\partial}

\newcommand{\MID}{\! \mid\!}
\newcommand{\MIDD}{\! \mid\!\!}

\newcommand{\cC}{\mathcal{C}}
\newcommand{\eps}{\epsilon}
\newcommand{\boldk}{\mathbf{k}}
\newcommand{\boldnull}{\mathbf{0}}
\newcommand{\boldone}{\mathbf{1}}
\newcommand{\OO}{\mathcal{O}}

\newcommand{\udk}{u^{\delta k}}
\newcommand{\udknull}{u^{\delta k}_0}
\newcommand{\udknullcircP}{\udknull\!\circ\! \opP}

\newcommand{\udks}{u^{\delta k}_s}

\newcommand{\Uk}{U^{k}}
\newcommand{\Uknull}{U^{k}_0}
\newcommand{\opP}{\mathtt{P}}
\newcommand{\UknullcircP}{\Uknull\!\circ\! \opP}
\newcommand{\Uks}{U^{k}_s}
\newcommand{\Ukapp}{\Uk_{\mathrm{app}}}

\newcommand{\udnoboldk}{u^{\delta  k}}

\newcommand{\udnoboldks}{u^{\delta  k}_s}
\newcommand{\udnoboldkf}{u^{\delta  k}_f}

\newcommand{\GKS}{\Gamma^\boldk_\sharp}
\newcommand{\GKA}{\Gamma^\boldk_\ast}

\newcommand{\GdKS}{\Gamma^{\delta\boldk}_\sharp}
\newcommand{\GdKA}{\Gamma^{\delta\boldk}_\ast}

\newcommand{\GKp}{\Gamma^\boldk_+}
\newcommand{\GKc}{\Gamma^\boldk_\times}
\newcommand{\GdKp}{\Gamma^{\delta\boldk}_+}
\newcommand{\GdKc}{\Gamma^{\delta\boldk}_\times}

\newcommand{\Gdkp}{\Gamma^{\delta k}_+}

\newcommand{\Neu}{\mathrm{N}}
\newcommand{\NEK}{\Neu_E^k}
\newcommand{\NdelEK}{\Neu_{\del E}^k}
\newcommand{\Reu}{\mathrm{R}}
\newcommand{\REK}{\Reu_E^k}
\newcommand{\RdelEK}{\Reu_{\del E}^k}
\newcommand{\NOKp}{\Neu_{\Omega,+}^\boldk}
\newcommand{\NOkp}{\Neu_{\Omega,+}^k}
\newcommand{\NdelOKp}{\Neu_{\del \Omega,+}^\boldk}
\newcommand{\NOKc}{\Neu_{\Omega,\times}^\boldk}
\newcommand{\NOkc}{\Neu_{\Omega,\times}^k}
\newcommand{\NdelOKc}{\Neu_{\del \Omega,\times}^\boldk}
\newcommand{\ROKp}{\Reu_{\Omega,+}^\boldk}
\newcommand{\RdelOKp}{\Reu_{\del \Omega,+}^\boldk}
\newcommand{\ROKc}{\Reu_{\Omega,\times}^\boldk}
\newcommand{\RdelOKc}{\Reu_{\del \Omega,\times}^\boldk}
\newcommand{\intd}{\mathrm{d}}
\newcommand{\II}{\mathrm{I}}
\newcommand{\KK}{\mathcal{K}}
\newcommand{\KdelOKc}{\KK^\boldk_{\del\Omega,\times}}
\newcommand{\KEK}{\KK^k_{\del E}}
\newcommand{\SSS}{\mathcal{S}}
\newcommand{\SdelOKc}{\SSS^\boldk_{\del\Omega,\times}}
\newcommand{\SEK}{\SSS^k_{\del E}}

\newcommand{\NDK}{\Neu_D^k}

\newcommand{\NDdK}{\Neu_D^{\delta k}}
\newcommand{\NdelDdK}{\Neu_{\del D}^{\delta k}}

\newcommand{\RdelDdK}{\Reu_{\del D}^{\delta k}}

\newcommand{\NOonedkp}{\Neu_{\Omega^1,+}^{\delta k}}

\newcommand{\NdelOonedkc}{\Neu_{\del \Omega^1,\times}^{\delta k}}

\newcommand{\ROonedkp}{\Reu_{\Omega^1,+}^{\delta k}}
\newcommand{\RdelOonedKp}{\Reu_{\del \Omega^1,+}^{\delta\boldk}}
\newcommand{\RdelOonedkp}{\Reu_{\del \Omega^1,+}^{\delta k}}

\newcommand{\NdelOonedkp}{\Neu_{\del \Omega^1,+}^{\delta k}}

\newcommand{\fdk}{f^{\delta k}}
\newcommand{\curlXe}{{\mathcal{X}^\epsilon}}
\newcommand{\curlYe}{{\mathcal{Y}^\epsilon}}

\newcommand{\kkkk}{\mathbb{K}}
\newcommand{\kkkkc}{\kkkk_{\CC}}

\newcommand{\calL}{\mathcal{L}}
\newcommand{\calR}{\mathcal{R}}
\newcommand{\calGG}{\widehat{\Gamma}}

\newcommand{\LLe}{\mathcal{L}^\eps}
\newcommand{\LLeinv}{(\LLe)^{-1}}
\newcommand{\pvinteps}{\pvint\limits_{-\eps}^\eps\,}
\newcommand{\KKe}{\mathcal{K}^\eps}
\newcommand{\RRdke}{\mathcal{R}^{\delta k, \eps}}
\newcommand{\RRdnoboldke}{\mathcal{R}^{\delta k, \eps}}
\newcommand{\AAdke}{\mathcal{A}^{\delta k, \eps}}
\newcommand{\AAdnoboldke}{\mathcal{A}^{\delta k, \eps}}
\newcommand{\inteps}{\int\limits_{-\eps}^\eps}
\newcommand{\calQ}{\mathcal{Q}}
\newcommand{\QQdke}{\calQ^{\delta k,\eps}}
\newcommand{\wLLdke}{\widetilde{\calL}^{\delta k,\eps}}
\newcommand{\wLLdkeinv}{(\wLLdke)^{-1}}
\newcommand{\SSdke}{\mathcal{S}^{\delta  k,\eps}}
\newcommand{\calI}{\,\mathcal{I}}

\newcommand{\Adke}{A^{\delta k, \eps}}
\newcommand{\Adkesqrt}{A^{\delta k, \eps}_{\star}}
\newcommand{\AdkesqrtNull}{A^{\delta k, \eps}_{\star,(0)}}

\newcommand{\TransT}{\mathrm{T}}
\newcommand{\Ltwo}{\mathrm{L}^{2}}

\newcommand{\RRdkeone}{\RRdnoboldke_1}
\newcommand{\RRdketwo}{\RRdnoboldke_2}
\newcommand{\RRdkethree}{\RRdnoboldke_3}

\newcommand{\ceps}{c_\eps}
\newcommand{\Sdke}{S^{\delta  k,\eps}}

\newcommand{\fdnoboldk}{f^{\delta k}}

\newcommand{\resSum}{\frac{\delta^2 k^2}{\delta k^{\delta, \eps}_+ - \delta k^{\delta, \eps}_-}\left(\frac{1}{\delta k-\delta k^{\delta, \eps}_+}-\frac{1}{\delta k-\delta k^{\delta, \eps}_-}\right)}
\newcommand{\resSumWOfirstfrac}{\left(\frac{1}{\delta k-\delta k^{\delta, \eps}_+}-\frac{1}{\delta k-\delta k^{\delta, \eps}_-}\right)}
\newcommand{\resSumWOfirstfracSHORT}{\left(\!\frac{1}{\delta k\!-\!\delta k^{\delta, \eps}_+}\!-\!\frac{1}{\delta k\!-\!\delta k^{\delta, \eps}_-}\!\right)}
\newcommand{\dkdep}{\delta k^{\delta, \eps}_+}
\newcommand{\dkdem}{\delta k^{\delta, \eps}_-}

\newcommand{\mudkstar}{\mu^{\delta k}_\star}
\newcommand{\mudkt}{\mu^{\delta k}_\sim}

\newcommand{\Seu}{\mathcal{S}}
\newcommand{\Leu}{\mathrm{L}}
\newcommand{\Teu}{\mathcal{T}}

\newcommand{\cIBC}{c_{\mathrm{IBC}}}

\newcommand{\BOm}{\bm{\Omega}}
\newcommand{\BXi}{\bm{\Xi}}

\newcommand{\BAAdke}{\bm{\mathcal{A}}^{\delta k, \eps}}
\newcommand{\BAAdnoboldke}{\bm{\mathcal{A}}^{\delta k, \eps}}
\newcommand{\NDjdk}{\Neu^{\delta k}_{D_j}}
\newcommand{\NdelDjdk}{\Neu^{\delta k}_{\del D_j}}
\newcommand{\NBOonedkp}{\Neu^{\delta k}_{\BOm^1,+}}
\newcommand{\NdelBOonedkp}{\Neu^{\delta k}_{\del\BOm^1,+}}
\newcommand{\NdelBOonedkc}{\Neu^{\delta k}_{\del\BOm^1,\times}}
\newcommand{\Ofdk}{f^{\delta k}_{\del \BOm^1}}
\newcommand{\Bfdk}{\bm{f}^{\delta k}}
\newcommand{\BcurlXe}{{\bm{\mathcal{X}}^\epsilon}}
\newcommand{\BcurlYe}{{\bm{\mathcal{Y}}^\epsilon}}

\newcommand{\BcalGG}{\bm{\widehat{\Gamma}}}

\newcommand{\BLLe}{\bm{\mathcal{L}}^\eps}
\newcommand{\BLLeinv}{(\BLLe)^{-1}}
\newcommand{\BKKe}{\bm{\mathcal{K}}^\eps}
\newcommand{\BRRdke}{\bm{\mathcal{R}}^{\delta k, \eps}}
\newcommand{\BRRdnoboldke}{\bm{\mathcal{R}}^{\delta k, \eps}}
\newcommand{\BcalQ}{\bm{\mathcal{Q}}}
\newcommand{\BQQdke}{\BcalQ^{\delta k,\eps}}
\newcommand{\BwLLdke}{\bm{\widetilde{\calL}}^{\delta k,\eps}}
\newcommand{\BwLLdkeinv}{(\BwLLdke)^{-1}}
\newcommand{\BSSdke}{\bm{\mathcal{S}}^{\delta  k,\eps}}
\newcommand{\BcalI}{\,\bm{\mathcal{I}}}

\newcommand{\RBOonedkp}{\Reu_{\BOm^1,+}^{\delta k}}
\newcommand{\RdelBOonedKp}{\Reu_{\del \BOm^1,+}^{\delta\boldk}}
\newcommand{\RdelBOonedkp}{\Reu_{\del \BOm^1,+}^{\delta k}}

\newcommand{\RdelDjdK}{\Reu_{\del D_j}^{\delta k}}

\newcommand{\boldeone}{\bm{\mathrm{e}}_1}
\newcommand{\boldetwo}{\bm{\mathrm{e}}_2}
\newcommand{\boldej}{\bm{\mathrm{e}}_j}
\newcommand{\boldejhat}{\bm{\mathrm{e}}_{\hat{j}}}
\newcommand{\leps}{\mathfrak{c}_{\eps}}

\newcommand{\BAdkesqrt}{\bm{A}^{\delta k, \eps}_{\star}}

\newcommand{\BAdke}{\bm{A}^{\delta k, \eps}}
\newcommand{\BAdk}{\bm{A}^{\delta k}}
\newcommand{\BSdke}{\bm{S}^{\delta  k,\eps}}

\newcommand{\BRRdkeone}{\BRRdnoboldke_1}
\newcommand{\BRRdketwo}{\BRRdnoboldke_2}
\newcommand{\BRRdkethree}{\BRRdnoboldke_3}

\newcommand{\Balphanull}{(\bm{\alpha_{(0)}})}
\newcommand{\Balphaone}{(\bm{\alpha_{(1)}})}
\newcommand{\BT}{(\bm{\mathrm{T}_{(0)}})}
\newcommand{\BS}{(\bm{\mathrm{T}_{(1)}})}
\newcommand{\BTnull}{(\bm{\mathrm{T}_{(0)}})}
\newcommand{\BTone}{(\bm{\mathrm{T}_{(1)}})}
\newcommand{\BBalphanull}{\bm{\alpha_{(0)}}}
\newcommand{\BBalphaone}{\bm{\alpha_{(1)}}}

\newcommand{\BBTnull}{\bm{\mathrm{T}_{(0)}}}
\newcommand{\BBTone}{\bm{\mathrm{T}_{(1)}}}

\newcommand{\rmY}{\mathrm{Y}^\eps}
\newcommand{\rmYss}{\mathrm{Y}^\eps_{\star,\star}}
\newcommand{\BrmY}{\bm{\mathrm{Y}}^\eps}
\newcommand{\BrmYss}{\bm{\mathrm{Y}}^\eps_{\star,\star}}

\newcommand{\BrmMss}{\bm{\mathrm{M}}^\eps_{\star,\star}}
\newcommand{\BA}{\bm{A}}

\newcommand{\BAdkess}{\BA^{\eps}_{\star,\star}}
\newcommand{\BrmMdkone}{\bm{\mathrm{M}}^{\delta k}_{+}}
\newcommand{\BrmMdktwo}{\bm{\mathrm{M}}^{\delta k}_{-}}

\newcommand{\BMdkerest}{\bm{\mathrm{M}}^{\delta k,\eps}_{\mathrm{rest}}}

\newcommand{\Bmu}{\bm{\mu}}
\newcommand{\Bmudkstar}{\Bmu^{\delta k}_\star}
\newcommand{\Bmudkt}{\Bmu^{\delta k}_\sim}

\newcommand{\uzke}{u_z^{k,\eps}}
\newcommand{\uxSke}{u_{x_S}^{k,\eps}}
\newcommand{\Heu}{\mathcal{H}}
\newcommand{\Geu}{\mathrm{G}}
\newcommand{\GOk}{\Geu_\Om^k}

\newcommand{\RGOk}{\Reu_{\Geu,\Om}^k}
\newcommand{\RdelGOk}{\Reu_{\del\Geu,\Om}^k}

\newcommand{\vzke}{v^{k,\eps}_{z}}

\newcommand{\arcYe}{{\mathcal{Y}}^\epsilon_\Pi}

\newcommand{\Jeu}{\mathcal{J}}
\newcommand{\JJe}{\Jeu^{\eps}}
\newcommand{\HHe}{\Heu^{\eps}}
\newcommand{\JJeinv}{(\Jeu^{\eps})^{-1}}
\newcommand{\HHed}{(\Heu^{\eps})^{\dagger}}
\newcommand{\calM}{\mathcal{M}}
\newcommand{\sbsa}{\frac{\sqrt{\beta}}{\sqrt{\alpha}}}
\newcommand{\rmI}{\,\mathrm{I}}
\newcommand{\rmH}{\,\mathrm{H}}

\title{Wave-Field Shaping in Cavities by Tunable Metasurfaces}
\author{ Habib Ammari\thanks{\footnotesize Department of Mathematics, 
ETH Z\"urich, 
R\"amistrasse 101, CH-8092 Z\"urich, Switzerland (habib.ammari@math.ethz.ch, kimeri@student.ethz.ch, wei.wu@sam.math.ethz.ch).} \and Kthim Imeri\footnotemark[1]
\and Wei Wu\footnotemark[1]}
\date{}

\begin{document}
	\maketitle

\begin{abstract}
Cavities, because they trap waves for long times due to their reflecting walls, are used in a vast number of scientific domains. Indeed, in these closed media and due to interferences, the free space continuum of solutions becomes a discrete set of stationary eigenmodes. These enhanced stationary fields are commonly used in fundamental physics to increase wave-matter interactions. The eigenmodes and associated eigenfrequencies of a cavity are imposed by its geometrical properties through the boundary conditions.  In this paper, we show that one can control  the wave fields created by point sources inside cavities by tailoring only the boundaries of the cavities. This is achieved through the use of a tunable reflecting metasurface, which is part of the frontiers of the cavity, and can switch its boundary conditions from Dirichlet to Neumann. Based on the use of arrays of subwavelength resonators, a mathematical modeling of the physical mechanism underlying the concept of tunable metasurfaces is provided.

\end{abstract}

\def\keywords2{\vspace{.5em}{\textbf{  Mathematics Subject Classification
(MSC2000).}~\,\relax}}
\def\endkeywords2{\par}
\keywords2{35R30, 35C20.}

\def\keywords{\vspace{.5em}{\textbf{ Keywords.}~\,\relax}}
\def\endkeywords{\par}
\keywords{Subwavelength resonance, Helmholtz resonator, hybridization, metasurfaces,  cavity.}

\tableofcontents
%\mainmatter

%% Your real content!
\section{Introduction}

Controlling waves in cavities, which are used in numerous domains of applied and fundamental physics, has become a major topic of interest \cite{cavity1,cavity3,cavity2, mur-mathias}. The wave fields established in cavities are  fixed by their geometry. They are usually modified by using mechanical parts. Nevertheless, tailoring the cavity boundaries permits one to design at will the wave fields they support. Here, we show  that it is achievable simply by using tunable metasurfaces that locally modify the boundaries, switching them from Dirichlet to Neumann conditions. The concept of metasurfaces is a powerful tool to shape waves by governing precisely the phase response of each constituting element through its subwavelength resonance properties \cite{METASCREEM, MetaSurfaceIBC, metasurfaces2, metasurfaces}. Subwavelength resonators have been also used as the building block of super-resolution imaging \cite{HaiHabib2, HaiHabib3,HaiHabib}. 

A metasurface is a thin sheet with patterned subwavelength structures, which nevertheless has a macroscopic effect on wave propagation. Based on the concept of hybridized resonators, a tunable metasurface can be designed. Hence, it can be transformed into a tunable component that allows shaping waves dynamically in unprecedented ways \cite{mur-lemoult, mur-hybridized, mur-mathias}. The mechanism is based on the very general concept of hybridized coupled resonant elements whose resonant frequencies can be tuned by adjusting the coupling strength.  The idea from \cite{mur-hybridized} is to design a metasurface that can be either resonant or not resonant at a given operating frequency.  In the first case, the collective resonant behaviour of the subwavelength resonators provides a change of the boundary condition  while in the second case, the metasurface is transparent to the incident wave. To that aim, one can take as unit cell of the metasurface a system made out of two individual subwavelength resonators: one static resonator (referred to as the main resonator), whose frequency is fixed to the operating frequency and one tunable parasitic resonator (referred to as the parasitic resonator) whose frequency can be wisely adjusted by a given tunable mechanism. In the first case, the resonance frequency of the parasitic resonator is different enough from the resonance frequency of the main resonator, so that the two resonant elements do not couple. At the operating frequency, the metasurface is then resonant. In the second case however, one sets the resonance frequency of the parasitic resonator to match that of the static resonator. In that case, the subwavelength resonators hybridize to create a dimer whose eigenfrequencies are respectively under and above the initial resonance frequency. 

In this paper, we mathematically and numerically model the physical mechanism underlying the concept of tunable metasurfaces. We consider Helmholtz resonators. We show that an array of Helmholtz resonators behaves as an equivalent surface with Neumann boundary condition at a resonant frequency which corresponds to a wavelength much greater than the size of the Helmholtz resonators. Analytical formulas for the hybridized  resonances of coupled Helmholtz resonators are also derived. Numerical simulations confirm their accuracy.  We also propose an efficient approach to characterize the Green's function of a cavity with mixed (Dirichlet and Neumann) boundary conditions. The use of tunable metasurfaces allows us to find a criterion ensuring 
that modifying parts of a cavity's boundaries turn it into a completely different one. 
We provide a new and simple procedure for maximizing the Green's function between two points at a chosen frequency in terms of the  boundary conditions. Our algorithm is a one shot optimization algorithm and can then be used in real-time to focalize the wave on a given spot by maximizing the transmission between an emitter and a receiver through specific eigenmodes of the cavity or on the contrary, to minimize the field on a receiver.

The paper is organized as follows. Section \ref{ch2} is devoted to give preliminary results on the so-called Neumann functions, which play a key role in proving the results in Sections \ref{ch1HR} and \ref{ch:2HR}. We first introduce the quasi-periodic fundamental solutions to the Helmholtz equation and recall in Lemmas \ref{prop:GammaExpansion}--\ref{prop:GKSFormulaFarFieldz2>x2} some key results from \cite{METASCREEM}. Then, we consider the Neumann functions, which depends crucially on certain remainder functions, for which we provide exact formulas in Lemmas \ref{prop:FormulaForR}--\ref{lemma:1=2K[1]}.

In Section \ref{ch1HR}, we look at one periodically repeated Helmholtz resonator above a ground plate. After treating them with an incident wave, we obtain a scattered wave, whose resonant values are discussed and its behavior at the far-field are examined. 
We  show in Theorem \ref{THM2:1HR} that the structure behaves as an equivalent surface with Neumann boundary condition at the resonant frequencies characterized in Theorem \ref{THM1:1HR}.  
The proof uses a combined technique of \cite{HaiHabib} and \cite{METASCREEM}. 

Section \ref{ch:2HR} has the same objective as the previous section, but this time we have two periodically repeated Helmholtz resonator above a ground plate. As shown in Theorem \ref{THM1:2HR}, the strong coupling between the periodically repeated pair of resonators leads us to hybridized resonances. It is shown in Theorem \ref{THM2:2HR} that at only these hybridized frequencies the structure behaves as  an equivalent surface with Neumann boundary condition.  

Section \ref{ch5} is devoted to the derivation of an asymptotic formula of the Green's function of a cavity with mixed boundary conditions in terms of the size of the part of the cavity boundary  where the boundary condition is switched from Dirichlet to Neumann. A closed form of the derivative of the Green's function with respect to changes in the boundary condition is given in Theorem \ref{thm:ch4mainresult}. 
 
In Section \ref{ch6f}, we consider the  problem where we have a source in a bounded domain and we want to determine whether we activate a small part of the boundary to be reflecting or not in such a way  the signal at a given receiving point is significantly enhanced. 
Based on Theorem \ref{thm:ch4mainresult}, we propose a simple strategy aiming to maximize the norm of the Green's function by nucleating Neumann boundary conditions. The basic idea follows the concept of topological derivative. The switching of parts of the boundary from Dirichlet to Neumann where the topological derivative of the norm of the Green's function is positive allows for an increasing of the transmission between the point source and the receiver. Finally, we present some numerical experiments to show the applicability of the proposed methodology.

\section{Preliminaries}\label{ch2}

In this section, we introduce the quasi-periodic fundamental solutions to the  Helmholtz equation. The explicit formula derived in \cite{METASCREEM} will be helpful for us. Then we consider the Neumann functions and their remainders.

\subsection{Quasi-Periodic Fundamental Solution to the Helmholtz Equation with Dirichlet Boundary Conditions}
Let $\boldk\DEF(k_1,k_2)^\TransT\in\RR^2$ and $k\DEF|\boldk|\DEF(k_1^2+k_2^2)^{1/2}\in [0,\infty)$, where $\TransT$ denotes the transpose. Then for $k\in (0,\infty)$ we define $\Gamma^k:\RR^2\setminus\{0\}\rightarrow\CC$ as the fundamental solution to $\left( \Laplace + k^2  \right) \Gamma^k(x) = \delta_\boldnull(x)$ satisfying the Sommerfeld radiation condition. For $k=0$ we define $\Gamma^0:\RR^2\setminus\{0\}\rightarrow\CC$ as the fundamental solution to $ \Laplace \Gamma^0(x) = \delta_\boldnull(x)$. For $k\in[0,\infty)$, we define $\Gamma^k(z,x)\DEF \Gamma^k(z-x)$.
   For $x\in\RR^2\setminus\{0\}$ we have that
   \begin{align*}
      \Gamma^0(x)&=\frac{1}{2\pi}\ln |x|,\\
      \Gamma^k(x)&=-\frac{i}{4}H_0^{(1)}(k|x|),
   \end{align*}
   where $H_0^{(1)}$ is the Hankel function of the first kind. For $x$ near zero we have
   \begin{align}\label{equ:HankelExpansion}
      -\frac{i}{4}H_0^{(1)}(k|x|)= \frac{1}{2\pi}\ln |x|+\eta_k+\sum_{j=1}^{+\infty}(b_j\ln (k|x|)+c_j)(k|x|)^{2j},
   \end{align}
   where
   \begin{align*}
      b_j=\frac{(-1)^j}{2\pi}\frac{1}{2^{2j}(j!)^2},\;\; c_j\!=\!-b_j\left(\! \gamma-\ln 2-\frac{i\pi}{2}-\sum_{l=1}^j\frac{1}{l} \!\right),\;\; \eta_k=\frac{1}{2\pi}(\ln k +\gamma-\ln 2)-\frac{i}{4},
   \end{align*}
   where $\gamma$ is the Euler-Mascheroni constant.

The quasi-periodic fundamental solutions $\Gamma^\boldk_\sharp:\RR^2\setminus\{(np,0)^\TransT|n\in\ZZ\}\rightarrow\CC$ and $\Gamma^\boldk_\ast:\RR^2\setminus\{(np,0)^\TransT|n\in\ZZ\}\rightarrow\CC$ are defined by
\begin{align*}
   \Gamma^\boldk_\sharp(x)&=\sum_{n\in\ZZ}\Gamma^k\left(x+\begin{pmatrix}np\\0\end{pmatrix}\right)e^{-ik_1np},\\
   \GKA(x)&=\sum_{n\in\ZZ}\Gamma^k\left(x+\begin{pmatrix}np\\0\end{pmatrix}\right)e^{ik_1np}.
\end{align*}
 $\GKS$ and $\GKA$ are solutions to 
\begin{align*}
   \left( \Laplace + k^2  \right) \GKS(x)=\sum_{n\in\ZZ}e^{-ik_1np} \delta_{(np,0)^\TransT}(x),\\
   \left( \Laplace + k^2  \right) \GKA(x)=\sum_{n\in\ZZ}e^{+ik_1np} \delta_{(np,0)^\TransT}(x).
\end{align*}
Using Poisson summation formula, this leads us to the following formulas \cite[Lemma 3.2]{METASCREEM}:
\begin{lemma}\label{prop:ch2:expansionGammasharp}
   For the case $\boldk=\boldnull$ we have
   \begin{align}
      \Gamma^\boldnull_\sharp(x)\;=\;\Gamma^\boldnull_\ast(x)\; = \;\frac{ | x_2 |}{2p}\; - \!\sum_{\substack{n\in\ZZ\setminus\{0\}\\ l\DEF 2\pi n/p}} \frac{1}{2 p |l|} e^{- | l | \, |x_2|} e^{ i (l \, x_1)}. \label{equ:ch2:k=0GammaSharp}
   \end{align}
   If $\boldk$ satisfies $k^2<\inf\{|l-k_1|^2\MID l\DEF 2\pi n/p,\; n\in\ZZ\setminus\{0\}\}$, we have
   \begin{align*}
      \GKS(x)=\dfrac{e^{-i k_1  x_1} e^{i k_2  \, |x_2|}}{2 i k_2 p} - \sum_{\substack{n\in\ZZ\setminus\{0\}\\ l\DEF 2\pi n/p}} \frac{e^{-i k_1 x_1}}{2 p \sqrt{| l - k_1|^2 - k^2}} e^{- \sqrt{| l - k_1|^2 - k^2} | x_2 |} e^{i (l x_1)},\\
      \GKA(x)=\dfrac{e^{i k_1  x_1} e^{i k_2  \, |x_2|}}{2 i k_2 p} - \sum_{\substack{n\in\ZZ\setminus\{0\}\\ l\DEF 2\pi n/p}} \frac{e^{i k_1 x_1}}{2 p \sqrt{| l + k_1|^2 - k^2}} e^{- \sqrt{| l + k_1|^2 - k^2} | x_2 |} e^{i (l x_1)}.
   \end{align*}
\end{lemma}

Using Lemma \ref{prop:ch2:expansionGammasharp} we can expand $\GdKS$ and $\GdKA$ with respect to $\delta$ near $0$ and obtain
\begin{align}\label{equ:ch2:expansionGdKS}
   \GdKS(x)&=\frac{1}{2i\delta k_2 p}+\Gamma^\boldk_{0,\sharp}(x)+\sum_{n=1}^\infty \delta^n\Gamma^\boldk_{n,\sharp}\;,\\\label{equ:ch2:expansionGdKA}
   \GdKA(x)&=\frac{1}{2i\delta k_2 p}+\Gamma^\boldk_{0,\ast}(x)+\sum_{n=1}^\infty \delta^n\Gamma^\boldk_{n,\ast}\;,
\end{align}
where the kernels $\Gamma^\boldk_{0,\sharp}$ and $\Gamma^\boldk_{0,\ast}$ for $n\geq 1$ can be computed explicitly and shown to be smoother than $\Gamma^\boldk_{0,\sharp}$ and $\Gamma^\boldk_{0,\ast}$, see \cite[Chapters 7.3, 7.4]{LPTSA}. For the zeroth kernel we have with Equation (\ref{equ:ch2:k=0GammaSharp}), for $\boldk \neq \boldnull$, that
\begin{align} \label{equ:ch2:GKSandGKAalmostG0S}
   \Gamma_{0,\sharp}^\boldk(x)=\Gamma_\sharp^\boldnull(x)-\frac{k_1 x_1}{2 k_2 p}, \quad
   \Gamma_{0,\ast}^\boldk(x)=\Gamma_\ast^\boldnull(x)+\frac{k_1 x_1}{2 k_2 p}.
\end{align} 
We define $\GKS(z,x)\DEF\GKS(z-x)$ and $\GKA(z,x)\DEF\GKA(z-x)$. Consider that $\GKS$ and $\GKA$ are not symmetric for $\boldk\neq\boldnull$ in general, that is $\GKS(z,x)\neq\GKS(x,z)$. 

We define the parity operator $\opP:\RR^2\rightarrow \RR^2$ as 
\begin{align} \label{parity}
   \opP(x_1,x_2)=\begin{pmatrix}x_1\\-x_2\end{pmatrix}.
\end{align}

We introduce the quasi-periodic fundamental solutions to the Helmholtz equation with Dirichlet boundary condition $\GKp,\, \GKc:\{ (z,x)\in \RR^2\times\RR^2 \MID \neg(\exists n\in\ZZ:|z_1-x_1|=np \wedge z_2= x_2 )\} \rightarrow \CC$ by
\begin{align}\label{def:ch2:GKp,GKc}
   \Gamma^\boldk_+(z,x)\DEF\GKS(z,x)-\GKS(z,\opP(x)),\qquad \Gamma^\boldk_\times(z,x)\DEF\GKA(z,x)-\GKA(z,\opP(x)).
\end{align}
From Lemma \ref{prop:ch2:expansionGammasharp}, we see that for $z\in\del\RR^2_+$ or $x\in\del\RR^2_+$, we have that $\GKp(z,x)=0$ and $\GKc(z,x)=0$. Moreover, $\GKp$ and $\GKc$ are not symmetric, in general, and $\GKp$ and $\GKc$ are not translation invariant in general, that is $\GKp(z,x)\neq \GKp(z-x,0)$. The following results hold from \cite{MetaSurfaceIBC}. 
\begin{lemma}\label{prop:GammaExpansion}
   $\GdKp$ and $\GdKc$ admit the expansions of the form
   \begin{align*}
      \GdKp(z,x)=\Gamma_{+,0}^\boldk(z,x)+\sum_{n=1}^\infty \delta^n \Gamma_{+,n}^\boldk(z,x),\quad \GdKc(z,x)=\Gamma_{\times,n}^\boldk(z,x)+\sum_{n=1}^\infty \delta^n \Gamma_{\times,n}^\boldk(z,x),
   \end{align*}
   where
   \begin{align}
      \Gamma_+^\boldnull(z,x) &= \Gamma_{+,0}^\boldk(z,x) = \Gamma_{\times,0}^\boldk(z,x) \nonumber \\
      &= \frac{1}{4\pi}\bigg[ \log\bigg(\sinh\bigg(\frac{\pi}{p}(z_2-x_2)\bigg)^2+\sin\bigg(\frac{\pi}{p}(z_1-x_1)\bigg)^2 \bigg)\nonumber\\
      & \qquad\;-\log\bigg(\sinh\bigg(\frac{\pi}{p}(z_2+x_2)\bigg)^2+\sin\bigg(\frac{\pi}{p}(z_1-x_1)\bigg)^2\bigg) \bigg]\,.\label{equ:ch2:logsinhsin}
   \end{align}
\end{lemma}

\begin{lemma} \label{prop:GKSFormulaFarFieldx2>z2}
   Let $z,x\in\RR^2$, such that $x_2>z_2>0$, and let $k$ be small enough, then
   \begin{align*}
      \GKp(z,x)=\Gamma^\boldk_{+,p}(z,x)+\Gamma^\boldk_{+,e}(z,x)\;,
   \end{align*}
   where
   \begin{align*}
      \Gamma^\boldk_{+,p}(z,x)\DEF&-\frac{\sin(k_2z_2)}{k_2p}e^{i(k_2x_2-k_1(z_1-x_1))}, \\
      \Gamma^\boldk_{+,e}(z,x)\DEF&-\sum_{\substack{n\in\ZZ\setminus\{0\}\\ l\DEF 2\pi n/p}}\left( \frac{e^{i(l-k_1)(z_1-x_1)}}{p\sqrt{|l-k_1|^2-k^2}}\sinh\left( \sqrt{|l-k_1|^2-k^2}\,z_2 \right) \right)e^{-\sqrt{|l-k_1|^2-k^2}\,x_2}\;.
   \end{align*}
   We also have that
   \begin{align*}
      \nabla_x \Gamma^\boldk_{+,p}(z,x)=& i\,\boldk\, \Gamma^\boldk_{+,p}(z,x)\,,\\
      \nabla_x \Gamma^\boldk_{+,e}(z,x)=& -\sum_{\substack{n\in\ZZ\setminus\{0\}\\ l\DEF 2\pi n/p}}\bigg(\begin{pmatrix}-i(l-k_1)\\-\sqrt{|l-k_1|^2-k^2}\end{pmatrix}\frac{e^{i(l-k_1)(z_1-x_1)}}{p\sqrt{|l-k_1|^2-k^2}}\nonumber\\
      &\cdot\sinh\left( \sqrt{|l-k_1|^2-k^2}\,z_2 \right)e^{-\sqrt{|l-k_1|^2-k^2}\,x_2} \bigg)\;.
   \end{align*}
\end{lemma}

\begin{lemma} \label{prop:GKSFormulaFarFieldz2>x2}
   Let $z,x\in\RR^2$, such that $z_2>x_2>0$, and let $k$ be small enough, then
   \begin{align*}
      \GKp(z,x)=\Gamma^\boldk_{+,p}(z,x)+\Gamma^\boldk_{+,e}(z,x)\;,
   \end{align*}
   where
   \begin{align*}
      \Gamma^\boldk_{+,p}(z,x)\DEF&-\frac{\sin(k_2x_2)}{k_2p}e^{i(k_2z_2-k_1(z_1-x_1))},\\
      \Gamma^\boldk_{+,e}(z,x)\DEF&-\sum_{\substack{n\in\ZZ\setminus\{0\}\\ l\DEF 2\pi n/p}}\left( \frac{e^{i(l-k_1)(z_1-x_1)}}{p\sqrt{|l-k_1|^2-k^2}}\sinh\left( \sqrt{|l-k_1|^2-k^2}\,x_2 \right) \right)e^{-\sqrt{|l-k_1|^2-k^2}\,z_2}\;.
   \end{align*}
   We also have that
   \begin{align*}
      \nabla_x \Gamma^\boldk_{+,p}(z,x)=& \begin{pmatrix} -\frac{ik_1}{k_2p}\sin(k_2x_2) \\ -\frac{1}{p}\cos(k_2x_2) \end{pmatrix} e^{i(k_2z_2-k_1(z_1-x_1))} ,\\
      \nabla_x \Gamma^\boldk_{+,e}(z,x)=& -\sum_{\substack{n\in\ZZ\setminus\{0\}\\ l\DEF 2\pi n/p}}
      \left(\begin{pmatrix} -i(l-k_1)\sinh\left( \sqrt{|l-k_1|^2-k^2}\,x_2 \right) \\ \sqrt{|l-k_1|^2-k^2}\cosh\left( \sqrt{|l-k_1|^2-k^2}\,x_2 \right) \end{pmatrix}\right.\nonumber\\
      &\left.\cdot\frac{e^{i(l-k_1)(z_1-x_1)}}{p\sqrt{|l-k_1|^2-k^2}}\,e^{-\sqrt{|l-k_1|^2-k^2}\,z_2}\right)\;.
   \end{align*}
\end{lemma}

\subsection{Layer Potentials}
Let $E\subset \RR^2$ be an open, bounded, simply connected Lipschitz domain. Then we define the double-layer potential $\SEK: \Heu^{-1/2}(\del E)\rightarrow\Heu^{1/2}(\RR^2)$ and the Neumann-Poincar\'e operator $\KEK: \Heu^{1/2}(\del E)\rightarrow\Heu^{1/2}(\del E)$ as
\begin{align*}
	\SEK[\phi](x)&\DEF\int_{\del E} \Gamma^k(x,y)\,\phi(y)\intd \sigma_y\,,\\
	\KEK[\phi](x)&\DEF\pvint_{\del E} \del_{\nu_y}\Gamma^k(x,y)\,\phi(y)\intd \sigma_y\,,
\end{align*}
where 'p.v.' denotes the principal value. $\KEK$ is a bounded function and if $E$ has a $\cC^2$-boundary, $\KEK$ is a compact operator. Moreover, we have the following result (\cite[Propositions 2.5 and 2.6]{LPTSA}):

\begin{lemma}\label{lemma:LPL1}
	The operator $( -\tfrac{1}{2} )\calI+\KEK$ is invertible if and only if $k$ is not a Neumann eigenvalue of the operator $-\Laplace$. Then, $( (-\tfrac{1}{2}) \calI+\KEK)^{-1}$ has a continuation to an operator-valued meromorphic function in $\CC$. Also, all Neumann eigenvalues of the operator $-\Laplace$ are characteristic values of the operator $( -\tfrac{1}{2} )\calI+\KEK$.
\end{lemma}

We define now the double-layer potential and the Neumann-Poincar\'e operator on a periodic structure. Let $q_1,q_2\in\RR^2$ and let $E\Subset \{ y\in\RR^2\MID |y_1|<q_1/2\;\text{and}\; 0<y_2<q_2 \}$ be a open, bounded, simply connected domain. We define $\Omega\DEF \bigcup_{n\in\ZZ} E+(nq_1,0)^\TransT$ and with that let $\SdelOKc: \Heu^{-1/2}(\del E)\rightarrow\Heu^{1/2}(\RR^2)$ and $\KdelOKc: \Heu^{1/2}(\del E)\rightarrow\Heu^{1/2}(\del E)$ be defined as
\begin{align*}
	\SdelOKc[\phi](x)&\DEF\int_{\del \Omega} \GKc(x,y)\,\phi(y)\intd \sigma_y\,,\\
	\KdelOKc[\phi](x)&\DEF\pvint_{\del \Omega} \del_{\nu_y}\GKc(x,y)\,\phi(y)\intd \sigma_y\,.
\end{align*}

We know that if $E$ has a $\cC^2$-boundary, then $\KdelOKc$ is compact. We also have the following standard result (see \cite[Lemma 7.4]{LPTSA}):
\begin{lemma}
	The operator $ \tfrac{1}{2} \calI+\KdelOKc: \Heu^{1/2}(\del E)\rightarrow\Heu^{1/2}(\del E)$ is invertible.
\end{lemma}

\subsection{Neumann Functions and their Remainders}
Let $k_{E, \min,\Laplace}$ be the first non-zero eigenvalue of the operator $-\!\Laplace$ with Neumann conditions on the boundary $\del E$ of an open bounded domain $E\subset\RR^2$. We now can define the following functions:
\begin{definition}\label{def:NEK}
   Let $E\subset \RR^2$ be a open, bounded, simply connected $\cC^2$-domain. Let $0\neq |\boldk|< k_{E, \min, \Laplace}/2$. Then we define $\NEK: E\times \overline{E}\setminus\{ (z,x)\in\RR^2\times\RR^2\MID z= x\}\rightarrow\CC$ and $\NdelEK: \del E\times \overline{E}\setminus\{ (z,x)\in\RR^2\times\RR^2\MID z= x\}\rightarrow \CC$ as
   \begin{align*}
      \NEK(z,x) &\DEF \frac{1}{2\pi}\log|z-x|+\frac{1/|E|}{k^2}+\REK(z,x),\\
      \NdelEK(z,x) &\DEF \frac{1}{\pi}\log|z-x|+\frac{1/|E|}{k^2}+\RdelEK(z,x),
   \end{align*}
   where $|E|$ denotes the area of $E$ and $\NEK$ and $\NdelEK$ are solutions to
   \begin{align*}
      \left\{ 
	     \begin{aligned}
	        (\Laplace_x+k^2)\NEK(z,x) &=\delta_\boldnull(z-x) &&\quad\text{for}\quad x\in E,\\
	        \del_{\nu_x}\NEK(z,x) &= 0 &&\quad\text{for}\quad x\in\del E,
	     \end{aligned}
	  \right.
   \end{align*}
   where $z\in E$ is fixed and where $\nu_x$ denotes the outside normal on $\del E$ with respect to $x$, and
   \begin{align*}
      \left\{ 
	     \begin{aligned}
	        (\Laplace_x+k^2)\NdelEK(z,x) &=\delta_\boldnull(z-x)  &&\quad\text{for}\quad x\in E,\\
	        \del_{\nu_x}\NdelEK(z,x) &=0 &&\quad\text{for}\quad x\in\del E,
	        %-\left(\frac{1}{\pi}\frac{\nu_x\cdot(x-z)}{|z-x|^2}\right)
	     \end{aligned}
	  \right.
   \end{align*}
   where $z\in \del E$ is fixed.
\end{definition}
With this definition the function $\NEK$ is a solution to the Helmholtz equation with $\delta_\boldnull(z-x)$ on the right-hand side and has a vanishing normal derivative on the boundary. The same is valid for the function $\NdelEK$, although the source point $z$ is on the boundary. The remainders $\REK(z,\cdot)$ and $\RdelEK(z,\cdot)$ are smooth enough. We refer to \cite[Chapter 2.3.5.]{LPTSA}.
\begin{definition}\label{def:NOKp}
   Let $q_1,q_2\in\RR^2$ and let $E\Subset \{ y\in\RR^2\MID |y_1|<q_1/2\;\text{and}\; 0<y_2<q_2 \}$ be a open, bounded, simply connected $\cC^2$-domain. Let $0\neq |\boldk|< k_{E, \min, \Laplace}/2$. We define $\Omega\DEF \bigcup_{n\in\ZZ} E+(nq_1,0)^\TransT$. Then we define
   \begin{align*}
      \NOKp: &\;\{ (z,x)\in(\RR^2_+\setminus\overline{\Omega})\times (\RR^2_+\setminus\Omega) \MID \neg(\exists n\in\ZZ: z_1-x_1=nq_1 \wedge z_2=x_2)\}\rightarrow\CC,\\
      \NdelOKp: &\;\{ (z,x)\in\del\Omega\times (\RR^2_+\setminus\Omega)\MID \neg(\exists n\in\ZZ: z_1-x_1=nq_1 \wedge z_2=x_2)\}\rightarrow\CC,
%      \NOKc: &\;\{ (z,x)\in(\RR^2_+\setminus\overline{\Omega})\times (\RR^2_+\setminus\Omega) \MID \neg(\exists n\in\ZZ: z_1-x_1=nq_1 \wedge z_2=x_2)\}\rightarrow\RR,\\
%      \NdelOKc: &\;\{ (z,x)\in\del\Omega\times (\RR^2_+\setminus\Omega)\MID \neg(\exists n\in\ZZ: z_1-x_1=nq_1 \wedge z_2=x_2)\}\rightarrow\RR,
   \end{align*}    
   as
   \begin{align*}
      \NOKp(z,x) &\DEF \GKp(z,x) + \ROKp(z,x),\\
      \NdelOKp(z,x) &\DEF 2\GKp(z,x)+\RdelOKp(z,x),
%%%%%      \NOKc(z,x) &\DEF \GKc(z,x) + \ROKc(z,x),\\
%%%%%      \NdelOKc(z,x) &\DEF 2\GKc(z,x)+\RdelOKc(z,x),
   \end{align*}
%  where $\ROKp$, $\RdelOKp$, $\ROKc$ and $\RdelOKc$ are solutions to
   with $\ROKp$, $\RdelOKp$ being solutions to
   \begin{align*}
      \left\{ 
	     \begin{aligned}
	        (\Laplace_x+k^2)\ROKp(z,x) &=0 &&\quad\text{for}\quad x\in \RR^2_+\setminus\overline{\Omega},\\
	        \del_{\nu_x}\ROKp(z,x) &=-\del_{\nu_x}\GKp(z,x) &&\quad\text{for}\quad x\in\del \Omega,\\
	        \ROKp(z,x)&=0 &&\quad\text{for}\quad x\in\del \RR^2_+,
	     \end{aligned}
	  \right.
   \end{align*}
   where $z\in \RR^2_+\setminus\overline{\Omega}$ is fixed, where $\nu_x$ denotes the outside normal on $\del \Omega$ with respect to $x$, and
   \begin{align*}
      \left\{ 
	     \begin{aligned}
	        (\Laplace_x+k^2)\RdelOKp(z,x) &=0 &&\quad\text{for}\quad x\in \RR^2_+\setminus\overline{\Omega},\\
	        \del_{\nu_x}\RdelOKp(z,x) &=-2\del_{\nu_x}\GKp(z,x) &&\quad\text{for}\quad x\in\del \Omega,\\
	        \RdelOKp(z,x)&=0 &&\quad\text{for}\quad x\in\del \RR^2_+.
	     \end{aligned}
	  \right.
   \end{align*}
   Here, $z\in \del \Omega$ is fixed, and $\ROKp$, $\RdelOKp$ satisfy the outgoing radiation condition (see \cite{RadCond}), thus in particular
%   \begin{align}
%      \left\{ 
%	     \begin{aligned}
%	        (\Laplace_x+k^2)\ROKc(z,x) &=0 &&\quad\text{for}\quad x\in \RR^2_+\setminus\overline{\Omega},\\
%	        \del_{\nu_x}\ROKc(z,x) &=-\del\GKc(z,x) &&\quad\text{for}\quad x\in\del \Omega,
%	     \end{aligned}
%	  \right.
%   \end{align}
%   where $z\in \RR^2_+\setminus\overline{\Omega}$ is fixed, and
%   \begin{align}
%      \left\{ 
%	     \begin{aligned}
%	        (\Laplace_x+k^2)\RdelOKc(z,x) &=0 &&\quad\text{for}\quad x\in \RR^2_+\setminus\overline{\Omega},\\
%	        \del_{\nu_x}\RdelOKc(z,x) &=-2\del\GKc(z,x) &&\quad\text{for}\quad x\in\del \Omega,
%	     \end{aligned}
%	  \right.
%   \end{align}
%   where $z\in \del \Omega$ is fixed, 
   \begin{align*}
      |\del_{x_2}\ROKp(z,x)-ik_2\ROKp(z,x)|\rightarrow 0\quad\text{for}\quad x_2\rightarrow\infty.
   \end{align*}
%   \ki{Not sure about that power. However, I actually only need convergence.} 
   Analogously, we define $\NOKc$, $\NdelOKc$, $\ROKc$ and $\RdelOKc$.
\end{definition}

For the remainder functions we have the following formulas:
\begin{lemma}\label{prop:FormulaForR}
   Let $z,x\in\RR^2$, such that $z_2>q_2$ and $x_2\in\del E$ then
   \begin{align*}
      \ROKp(z,x)=-\int_{\del E}\NdelOKc(x,y)\del_{\nu_y}\GKp(z,y)\intd \sigma_y.
   \end{align*}
\end{lemma}

\begin{lemma} \label{prop:(I+K)[R]=S[G]}
   Let $z,x\in\del E$ and $k\neq 0$, then
   \begin{align*}
      \left(\frac{1}{2}\II+\KdelOKc\right)\left[\RdelOKp(z,\cdot)\right](x)=\SdelOKc\left[ -2\del_{\nu_\cdot}\GKp(z,\cdot) \right](x).
   \end{align*}
\end{lemma}

\begin{lemma}\label{prop:FormulaForRdelEK}
   Let $z,x\in\del E$ and $0\neq k< k_{E, \min, \Laplace}/2$ then
   \begin{align*}
      \left(\frac{1}{2}\II-\KEK\right)\!\left[\RdelEK(z,\cdot)\right]\!(x)
      \!=\!-\int_E \left(2\,k^2 \,\Gamma^0(z,y) +\!\frac{1}{|E|}\!\right) \!\Gamma^k(x,y)\intd y+
      \SEK\left[ 2\del_{\nu_x}\Gamma^0(z,\cdot) \right]\!(x).
   \end{align*}
\end{lemma}

\begin{lemma}\label{lemma:1=2K[1]}
	For $z\in\del E$ we have that
	\begin{align} \label{equ:1=2K[1]}
    	\int_E \left(\frac{1}{|E|}\right) \frac{1}{2\pi}\log(k)\intd y=2\int_{\del E}\,\del_{\nu_y}\Gamma^0(z,y)\,\frac{1}{2\pi}\log(k)\intd \sigma_y\,.
\end{align}
\end{lemma}

\section{One Periodically Arranged Helmholtz Resonator}\label{ch1HR}
In this section, we look at a bounded, connected, domain $D$, which has height $h$. Additionally, $D$ has a gap $\Lambda$ at its boundary, which allows the incident wave $\Uknull$ to pass through. The incident wave rebounds inside $D$ and leaves at the gap $\Lambda$, which then leads to the scattered wave $\Uks$. We repeat the geometry with periodicity $p$ along the $x_1$-axis and scale it by a factor $\delta$.

We look for an accurate approximation of the resonance as well as the scattered wave in the far-field. We will see that, this approximation satisfies the Helmholtz equation with a Robin boundary condition at the $x_1$-axis which approximates a Dirichlet boundary or a Neumann boundary depending on the magnitude of the incoming wave vector $\boldk$ and $\delta$.
\subsection{Mathematical Description of the Physical Problem}
\subsubsection{Geometry}
\begin{figure}[h]
  \begin{subfigure}{0.49\textwidth}
    \centering
    \includegraphics[width=0.99\textwidth]{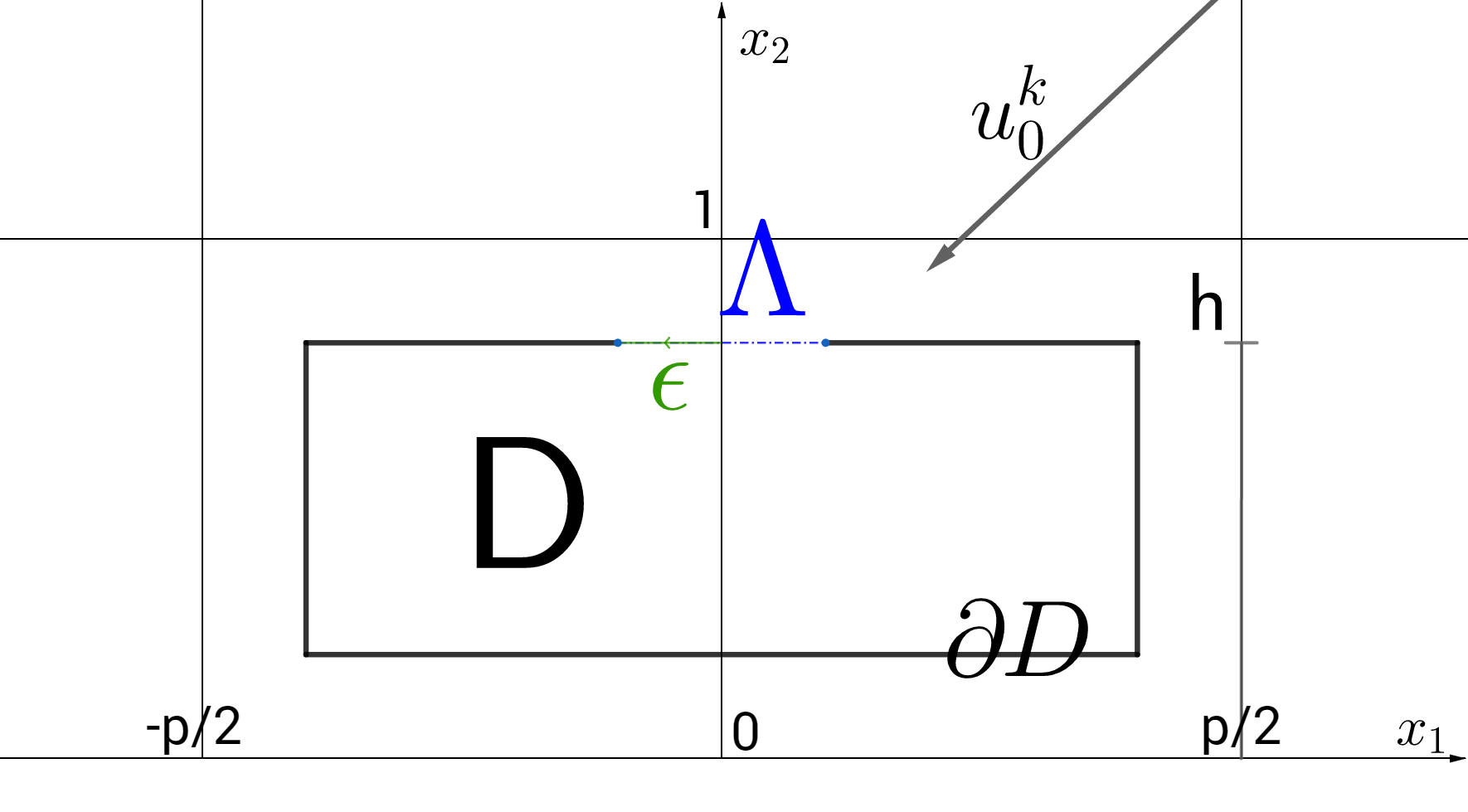}
    \caption[NonPeriodic1HR]{Microscopic, non-periodic view of our Helmholtz resonator. The Helmholtz resonator is contained in the unit cell $(-p/2,p/2)\times(0,1)$. It has a gap $\Lambda$ of length $2\eps$, which is parallel to the $x_1$ axis and centered in $(0,h)^\TransT$, where $h\in (0,1)$. $u_0^k$ denotes the incident wave. $D$ has not to be rectangular in shape.}
    \label{fig:NonPeriodic1HR}
  \end{subfigure}\hfill %half fill
  \begin{subfigure}{0.49\textwidth} 
    \centering
    \includegraphics[width=0.99\textwidth]{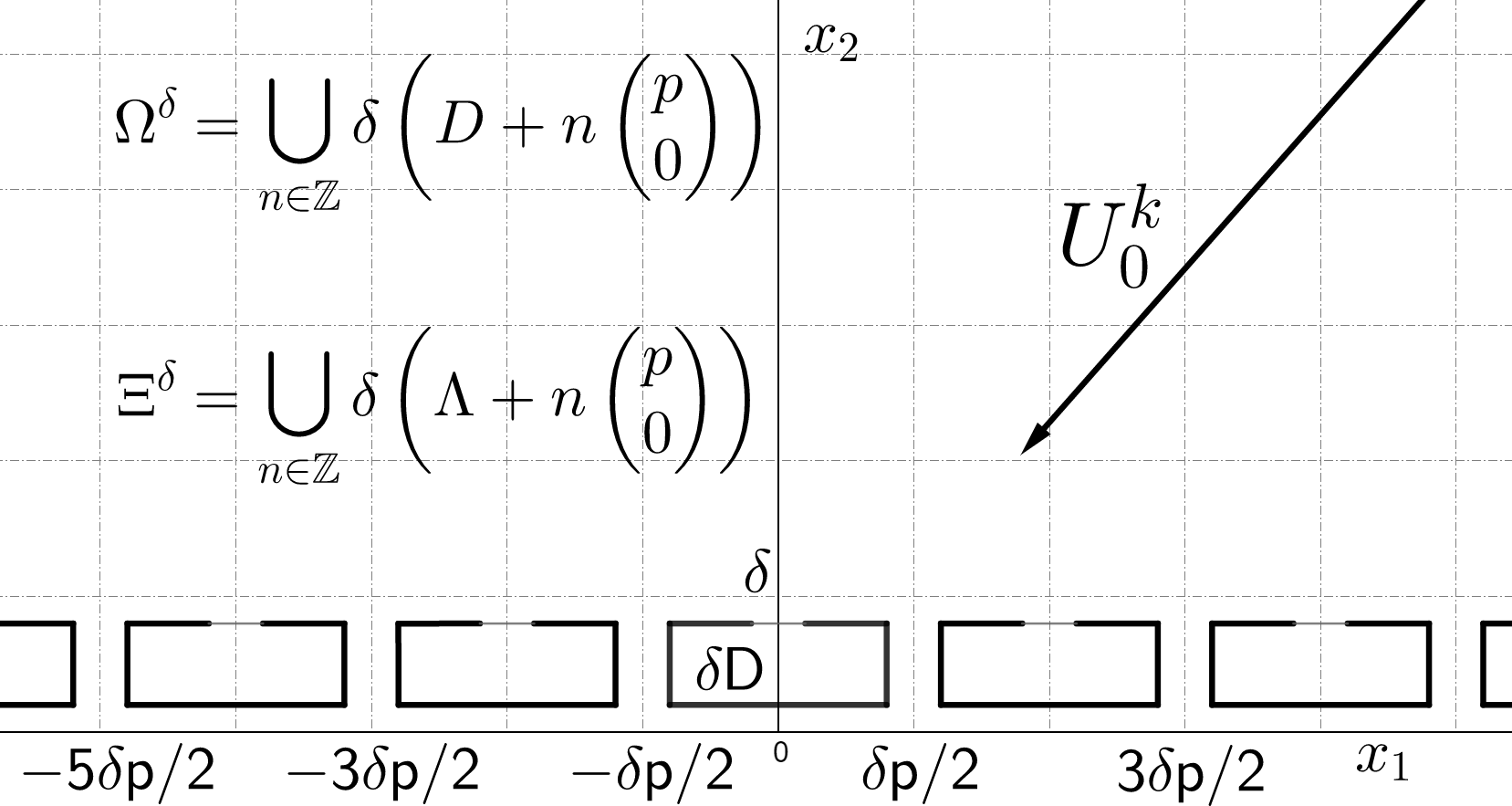}
    \caption[Periodic1HR]{Macroscopic view of our periodically arranged Helmholtz resonators, with periodicity $\delta p$. All Helmholtz resonators have the form of the Helmholtz resonator depicted in $(a)$, but are scaled with the factor $\delta$. $U_0^k$ denotes the incident wave. $\Omega^\delta$ is the collection of all Helmholtz resonators and $\Xi^\delta$ is the collection of all gaps.}
    \label{fig:Periodic1HR}
  \end{subfigure}
  \caption{The physical setup. In $\mathrm{(a)}$ we have the microscopic, non-periodic view. In $\mathrm{(b)}$ we have the macroscopic, periodic view.}
\end{figure}

Before we consider the periodic and macroscopic problem, we first define the geometry of our Helmholtz resonator in the unit cell. Let $D\Subset (-p/2,p/2)\times(0,1)$ be a open, bounded, simply connected and connected domain, where $p\in\RR$ and $p$ is close enough to $1$. For sake of simplicity we assume that $D$ is a $\cC^2$-domain. We define $\Lambda\subset \del D$ to be the gap of $D$, where $\Lambda$ is a line segment parallel to the $x_1$-axis. $\Lambda$ is centered at $(0,h)^\TransT$, where $h\in (0,1)$ is the height of $\Lambda$, and $\Lambda$ has length $2\eps$, where $\eps\in(0,1)$ and it is small enough. To facilitate future computations we assume that $h/p\geq 1/2$. 

Let us define the macroscopic view. We define the collection of periodically arranged Helmholtz resonators $\Omega^\delta$, with period $\delta p$, and the collection of gaps of those Helmholtz resonators $\Xi^\delta$, where a single gap has length $2\delta\eps$, as
\begin{align*}
   \Omega^\delta&\DEF\bigcup_{n\in\mathbb{Z}}\delta\left(D+n\begin{pmatrix}p\\ 0 \end{pmatrix}\right),\\
   \Xi^\delta&\DEF\bigcup_{n\in\mathbb{Z}}\delta\left(\Lambda+n\begin{pmatrix}p\\ 0 \end{pmatrix}\right).
\end{align*} 

\subsubsection{Incident Wave}
Let $\boldk\DEF(k_1,k_2)^\TransT\in\RR^2$ be the wave vector. We will fix the direction of the wave vector, that is, $k_1/k$ and $k_2/k$, where $k\DEF|\boldk|\DEF(k_1^2+k_2^2)^{1/2}\in [0,\infty)$, but let the magnitude $k$ vary. With that, we define the function $\Uknull:\RR^2\rightarrow \CC$ as
\begin{align*}
   \Uknull(x)\DEF a_0 e^{-ik_1x_1}e^{-ik_2x_2},
\end{align*}
where $a_0\in \RR$ denotes the amplitude. $\Uknull$ will be our incident wave. 

From (\ref{parity}), it follows that
\begin{align*}
   \Uknull\circ \opP(x)= a_0 e^{-ik_1x_1}e^{ik_2x_2},
\end{align*}
and 
\begin{align*}
   (\Uknull-\UknullcircP)(x)=-2ia_0e^{-ik_1x_1}\sin(k_2x_2).
\end{align*}
We will also need the following equation
\begin{align*}
   \nabla (\Uknull-\UknullcircP)(x)=\begin{pmatrix}-2a_0k_1e^{-ik_1x_1}\sin(k_2x_2)\\-2ia_0k_2e^{-ik_1x_1}\cos(k_2x_2)\end{pmatrix}.
\end{align*}
Consider also that $\Uknull$ and $\UknullcircP$ are quasi-periodic with quasi-momentum $-k_1p$, that is, 
\begin{align*}
   \Uknull\left( x+\begin{pmatrix}p\\0\end{pmatrix} \right)&=e^{-ik_1p}\Uknull(x),\\
   \UknullcircP\left( x+\begin{pmatrix}p\\0\end{pmatrix} \right)&=e^{-ik_1p}\UknullcircP(x).
\end{align*}

\subsubsection{The Resulting Wave}
With the geometry and the incident wave, we model the electromagnetic scattering problem and the resulting wave $\Uk:\RR^2_+\setminus\del\Omega^\delta\rightarrow\CC$ by the following system of equations:
\begin{equation} \label{equ:ch2:ThePDEforuk}
	\left\{ 
	\begin{aligned}
		 \left( \Laplace + k^2  \right) \Uk &= 0 \quad &&\text{in} \quad \RR_+^2\setminus \del\Om^\delta, \\
		 \Uk \MID_+\! - \Uk\MID_- &= 0 \quad &&\text{on} \quad \Xi^\delta ,\\
		 \del_\nu \Uk \MID_+\! - \del_\nu \Uk\mid_- &= 0 \quad &&\text{on} \quad \Xi^\delta ,\\
		 \del_\nu \Uk \MID_+\! &= 0 \quad &&\text{on} \quad \del\Om^\delta\setminus\Xi^\delta ,\\
		 \del_\nu \Uk \MID_-\! &= 0 \quad &&\text{on} \quad \del\Om^\delta\setminus\Xi^\delta ,\\
		 \Uk&=0 \quad &&\text{on} \quad \del \RR_+^2,
	\end{aligned}
	\right.
\end{equation}
where $\cdot\!\mid_+$ denotes the limit from outside of $\Omega^\delta$ and $\cdot\!\mid_-$ denotes the limit from inside of $\Omega^\delta$, and $\del_\nu$ denotes the normal derivative on $\del\Om^\delta$. Similar to diffraction problems for gratings, the above system of equations is complemented by a certain outgoing radiation condition imposed on the scattered field $\Uks\DEF \Uk-(\Uknull-\UknullcircP)$ and quasi-periodicity on $\Uk$. More precisely, we are interested in the quasi-periodic solutions, that is,
\begin{align*}
   \Uk\left( x+ \begin{pmatrix}p\\0\end{pmatrix} \right) =e^{-ik_1p}\Uk(x) \quad\text{for}\quad x\in \RR_+^2,
\end{align*}
and solutions satisfying the outgoing radiation condition, thus we have
\begin{align*}
   \left| \del_{x_2} \Uks -ik_2\Uks \right|\rightarrow 0\quad\text{for}\quad x_2\rightarrow\infty.
\end{align*}

Then the outgoing radiation condition can be imposed by assuming that all the modes in the Rayleigh-Bloch expansion are either decaying exponentially or propagating along the $x_2$-direction. Since in our case we assume that the period of the resonator structure $\delta p$ is much smaller than $k$, the outgoing radiation condition takes the following specific form:
\begin{align*}
   (\Uk-\Uknull)(x) &\sim ae^{-ik_1x_1}e^{ik_2x_2}\quad \text{as}\quad x_2\rightarrow\infty,\\
   \Uks(x) &\sim (a+1)e^{-ik_1x_1}e^{ik_2x_2}\quad \text{as}\quad x_2\rightarrow\infty,
\end{align*}
for some constant amplitude $a\in\RR$.

As a remark, in the general case where $\Uknull$ is a superposition of plane-waves, we can decompose $\Uknull$ using Bloch-Flocquet theory \cite{MR0493421, LPTSA}. We obtain a family of problems to solve, each one with its own outgoing radiation condition. The final solution is then the superposition of all these solutions.

Consider also that in absence of Helmholtz resonators the solution to (\ref{equ:ch2:ThePDEforuk}) is given by $\Uknull-\UknullcircP$.

\subsubsection{The Resulting Wave in the Microscopic View}\label{sec:reswaveinmicroscopicview}
Given the resulting wave $\Uk(x)$, the function $\udk(x): \RR^2_+\setminus\del\Omega^1\rightarrow\CC$, $\udk(x)\DEF \Uk(\delta x)$ represents the resulting wave, but where the Helmholtz resonators are scaled-back and thus are of height $h$ and not $\delta h$. $\udk$ satisfies
\begin{equation} \label{equ:ch2:ThePDEfortudk}
	\left\{ 
	\begin{aligned}
		 \left( \Laplace + (\delta k)^2  \right) \udk &= 0 \quad &&\text{in} \quad \RR_+^2\setminus \del\Om^1, \\
		 \udk \MID_+\! - \udk\mid_- &= 0 \quad &&\text{on} \quad \Xi^1 ,\\
		 \del_\nu \udk \MID_+\! - \del_\nu \udk\mid_- &= 0 \quad &&\text{on} \quad \Xi^1 ,\\
		 \del_\nu \udk \MID_+\! &= 0 \quad &&\text{on} \quad \del\Om^1\setminus\Xi^1 ,\\
		 \del_\nu \udk \MID_-\! &= 0 \quad &&\text{on} \quad \del\Om^1\setminus\Xi^1 ,\\
		 \udk&=0 \quad &&\text{on} \quad \del \RR_+^2,
	\end{aligned}
	\right.
\end{equation}
where $\cdot\!\mid_+$ denotes the limit from outside of $\Omega^1$ and $\cdot\!\mid_-$ denotes the limit from inside of $\Omega^1$, and $\del_\nu$ denotes the normal derivative on $\del\Om^1$.

We can adopt the quasi-periodicity from the macroscopic view and obtain
\begin{align*}
   \udk\left( x+ \begin{pmatrix}p\\0\end{pmatrix} \right) =e^{-i\delta k_1p}\udk(x) \quad\text{for}\quad x\in \RR_+^2.
\end{align*}
Defining $\udks\DEF \udk-(\udknull-\udknull\!\circ\! \opP)$ we also get that
\begin{align*}
   \left| \del_{x_2} \udks -i\delta k_2\udks \right|\rightarrow 0\quad\text{for}\quad x_2\rightarrow\infty.
\end{align*} 

We see that $\udk$ solves the same partial differential equation like $\Uk$ in the rescaled geometry, but with the scaled wave vector $\delta k$. We will see that we can express $\udk$ as an expansion in terms of $\delta$ and we will give an analytic expression for the first order term.

\subsection{Main Results}
We assume that $\delta k\in \kkkk \DEF \{ \hat{k} \in \RR \MID  0\neq |\hat{k}|< k_{D, \min, \Laplace}/2 \text{ and } |\hat{k}|^2<\inf\{|l-\hat{k}\,e_1|^2\MID l\DEF 2\pi n/p,\; n\in\ZZ\setminus\{0\}\}\}$, where $k_{D, \min, \Laplace}$ is defined as the first non-zero eigenvalue of the operator $-\!\Laplace$ with Neumann conditions on the boundary $\del D$ . If we would extend the domain $\kkkk$ to $\kkkkc\DEF\{k_\ast\in \CC\MID \sqrt{k_\ast \bar{k}_\ast}< k_{D,\min,\Laplace}/2\}$, we would obtain following resonance values for our physical problem:
\begin{theorem}\label{THM1:1HR}
We have exactly two resonance values in $\kkkkc$ for $\Uk$. These are
	\begin{align*}
		k^{\delta,\eps}_+ =&\, k^{\delta,\eps}_0+ \frac{\alpha_0\,|D|}{2\delta}\,\ceps^{3/2}+\frac{\alpha_1\,|D|}{2\delta}\,\ceps^2+\dfrac{1}{\delta}\OO\left(\ceps^{5/2}\right)\,,\\
		k^{\delta,\eps}_- =-&\, k^{\delta,\eps}_0- \frac{\alpha_0\,|D|}{2\delta}\,\ceps^{3/2}+\frac{\alpha_1\,|D|}{2\delta}\,\ceps^2+\dfrac{1}{\delta}\OO\left(\ceps^{5/2}\right)\,,\\
		&\delta k^{\delta,\eps}_0 \DEF \ceps^{1/2}\DEF\sqrt{\frac{-\pi}{2|D|\log{(\eps/2)}}}\,,
	\end{align*}
	where $\alpha_0$ and $\alpha_1$ are given by
\begin{align}
	\alpha_0 \DEF&  \Reu_{\del D}^0 \label{equdef:alpha0}
      		\left(\!\!
      			\begin{pmatrix*} 0 \\ h \end{pmatrix*}, 
      			\begin{pmatrix*} 0 \\ h \end{pmatrix*}
      		\!\!\right)\!+\!
      		\Reu_{\del\Omega^1,+}^0
      		\left(\!\!
      			\begin{pmatrix*} 0 \\ h \end{pmatrix*}, 
      			\begin{pmatrix*} 0 \\ h \end{pmatrix*}
      		\!\!\right) 
      		\!+\! \frac{1}{\pi}\log\bigg(\frac{\pi}{p}\bigg)
      		\!-\! \frac{1}{\pi}\log\left(\sinh\left(\frac{\pi}{p}2\,h\right)\right)\,,\\
	\alpha_1\DEF& \del_{\delta k}\Reu_{\del D}^0 \label{equdef:alpha1}
      		\left(\!\!
      			\begin{pmatrix*} 0 \\ h \end{pmatrix*}, 
      			\begin{pmatrix*} 0 \\ h \end{pmatrix*}
      		\!\!\right) +
      		\del_{\delta k} \Reu_{\del\Omega^1,+}^0
      		\left(\!\!
      			\begin{pmatrix*} 0 \\ h \end{pmatrix*}, 
      			\begin{pmatrix*} 0 \\ h \end{pmatrix*}
      		\!\!\right)\,.
\end{align}	

\end{theorem}

We have the following approximation for the resulting wave $\Uk$:
\begin{theorem}\label{THM2:1HR}
	Let $V_r\DEF\{ z\in\RR^2_+\MID z_2>r \}$. There exist constants $C_{(\ref{Thm2:1HR})},\widetilde{C}_{(\ref{Thm2:1HR})}>0$ such that
	\begin{align}\label{Thm2:1HR}
		&			\NORM{\Uks\!-\!(U^k_{\Seu_p^\star}\!+\!U^k_{\Teu_p^\star}\!+\!U^k_{\mathrm{RHS},p})}_{\Leu^\infty(V_r)}\!\!
					+\!\NORM{\nabla\left[\Uks\!-\!(U^k_{\Seu_p^\star}\!+\!U^k_{\Teu_p^\star}\!+\!U^k_{\mathrm{RHS},p})\right]}_{\Leu^\infty(V_r)}\\
		&\quad		\leq C_{(\ref{Thm2:1HR})}\left(\!\frac{\delta}{|\log(\eps)|^3}\resSumWOfirstfrac\!+\!\frac{\delta}{|\log(\eps)|^2}\!+\! \delta\eps +\delta\,e^{-\widetilde{C}_{(\ref{Thm2:1HR})}r}+\delta^2\right)\,,\nonumber
	\end{align}
	for $\delta$, $\eps$ small enough and $r$ large enough, where 
	\begin{align*}
		\Uk_{\Seu_p^\star}(z)
			=\,& 	-e^{i(k_2 z_2-k_1 z_1)}\frac{h}{p}\int_\Lambda \mudkstar(y)\intd \sigma_y\,,\\
		\Uk_{\Teu_p^\star}(z)
			=\,&	 	e^{i(k_2 z_2-k_1 z_1)}\frac{1}{p}\int_\Lambda\int_{\del D}\mudkstar(y)\NdelOonedkc(y,w)\nu_w\!\cdot\!\begin{pmatrix}0\\1\end{pmatrix}\,\intd \sigma_w\intd \sigma_y\,,\\
		\Uk_{\mathrm{RHS},p}(z)
			=\,& 		e^{i\delta (k_2z_2\!-\!k_1z_1)}\left[\int_{\del D} \del_\nu (\udknull-\udknull\circ \opP)(y)\frac{\sin(\delta k_2y_2)\,e^{i\delta k_1y_1}}{\delta k_2 p}\intd \sigma_y \right.\\
			&\mkern-50mu	\left.-\int_{\del D}\int_{\del D}\del_\nu (u^{\delta k}_0-u^{\delta k}_0\!\!\circ\! \opP)(y) \NdelOonedkc(y,w)\nu_{w}\cdot\begin{pmatrix}\frac{i\,k_1}{p\,k_2}\sin(\delta k_2 w_2)\\ \frac{1}{p}\cos(\delta k_2w_2)\end{pmatrix}e^{i\delta k_1 w_1}\intd \sigma_w\intd \sigma_y \right]\nonumber\,.
	\end{align*}
	and where, for $y\in\Lambda$, we have
	\begin{align*}
		\mudkstar(y)\!=\!\frac{1}{\sqrt{\eps^2-y_1^2}}\left(
		\frac{-\pi}{4|D|(\log(\eps/2))^2}\frac{f^{\delta k}(0)}{\dkdep \!-\! \dkdem}\resSumWOfirstfracSHORT\!+\!\frac{f^{\delta k}(0)}{2\log(\eps/2)}\right)\,,
	\end{align*}
	with the constant $f^{\delta k}(0)\in\RR$ being given by
	\begin{align*}
		f^{\delta k}(0) =
			2ia_0\sin(\delta k_2 h)-2a_0\delta\int_{\del D} \nu_y \cdot \begin{pmatrix}k_1e^{-i\delta k_1y_1}\sin(\delta k_2 y_2) \\ ik_2e^{-i\delta k_1y_1}\cos(\delta k_2 y_2)\end{pmatrix}\NdelOonedkp\left(\!\begin{pmatrix}0\\h\end{pmatrix},y\right)\intd \sigma_y\,.
	\end{align*}
\end{theorem}

We see from Theorem \ref{THM2:1HR} that the function
\begin{align*}
	\Ukapp(z)
		\DEF& (\Uknull-\UknullcircP)(z)+\Uk_{\Seu_p^\star}(z)+\Uk_{\Teu_p^\star}(z) +\Uk_{\mathrm{RHS},p}(z)\,,
\end{align*}
gives an accurate approximation of $\Uk$ in the far-field. Moreover, it satisfies the 
Helmholtz equation in $\RR^2_+$ with the boundary condition %HACK
\begin{align*}
	\Ukapp(z)+\delta\cIBC\del_{z_2}\Ukapp(z)=0\,,\quad\text{ for } z\in\del\RR^2_+\,,
\end{align*}
and $\Ukapp-\Uknull$ satisfies the outgoing radiation condition. The boundary condition is called 'Impedance Boundary Condition' and for $\cIBC=0$ it yields a Dirichlet boundary condition while for $\cIBC \gg 1$ it approximates a Neumann boundary condition for $\Ukapp$. Using Theorem \ref{THM2:1HR}, we can express $\cIBC$ as follows.
\begin{theorem}\label{THM3:1HR}
	The constant in the impedance boundary condition is given by
	\begin{align*}
		\cIBC = \frac{1}{2ia_0\delta k_2}\,C_{(\ref{Thm3:HR1})}^{\delta k}+\OO(\delta)\,,
	\end{align*}
	where $C_{(\ref{Thm3:HR1})}^{\delta k}\in\RR$ is defined as
	\begin{align}\label{Thm3:HR1}
		C_{(\ref{Thm3:HR1})}^{\delta k} = 
			&		-\frac{h}{p}\int_\Lambda \mudkstar(y)\intd \sigma_y
					+\frac{1}{p}\int_\Lambda\int_{\del D}\mudkstar(y)\NdelOonedkc(y,w)\nu_w\!\cdot\!\begin{pmatrix}0\\1\end{pmatrix}\,\intd \sigma_w\intd \sigma_y \nonumber\\
			&		+\int_{\del D} \del_\nu (\udknull-\udknull\circ \opP)(y)\frac{\sin(\delta k_2y_2)\,e^{i\delta k_1y_1}}{\delta k_2 p}\intd \sigma_y \\
			&		-\int_{\del D}\int_{\del D}\del_\nu (u^{\delta k}_0-u^{\delta k}_0\!\!\circ\! \opP)(y) \NdelOonedkc(y,w)\nu_{w}\cdot\begin{pmatrix}\frac{i\,k_1}{p\,k_2}\sin(\delta k_2 w_2)\\ \frac{1}{p}\cos(\delta k_2w_2)\end{pmatrix}e^{i\delta k_1 w_1}\intd \sigma_w\intd \sigma_y \nonumber\,.
	\end{align}
\end{theorem}

%\section{Expansion of the Resulting Wave in Terms of Delta}
\subsection{Proof of the Main Results}

We want to proof Theorems \ref{THM1:1HR} -- \ref{THM3:1HR}. First, we express the resulting wave outside the Helmholtz resonators and the resulting wave inside  the Helmholtz resonators through operators acting on the resulting wave, but restricted on the gap. This leads us to a condition with the linear operator $\AAdnoboldke$, whose solution is the resulting wave on the gap up to a term of order $\delta^2$. We solve this linear system based on the procedure given in \cite{HaiHabib}. We will see that it is solvable for a complex wave vector near $0$ except in three points, two of which are the resonances of our system. With this we obtain the resulting wave on the gap. Then we recover the resulting wave outside the resonators up to a term of order $\delta^2$. We will see, that we can split the resulting wave into a propagating wave and an evanescent one. The propagating one leads us to the impedance boundary condition constant $\cIBC$.

\subsubsection{Collapsing the Wave-Informations on the Gap}
Let us consider the resulting wave $\udk$ in the microscopic view described in Subsection \ref{sec:reswaveinmicroscopicview}. We will keep the microscopic view until Subsection \ref{subsec:IBC}. We define the main strip $Y\DEF\{ y\in\RR^2_+\MID |y_1|<p/2\}$. $D$ is the Helmholtz resonator on that strip and $\Lambda$ the gap on $\del D$. Furthermore, we fix $k_1/k\FED e_1$ and $k_2/k\FED e_2$ and assume that $\delta k\in \kkkk \DEF \{ \hat{k} \in \RR \MID  0\neq |\hat{k}|< k_{D, \min, \Laplace}/2 \text{ and } |\hat{k}|^2<\inf\{|l-\hat{k}\,e_1|^2\MID l\DEF 2\pi n/p,\; n\in\ZZ\setminus\{0\}\}\}$. Consider that $\udk$ is continuous on the gap $\Lambda$, thus $\udk(z)$ is well-defined for $z\in\Lambda$.
\begin{proposition}\label{prop:udkFormulaOnD}
   Let $\NDK$ be as in Definition \ref{def:NEK}. Let $z\in  D$, then we have 
   \begin{align*}
      \udk(z)=-\int_\Lambda\del_\nu \udk(y)\NDdK(z,y)\intd \sigma_y\,.
   \end{align*}
   Let $z\in\Lambda$. Then we have 
   \begin{align}\label{equ:udkFormulaOndelD}
      \udk(z)=-\int_\Lambda\del_\nu \udk(y)\NdelDdK(z,y)\intd \sigma_y\,.
   \end{align}
\end{proposition}

\begin{proposition}\label{prop:udkFormulaOnComplementOfOmega}
   Let $\NOkp$, $\NOkc$ be as in Definition \ref{def:NOKp}. Let $z\in  Y\setminus\overline{D}$. Then it follows that
   \begin{align*}
      \udks(z)=\int_\Lambda\del_\nu \udk(y)\NOonedkp(z,y)\intd \sigma_y - \int_{\del D} \del_\nu (\udknull-\udknull\circ \opP)(y)\NOonedkp(z,y)\intd \sigma_y\,.
   \end{align*}
   Let $z\in\Lambda$. Then we have 
   \begin{align}\label{equ:udkFormulaOnDelOmega}
      \udks(z)=\int_\Lambda\del_\nu \udk(y)\NdelOonedkp(z,y)\intd \sigma_y \!-\! \int_{\del D} \del_\nu (\udknull-\udknull\circ \opP)(y)\NdelOonedkp(z,y)\intd \sigma_y\,.
   \end{align}
\end{proposition}

Using that $\udk$ is continuous on the gap we can deduce from the following proposition a necessary condition for $\del_\nu\udk\MID_\Lambda$. Assume we can obtain a solution $\del_\nu\udk$ from that condition, then from Propositions \ref{prop:udkFormulaOnD} and \ref{prop:udkFormulaOnComplementOfOmega} we can recover the resulting wave on $Y$.
\begin{proposition}[Gap Formula]\label{prop:CompressedInfosOnGap}
   Let $z\in \Lambda$ then
   \begin{multline}\label{equ:CompressedInfosOnGap}
      \int_\Lambda \del_\nu \udk(y)\left( \NdelOonedkp(z,y) +\NdelDdK(z,y) \right) \intd \sigma_y = \\ 
      \int_{\del D} \del_\nu (\udknull-\udknull\circ \opP)(y)\NdelOonedkp(z,y)\intd \sigma_y-(\udknull-\udknull\circ \opP)(z).
   \end{multline}
\end{proposition}

Consider that the right-hand side in (\ref{equ:CompressedInfosOnGap}) does not depend on $\udk$ and it is computable.

\begin{proof}[Proposition \ref{prop:udkFormulaOnD}]
   Let us look at  (\ref{equ:udkFormulaOndelD}) first. Let $z\in\del D$ then using Green's formula with $\left( \Laplace + k^2 \right)\udk=0$  we have
   \begin{align*}
      \udk(z)=&\int_D \udk(y)\left( \Laplace_y + k^2 \right)\NdelDdK(z,y)\intd y \\
      =& \int_{\del D} \udk(y)\del_{\nu_y}\NdelDdK(z,y)\intd \sigma_y - \int_{\del D} \del_\nu\udk(y)\NdelDdK(z,y)\intd \sigma_y\,.
   \end{align*}
   Using that $\del_{\nu_y}\NdelDdK(z,y)=0$ on $\del D$ and $\del_\nu\udk(y)=0$ on $\del D\setminus\Lambda$ we obtain the desired equation.
   We get the other equation analogously.
\end{proof}

\begin{proof}[Proposition \ref{prop:udkFormulaOnComplementOfOmega}]
   Let us first look at  (\ref{equ:udkFormulaOnDelOmega}).
   Let $r>0$ and $U_r\DEF \{ y\in \RR^2\setminus\overline{D}\;\MID |y_1|<p /2 \wedge 0<y_2<r \}$, $\del U_{r,0}\DEF \{ y\in \del U_r\,\MID y_2=0 \}$, $\del U_{r,-}\DEF \{ y\in \del U_r\,\MID y_1=-p/2 \}$, $\del U_{r,+}\DEF \{ y\in \del U_r\,\MID y_1=+p/2 \}$ and $\del U_{r,r}\DEF \{ y\in \del U_r\,\MID y_2=r \}$.
   Then
   \begin{align*}
      \udks(z)=\lim_{r\rightarrow\infty}\int_{U_r}\udks(y) (\Laplace_y+k^2)\NdelOonedkp(z,y)dy.
   \end{align*}
  Using Green's formula, we have
   {\setlength{\belowdisplayskip}{0pt} \setlength{\belowdisplayshortskip}{0pt}\setlength{\abovedisplayskip}{0pt} \setlength{\abovedisplayshortskip}{0pt}
   \begin{flalign}\label{equ:PFudks=-INTdeludksNeu-rhs:1}
   \udks&(z)=\lim_{r\rightarrow\infty} \bigg( \int_{U_r}(\Laplace +k^2)\udks(y)\,\NdelOonedkp(z,y)\intd y
   \end{flalign}
   \begin{align} 
   &-\int_{\del D}\udks(y)\del_{\nu_y}\NdelOonedkp(z,y)\intd \sigma_y+\int_{\del D}\del_{\nu}\udks(y)\NdelOonedkp(z,y)\intd \sigma_y \label{equ:PFudks=-INTdeludksNeu-rhs:2}\\
   &+\int_{\del U_{r,0}}\udks(y)\del_{\nu_y}\NdelOonedkp(z,y)\intd \sigma_y-\int_{\del U_{r,0}}\del_{\nu}\udks(y)\NdelOonedkp(z,y)\intd \sigma_y \label{equ:PFudks=-INTdeludksNeu-rhs:3}\\
   &+\int_{\del U_{r,-}}\udks(y)\del_{\nu_y}\NdelOonedkp(z,y)\intd \sigma_y-\int_{\del U_{r,-}}\del_{\nu}\udks(y)\NdelOonedkp(z,y)\intd \sigma_y \label{equ:PFudks=-INTdeludksNeu-rhs:4}\\
   &+\int_{\del U_{r,+}}\udks(y)\del_{\nu_y}\NdelOonedkp(z,y)\intd \sigma_y-\int_{\del U_{r,+}}\del_{\nu}\udks(y)\NdelOonedkp(z,y)\intd \sigma_y \label{equ:PFudks=-INTdeludksNeu-rhs:5}\\
   &+\int_{\del U_{r,r}}\udks(y)\del_{\nu_y}\NdelOonedkp(z,y)\intd \sigma_y-\int_{\del U_{r,r}}\del_{\nu}\udks(y)\NdelOonedkp(z,y)\intd \sigma_y \label{equ:PFudks=-INTdeludksNeu-rhs:6} \bigg).
   \end{align}}
   The right-hand side in  (\ref{equ:PFudks=-INTdeludksNeu-rhs:1}) vanishes because $\udks$ satisfies the homogeneous Helmholtz equation. The left term in (\ref{equ:PFudks=-INTdeludksNeu-rhs:2}) vanishes because $\NdelOonedkp$ has a vanishing normal derivative on $\del \Omega^1$. Both terms in  (\ref{equ:PFudks=-INTdeludksNeu-rhs:3}) vanish because of the Dirichlet boundary. The terms in  (\ref{equ:PFudks=-INTdeludksNeu-rhs:4}) and in  (\ref{equ:PFudks=-INTdeludksNeu-rhs:5}) cancel each other out because of the quasi-periodicity with quasi-momentum $-k_1p$ for $\udks$ and the quasi-periodicity with quasi-momentum $k_1p$ for $\NdelOonedkp$, together with the explicit expression for the normal on $\del U_{r,-}$, which is $(-1,0)^\TransT$, and the explicit expression for the normal on $\del U_{r,+}$, which is $(1,0)^\TransT$. Thus we are left with 
   \setlength{\belowdisplayskip}{1.5pt} \setlength{\belowdisplayshortskip}{1.5pt}\setlength{\abovedisplayskip}{1.5pt} \setlength{\abovedisplayshortskip}{1.5pt}
   \begin{multline}
      \udks(z,x)=\int_{\del D}\del_{\nu}\udks(y)\NdelOonedkp(z,y)\intd \sigma_y \\
      +\lim_{r\rightarrow\infty}\left(\int_{\del U_{r,r}}\udks(y)\del_{\nu_y}\NdelOonedkp(z,y)\intd \sigma_y-\int_{\del U_{r,r}}\del_{\nu}\udks(y)\NdelOonedkp(z,y)\intd \sigma_y\right). 
   \end{multline}
   Using that $\udks$ and $\NdelOonedkp$ satisfy the outgoing radiation condition, we can write $\NdelOonedkp(z,y)=\frac{1}{ik_2}\del_{y_2}\NdelOonedkp(z,y)+o(1)$ and $\del_{y_2}\udks(y)=ik_2\udks(y)+0(1)$ for $y_2\rightarrow\infty$. With that we can eliminate the integrals within the limes. 
   
   Finally, using that $\del_{\nu}\udk\MID_{\del D\setminus\Lambda}=0$ and the definition of $\udks$, we proved  (\ref{equ:udkFormulaOnDelOmega}).
   
   We get the other equation analogously.
\end{proof}

\begin{proof}[Proposition \ref{prop:CompressedInfosOnGap}]
   Using that $\udk$ is continuous at $\Lambda$ we have that $\udk\MIDD_+(z)-\udk\MIDD_-(z)=\udk(z)-\udk(z)=0$, for $z\in\Lambda$. From  (\ref{equ:udkFormulaOnDelOmega}) and  (\ref{equ:udkFormulaOndelD}), we obtain  (\ref{equ:CompressedInfosOnGap}).
\end{proof}

\subsubsection{Expanding the Gap Formula in Terms of $\delta$}
We define $\fdk:\Lambda\rightarrow\CC$ as the right-hand side of the Gap Formula \ref{prop:CompressedInfosOnGap}, that is, 
\begin{align}\label{equdef:fdk}
   \fdk(z)\DEF \int_{\del D} \del_\nu (\udknull-\udknull\circ \opP)(y)\NdelOonedkp(z,y)\intd \sigma_y-(\udknull-\udknull\circ \opP)(z).
\end{align}
We identify $f^{\delta k}(\tau)$ with $f^{\delta k}((\tau\,,h)^\TransT)$, for $\tau\in (-\eps\,,\eps)$.

Let us define the following operator spaces and their respective norms:
\begin{definition}\label{def:curlXe}
   Let $\mu'$ represent the distributional derivative of $\mu$. We define
   \begin{align*}
      \curlXe&\DEF \left\{ \mu\in L^2((-\eps,\eps))\middle| \int_{-\eps}^\eps \sqrt{\eps^2-t^2}|\mu(t)|^2 \intd t < \infty \right\}, \\
      \NORM{\mu}_{\curlXe} &\DEF \left( \int_{-\eps}^\eps \sqrt{\eps^2-t^2} |\mu(t)|^2\intd t \right)^{1/2},\\
      \curlYe&\DEF \left\{ \mu\in\cC^0([-\eps,\eps])\middle| \mu'\in\curlXe \right\},\\
      \NORM{\mu}_{\curlYe}&\DEF\left( \NORM{\mu}^2_{\curlXe} + \NORM{\mu'}^2_{\curlXe} \right)^{1/2}.
   \end{align*}
\end{definition}
Consider that for $\mu\in\curlXe$
\begin{align}\label{equ:intleqXenorm}
	\inteps \mu(t) \intd t = \inteps	\mu(t) \frac{\sqrt[4]{\eps^2-t^2}}{\sqrt[4]{\eps^2-t^2}}\intd t  
			\leq		\sqrt{\pi}\NORM{\mu}_\curlXe\,,
\end{align}
because of the $\Leu^2$-Cauchy-Schwarz inequality.
\begin{definition}
	Let $\mu\in\curlXe$ and $\alpha>0$. We say $\mu=\OO_\curlXe(\alpha)$ for $\alpha\rightarrow 0$ if 
	$
		\frac{\NORM{\mu}_{\curlXe}}{\alpha}
	$
	is bounded as $\alpha\rightarrow 0$.
\end{definition}
With those spaces we can define the following operators:
\begin{definition}\label{def:OperatorsFor1HR}
   The following operators are defined as functions from $\curlXe$ to $\curlYe$. Let $n \in \NN\setminus\{ 0\}$, then
   \begin{align*}
      \calL^\eps[\mu](\tau)\DEF&\inteps\mu(t)\log(|\tau-t|)\intd t\,,\\
      \KKe[\mu](\tau)\DEF&\frac{1}{|D|}\inteps\mu(t)\intd t,\\
      \calR^{\delta k,\eps}[\mu](\tau)\DEF&\inteps\mu(t)\Bigg[\RdelDdK 
      		\begin{pmatrix*} 
      			\begin{pmatrix*} \tau \\ h \end{pmatrix*}, 
      			\begin{pmatrix*} t \\ h \end{pmatrix*}
      		\end{pmatrix*} +
      		\RdelOonedkp
      		\begin{pmatrix*} 
      			\begin{pmatrix*} \tau \\ h \end{pmatrix*}, 
      			\begin{pmatrix*} t \\ h \end{pmatrix*}
      		\end{pmatrix*} \nonumber\\
      		&+ \frac{1}{\pi}\!\log\bigg(\frac{\pi}{p}\bigg)+\frac{1}{\pi}\!\log\bigg( \sinc\bigg|\frac{\pi}{p}(\tau- t)\bigg|\bigg) \nonumber\\
      		&- \frac{1}{2\pi}\log\left(\sinh\left(\frac{\pi}{p}2\,h\right)^2+\sin\left(\frac{\pi}{p}(\tau-t)\right)^2\right)\Bigg]\intd t, \\
      \calGG^{k,\eps}_{+,n}[\mu](\tau)\DEF&\inteps\mu(t)\,\Gamma^k_{+,n}
      		\begin{pmatrix*} 
      			\begin{pmatrix*} \tau \\ h \end{pmatrix*}, 
      			\begin{pmatrix*} t \\ h \end{pmatrix*}
      		\end{pmatrix*}
      		\intd t,\\
      \AAdnoboldke[\mu](\tau)\DEF& \frac{2}{\pi}\LLe[\mu](\tau)+\frac{\KKe[\mu](\tau)}{\delta^2 k^2}+\RRdnoboldke[\mu](\tau),
   \end{align*}
   where $\Gamma^k_{+,n}$ is given in Lemma \ref{prop:GammaExpansion}.
\end{definition}

Later, we will show that 
\begin{align*}
   \del_\nu \udk\MID_\Lambda = (\AAdke)^{-1}[\fdk] + \OO(\delta^2).
\end{align*}

\begin{proposition} \label{prop:GapFormulaINOperators}
   Let $2\eps<p$, let $\tau\in (-\eps,\eps)$, then
   \begin{align*}
   		\AAdke[\del_\nu \udk\MID_\Lambda](\tau)+\sum_{n=1}^\infty\delta^n\,2 \,\calGG^{k,\eps}_{+,n}[\del_\nu \udk\MID_\Lambda](\tau) = \fdk((\tau,h)^\TransT).
   \end{align*}
\end{proposition}

\begin{proof}
 	Let $z\DEF(\tau,h)^\TransT\in\Lambda$. From Proposition \ref{prop:CompressedInfosOnGap}, it follows that
 	\begin{align*}
 		\int_\Lambda \del_\nu \udk(y)\left(\! 2\Gdkp(z,y)+\RdelOonedkp(z,y) +
 		\!\frac{1}{\pi}\log|z-y|+
 		\!\frac{2/|D|}{\delta^2 k^2}\!+
 		\RdelDdK(z,y)\!\right) \intd \sigma_y 
 		= \fdk(z).
 	\end{align*}
 	Using Lemma \ref{prop:GammaExpansion}, we can rearrange the last equation and obtain
 	\begin{multline}
 		\int_\Lambda \del_\nu \udk(y)\bigg( 
 		2\Gamma_{+}^0(z,y)+
 		\frac{1}{\pi}\log|z-y|+
 		\frac{1/|D|}{\delta^2 k^2}\\ 
 		+\RdelOonedKp(z,y) +
 		\RdelDdK(z,y)+
 		\sum_{n=1}^\infty \delta^n \,2\Gamma_{+,n}^k(z,y)
 		\bigg) \intd \sigma_y = \fdk(z).
 	\end{multline}
 	Using that $\Lambda$ is a line segment parallel to the $x_1$-axis, we have that $\intd\sigma_y = \intd t$, by writing $y=(t,h)^\TransT$. Using  (\ref{equ:ch2:logsinhsin}) for $\Gamma_{+}^0(z,y)$ and using that the expansion in $\delta$ (see Lemma \ref{prop:GammaExpansion}) is uniform, we can interchange the infinite sum and the integration. Let $\mu(t)\DEF\del_\nu \udk(y)$, we have that
 	\begin{multline}\label{equ:PFGapFormulaINOperators:1}
 	  	\fdk(z)\!=\!\inteps \mu(t)\bigg[ 
 	  	\frac{1}{\pi}\log|\tau\!-\!t|+
		\!\frac{1}{2\pi}\! \log\bigg(\!\!\sin\bigg(\!\frac{\pi}{p}(\tau\!-\! t)\bigg)^2 \bigg)\!+\!
 		\RdelOonedkp(z,y) \!+\!
 		\RdelDdK(z,y) \\
 		-\!\frac{1}{2\pi}\log\bigg(\!\sinh\bigg(\frac{\pi}{p}2\,h\bigg)^2\!\!+\!\sin\bigg(\!\frac{\pi}{p}(\tau\!-\!t)\bigg)^2\bigg)\bigg] \intd t +
 		\frac{\KKe[\mu](\tau)}{\delta^2 k^2}+
 		\sum_{n=1}^\infty\delta^n \,2\calGG^{k,\eps}_{+,n}[\mu](\tau).
 	\end{multline}
 	Now consider that for $2\eps<p$, we have
 	\begin{multline}
 	 	\frac{1}{2\pi}\!\log\bigg(\sin\bigg(\frac{\pi}{p}(\tau- t)\bigg)^2 \bigg) =
 	 	\frac{1}{\pi}\!\log\bigg(\sin\bigg|\frac{\pi}{p}(\tau- t)\bigg| \bigg)\\
 	 	= \frac{1}{\pi}\!\log\bigg(\frac{\pi}{p}\bigg)+ 
 	 	\frac{1}{\pi}\!\log|\tau-t| +
 	 	\frac{1}{\pi}\!\log\bigg( \sinc\bigg|\frac{\pi}{p}(\tau- t)\bigg| \bigg).
 	\end{multline}
 	Inserting the last equation into  (\ref{equ:PFGapFormulaINOperators:1}), we obtain that
 	\begin{align*}
 		\frac{2}{\pi}\LLe[\mu](\tau) \!+\!
		\RRdke[\mu](\tau)\!+\!
		\frac{\KKe[\mu](\tau)}{\delta^2 k^2}\!+\!
 		\sum_{n=1}^\infty\delta^n \,\calGG^{k,\eps}_{+,n}[\mu](\tau)
 		=\fdk(z). 			
 	\end{align*}
 	With that we have proven Proposition \ref{prop:GapFormulaINOperators}.
\end{proof}

Let us show that $\LLe$ is bijective and let us consider its inverse. 
\begin{proposition}\label{prop:LLEisInjective}
	Let $0<\eps<2$. The operator $\LLe: \curlXe\rightarrow\curlYe$ is linear, bounded and invertible and has the inverse
	\begin{align*}
		\LLeinv[\eta](t) = -\frac{1}{\pi^2\sqrt{\eps^2-t^2}}\pvinteps\frac{\sqrt{\eps^2-\tau^2}\;\eta'(\tau)}{t-\tau}\intd \tau+\frac{C_{\calL}[\eta]}{\pi\log(\eps/2)\sqrt{\eps^2-t^2}}\,
	\end{align*}	 
	where
	\begin{align} \label{eq:cl}
	 	C_{\calL}[\eta]\DEF \eta(\tau)-\LLe\bigg[ -\frac{1}{\pi^2\sqrt{\eps^2-t^2}}\pvinteps\frac{\sqrt{\eps^2-s^2}\;\eta'(s)}{t-s}\intd s \bigg]
	\end{align}
	is a constant depending on $\eta$ and it is linear in $\eta$.
\end{proposition}

\begin{proof}
	The proof for invertiblity is given in \cite[Chapter 11.5]{SV}, the exact formula is derived in \cite[Chapter 5.2.3]{LPTSA}.
\end{proof}

\begin{lemma}\label{lemma:exactLLeValues}
	We have that
	\begin{align}
		\LLe\bigg[ t\mapsto\frac{1}{\sqrt{\eps^2-t^2}} \bigg] (\tau)=&\;\pi\log(\eps/2),\label{equ:exactLLeValues:1}\\
		\LLe[1](\tau)=& \;-2\eps+(\eps+\tau)\log(\eps+\tau)+(\eps-\tau)\log(\eps-\tau)\,,\label{equ:exactLLeValues:2}\\
		\LLe\left[\frac{t}{\sqrt{\eps^2-t^2}}\right](\tau)=&\; -\pi\tau \,,\label{equ:exactLLeValues:4}\\
		\LLeinv[1](t)=&\; \frac{1}{\pi\log(\eps/2)\sqrt{\eps^2-t^2}}\,.\label{equ:exactLLeValues:3}
	\end{align}
\end{lemma}

\begin{proof}
	Equations (\ref{equ:exactLLeValues:1})--(\ref{equ:exactLLeValues:4}) follow from straightforward calculations and  (\ref{equ:exactLLeValues:3}) follows using Proposition \ref{prop:LLEisInjective}.
\end{proof}

From Lemma \ref{lemma:exactLLeValues}, we also readily compute the following lemma:
\begin{lemma}\label{lemma:exactLnorms}
	We have that
	\begin{align*}
		\NORM{\LLeinv[1]}_{\curlXe}=&\frac{1}{\sqrt{\pi}\,|\log(\eps/2)|}\,,\\
		\NORM{\LLeinv[\tau]}_{\curlXe}=&\frac{\eps}{\sqrt{2\pi}}\,.
	\end{align*}
\end{lemma}

Since $\RdelDdK$ and $\RdelOonedKp$ are continuous, $\RRdke$ is a compact operator. Thus we have that $\tfrac{2}{\pi}\LLe+\RRdke$ is a Fredholm operator of index zero. Hence, for the operator $\AAdke$, extending the domain $\kkkk$ to the complex numbers in a disk-shaped form, we will see that $\AAdke$ is invertible except for a finite amount of values of $\delta k$. Some of those values are the resonances of our physical system. To that end, we will need the following result.

\begin{lemma}\label{lemma:normLinvR}
	Let $\calR$ be the integral operator defined from $\curlXe$ into $\curlYe$ by
	\begin{align*}
		\calR[\mu](\tau)\,=\, \inteps R(\tau ,t)\, \mu(t)\intd t\,,
	\end{align*}
	where $R(\tau,t)$ is of class $\cC^{1,\eta}$ in $\tau$ and $t$, for $0<\eta<1$.
	There exists a positive constant $C_{(\ref{equ:normLinvR})}$, independent of $\eps$, such that
	\begin{align}\label{equ:normLinvR}
		\NORM{\LLeinv\calR}_{\calL(\curlXe,\curlXe)}\leq \frac{C_{(\ref{equ:normLinvR})}}{|\log(\eps)|},
	\end{align}
	where ${\calL(\curlXe,\curlXe)}$ denotes the space of bounded linear operators from $\curlXe$ to $\curlXe$.
\end{lemma}

\begin{proof}
	The proof is given in \cite[Lemma 5.4]{LPTSA}
\end{proof}

\subsubsection{Characteristic Values of $\AAdnoboldke$ and the two Resonance Values}
Let us first look at the characteristic values of 
\begin{align*}
	\QQdke[\mu] \DEF \frac{2}{\pi}\LLe[\mu]+\frac{\KKe[\mu]}{\delta^2 k^2},
\end{align*}
where $\mu\in\curlXe$ and $\delta k\in\kkkkc\DEF\{k_\ast\in \CC\MID \sqrt{k_\ast \bar{k}_\ast}< k_{D,\min,\Laplace}/2\}$. Let $(\cdot\,, \cdot)_\eps$ be defined as the $\Ltwo((-\eps,\eps))$ inner product.

\begin{lemma}\label{lemma:ZeroOrderRes}
	$\QQdke$ has exactly the two characteristic values $\pm k^{\delta,\eps}_0$ where both have the characteristic function $\mu^{\delta,\eps}_0$, with
	\begin{align*}
		k^{\delta,\eps}_0 =&
		\frac{1}{\delta}\left( -\frac{\pi}{2|D|} \Big(\LLeinv[1]\,,1\Big)_\eps \right)^{1/2} 
		= \frac{1}{\delta}\left(-\frac{\pi}{2|D|\log{(\eps/2)}}  \right)^{1/2} , \\
		\mu^{\eps}_0 =& 
		-\frac{\pi}{2|D|}\frac{\LLeinv[1]}{(\delta k^{\delta,\eps}_0)^2}\,,
	\end{align*}
	after imposing $(\mu,1)_\eps=1$.
\end{lemma}

Consider that $k^{\delta,\eps}_0$ is real and positive.

\begin{proof}
	We are looking for $\mu\in\curlXe$ such that
	\begin{align*}
		\QQdke[\mu]=0\,.
	\end{align*}
	Since $\LLeinv[0]=0$, the last equation is equivalent to
	\begin{align}\label{equ:PFZeroOrderRes:2}
		\frac{2}{\pi}\mu+\frac{1}{|D|}\frac{(\mu\,, 1)_\eps\,\LLeinv[1]}{(\delta k)^2}=0\,.
	\end{align}
	Applying $(\cdot\,, 1)_\eps$ on both sides yields
	\begin{align}\label{equ:PFZeroOrderRes:3}
		(\mu\,,1)_\eps\left( \frac{2}{\pi}+\frac{1}{|D|}\frac{(\LLeinv[1]\,,1)_\eps}{(\delta k)^2} \right)=0\,.
	\end{align}
	If $(\mu\,, 1)_\eps=0$, then $\LLe[\mu]=0$, because of the condition $\QQdke[\mu]=0$, and then $\mu=0$, since $\LLe$ is invertible and linear. But $0$ cannot be a characteristic function, by definition. Hence $(\mu\,, 1)\neq 0$. Thus the second factor in  (\ref{equ:PFZeroOrderRes:3}) has to be zero. This leads us to
	\begin{align*}
		\delta^2 k^2 =-\frac{\pi}{2|D|}\Big(\LLeinv[1]\,,1\Big)_\eps\,.
	\end{align*}
	Using Lemma \ref{lemma:exactLLeValues}, we can calculate that $(\LLeinv[1]\,,1)_\eps\,= \frac{1}{\log(\eps/2)}$, and obtain the characteristic values.
	
	As for the characteristic functions, we rewrite  (\ref{equ:PFZeroOrderRes:2}) as
	\begin{align*}
		\frac{\mu}{(\mu,1)}=-\frac{\pi}{2|D|}\frac{\LLeinv[1]}{(\delta k)^2}\,.
	\end{align*}
	Imposing the normalization on $\mu$ we have proven our statement.
\end{proof}

To facilitate future expressions we define
\begin{align*}
	\ceps\DEF -\frac{\pi}{2|D|\log{(\eps/2)}} = -\frac{\pi}{2|D|}(\LLeinv[1]\,,1)_\eps\,.
\end{align*}
Next we will look at the characteristic values of $\AAdke$. Denote $\wLLdke\DEF \LLe+\frac{\pi}{2}\RRdnoboldke$ and $\SSdke\DEF \frac{\pi}{2}\LLeinv\RRdnoboldke$, where we fixed the angles of the incoming wave vector, but let the magnitude be complex. Using that $\RRdnoboldke$ is in $\cC^{1,\eta}$, for $\eta\in (0,1)$, and $\LLe$ invertible and using Lemma \ref{lemma:normLinvR}, we can apply the Neumann series, whenever $\eps$ is small enough, and thus we have
\begin{align}\label{equ:wLLdkeExpansion}
	(\wLLdke)^{-1}=\big( \calI + \SSdke \big)^{-1}\LLeinv = \sum_{l=0}^{\infty} (-\SSdke)^{l}\LLeinv,
\end{align}
where $\calI$ denotes the identity function in $\curlXe$. We then define
\begin{align}\label{equ:SumOfAdke}
	\Adke\DEF -\frac{\pi}{2|D|}\left(\wLLdkeinv[1]\,,1\right)_\eps\,.
\end{align}

\begin{lemma}
	The characteristic values of $\AAdnoboldke$ are zeros of the function $\delta^2 k^2 - \Adke$ and we have the asymptotic formula
	\begin{align*}
		\delta^2 k^2&= (\delta k^{\delta,\eps}_0)^2+ \delta^2\OO(\ceps^2),&&\quad\text{for } \eps\rightarrow 0\,.
		%\mu&= \mu^{\eps}_0+ \OO_\curlXe(\ceps),&&\quad\text{for } \eps\rightarrow 0\,.
	\end{align*}
\end{lemma}

\begin{proof}
	We prove that the zeros of $\delta^2 k^2 - \Adke$ are the characteristic values in the same way as in the proof of Lemma \ref{lemma:ZeroOrderRes}, but substitute $\LLe$ with $\wLLdke$. In order to derive  the asymptotic formulas, we use  (\ref{equ:wLLdkeExpansion}) and Lemma \ref{lemma:normLinvR}.
\end{proof}

\begin{proposition}
There exist two characteristic values, counting multiplicity, for the operator $\AAdnoboldke$ in $\kkkkc$. Moreover, they have the asymptotic expansions
	\begin{align*}
		\delta k&= \pm \delta k^{\delta,\eps}_0+ \delta\OO(\ceps),&&\quad\text{for } \eps\rightarrow 0 \,.
	\end{align*}
\end{proposition}
\begin{proof}
	Recall that the operator-valued analytic function $\QQdke$ is finitely meromorphic and of Fredholm type. Moreover, it has two characteristic values $\pm k_0^{\delta k,\eps}$ , and has a pole at $0$ with order two in $\kkkkc$. Thus, the multiplicity of $\QQdke$ is zero in $\kkkkc$. Note that for $\delta k \in \kkkkc\setminus\{ 0,\pm \delta k^{\delta,\eps}_0\}$, the operator $\QQdke$ is invertible, because it is of Fredholm type and because it is injective due to Lemma \ref{lemma:ZeroOrderRes}. With that,
\begin{align*}
	(\QQdke)^{-1}\RRdnoboldke =\frac{2}{\pi}\left( \frac{2}{\pi}\calI - \frac{\LLeinv\KKe}{\delta^2 k^2} \right)^{-1}\SSdke\,.
\end{align*}	 
Thus, $\NORM{(\QQdke)^{-1}\RRdnoboldke}_{\calL(\curlXe,\curlXe)}=\OO(\ceps)$ uniformly for $\delta k\in\del\kkkkc$.
By the generalized Rouch\'e’s theorem \cite[Theorem 1.15]{LPTSA}, we can conclude that for $\eps$ sufficiently small,
the operator $\AAdnoboldke$ has the same multiplicity as the operator $\QQdke$ in $\kkkkc$ , which is zero.
Since $\AAdnoboldke$ has a pole of order two, we derive that $\AAdnoboldke$ has either one characteristic
value of order two or two characteristic values of order one. This completes the proof of
the proposition.
\end{proof}

Now, let us give an asymptotic expression for those resonances. We recall 
that $\alpha_0$ and $\alpha_1$, defined in (\ref{equdef:alpha0}) and  (\ref{equdef:alpha1}), are used in the decomposition of $\RRdnoboldke$, that is, $\RRdnoboldke = \RRdkeone + \RRdketwo + \RRdkethree$, with
\begin{align}
	\RRdkeone[\mu] \DEF&\alpha_0 (\,\mu,1)_\eps\,,\label{equdef:RRdkeone}\\
	\RRdketwo[\mu] \DEF&\delta k\,\alpha_1 \,(\mu,1)_\eps\,,\label{equdef:RRdketwo}\\
	\RRdkethree[\mu](\tau) \DEF& \inteps \mu(t) \left( 
	\tau \,\widetilde{\del_{\tau}} R^{\delta k}(\tau, t)+
	t\, \widetilde{\del_{t}} R^{\delta k}(\tau, t) +
	\delta^2 k^2\, \widetilde{\del_{\delta k}}^2 R^{\delta k}(\tau, t)
 	\right)\intd t\,,\label{equdef:RRdkethree}
\end{align}
with $R^{\delta k}(\tau,t)$ denotes the kernel of $\RRdnoboldke$, see Definition \ref{def:OperatorsFor1HR}, and $\widetilde{\del}$ denotes the derivative part of the remainder in Taylor's theorem in the Peano form. 

\begin{lemma}
	For all $\delta k \in \kkkkc$, $\RRdkethree$ satisfies the estimate
	\begin{align}\label{equdef:SSdkel}
		\NORM{\LLeinv\RRdkethree}_{\calL(\curlXe,\curlXe)}=\OO(\eps\cdot\ceps+ \ceps\cdot \delta^2 k^2)\,.
	\end{align}
\end{lemma}

\begin{proof}
	Using that $\LLeinv$ is linear we divide the proof according to the three terms in  (\ref{equdef:RRdkethree}). The proof for the term with $\delta^2k^2$ in it, can be seen immediately using Lemma \ref{lemma:normLinvR}.
	
	For the $\inteps\mu(t)t\, \del_{t} R^{\delta k}(\tau, t)\intd t \FED \eta(\tau)$ term we have
	\begin{align}
		\NORM{\LLeinv\eta}_\curlXe
			&=	\NORM{-\frac{1}{\pi^2\sqrt{\eps^2-t^2}}\pvinteps\frac{\sqrt{\eps^2-\tau^2}\;\eta'(\tau)}{t-\tau}\intd \tau+\frac{C_{\calL}[\eta]}{\pi\log(\eps/2)\sqrt{\eps^2-t^2}}}_\curlXe \nonumber \\
			&\leq	\NORM{\pvinteps\frac{\sqrt{\eps^2-\tau^2}\;\eta'(\tau)}{\pi^2\sqrt{\eps^2-t^2}(t-\tau)}\intd \tau}_\curlXe+\frac{|C_{\calL}[\eta]|}{\log(\eps/2)}\,.\label{equ:tdelRdk}
	\end{align}	
	Let us consider the left term first. We split $\eta'(\tau)=\eta'(t) +(\eta'(\tau)-\eta'(t))$ and thus can split the principal value integral into an principal value integral with $\eta'(t)$ and into a normal integral with $\eta'(\tau)-\eta'(t)$. Using $\pvinteps \frac{\sqrt{\eps^2-s^2}}{t-s}\intd s = \pi t$ and that $$\inteps \frac{\sqrt{\eps^2-\tau^2}}{\sqrt{|t-\tau|}}\intd\tau\leq C_{(\ref{equ:tdelRdk:1})}\eps\sqrt{\eps},$$ where $C_{(\ref{equ:tdelRdk:1})}$ is independent of $\eps$ and $t$, see also \cite[Proof of Lemma 5.4]{LPTSA}, we can estimate the left term in Inequality (\ref{equ:tdelRdk:1}) to be smaller or equal to
	\begin{align}
		\pi |t| \, \NORM{\eta'(\tau)}_{\cC^0([-\eps,\eps])} + C_{(\ref{equ:tdelRdk:1})}\eps\,\sqrt{\eps} \left| \eta'(\tau) \right|_{\cC^{0,\,1/2}([-\eps,\eps])}\,.\label{equ:tdelRdk:1}
	\end{align}
	
	As for the right term in Inequality (\ref{equ:tdelRdk}), consider that
	\begin{align*}
		|C_{\calL}[\eta]|\leq |\eta(0)|+\LLe\left[\frac{1/\pi^2}{\sqrt{\eps^2-t^2}}\pvinteps\frac{\sqrt{\eps^2-\tilde{\tau}^2}\;\eta'(\tilde{\tau})}{t-\tilde{\tau}}\intd \tilde{\tau}\right](0)\,.
	\end{align*}
	Similarly to the argumentation above and using $|\eta(0)|\leq \eps$, we can infer that
	\begin{align*}
		&|C_{\calL}[\eta]|
		\leq	\eps\!+\!\inteps 
			\frac{|\log(|t|)|/\pi^2}{\sqrt{\eps^2-t^2}}
			\left( \pi |t| \, \NORM{\eta'(\tau)}_{\cC^0} + C_{(\ref{equ:tdelRdk:1})}\eps\,\sqrt{\eps} \left| \eta'(\tau) \right|_{\cC^{0,\,1/2}} \right)
			\intd t \nonumber\\
		&\quad= \eps+\frac{2\eps}{\pi}(-\log(2\eps)+1)\NORM{\eta'(\tau)}_{\cC^0} 
			+ \frac{|\log(2\eps)|}{\pi}C_{(\ref{equ:tdelRdk:1})}\eps\,\sqrt{\eps}\left| \eta'(\tau) \right|_{\cC^{0,\,1/2}}\,,
	\end{align*}
	thus we have shown with the term in (\ref{equ:tdelRdk}) the desired estimation for $\inteps\mu(t)t\, \del_{t} R^{\delta k}(\tau, t)\intd t$.

	For the term $\inteps\mu(t)\tau\, \del_{t} R^{\delta k}(\tau, t)\intd t$, we can use the same argumentation, where this time $|\eta(0)|=0$. This concludes the proof.
\end{proof}

Let $l\geq 1$ be an integer, we define
\begin{align*}
	\Sdke_{(l)}\DEF -\frac{\pi}{2|D|}\left( (-\SSdke)^{l}\LLeinv[1]\,, 1\right)_\eps\,.
\end{align*}
Because of  (\ref{equ:SumOfAdke}), we can write
\begin{align*}
	\Adke= \ceps +\sum_{l=1}^\infty \Sdke_{(l)}\,.
\end{align*}
We want to give a second order analytic expression for \\$\Sdke_{(1)}=\frac{\pi}{2|D|}\left(\frac{\pi}{2} \LLeinv\RRdnoboldke\LLeinv[1]\,, 1\right)_\eps$. To this end, we define	%HACK
\begin{align}
	\Sdke_{(1,1)}=\frac{\pi}{2|D|}\left(\frac{\pi}{2} \LLeinv\RRdkeone\LLeinv[1]\,, 1\right)_\eps\,,\label{equdef:Sdke11}\\
	\Sdke_{(1,2)}=\frac{\pi}{2|D|}\left(\frac{\pi}{2} \LLeinv\RRdketwo\LLeinv[1]\,, 1\right)_\eps\,,\label{equdef:Sdke12}\\
	\Sdke_{(1,3)}=\frac{\pi}{2|D|}\left( \frac{\pi}{2}\LLeinv\RRdkethree\LLeinv[1]\,, 1\right)_\eps\,.\label{equdef:Sdke13}
\end{align}
We obtain that second order analytic expression with the following lemma:
\begin{lemma}
	For all $l\in\NN$, we have that
	\begin{align*}
		|\Sdke_{(l)}|=&\OO(\ceps^{l+1})\,, \\
		\Sdke_{(1,1)}=& \alpha_0 \ceps^2 {|D|} \,, \\
		\Sdke_{(1,2)}=& \alpha_1\,\delta k\, \ceps^2 {|D|} \,, \\
		\Sdke_{(1,3)}=& \OO(\eps\,\ceps^2+ \ceps^2\, \delta^2 k^2) \,.
	\end{align*}
\end{lemma}
\begin{proof}
	The proof follows from straightforward calculation using Lemma \ref{lemma:normLinvR} and the expressions in  (\ref{equdef:SSdkel}), (\ref{equdef:RRdkeone})--(\ref{equdef:RRdkethree}) and (\ref{equdef:Sdke11})--(\ref{equdef:Sdke13}).
\end{proof}
Now we can deduce that
\begin{align*}
	\Adke = \ceps + 
	\alpha_0\, \ceps^2 {|D|} + 
	\alpha_1\,\delta k\, \ceps^2 {|D|} +
	\OO(\eps\,\ceps^2+ \ceps^2\, \delta^2 k^2) \,.
\end{align*} 
Now we are able to solve for the zeros in $\delta^2 k^2-\Adke$, and thus get the characteristic values of $\AAdnoboldke$. To this end, consider that $\sqrt{1+x}= 1 + \frac{1}{2}x-\OO(x^2)$ for $|x|<1$, thus
\begin{align}\label{equdef:adkesqrt}
	\Adkesqrt\DEF\sqrt{\Adke}=
	\sqrt{\ceps}\left( 1+\ceps\frac{\alpha_0}{2}{|D|}+\delta k\,\ceps\frac{\alpha_1}{2}{|D|} +\OO(\eps\,\ceps+ \ceps\, \delta^2 k^2) \right)\,,
\end{align}
which leads us to the factorization,
\begin{align*}
	\delta^2 k^2- \Adke = (\delta k - \Adkesqrt)\cdot(\delta k + \Adkesqrt).
\end{align*}
Thus, the roots for $\delta^2 k^2-\Adke$ are those for $\delta k-\Adkesqrt$ and $k+\Adkesqrt$.

\begin{proposition}\label{prop:resonances}
	There are exactly two characteristic values for the operator $\AAdke$ in $\kkkkc$. Those are approximated by
	\begin{align*}
		\delta k^{\delta,\eps}_+ =\,&\delta k^{\delta,\eps}_0+ \frac{\alpha_0\,|D|}{2}\,\ceps^{3/2}+\frac{\alpha_1\,|D|}{2}\,\ceps^2+\OO(\ceps^{5/2})\,,\\
		\delta k^{\delta,\eps}_- =-&\delta k^{\delta,\eps}_0- \frac{\alpha_0\,|D|}{2}\,\ceps^{3/2}+\frac{\alpha_1\,|D|}{2}\,\ceps^2+\OO(\ceps^{5/2})\,,\\
		&\delta k^{\delta,\eps}_0 \DEF\ceps^{1/2}\DEF\sqrt{\frac{-\pi}{2|D|\log{(\eps/2)}}}\,.
	\end{align*}
	%Furthermore, under the normalization condition $(\mu,1)=1$, the corresponding characteristic functions are given by
	%\begin{align}
	%	\mu_+=&\mu^{\eps}_0+\OO(\ceps)\,\\
	%	\mu_-=&\mu^{\eps}_0+\OO(\ceps)\,.
	%\end{align}
\end{proposition}

\begin{proof}
	We define
	\begin{align*}
		\AdkesqrtNull\DEF \sqrt{\ceps}\left( 1+\ceps\frac{\alpha_0}{2}{|D|}+\delta k\,\ceps\frac{\alpha_1}{2}{|D|} \right)\,.
	\end{align*}
	For $\eps$ small enough, the zeros of $\delta k -\AdkesqrtNull$ are exactly the zeros of $\delta k-\Adkesqrt$, up to a term in $\OO(\ceps^{5/2})$. This follows readily from Rouch\'e's Theorem.\\
	As for the zeros of $\delta k -\AdkesqrtNull$, we obtain them through the approach
	\begin{align*}
		\delta k = \gamma_1\ceps^{1/2}+\gamma_2\ceps^{1} +\gamma_3\ceps^{3/2}+\gamma_4\ceps^2.
	\end{align*}
	Inserting the approach into $\AdkesqrtNull$ and solving $\delta k -\AdkesqrtNull=0$, we obtain
	\begin{align*}
		\gamma_1=1,\quad\gamma_2=0,\quad
		\gamma_3=\frac{\alpha_0\,|D|}{2},\quad \gamma_4=\frac{\alpha_1\,|D|}{2}\,.
	\end{align*}
	An analogous argumentation leads to the zeros of $\delta k +\AdkesqrtNull\;$.
\end{proof}

With Proposition \ref{prop:resonances} we immediately obtain Theorem \ref{THM1:1HR}.

\subsubsection{Inversion of $\AAdnoboldke$ - Solving the First Order Linear Equation}
We know now that $\AAdnoboldke$ is invertible for $\delta k\in\kkkkc$, except at the characteristic values $k=k^{\delta,\eps}_+$ and $k=k^{\delta,\eps}_-$ and at the pole, $k=0$. Let us examine, how we can express $(\AAdnoboldke)^{-1}[f^{\delta k}]$, where $f^{\delta k}$ is given by  (\ref{equdef:fdk}).

First consider that we already know that the equation $\AAdke[\mu]=f^{\delta k}$ is equivalent to
\begin{align}\label{equ:FirstFormulaForAm=f:1}
	\frac{2}{\pi}\mu +\frac{(\mu\,, 1)_\eps \wLLdkeinv[1]}{|D|\,\delta^2 k^2}=\wLLdkeinv[f^{\delta k}]\,.
\end{align} 
Applying $(\cdot\,, 1)_\eps$ on both sides, we obtain
\begin{align*}
	(\mu\,,1)_\eps \left( 1 - \frac{\Adke}{\delta^2 k^2} \right)
		=	\frac{\pi}{2} (\wLLdkeinv[f^{\delta k}]\,, 1)_\eps \,.
\end{align*}
And thus
\begin{align}\label{equ:FirstFormulaForAm=f:3}
	(\mu\,,1)_\eps= \frac{\delta^2 k^2}{\delta^2 k^2-\Adke}\frac{\pi}{2}(\wLLdkeinv[f^{\delta k}]\,, 1)_\eps\,.
\end{align}

\begin{lemma}\label{lemma:1/dk^2-adke}
	For $\delta k\in\kkkkc$ and $\eps\rightarrow 0$ we have that
	\begin{align*}
		\frac{1}{\delta^2 k^2 - \Adke}= \frac{1}{(\delta k-\delta k^{\delta, \eps}_+)(\delta k-\delta k^{\delta, \eps}_-)}\left( 1+\OO(\ceps) \right)\,.
	\end{align*}
\end{lemma}

\begin{proof}
	Recall that $\delta^2 k^2-\Adke = (\delta k-\Adkesqrt)(\delta k + \Adkesqrt)$. Let us first investigate the function $\delta k-\Adkesqrt$. We already established that it has the unique root $\delta k^{\delta, \eps}_+$ in the set $\kkkkc$, for $\eps$ small enough. Thus $\delta k-\Adkesqrt$ can be written as
	\begin{align}\label{equ:PFRecidk-Adke:1}
		\delta k-\Adkesqrt=(\delta k- \delta k^{\delta,\eps}_+)(1+g^{\eps}(\delta k))\,,
	\end{align}
	where $g^{\eps}$ is an analytic function in $\kkkkc$, for $\eps$ small enough. Considering the definition of $\Adkesqrt$,  (\ref{equdef:adkesqrt}), and of $\delta k^{\delta,\eps}_+$, Proposition \ref{prop:resonances}, we can conclude that $g^{\eps}$ is smooth with respect to $\sqrt{\ceps}$. By the Taylor expansion, we can write $g^\eps$ in the form 
	\begin{align*}
		g^\eps(\delta k)= g_0(\delta k)+g_1(\delta k)\sqrt{\ceps}+g_2(\delta k,\sqrt{\ceps})\ceps\,,
	\end{align*}
	where $g_0$ and $g_1$ are analytic in $\delta k$ and the function $g_2$ is analytic in the first variable and is smooth in the second one. By comparing coefficients of different orders of $\sqrt{\ceps}$ on both sides of  (\ref{equ:PFRecidk-Adke:1}), we can deduce that $g_0(\delta k)=g_1(\delta k)=0$. Hence, we can conclude that $g^\eps(\delta k)=g_2(\delta k, \sqrt{\ceps})\ceps=\OO(\ceps)$.
	Similarly, we can prove that
	\begin{align*}
		\delta k+\Adkesqrt=(\delta k - \delta k^{\delta,\eps}_-)(1+\OO(\ceps))\,.
	\end{align*}
	Thus, we can conclude that
	\begin{align*}
		\frac{1}{\delta^2 k^2 - \Adke}=& \frac{1}{\delta k -\Adkesqrt}\frac{1}{\delta k +\Adkesqrt}\\
		=& \frac{1}{(\delta k -\delta k^{\delta,\eps}_+)(\delta k -\delta k^{\delta,\eps}_-)}\frac{1}{(1+\OO(\ceps))(1+\OO(\ceps))}\\
		=&\frac{1}{(\delta k -\delta k^{\delta,\eps}_+)(\delta k -\delta k^{\delta,\eps}_-)} (1+\OO(\ceps))\,,
	\end{align*}
	which completes the proof of the lemma.
\end{proof}

\begin{lemma}\label{lemma:EstLinvf}
	We have
	\begin{align*}
		\wLLdkeinv[f^{\delta k}](t)=&f^{\delta k}(0)\LLeinv[1]+\OO(\delta)\OO_\curlXe(\ceps^2) + \OO(\delta)\OO_\curlXe(\eps)\,,\\
		\wLLdkeinv[1](t)=&\LLeinv[1]+\OO_\curlXe(\ceps^2)\,,
	\end{align*}
	for $\eps\rightarrow 0$ in the $\curlXe$ norm and $\delta\rightarrow 0$.
\end{lemma}

\begin{proof}
	We write $f^{\delta k}(\tau)=f^{\delta k}(0)+\tau\,g^{\delta k}(\tau)$ with a smooth function $g^{\delta k}\in \cC^{1,\,1/2}([-\eps,\eps])$. We readily see that $f^{\delta k}=\OO(\delta)$ and $g^{\delta k}=\OO(\delta)$. Recall that
	\begin{align*}
		\wLLdkeinv=\sum_{l=0}^{\infty} (-\SSdke)^{l}\LLeinv\,.
	\end{align*}
	Using that $\NORM{\SSdke}_{\calL(\curlXe,\curlXe)}=\OO(\ceps)$ and according to Lemma \ref{lemma:exactLnorms}, $\NORM{\LLeinv[1]}_{\curlXe}=\OO(\ceps)$, we have for $n\in\NN$ that
	\begin{align*}
		\NORM{(\SSdke)^n\LLeinv[1]}_{\curlXe}=&\,\OO(\ceps^{n+1})\,,\\
		\NORM{(\SSdke)^n\LLeinv[\tau\,g^{\delta k}(\tau)]}_{\curlXe}=&\,\OO(\ceps^n)\NORM{\LLeinv[\tau\,g^{\delta k}(\tau)]}_\curlXe\,.
	\end{align*}
	Consider that $\eps=\OO(\ceps^n)$, thus it is enough to show
	\begin{align*}
		\NORM{\LLeinv[\tau\,g^{\delta k}(\tau)]}_\curlXe=\OO(\delta)\OO(\eps)\,.
	\end{align*}
	So, let us show that. Using Proposition \ref{prop:LLEisInjective} we have
	\begin{align}\label{equ:PFEstLinvf:5}
		&\NORM{\LLeinv[\tau\, g^{\delta k}(\tau)]}_{\curlXe}^2= 
		\inteps \sqrt{\eps^2-t^2}\left| \LLeinv[\tau g^{\delta k}(\tau)] \right|^2\intd t\\
		&\qquad=\inteps \sqrt{\eps^2-t^2}\bigg| 
		\frac{-1}{\pi^2\sqrt{\eps^2-t^2}}\pvinteps\frac{\sqrt{\eps^2-\tau^2}\;\del_\tau(\tau g^{\delta k}(\tau))}{t-\tau}\intd \tau
		+\frac{C_{\calL}[\tau g^{\delta k}(\tau)]}{\pi\log(\eps/2)\sqrt{\eps^2-t^2}}
		\bigg|^2\intd t\,.\label{equ:PFEstLinvf:5+}
	\end{align}
	Using $(a+b)^2\leq 2(a^2+b^2)$ and pulling the $\sqrt{\eps^2-t^2}$ term inside, we obtain that the term in  (\ref{equ:PFEstLinvf:5+}) is smaller or equal to
	\begin{align}\label{equ:PFEstLinvf:7}
		\inteps \left( 
		\frac{1/\pi^2}{\sqrt{\eps^2-t^2}}\left|\pvinteps\frac{\sqrt{\eps^2-\tau^2}\;\del_\tau(\tau g^{\delta k}(\tau))}{t-\tau}\intd \tau\right|^2 
		+ \frac{C_{\calL}[\tau g^{\delta k}(\tau)]^2/\pi^2}{(\log(\eps/2))^2\sqrt{\eps^2-t^2}}
		\right)\intd t\,.
	\end{align}
	Using $\inteps 1/\sqrt{\eps^2-s^2}\intd s=\pi$ and $C_{\calL}[\tau g^{\delta k}(\tau)]$ is a constant, the term is equal to
	\begin{align}\label{equ:PFEstLinvf:8}
		\inteps \left( 
		\frac{1/\pi^2}{\sqrt{\eps^2-t^2}}\left|\pvinteps\frac{\sqrt{\eps^2-\tau^2}\;\del_\tau(\tau g^{\delta k}(\tau))}{t-\tau}\intd \tau\right|^2 
		\right)\intd t
		+ \frac{C_{\calL}[\tau g^{\delta k}(\tau)]^2}{(\log(\eps/2))^2\pi}\,.
	\end{align}
	Let us estimate both terms in the sum. 
	
	Observe that
	\begin{align}
		&\left|\pvinteps\frac{\sqrt{\eps^2-\tau^2}\;\del_{\tau}(\tau g^{\delta k}(\tau))}{t-\tau}\intd \tau\right| \nonumber\\
		&=\left|\pvinteps\frac{\sqrt{\eps^2-\tau^2}\;\del_{t}(t g^{\delta k}(t))}{t-\tau}\intd \tau
			+\inteps\frac{\sqrt{\eps^2-\tau^2}\;
				(\del_{\tau}(\tau g^{\delta k}(\tau))
				-\del_{t}(t g^{\delta k}(t)))
				}{t-\tau}\intd \tau
			\right| \nonumber \\
		&\leq \pi |t| \, |\del_{t}(t g^{\delta k}(t))| +
			\inteps \frac{\sqrt{\eps^2-\tau^2}}{\sqrt{|t-\tau|}}
			\frac{|\del_{\tau}(\tau g^{\delta k}(\tau))
			-\del_{t}(t g^{\delta k}(t))|}{\sqrt{|t-\tau|}}\intd \tau \nonumber \\
		&\leq \pi |t| \, \NORM{\del_{\tau}(\tau g^{\delta k}(\tau))}_{\cC^0([-\eps,\eps])} + C_{(\ref{equ:PFEstLinvf:Obs})}\eps\,\sqrt{\eps} \left| \del_\tau(\tau g^{\delta k}(\tau)) \right|_{\cC^{0,\,1/2}([-\eps,\eps])}\,,\label{equ:PFEstLinvf:Obs}
	\end{align}
	where we used that $\pvinteps \frac{\sqrt{\eps^2-s^2}}{t-s}\intd s = \pi t$, and that $\inteps \frac{\sqrt{\eps^2-\tau^2}}{\sqrt{|t-\tau|}}\intd\tau\leq C_{(\ref{equ:PFEstLinvf:Obs})}\eps\sqrt{\eps}$, with $C_{(\ref{equ:PFEstLinvf:Obs})}$ being independent of $\eps$ and $t$, see \cite[Proof of Lemma5.4]{LPTSA}.
	
	With that, we can estimate the left term in (\ref{equ:PFEstLinvf:8}) to be smaller or equal to 
	\begin{align}
		\inteps 
			\frac{1/\pi^2}{\sqrt{\eps^2-t^2}} 2 &\left( 
				\pi^2 |t|^2 \, \NORM{\del_{\tau}(\tau g^{\delta k}(\tau))}_{\cC^0([-\eps,\eps])}^2 + C_{(\ref{equ:PFEstLinvf:Obs})}^2\eps^3 \left| \del_\tau(\tau g^{\delta k}(\tau)) \right|_{\cC^{0,\,1/2}([-\eps,\eps])}^2 
		\right)\intd t \nonumber \\
		\leq & \eps^2 \pi \NORM{\del_{\tau}(\tau g^{\delta k}(\tau))}_{\cC^0([-\eps,\eps])}^2 +\frac{2C_{(\ref{equ:PFEstLinvf:Obs})}^2}{\pi}\eps^3 \left| \del_\tau(\tau g^{\delta k}(\tau)) \right|_{\cC^{0,\,1/2}([-\eps,\eps])}^2\,,\label{equ:PFEstLinvf:Intfinal}
	\end{align}
	where we used that $\inteps s^2/\sqrt{\eps^2-s^2}\intd s= \eps^2\pi/2$.

	Let us consider the right term in (\ref{equ:PFEstLinvf:8}). To compute $C_{\calL}[\tau g^{\delta k}(\tau)]$ we can pick any $\tau\in(-\eps\,,\eps)$. We pick $\tau=0$. Thus,
	\begin{align}\label{equ:PFEstLinvf:C}
		C_{\calL}[\tau g^{\delta k}(\tau)]=& 0 + \LLe\left[\frac{1/\pi^2}{\sqrt{\eps^2-t^2}}\pvinteps\frac{\sqrt{\eps^2-\tilde{\tau}^2}\;\del_{\tilde{\tau}}(\tilde{\tau} g^{\delta k}(\tilde{\tau}))}{t-\tilde{\tau}}\intd \tilde{\tau}\right](0)\,.
	\end{align}
	With the observation in Inequality (\ref{equ:PFEstLinvf:Obs}), we can infer that
	\begin{align}
		&|C_{\calL}[\tau g^{\delta k}(\tau)]|\leq
		\inteps 
			\frac{|\log(|t|)|/\pi^2}{\sqrt{\eps^2-t^2}}
			\left( \pi |t| \, \NORM{\del_{\tau}(\tau g^{\delta k}(\tau))}_{\cC^0} + C_{(\ref{equ:PFEstLinvf:Obs})}\eps\,\sqrt{\eps} \left| \del_\tau(\tau g^{\delta k}(\tau)) \right|_{\cC^{0,\,1/2}} \right)
			\intd t \nonumber\\
		&= \frac{2\eps}{\pi}(-\log(2\eps)+1)\NORM{\del_{\tau}(\tau g^{\delta k}(\tau))}_{\cC^0} 
			+ \frac{|\log(2\eps)|}{\pi}C_{(\ref{equ:PFEstLinvf:Obs})}\eps\,\sqrt{\eps}\left| \del_\tau(\tau g^{\delta k}(\tau)) \right|_{\cC^{0,\,1/2}}\label{equ:PFEstLinvf:CFinal}\,,
	\end{align}
	where we used that $$\inteps |\log(|t|)\,t|/\sqrt{\eps^2-t^2}=2\eps(-\log(2\eps)+1)$$ and where we assumed that $\eps$ is small enough, so that $|\log(|t|)|=-\log(|t|)$.
	Thus, we have
	\begin{align}\label{equ:PFEstLinvf:Cfinal}
		\frac{C_{\calL}[\tau g^{\delta k}(\tau)]^2}{\log(\eps/2)^2\pi} 
		\!\leq\! \frac{8\eps^2}{\pi^3}\frac{(1\!\!-\!\!\log(2\eps))^2}{\log(\eps/2)^2}\!\!\NORM{\del_{\tau}(\tau g^{\delta k}(\tau))}_{\cC^0}^2 
		\!\!+\!\frac{2C_{(\ref{equ:PFEstLinvf:Obs})}^2\eps^3}{\pi^3}\frac{\log(2\eps)^2}{\log(\eps/2)^2}\!\!\left| \del_\tau(\tau g^{\delta k}(\tau)) \right|_{\cC^{0,\,1/2}}^2 .
	\end{align}
	
%	We can conclude that
%	\begin{align}
%		\NORM{\LLeinv[\tau\, g^{\delta k}(\tau)]}_{\curlXe}^2
%		\leq \frac{1}{\pi} \eps^2 C_{(\ref{equ:HilbertrafoEST})}^2\, C_{g^{\delta k}}^2 
%		+ \frac{1}{\pi^3}\,4\eps^3\left| \del_s(s g^{\delta k}(s)) \right|_{\cC^{0,\,1/2}}^2 \\
%		\leq \eps^2 \left( \frac{1}{\pi} C_{(\ref{equ:HilbertrafoEST})}^2\, C_{g^{\delta k}}^2 
%	\end{align}
%	by inserting Inequality (\ref{equ:PFEstLinvf:CFinal}) and Term (\ref{equ:PFEstLinvf:IntFinal}) into Term (\ref{equ:PFEstLinvf:8}). 

	Combining Inequality (\ref{equ:PFEstLinvf:Cfinal}) and Inequality (\ref{equ:PFEstLinvf:Intfinal}) and using that $\NORM{\del_{\tau}(\tau g^{\delta k}(\tau))}_{\cC^0}=\OO(\delta)$ and $\left| \del_\tau(\tau g^{\delta k}(\tau)) \right|_{\cC^{0,\,1/2}}=\OO(\delta)$, we have shown that
	\begin{align*}
		\NORM{\LLeinv[\tau\,g^{\delta k}(\tau)]}_\curlXe=\OO(\delta)\OO(\eps)\,
	\end{align*}
	and proved Lemma \ref{lemma:EstLinvf}.
	%
%	\todo{Check this proof again at the end.}
\end{proof}

\begin{proposition}\label{prop:LSEsolution}
	Let $\delta k\in\kkkkc\setminus\{ 0, \dkdep, \dkdem \}$, there exists a unique solution to the equation $\AAdnoboldke[\mu]=f^{\delta k}$. Moreover, the solution can be written as $\mu=\mu_\star+\mu_\sim$, where
	\begin{align*}
		\mu_\star=&  \frac{\pi\,\ceps}{2}\frac{ f^{\delta k}(0)\,\LLeinv[1]}{\dkdep - \dkdem}\resSumWOfirstfrac\!+\!\frac{\pi}{2}f^{\delta k}(0)\LLeinv[1]\,,\\
		\mu_\sim=&  \OO(\delta)\left(\OO_{\curlXe}(\ceps^3)\resSumWOfirstfrac+\OO_{\curlXe}(\ceps^2) + \OO_{\curlXe}(\eps)\right)\,.
	\end{align*}
\end{proposition}

\begin{proof}
	From $\AAdnoboldke[\mu]=\fdnoboldk$ we get that
	\begin{align}\label{equ:PFLSEsolution:1}
		\frac{2}{\pi}\mu +\frac{(\mu\,, 1)_\eps \wLLdkeinv[1]}{|D|\,\delta^2 k^2}=\wLLdkeinv[f^{\delta k}]\,.
	\end{align}	 
	After rearranging, see  (\ref{equ:FirstFormulaForAm=f:1}), we derive
	\begin{align*}
		 (\mu\,,1)_\eps
		 		= 	\frac{\delta^2 k^2}{\delta^2 k^2-\Adke}\frac{\pi}{2}(\wLLdkeinv[f^{\delta k}]\,, 1)_\eps\,.
	\end{align*}
	Then, we obtain by applying Lemma \ref{lemma:1/dk^2-adke} that
	\begin{align*}
		(\mu\,,1)_\eps =& 
				\frac{\pi}{2}(\wLLdkeinv[f^{\delta k}]\,, 1)_\eps\,\frac{\delta^2 k^2}{(\delta k-\delta k^{\delta, \eps}_+)(\delta k-\delta k^{\delta, \eps}_-)}\left( 1+\OO(\ceps) \right)\,.
	\end{align*}
	Using that $\frac{1}{(x-a)(x-b)}=\frac{1}{a-b}\left( \frac{1}{x-a}-\frac{1}{x-b} \right)$, we have
	\begin{align}\label{equ:PFLSEsolution:4}
		(\mu\,,1)_\eps 
				=& \frac{\pi}{2}(\wLLdkeinv[f^{\delta k}]\,, 1)_\eps\,\resSum\left( 1+\OO(\ceps) \right)\,.
	\end{align}
	To get the the solution we use (\ref{equ:PFLSEsolution:1}), but insert  (\ref{equ:PFLSEsolution:4}) into it to arrive at
	\begin{align*}
		\frac{2}{\pi}\mu =&-\frac{\pi}{2}\frac{ (\wLLdkeinv[f^{\delta k}]\,, 1)_\eps\,\wLLdkeinv[1]}{|D|(\dkdep - \dkdem)}\resSumWOfirstfrac\left( 1+\OO(\ceps) \right) \nonumber\\
		&+\wLLdkeinv[f^{\delta k}]\,.
	\end{align*}
	Then from Lemma \ref{lemma:EstLinvf} it follows that
	\begin{align*}
		\frac{2}{\pi}\mu 
		=&-\frac{\pi}{2}\frac{ ((f^{\delta k}(0)\LLeinv[1]\,, 1)_\eps+\OO(\delta \ceps^2) + \OO(\delta\eps))\cdot(\LLeinv[1]+\OO_\curlXe(\ceps^2))}{|D|(\dkdep - \dkdem)} \\
		&\cdot \resSumWOfirstfracSHORT\left( 1+\OO(\ceps) \right)+f^{\delta k}(0)\LLeinv[1]+\OO_\curlXe(\delta\ceps^2) + \OO_\curlXe(\delta\eps)\nonumber\\
		=&-\frac{\pi}{2|D|}\frac{ f^{\delta k}(0)(\LLeinv[1]\,, 1)_\eps\,\LLeinv[1]+\OO_\curlXe( \delta\ceps^3) + \OO_\curlXe(\delta\eps\ceps^3)}{\dkdep - \dkdem} \\
		&\cdot \resSumWOfirstfracSHORT\left( 1+\OO(\ceps) \right)+f^{\delta k}(0)\LLeinv[1]+\OO_\curlXe(\delta\ceps^2) + \OO_\curlXe(\delta\eps)\nonumber\\
		=&\frac{-\pi}{2|D|}\frac{ f^{\delta k}(0)(\LLeinv[1]\,, 1)_\eps\,\LLeinv[1]}{\dkdep - \dkdem}\resSumWOfirstfracSHORT\!+\!f^{\delta k}(0)\LLeinv[1] \nonumber\\
		&+\OO(\delta)\OO_\curlXe(\ceps^3)\resSumWOfirstfracSHORT+\OO(\delta)\OO_\curlXe(\ceps^2) + \OO(\delta)\OO_\curlXe(\eps)\,.
	\end{align*}
	With that, and using $(\LLeinv[1]\,, 1)_\eps=-\frac{2|D|}{\pi}\ceps$, the proof for Proposition \ref{prop:LSEsolution} readily follows.
\end{proof}

\subsubsection{Asymptotic Expansion of our Solution to the Physical Problem}
In Proposition \ref{prop:GapFormulaINOperators} we established
\begin{align}\label{equ:GapFormulaINOperatorsRepeat:1}
   	\AAdnoboldke[\del_\nu \udnoboldk\MID_\Lambda](\tau)+\sum_{n=1}^\infty\delta^n\,2 \,\calGG^{k,\eps}_{+,n}[\del_\nu \udnoboldk\MID_\Lambda](\tau) = f^{\delta k}(\tau)\,,
\end{align}
and we know that for $\delta k\in\kkkkc\setminus\{ 0, \dkdep, \dkdem \}$ that $\AAdnoboldke$ is invertible, see Proposition \ref{prop:LSEsolution}. Then for $\delta$ small enough we can use the Neumann series and obtain
{\setlength{\belowdisplayskip}{0pt} \setlength{\belowdisplayshortskip}{0pt}\setlength{\abovedisplayskip}{0pt} \setlength{\abovedisplayshortskip}{0pt}
\begin{align*}
	(\AAdnoboldke+\sum_{n=1}^\infty 2\,\delta^n \,\calGG^{k,\eps}_{+,n})^{-1}
	=&\left(\calI +(\AAdnoboldke)^{-1} \sum_{n=1}^\infty 2\,\delta^n \,\calGG^{k,\eps}_{+,n}\right)^{-1}(\AAdnoboldke)^{-1} \\
	=&\sum_{m=0}^\infty\left(-(\AAdnoboldke)^{-1} \sum_{n=1}^\infty 2\,\delta^n \,\calGG^{k,\eps}_{+,n}\right)^m(\AAdnoboldke)^{-1}\\
	=&\sum_{m=0}^\infty\left( \sum_{n=1}^\infty -2\,\delta^n \,(\AAdnoboldke)^{-1}\calGG^{k,\eps}_{+,n}\right)^m(\AAdnoboldke)^{-1}\\
	=& (\AAdnoboldke)^{-1} -2\delta(\AAdnoboldke)^{-1}\calGG^{k,\eps}_{+,1} + \OO(\delta^2) \OO_{\calL(\curlXe,\curlXe)}(1)\,.
\end{align*}}

Thus we have from solving  (\ref{equ:GapFormulaINOperatorsRepeat:1})
\begin{align*}
	\del_\nu \udnoboldk\MID_\Lambda = (\AAdnoboldke)^{-1} [f^{\delta k}] + \OO(\delta^2)\,,
\end{align*}
where we use that $f^{\delta k}=\OO(\delta)$, that $\calGG^{k,\eps}_{+,1}$ is linear, and the formula for $(\AAdnoboldke)^{-1}$. With Proposition \ref{prop:LSEsolution} we split $(\AAdnoboldke)^{-1} [f^{\delta k}]$ into $\mu^{\delta k}_\star$ and $\mu^{\delta k}_\sim$ and thus we have for $t\in(-\eps\,,\eps)$
\begin{align*}
	\del_\nu \udnoboldk \left(\!\!\begin{pmatrix} t \\ h \end{pmatrix}\!\!\right)  = \mu^{\delta k}_\star(t) + \mu^{\delta k}_\sim(t) + \OO(\delta^2)\,,
\end{align*}
and we see from Proposition \ref{prop:LSEsolution} also that $\del_\nu \udnoboldk\MID_\Lambda = \OO(\delta)$.

Now we want to calculate the first order expansion term in $\delta$ for the solution in the far-field. Let $z\in \RR^2$, where $z_2\gg 1$. The following asymptotic expansion holds. 

\begin{lemma} \label{lemma:us=urhs+S+T}
We have for $z\in\RR^2_+$, $z_2>1$,
	\begin{multline}\label{equ:us=urhs+S+T}
		\udnoboldks(z)= u^{\delta k}_{\mathrm{RHS},p}(z) + u^{\delta k}_{\mathrm{RHS},e}(z)
			+ \Seu_{+,p}^{\delta k}[\del_\nu \udnoboldk\MID_\Lambda](z)+\Seu_{+,e}^{\delta k}[\del_\nu \udnoboldk\MID_\Lambda](z) \\
			+ \Teu^{\delta k}_{+, p}[\del_\nu \udnoboldk\MID_\Lambda](z)+ \Teu^{\delta k}_{+, e}[\del_\nu \udnoboldk\MID_\Lambda](z)+ \OO(\delta^2)\,,
	\end{multline}
	where for $\mu\in\curlXe$
	\begin{align*}
		\Seu_{+,p}^{\delta k}[\mu](z)\DEF & \int_\Lambda\mu(y)\Gamma^{\delta k}_{+,p}(z,y)\intd \sigma_y\,,\\
		\Seu_{+,e}^{\delta k}[\mu](z)\DEF & \int_\Lambda\mu(y)\Gamma^{\delta k}_{+,e}(z,y)\intd \sigma_y\,, \\
		\Teu_{+,p}^{\delta k}[\mu](z)\DEF & -\int_\Lambda\int_{\del D}\mu(y)\NdelOonedkc(y,w)\del_{\nu_w}\Gamma^{\delta k}_{+,p}(z,w)\intd \sigma_w\intd \sigma_y\,,\\
		\Teu_{+,e}^{\delta k}[\mu](z)\DEF & -\int_\Lambda\int_{\del D}\mu(y)\NdelOonedkc(y,w)\del_{\nu_w}\Gamma^{\delta k}_{+,e}(z,w)\intd \sigma_w\intd \sigma_y\,, 
	\end{align*}
	and
	\begin{align*}
		u^{\delta k}_{\mathrm{RHS},p}(z)
			=&	\, e^{i\delta (k_2z_2\!-\!k_1z_1)}\left[\int_{\del D} \del_\nu (\udknull-\udknull\circ \opP)(y)\frac{\sin(\delta k_2y_2)\,e^{i\delta k_1y_1}}{\delta k_2 p}\intd \sigma_y \right.\\
			&\mkern-50mu	\left.-\int_{\del D}\int_{\del D}\del_\nu (u^{\delta k}_0-u^{\delta k}_0\!\!\circ\! \opP)(y) \NdelOonedkc(y,w)
				\nu_{w}\cdot\begin{pmatrix}
					\frac{i\,k_1}{p\,k_2}\sin(\delta k_2 w_2)\\ \frac{1}{p}\cos(\delta k_2w_2)
				\end{pmatrix}e^{i\delta k_1 w_1}
				\intd \sigma_w\intd \sigma_y \!\right]\,,\nonumber\\
		u^{\delta k}_{\mathrm{RHS},e}(z)
			=&	\left[
				-\int_{\del D} \del_\nu (\udknull-\udknull\circ \opP)(y)\Gamma^{\delta k}_{+,e}(z,y)\intd \sigma_y \right.\\
			&\mkern-50mu	\left.+\int_{\del D}\int_{\del D}\del_\nu (u^{\delta k}_0-u^{\delta k}_0\!\!\circ\! \opP)(y) \NdelOonedkc(y,w)\,\del_{\nu_w} \Gamma^{\delta k}_{+,e}(z,w) \intd \sigma_w\intd \sigma_y \!\right]\,.\nonumber		
	\end{align*}
\end{lemma}
\begin{proof}
	We define $\udnoboldkf(z) \DEF \udnoboldk(z) + \fdnoboldk(z)\FED \udks(z)-u^{\delta k}_{\mathrm{RHS}}(z)$. Then we have from Proposition \ref{prop:udkFormulaOnComplementOfOmega} that
	\begin{align*}
		\udnoboldkf(z)
			&=		\int_\Lambda\del_\nu \udnoboldk(y)\NOonedkp(z,y)\intd \sigma_y\\
			&=		\int_\Lambda\del_\nu \udnoboldk(y)\Gamma^{\delta k}_+(z,y)\intd \sigma_y+\int_\Lambda\del_\nu \udnoboldk(y)\ROonedkp(z,y)\intd \sigma_y\\
			&\FED	\;\Seu^{\delta k}_{+}[\del_\nu \udnoboldk\MID_\Lambda] + \Teu^{\delta k}_{+}[\del_\nu \udnoboldk\MID_\Lambda]\,.
	\end{align*}
	Using the splitting $\GKp=\Gamma^\boldk_{+,p}+\Gamma^\boldk_{+,e}$, see Proposition \ref{prop:GKSFormulaFarFieldz2>x2}, we already obtain $\Seu_{+,p}^{\delta k}$ and $\Seu_{+,e}^{\delta k}$. To study the terms $u^{\delta k}_{\mathrm{RHS}}$ and $\Teu^{\delta k}_{+}$, consider that
	\begin{align*}
      \ROKp(z,y)=-\int_{\del E}\NdelOKc(y,w)\del_{\nu_y}\GKp(z,w)\intd \sigma_w\,,
   	\end{align*}
	from Proposition \ref{prop:FormulaForR}. Thus 
	\begin{align*}
		\int_\Lambda\del_\nu \udnoboldk(y)\ROonedkp(z,y)\intd \sigma_y 
			=	-\int_\Lambda\int_{\del D}\!\!\!\del_\nu \udnoboldk(y)\NdelOonedkc(y,w)\del_{\nu_w}\Gamma^{\delta k}_{+}(z,w)\intd \sigma_w \intd \sigma_y \,,
	\end{align*}
	and
	\begin{multline}
		\int_{\del D}\del_\nu (\udknull\!-\!\udknull\!\!\circ\! \opP)(y)\ROonedkp(z,y)\intd \sigma_y\\
			=	-\int_\Lambda\int_{\del D}\del_\nu (\udknull-\udknull\circ \opP)(y)\NdelOonedkc(y,w)\del_{\nu_w}\Gamma^{\delta k}_{+}(z,w)\intd \sigma_w \intd \sigma_y \,.
	\end{multline}
	Using again the splitting $\nabla\GKp=\nabla\Gamma^\boldk_{+,p}+\nabla\Gamma^\boldk_{+,e}$ and the explicit formula for $\Gamma^\boldk_{+,p}$ and $\nabla\Gamma^\boldk_{+,p}$ we obtain the formulas in Lemma \ref{lemma:us=urhs+S+T}.
\end{proof}

Let us approximate $\Seu_{+}^{\delta k}$ and $\Teu_{+}^{\delta k}$. Let $z_2>1$. We define
\begin{align}
	\Seu_{+,p,0}^{\delta k}[\mu](z)
		\DEF & 		-e^{i(\delta k_2 z_2-\delta k_1 z_1)}\frac{1}{p}\int_\Lambda y_2 \mu(y)\intd \sigma_y \,, \nonumber\\
	\Seu_{+,e,0}^{\delta k}[\mu](z)
		\DEF &\,	e^{-i\delta k_1 z_1} \bigg( \int_\Lambda \mu(y)\Gamma^{0}_{+}(z,y) \intd \sigma_y + \frac{1}{p} \int_\Lambda y_2 \mu(y) \intd \sigma_y \bigg) \nonumber \\
		=&			-e^{-i\delta k_1 z_1} \int_\Lambda \mu(y)
			\sum_{\substack{n\in\ZZ\setminus\{0\}\\ l\DEF 2\pi n/p}} \frac{1}{p|l|}e^{il(z_1-y_1)}e^{-|l|z_2}\sinh(|l|y_2) \intd \sigma_y\,,\label{equ:Seu_e0Extended}
\end{align}
and
\begin{align*}
	\Teu_{+,p,0}^{\delta k}[\mu](z)\DEF 
		& 	\;e^{i(\delta k_2 z_2-\delta k_1 z_1)}\frac{1}{p}\int_\Lambda\int_{\del D}\mu(y)\NdelOonedkc(y,w)
			\nu_w\!\cdot\!\begin{pmatrix}
			0\\1
			\end{pmatrix}\,
			\intd \sigma_w\intd \sigma_y \,,\\
	\Teu_{+,e,0}^{\delta k}[\mu](z)\DEF & \,-e^{-i\delta k_1 z_1} \bigg( \int_\Lambda\int_{\del D}\mu(y)\NdelOonedkc(y,w)\del_{\nu_w}\Gamma^{0}_{+}(z,w)\intd \sigma_w\intd \sigma_y \nonumber\\
		&	+ e^{-i\delta k_2 z_2} \Teu_{+,p,0}^{\delta k}[\mu](z)\bigg)\,.
%		=&	e^{-i\delta k_1 z_1} \int_\Lambda\int_{\del D}\mu(y)\NdelOonedkc(y,w)
%			\!\!\!\!\!\!\sum_{\substack{n\in\ZZ\setminus\{0\}\\ l\DEF 2\pi n/p}} \!\!\!\!\!\!
%					\nu_w\begin{pmatrix}
%						-i\,l\sinh(|l|w_2) \\ |l|\cosh(|l|w_2)
%					\end{pmatrix}
%			\intd \sigma_w\intd \sigma_y \nonumber\,.\label{equ:Teu_e0Extended}
\end{align*}
\begin{lemma}\label{lemma:S-S}
	There exist $C_{(\ref{equ:lemmaS-S:1})}>0$, $C_{(\ref{equ:lemmaS-S:2})}>0$, $C_{(\ref{equ:lemmaS-S:3})}>0$ and $C_{(\ref{equ:lemmaS-S:4})}>0$ such that, for all $z\in\RR^2_+$ with $z_2 > 1$ and all $\mu\in\curlXe$, it holds that
	\begin{align}\label{equ:lemmaS-S:1}
		\left|\left( \Seu_{+,p}^{\delta k}- \Seu_{+,p,0}^{\delta k}\right)[\mu](z)\right|+\frac{1}{\delta}\left|\nabla \left( \Seu_{+,p}^{\delta k}- \Seu_{+,p,0}^{\delta k} \right)[\mu](z)\right|\leq& C_{(\ref{equ:lemmaS-S:1})}\NORM{\mu}_{\curlXe}\delta^2\,,\\
		\left|\left( \Seu_{+,e}^{\delta k}- \Seu_{+,e,0}^{\delta k}\right)[\mu](z)\right|+\frac{1}{\delta}\left|\nabla \left( \Seu_{+,e}^{\delta k}- \Seu_{+,e,0}^{\delta k} \right)[\mu](z)\right|\leq& C_{(\ref{equ:lemmaS-S:2})}\NORM{\mu}_{\curlXe}\delta\,,\label{equ:lemmaS-S:2}
	\end{align}
	and
	\begin{align}\label{equ:lemmaS-S:3}
		\left|\left( \Teu_{+,p}^{\delta k}- \Teu_{+,p,0}^{\delta k}\right)[\mu](z)\right|+\frac{1}{\delta}\left|\nabla \left( \Teu_{+,p}^{\delta k}- \Teu_{+,p,0}^{\delta k} \right)[\mu](z)\right|\leq& C_{(\ref{equ:lemmaS-S:3})}\NORM{\mu}_{\curlXe}\delta\,,\\
		\left|\left( \Teu_{+,e}^{\delta k}- \Teu_{+,e,0}^{\delta k}\right)[\mu](z)\right|+\frac{1}{\delta}\left|\nabla \left( \Teu_{+,e}^{\delta k}- \Teu_{+,e,0}^{\delta k} \right)[\mu](z)\right|\leq& C_{(\ref{equ:lemmaS-S:4})}\NORM{\mu}_{\curlXe}\delta\,.\label{equ:lemmaS-S:4}
	\end{align}
\end{lemma}
\begin{proof}
	Let us consider $\Seu_{+}^{\delta k}$ first. Using the following splitting for $\Gamma^{\delta k}_+(z,y)$, which is given in Proposition \ref{prop:GKSFormulaFarFieldz2>x2}, we have
	\begin{align*}
		\Gamma^{\delta k}_+(z,y)=\Gamma^{\delta k}_{+,p}(z,x)+\Gamma^{\delta k}_{+,e}(z,x)\,,
	\end{align*}
	where
	\begin{align}
		\Gamma^{\delta k}_{+,p}(z,x)=&-\frac{\sin(\delta k_2x_2)}{\delta k_2p}e^{i\delta (k_2z_2-k_1z_1)}e^{i\delta k_1x_1}\,,\\
      \Gamma^{\delta k}_{+,e}(z,x)
      		=&		e^{-i\delta k_1z_1}\!\!\!\!\sum_{\substack{n\in\ZZ\setminus\{0\}\\ l\DEF 2\pi n/p}}\Bigg( \frac{-e^{il(z_1-x_1)+i\delta k_1x_1}}{p\sqrt{|l-\delta k_1|^2-\delta^2 k^2}}\nonumber\\
      		&		\cdot\sinh\left( \sqrt{|l-\delta k_1|^2-\delta^2 k^2}\,x_2 \right)\!\! \Bigg)e^{-\sqrt{|l-\delta k_1|^2-\delta^2 k^2}\,z_2}\;.\label{equ:PFS-SGammapluse}
	\end{align}
	$\Seu_{+,p,0}^{\delta k}$ is the zeroth order term of the Taylor expansion with respect to $\delta$ of $\Gamma^{\delta k}_{+,p}(z,x)$, but without the $e^{i\delta (k_2z_2-k_1z_1)}$ term, and $\Seu_{+,e,0}^{\delta k}$ is the zeroth order term of the Taylor expansion with respect to $\delta$ of $\Gamma^{\delta k}_{+,e}(z,x)$, but without the $e^{-i\delta k_1z_1}$ term, which is located before the evanescent sum in  (\ref{equ:PFS-SGammapluse}). To see that, write $\Seu_{+,e,0}^{\delta k}$ using
	\begin{align}\label{equ:PFS-SGammaplus0}
		\Gamma^{0}_+(z,x)=-\frac{x_2}{p}-\sum_{\substack{n\in\ZZ\setminus\{0\}\\ l\DEF 2\pi n/p}} \frac{1}{p|l|}e^{il(z_1-x_1)}e^{-|l|z_2}\sinh(|l|x_2)\,,
	\end{align}	
	which is given through  (\ref{equ:ch2:k=0GammaSharp}) for $z_2>1$, and rewrite $\int_\Lambda\del_\nu \udnoboldk(y)\Gamma^{0}_+(z,y)\intd \sigma_y$ with it. Then we see that the sum in  (\ref{equ:PFS-SGammaplus0}) is exactly the zeroth order term of the Taylor expansion. With that, we obtain the formula in  (\ref{equ:Seu_e0Extended}).
	
	Inserting these exact formulas into the expressions in Lemma \ref{lemma:S-S} and using
	\begin{align*}
		\inteps \mu(t) \phi(t) \intd t = 
			\inteps \phi(t) \frac{\sqrt[4]{\eps^2-t^2}}{\sqrt[4]{\eps^2-t^2}} \mu(t)\intd t\leq
			\NORM{\phi}_{\cC^0}\NORM{\mu}_{\curlXe}\sqrt{\pi}\,,
	\end{align*}
	where $\phi\in\cC^0{([-\eps,\eps])}$ and where we used the $\Leu^2$-Cauchy-Schwarz inequality, yields the desired estimations for $\Seu_{+}^{\delta k}$. For $\Teu_{+}^{\delta k}$ it works analogously. %\todo{Check this at the end again}
\end{proof}

We define
\begin{align*}
	u^{\delta k}_{\Seu_p^\star}(z)\DEF\,& \Seu_{+,p,0}^{\delta k}[\mudkstar](z)\,,\\
		% \quadu ^{\delta k}_{\Seu_p^\sim}(z)\DEF\, \Seu_{+,p,0}^{\delta k}[\mudkt](z)\,,\\
	u^{\delta k}_{\Seu_e^\star}(z)\DEF\,& \Seu_{+,e,0}^{\delta k}[\mudkstar](z)\,,\\
		%\quad u^{\delta k}_{\Seu_e^\sim}(z)\DEF\, \Seu_{+,e,0}^{\delta k}[\mudkt](z)\,,\\
	u^{\delta k}_{\Teu_p^\star}(z)\DEF\,& \Teu_{+,p,0}^{\delta k}[\mudkstar](z)\,, \\
		%\quad u^{\delta k}_{\Teu_p^\sim}(z)\DEF\, \Teu_{+,p,0}^{\delta k}[\mudkt](z)\,,\\
	u^{\delta k}_{\Teu_e^\star}(z)\DEF\,& \Teu_{+,e,0}^{\delta k}[\mudkstar](z)\,.
		%\quad u^{\delta k}_{\Teu_e^\sim}(z)\DEF\, \Teu_{+,e,0}^{\delta k}[\mudkt](z)\,.
\end{align*}
We then have the following proposition:
\begin{proposition}\label{prop:LinftyNormEstimateofu}
	Let $V_r\DEF\{ z\in\RR^2_+\MID z_2>r \}$. There exists a constant $C_{(\ref{equ:LinftyNormEstimateofu})}>0$ such that
	\begin{align}\label{equ:LinftyNormEstimateofu}
		&\NORM{\udks\!-\!(u^{\delta k}_{\Seu_p^\star}\!+\!u^{\delta k}_{\Seu_e^\star}\!+\!u^{\delta k}_{\Teu_p^\star}\!+\!u^{\delta k}_{\Teu_e^\star}\!+\!u^{\delta k}_{\mathrm{RHS},p}\!+\!u^{\delta k}_{\mathrm{RHS},e})}_{\Leu^\infty(V_r)}\nonumber\\
		&\qquad+\frac{1}{\delta}\NORM{\nabla\left[\udks\!-\!(u^{\delta k}_{\Seu_p^\star}\!+\!u^{\delta k}_{\Seu_e^\star}\!+\!u^{\delta k}_{\Teu_p^\star}\!+\!u^{\delta k}_{\Teu_e^\star}\!+\!u^{\delta k}_{\mathrm{RHS},p}\!+\!u^{\delta k}_{\mathrm{RHS},e})\right]}_{\Leu^\infty(V_r)}
		\\
		&\qquad\leq C_{(\ref{equ:LinftyNormEstimateofu})}\left(\delta\ceps^3\resSumWOfirstfrac+\delta\ceps^2 + \delta\eps^1 +\delta^2\right)\,,\nonumber
	\end{align}
	for $\delta$, $\eps$ small enough and $r$ large enough.
\end{proposition}

\begin{proof}
	According to (\ref{equ:us=urhs+S+T}), we have
	\begin{align*}
		\udks-u^{\delta k}_{\mathrm{RHS},p}-u^{\delta k}_{\mathrm{RHS},e}
		= \left(\Seu_{+,e}^{\delta k}+\Seu_{+,p}^{\delta k}+\Teu_{+,e}^{\delta k}+\Teu_{+,p}^{\delta k}\right)[u^{\delta k}|_\Lambda]+\OO(\delta^2)\,.
	\end{align*}
	With Lemma \ref{lemma:S-S} and the fact that $u^{\delta k}\MID_\Lambda= \mudkstar+\mudkt+\OO(\delta^2)$, $\NORM{\mudkstar}_{\curlXe}=\OO(\delta)$ and $\NORM{\mudkt}_{\curlXe}=\OO(\delta)$, we readily proved Proposition \ref{prop:LinftyNormEstimateofu}.
\end{proof}

\begin{proof}[Theorem \ref{THM2:1HR}]
	We see that $u^{\delta k}_{\Seu_e^\star}$, $u^{\delta k}_{\Teu_e^\star}$ and $u^{\delta k}_{\mathrm{RHS},e}$ are exponentially decaying in $z_2$ and are of order $\OO(\delta)$, thus their $\Leu^\infty(V_r)$-norm is of order $\delta e^{-C\,r}$ for some constant $C>0$. Then with Proposition \ref{prop:LinftyNormEstimateofu} and the change into the macroscopic view with $\Uk(x)=\udk(x/\delta)$, we obtain Theorem \ref{THM2:1HR}.
\end{proof}

\subsubsection{Evaluating the Impedance Boundary Condition}\label{subsec:IBC}

We switch back to the macroscopic variable $\Uk(x)=\udk(x/\delta)$. We approximate our solution in the far-field with the function
\begin{align*}
	\Ukapp(z)
		\DEF& (\udknull-\udknullcircP)(z/\delta)+u^{\delta k}_{\Seu_p^\star}(z/\delta)+u^{\delta k}_{\Teu_p^\star}(z/\delta) +u^{\delta k}_{\mathrm{RHS},p}(z/\delta) \nonumber \\
		=& -2ia_0e^{-ik_1z_1}\sin(k_2z_2)+ e^{i( k_2 z_2 - k_1 z_1)}\,C^{\delta k}_{(\ref{const:Uappconstant})}\,,\label{equdef:Uapp:2}
\end{align*}
where
\begin{align}\label{const:Uappconstant}
	C^{\delta k}_{(\ref{const:Uappconstant})}
		\DEF& 	-\frac{h}{p}\int_\Lambda \mudkstar(y)\intd \sigma_y
					+\frac{1}{p}\int_\Lambda\int_{\del D}\mudkstar(y)\NdelOonedkc(y,w)\nu_w\!\cdot\!\begin{pmatrix}0\\1\end{pmatrix}\,\intd \sigma_w\intd \sigma_y \nonumber\\
			&		+\int_{\del D} \del_\nu (\udknull-\udknull\circ \opP)(y)\frac{\sin(\delta k_2y_2)\,e^{i\delta k_1y_1}}{\delta k_2 p}\intd \sigma_y \\
			&		-\int_{\del D}\int_{\del D}\del_\nu (u^{\delta k}_0-u^{\delta k}_0\!\!\circ\! \opP)(y) \NdelOonedkc(y,w)\nu_{w}\cdot\begin{pmatrix}\frac{i\,k_1}{p\,k_2}\sin(\delta k_2 w_2)\\ \frac{1}{p}\cos(\delta k_2w_2)\end{pmatrix}e^{i\delta k_1 w_1}\intd \sigma_w\intd \sigma_y \nonumber\,.
\end{align}
Consider that $C^{\delta k}_{(\ref{const:Uappconstant})}=\OO(\delta)$ since $\mudkstar=\OO(\delta)$ and $(\udknull\!-\!\udknull\!\circ\! \opP)=\OO(\delta)$ and the other factors are of size $\OO(1)$.

We want that $\Ukapp$ satisfies the impedance boundary condition, that is $\Ukapp(z) + \delta \cIBC\del_{z_2}\Ukapp(z)=0$	at $\del\RR^2_+$. From (\ref{equdef:Uapp:2}), we obtain the condition
\begin{align*}
	e^{-i k_1 z_1}\,C^{\delta k}_{(\ref{const:Uappconstant})}
	+\delta \cIBC (-2ia_0k_2 e^{-ik_1z_1}+ i\,k_2 e^{-i k_1 z_1}\,C^{\delta k}_{(\ref{const:Uappconstant})})=0 \,,\quad \text{for all }z_2\in\RR\,.
\end{align*}
After rearranging the terms, we obtain
\begin{align*}
	\cIBC = \frac{-C^{\delta k}_{(\ref{const:Uappconstant})}}{i\,\delta k_2(\,C^{\delta k}_{(\ref{const:Uappconstant})}-2 a_0)}\,.
\end{align*}
Using that $\frac{1}{1+\OO(\delta)}=1-\OO(\delta)$, we have
\begin{align*}
	\cIBC = \frac{C^{\delta k}_{(\ref{const:Uappconstant})}}{2ia_0\,\delta k_2}+\OO(\delta)\,.
\end{align*}
This proves Theorem \ref{THM3:1HR}.

\subsection{Numerical Illustrations}
In this subsection we compute the impedance boundary condition constant $\cIBC$ with numerical means using Theorem \ref{THM3:1HR}. We use two geometries, both rectangles, but with different sizes and for each geometry we have different ranges for the wave vector $k$ and the gap length $\eps$. This is because the resonance value $k^{\delta,\eps}_{\mathrm{res}}\DEF k^{\delta, \eps}_{+}$ is proportional to the square root of the geometry area, and the same holds for the width of the resonance peak of $\cIBC$. However, we do not have to consider different values of $\delta$, since according to Theorem \ref{THM3:1HR} all computations are done with the input $\delta k$, thus a different $\delta$ would only scale the first coordinate axis, but would not change the shape of the curve itself. We fix $\delta=0.01$.

We implement $\GKS$ using Edwald's Method see \cite{Ewald} or \cite[Chapter 7.3.2]{LPTSA}. We implement the remainders using Lemmas \ref{prop:(I+K)[R]=S[G]}, \ref{prop:FormulaForRdelEK} and  \ref{lemma:1=2K[1]}. 

The first geometry has the following set-up. It is a rectangle with length $0.9$ and height $0.9$. The period is $p=1$ and $h=1$. The amplitude of the incident wave is $a_0=1$. The number of points, with which we approximate the boundary of the rectangle, is $200$. For the wave vector $k=k_2$, we take $341$ equidistant points on the interval $[30, 200]$. For $\eps$ we pick 5 values, those are $\{ 0.1, 0.05, 0.03, 0.01, 0.001 \}$. The result can be seen in Figure \ref{fig:cIBCPlotGeometry1_1HR}.

The second geometry has the following set-up. It is a rectangle with length $0.2$ and height $0.3$. The period is $p=1$ and $h=0.5$. The amplitude of the incident wave is $a_0=1$. The number of points, with which we approximate the boundary of the rectangle, is $200$. For the wave vector $k=k_2$, we take $301$ equidistant points on the interval $[100, 400]$. For $\eps$ we pick 5 values, those are $\{ 0.1, 0.05, 0.03, 0.01, 0.001 \}$. Consider that the case $\eps=0.1$ means that the whole upper boundary of our rectangle is the gap. The result can be seen in Figure \ref{fig:cIBCPlotGeometry2_1HR}.
\begin{figure}[h]
    \centering
    \includegraphics{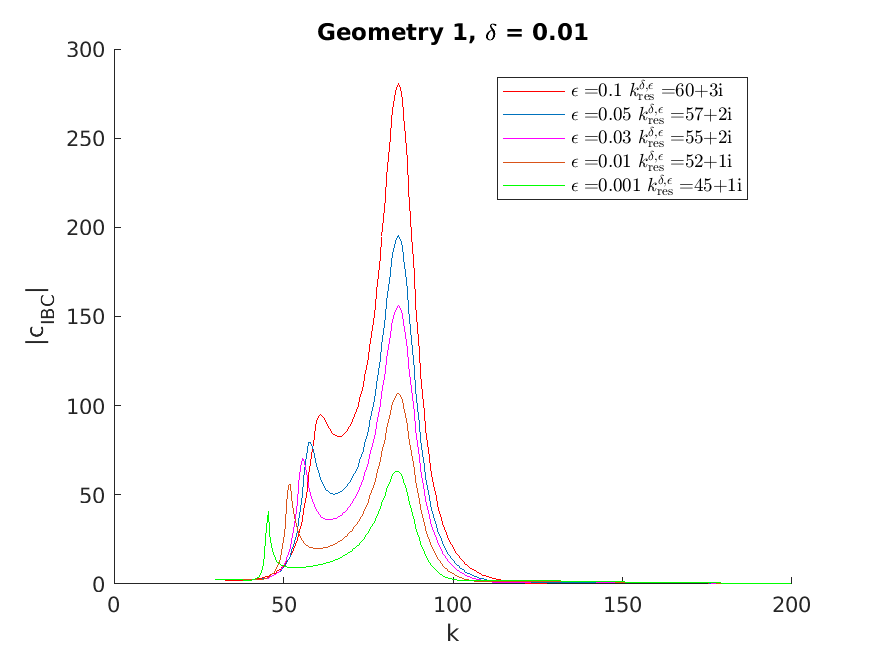}
    \caption[Geometry 1 Plot]{The plot of the absolute value of the variable $\cIBC$ depending on the wave vector $k$ for the first geometry. For every value of $\eps$, the rounded value of the resonance value $k^{\delta \eps}_{\mathrm{res}}$ is displayed.}
    \label{fig:cIBCPlotGeometry1_1HR}
\end{figure}

Consider that in Figure \ref{fig:cIBCPlotGeometry1_1HR}, for $\eps\leq 0.01$, the resonance splits, so we get two peaks, due to the geometry, which is large enough that the neighbouring resonators have an effect on the main one. In the second geometry however, it seems the resonators have a large enough distance from each other with respect to their width, such that no the neighbouring resonators do not affect the main one. See Figure \ref{fig:cIBCPlotGeometry2_1HR}.

\begin{figure}[h]
    \centering
    \includegraphics{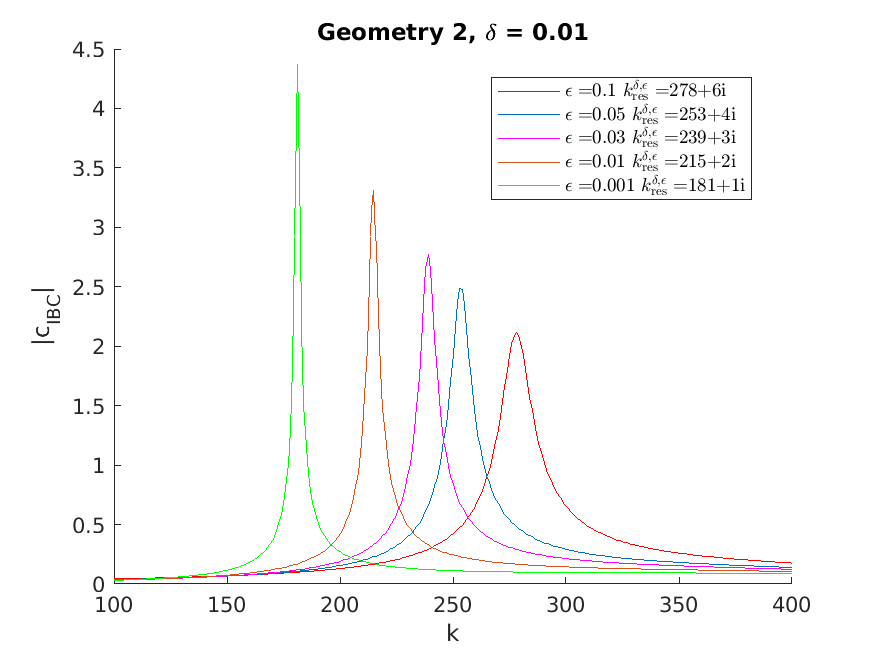}
    \caption[Geometry 2 Plot]{The plot of the absolute value of the variable $\cIBC$ depending on the wave vector $k$ for the second geometry. For every value of $\eps$, the rounded value of the resonance value $k^{\delta \eps}_{\mathrm{res}}$ is displayed.}
    \label{fig:cIBCPlotGeometry2_1HR}
\end{figure}

\section{Two Periodically Arranged Helmholtz Resonators}\label{ch:2HR}
In this section, we look at two domains $D_1$ and $D_2$ both bounded and simply connected domains, which have the same height $h$. We repeat $D_1\cup D_2$ periodically along the $x_1$-axis, with period $p$ and scale the whole geometry by a factor of $\delta$. Additionally, $D_1$ and $D_2$ have each a gap called $\Lambda_1$ and $\Lambda_2$ on their boundary, which allows the incident wave $\Uknull$ to pass through. The incident wave rebounds inside $D_1$ and $D_2$ and leaves at the gaps, which then leads to the scattered wave $\Uks$. 

We will have a good approximation of the scattered wave in the far-field. Moreover, this approximation satisfies the Helmholtz equation with a Robin boundary condition at the $x_1$-axis which approximates a Dirichlet or a Neumann boundary condition depending on the magnitude of the incoming wave vector $\boldk$ and $\delta$.

\subsection{Mathematical Description of the Physical Problem}
\subsubsection{Geometry}
\begin{figure}[h]
  \begin{subfigure}{0.49\textwidth}
    \centering
    \includegraphics[width=0.99\textwidth]{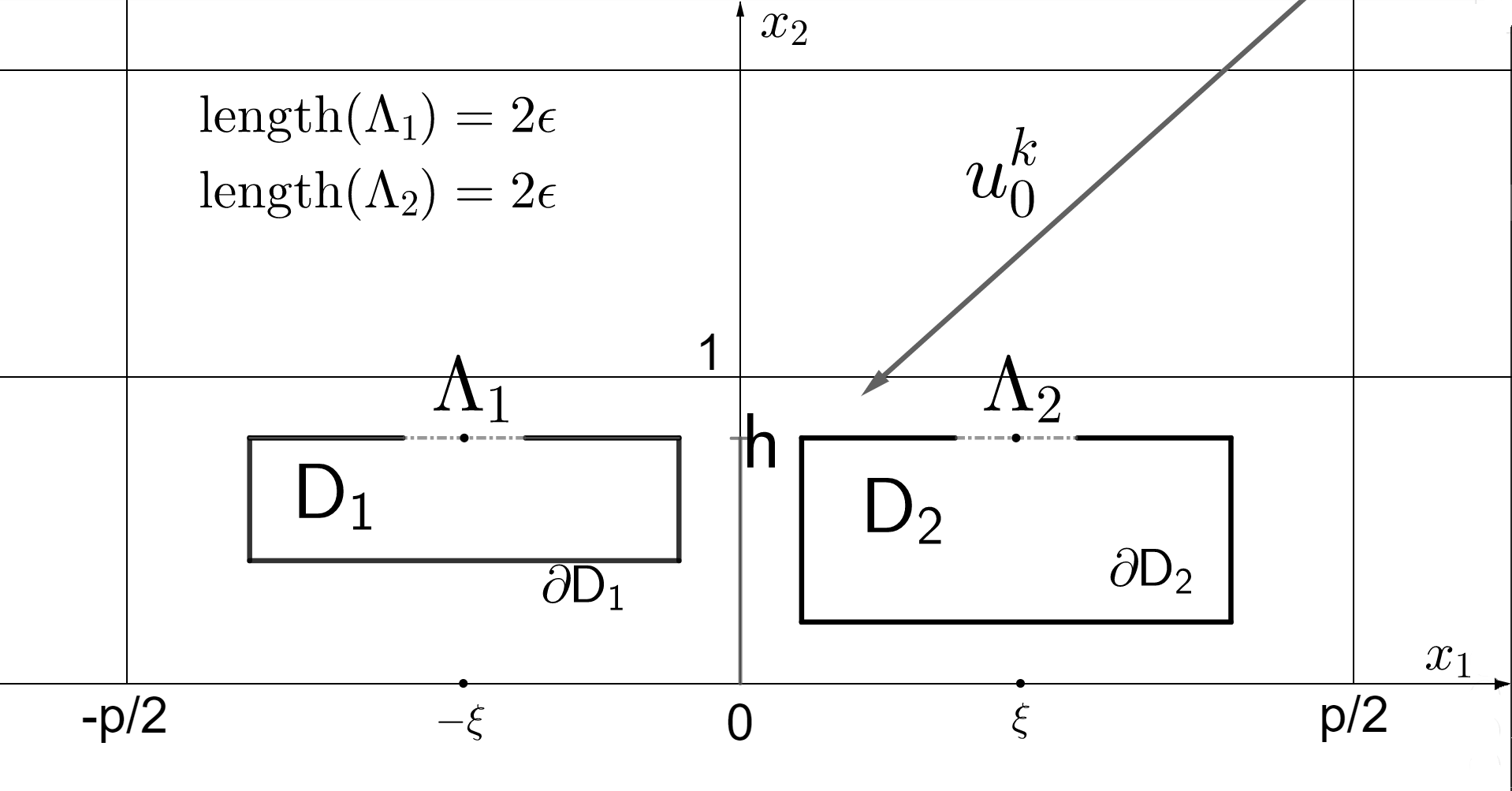}
    \caption[NonPeriodic1HR]{Microscopic, non-periodic view of our two Helmholtz resonators. The Helmholtz resonators are contained in the unit cell $(-p/2,p/2)\times(0,1)$. They have the gaps $\Lambda_1$ and $\Lambda_2$, both of length $2\eps$, which are parallel to the $x_1$ axis and centered at $(-\xi,h)^\TransT$ respectively $(\xi,h)^\TransT$, where $h\in (0,1)$. $u_0^k$ denotes the incident wave. $D_1$ and $D_2$ have not to be rectangular in shape.}
    \label{fig:NonPeriodic2HR}
  \end{subfigure}\hfill %half fill
  \begin{subfigure}{0.49\textwidth} 
    \centering
    \includegraphics[width=0.99\textwidth]{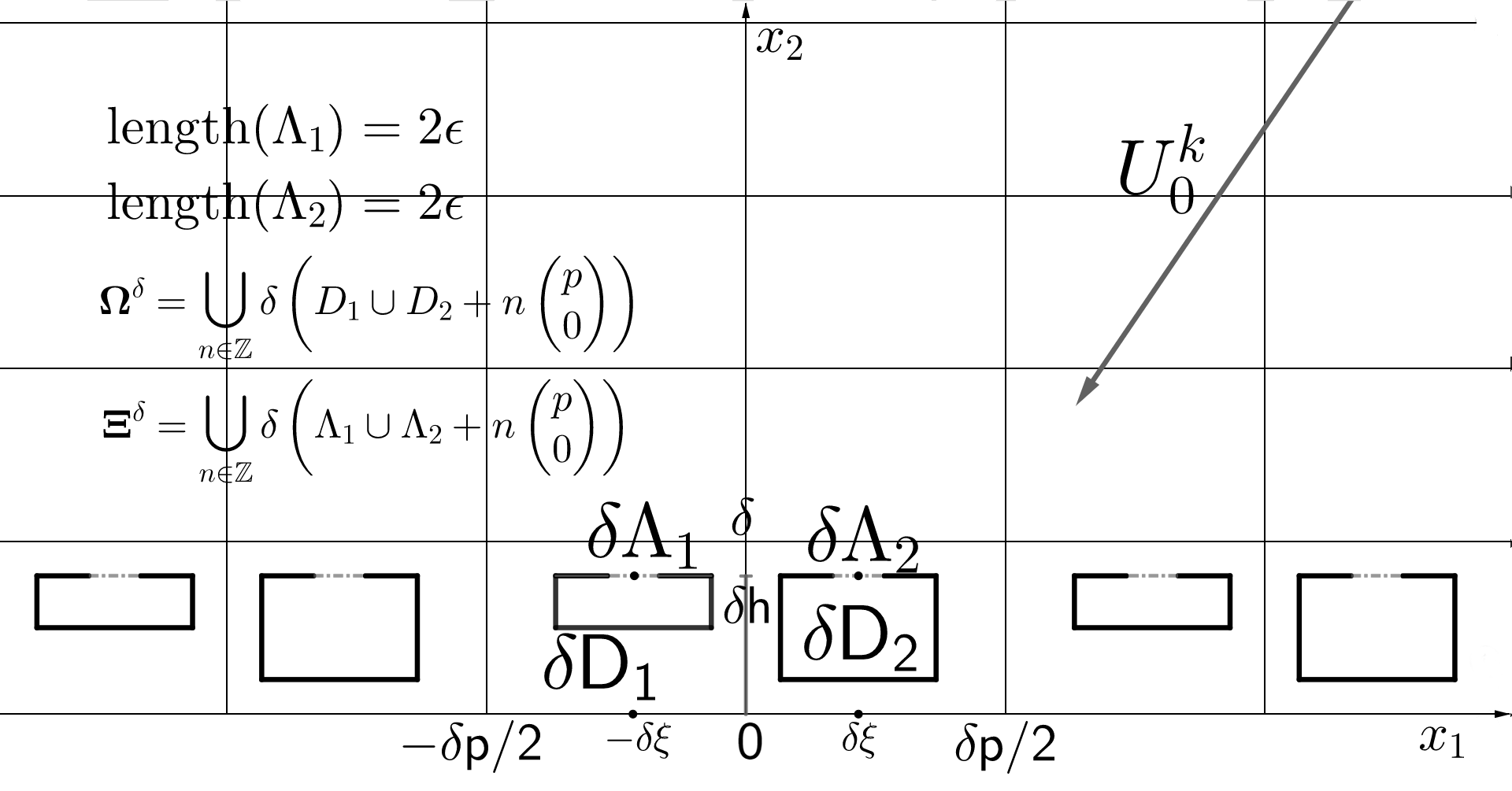}
    \caption[Periodic1HR]{Macroscopic view of our periodically arranged Helmholtz resonators, with periodicity $\delta p$. All Helmholtz resonators have the form of the Helmholtz resonators depicted in $(a)$, but are scaled with the factor $\delta$. $U_0^k$ denotes the incident wave. $\bm{\Omega}^\delta$ is the collection of all Helmholtz resonators and $\bm{\Xi}^\delta$ is the collection of all gaps.}
    \label{fig:Periodic2HR}
  \end{subfigure}
  \caption{The physical setup. In $\mathrm{(a)}$ we have the microscopic, non-periodic view. In $\mathrm{(b)}$ we have the macroscopic, periodic view.}
\end{figure}

Before we consider the periodic and macroscopic problem, we first define the geometry of our Helmholtz resonators in the unit cell. Let $D_1, D_2\Subset (-p/2,p/2)\times(0,1)$ be two open, bounded, and simply connected domains, such that $\overline{D_1}$ and $\overline{D_2}$ do not intersect and where $p\in\RR$ is close enough to $1$. We assume that $D_1\cup D_2$ is a $\cC^2$-domain. We define $\Lambda_1\subset \del D$ and $\Lambda_2\subset \del D$ to be the gap of $D_1$ and $D_2$ respectively, where $\Lambda_1$ and $\Lambda_2$ are both line segment parallel to the $x_1$-axis. $\Lambda_1$ is centered at $(-\xi,h)^\TransT$ and $\Lambda_2$ is centered at $(\xi,h)^\TransT$, where $h\in (0,1)$ and $\xi\in (0,p)$, and $\Lambda_1$ and $\Lambda_2$ have both length $2\eps$, where $\eps\in(0,1)$ small enough. To facilitate future computations we assume that $h/p\geq 1/2$. 

Now we define the macroscopic view, that is, we shrink our domain by the factor $\delta\in (0,\infty)$. We define the collection of periodically arranged Helmholtz resonators $\BOm^\delta$, with period $\delta p$, and the collection of gaps of those Helmholtz resonators $\BXi^\delta$, where a single gap has length $2\delta\eps$, as
\begin{align*}
   \BOm^\delta&\DEF\bigcup_{n\in\mathbb{Z}}\delta\left(D_1\cup D_2+n\begin{pmatrix}p\\ 0 \end{pmatrix}\right),\\
   \BXi^\delta&\DEF\bigcup_{n\in\mathbb{Z}}\delta\left(\Lambda_1\cup\Lambda_2+n\begin{pmatrix}p\\ 0 \end{pmatrix}\right).
\end{align*} 

\subsubsection{Incident Wave}
Let $\boldk\DEF(k_1,k_2)^\TransT\in\RR^2$ be the wave vector. We will fix the direction of the wave vector, that is $k_1/k$ and $k_2/k$, where $k\DEF|\boldk|\DEF(k_1^2+k_2^2)^{1/2}\in [0,\infty)$, but let the magnitude $k$ vary. With that, we define the function $\Uknull:\RR^2\rightarrow \CC$ as
\begin{align*}
   \Uknull(x)\DEF a_0 e^{-ik_1x_1}e^{-ik_2x_2},
\end{align*}
where $a_0\in \RR$ denotes the amplitude. $\Uknull$ will be our incident wave. We define the parity operator $\opP:\RR^2\rightarrow \RR^2$ as 
\begin{align*}
   \opP(x_1,x_2)=\begin{pmatrix}x_1\\-x_2\end{pmatrix}.
\end{align*}
With this we have the reflected incident wave
\begin{align*}
   \Uknull\circ \opP(x)= a_0 e^{-ik_1x_1}e^{ik_2x_2}.
\end{align*}
Moreover, 
\begin{align*}
   (\Uknull-\UknullcircP)(x)=-2ia_0e^{-ik_1x_1}\sin(k_2x_2).
\end{align*}
We will also need the following equation
\begin{align*}
   \nabla (\Uknull-\UknullcircP)(x)=\begin{pmatrix}-2a_0k_1e^{-ik_1x_1}\sin(k_2x_2)\\-2ia_0k_2e^{-ik_1x_1}\cos(k_2x_2)\end{pmatrix}.
\end{align*}
Consider also that $\Uknull$ and $\UknullcircP$ are quasi-periodic with quasi-momentum $-k_1p$, that is
\begin{align*}
   \Uknull\left( x+\begin{pmatrix}p\\0\end{pmatrix} \right)&=e^{-ik_1p}\Uknull(x),\\
   \UknullcircP\left( x+\begin{pmatrix}p\\0\end{pmatrix} \right)&=e^{-ik_1p}\UknullcircP(x).
\end{align*}

\subsubsection{The Resulting Wave}
With the geometry and the incident wave, we model the electromagnetic scattering problem and the resulting wave $\Uk:\RR^2_+\setminus\del\BOm^\delta\rightarrow\CC$ by the following system of equations:
\begin{equation} \label{equ:ch3:ThePDEforuk}
	\left\{ 
	\begin{aligned}
		 \left( \Laplace + k^2  \right) \Uk &= 0 \quad &&\text{in} \quad \RR_+^2\setminus \del\BOm^\delta, \\
		 \Uk \MID_+\! - \Uk\MID_- &= 0 \quad &&\text{on} \quad \BXi^\delta ,\\
		 \del_\nu \Uk \MID_+\! - \del_\nu \Uk\mid_- &= 0 \quad &&\text{on} \quad \BXi^\delta ,\\
		 \del_\nu \Uk \MID_+\! &= 0 \quad &&\text{on} \quad \del\BOm^\delta\setminus\BXi^\delta ,\\
		 \del_\nu \Uk \MID_-\! &= 0 \quad &&\text{on} \quad \del\BOm^\delta\setminus\BXi^\delta ,\\
		 \Uk&=0 \quad &&\text{on} \quad \del \RR_+^2,
	\end{aligned}
	\right.
\end{equation}
where $\cdot\!\mid_+$ denotes the limit from outside of $\BOm^\delta$ and $\cdot\!\mid_-$ denotes the limit from inside of $\BOm^\delta$, and $\del_\nu$ denotes the normal derivative on $\del\BOm^\delta$. Similar to diffraction problems for gratings, the above system of equations is complemented by a certain outgoing radiation condition imposed on the scattered field $\Uks\DEF \Uk-(\Uknull-\UknullcircP)$ and quasi-periodicity on $\Uk$. More precisely, we are interested in the quasi-periodic solutions, that is
\begin{align*}
   \Uk\left( x+ \begin{pmatrix}p\\0\end{pmatrix} \right) =e^{-ik_1p}\Uk(x) \quad\text{for}\quad x\in \RR_+^2,
\end{align*}
and solutions satisfying the outgoing radiation condition, thus
\begin{align*}
   \left| \del_{x_2} \Uks -ik_2\Uks \right|\rightarrow 0\quad\text{for}\quad x_2\rightarrow\infty.
\end{align*}

Then the outgoing radiation condition can be imposed by assuming that all the modes in the Rayleigh-Bloch expansion are either decaying exponentially or propagating along the $x_2$-direction. Since in our case we assume that the period of the resonator structure $\delta p$ is much smaller than $k$, the outgoing radiation condition takes the following specific form:
\begin{align*}
   (\Uk-\Uknull)(x) &\sim ae^{-ik_1x_1}e^{ik_2x_2}\quad \text{as}\quad x_2\rightarrow\infty,\\
   \Uks(x) &\sim (a+1)e^{-ik_1x_1}e^{ik_2x_2}\quad \text{as}\quad x_2\rightarrow\infty,
\end{align*}
for some constant amplitude $a\in\RR$.

Consider also that in absence of Helmholtz resonators the solution to (\ref{equ:ch3:ThePDEforuk}) is given by $\Uknull-\UknullcircP$.

\subsubsection{The Resulting Wave in the Microscopic View}\label{sec:reswaveinmicroscopicview:2HR}
Given the resulting wave $\Uk(x)$, the function $\udk(x): \RR^2_+\setminus\del\BOm^1\rightarrow\CC$, $\udk(x)\DEF \Uk(\delta x)$ represents the resulting wave, but where the Helmholtz resonators are scaled-back and thus are of height $h$ and not $\delta h$. $\udk$ satisfies
\begin{equation} \label{equ:ch3:ThePDEfortudk}
	\left\{ 
	\begin{aligned}
		 \left( \Laplace + (\delta k)^2  \right) \udk &= 0 \quad &&\text{in} \quad \RR_+^2\setminus \del\BOm^1, \\
		 \udk \MID_+\! - \udk\mid_- &= 0 \quad &&\text{on} \quad \BXi^1 ,\\
		 \del_\nu \udk \MID_+\! - \del_\nu \udk\mid_- &= 0 \quad &&\text{on} \quad \BXi^1 ,\\
		 \del_\nu \udk \MID_+\! &= 0 \quad &&\text{on} \quad \del\BOm^1\setminus\BXi^1 ,\\
		 \del_\nu \udk \MID_-\! &= 0 \quad &&\text{on} \quad \del\BOm^1\setminus\BXi^1 ,\\
		 \udk&=0 \quad &&\text{on} \quad \del \RR_+^2,
	\end{aligned}
	\right.
\end{equation}
where $\cdot\!\mid_+$ denotes the limit from outside of $\BOm^1$ and $\cdot\!\mid_-$ denotes the limit from inside of $\BOm^1$, and $\del_\nu$ denotes the normal derivative on $\del\BOm^1$.

We can adopt the quasi-periodicity from the macroscopic view and obtain
\begin{align*}
   \udk\left( x+ \begin{pmatrix}p\\0\end{pmatrix} \right) =e^{-i\delta k_1p}\udk(x) \quad\text{for}\quad x\in \RR_+^2.
\end{align*}
Defining $\udks\DEF \udk-(\udknull-\udknull\!\circ\! \opP)$ we also get that
\begin{align*}
   \left| \del_{x_2} \udks -i\delta k_2\udks \right|\rightarrow 0\quad\text{for}\quad x_2\rightarrow\infty.
\end{align*} 

We see that $\udk$ solves the same partial differential equation like $\Uk$ in the re-scaled geometry, but with the scaled wave vector $\delta k$. We will see that we can express $\udk$ as an expansion in terms of $\delta$ and we will give an analytic expression for the first order term.

\subsection{Main Results}

We assume that $\delta k\in \kkkk \DEF \{ \hat{k} \in \RR \MID  0\neq |\hat{k}|< k_{{D_1\cup D_2}, \min, \Laplace}/2 \text{ and } |\hat{k}|^2<\inf\{|l-\hat{k}\,e_1|^2\MID l\DEF 2\pi n/p,\; n\in\ZZ\setminus\{0\}\}\}$, where $k_{{D_1\cup D_2}, \min, \Laplace}$ is defined as the first non-zero eigenvalue of the operator $-\!\Laplace$ with Neumann conditions on the boundary $\del D_1\cup\del D_2\,$. If we would extend the domain $\kkkk$ to $\kkkkc\DEF\{k_\ast\in \CC\MID \sqrt{k_\ast \bar{k}_\ast}< k_{{D_1\cup D_2},\min,\Laplace}/2\}$, we would obtain the following resonance values for our physical problem:
\begin{theorem}\label{THM1:2HR}
	There exists exactly four resonance values in $\kkkkc$ for $\Uk$. These are for $j\in\{1,2\}$
	\begin{align*}
		\delta k^{\delta, \eps}_{j,+} =& \delta k^{\delta,\eps}_{\star,j}\bigg(1+\frac{(\pi\leps)^{\tfrac{3}{2}}}{4}\mathrm{e}_j^\TransT(\BrmYss)^{-1}\BA_{\star, (1)}(\rmYss)_j\bigg)+\OO(\leps^{5/2})\,,\\
		\delta k^{\delta, \eps}_{j,-} =& \delta k^{\delta,\eps}_{\star,j}\bigg(-\!1+\frac{(\pi\leps)^{\tfrac{3}{2}}}{4}\mathrm{e}_j^\TransT(\BrmYss)^{-1}\BA_{\star, (1)}(\rmYss)_j\bigg)+\OO(\leps^{5/2})\,,
	\end{align*}
	where 
	\begin{align*}
		\delta k^{\delta,\eps}_{\star, 1} 
			=&	\sqrt{\frac{\pi}{2}\leps}
				\left(
					\frac{1}{2}
					\left(
						k^{\eps}_{\star,\mathrm{tr1}}
						+k^{\eps}_{\star,\mathrm{tr2}}
					\right)
					+
					\left[
						\frac{1}{4}
						\left(
							k^{\eps}_{\star,\mathrm{tr1}}
							-k^{\eps}_{\star,\mathrm{tr2}}
						\right)^2
						-k^{\eps}_{\star,\mathrm{det}}
					\right]^{1/2}
				\right)\,,\\
		\delta k^{\delta,\eps}_{\star, 2} 
			=&	\sqrt{\frac{\pi}{2}\leps}
				\left(
					\frac{1}{2}
					\left(
						k^{\eps}_{\star,\mathrm{tr1}}
						+k^{\eps}_{\star,\mathrm{tr2}}
					\right)
					-
					\left[
						\frac{1}{4}
						\left(
							k^{\eps}_{\star,\mathrm{tr1}}
							-k^{\eps}_{\star,\mathrm{tr2}}
						\right)^2
						-k^{\eps}_{\star,\mathrm{det}}
					\right]^{1/2}
				\right)\,,
	\end{align*}
	\begin{align*}
		k^{\eps}_{\star,\mathrm{tr1}} 
				=&	\frac{1}{\sqrt{|D_1|}}\left( 1+\frac{\Balphanull_{1,1}\pi\leps}{4} \right)\,,\quad
		k^{\eps}_{\star,\mathrm{tr2}} 
				=	\frac{1}{\sqrt{|D_2|}}\left( 1+\frac{\Balphanull_{2,2}\pi\leps}{4} \right)\,,\\
		k^{\eps}_{\star,\mathrm{det}} 
				=&	\frac{\pi^2\leps^2}{4}\frac{\BTnull_{1,2}\BTnull_{2,1}}{(\sqrt{|D_1|}+\sqrt{|D_2|})}\,, \quad
		\leps
				= \frac{-1}{\log(\eps/2)}\,.
	\end{align*}
	Here, $\BrmYss=[(\rmYss)_1\,, (\rmYss)_2]$, with $(\rmYss)_j$ being the normalized eigenvector to the eigenvalue $\delta k^{\delta,\eps}_{\star,j}$ of $\BAdkess$, where $\BAdkess$ is given in  (\ref{equ:BAestarstar}) and $\BA_{\star,D}$, $\BA_{\star,(0)}$ and $\BA_{\star,(1)}$ are given in Lemma \ref{lemma:BAsqrtleps3/2Expansion}, where $\BS_{2,1}$, $\BS_{1,2}$, $\BT_{2,1}$, $\BT_{1,2}$, $\Balphaone_{2,2}$, $\Balphaone_{1,1}$, $\Balphanull_{2,2}$ and $\Balphanull_{1,1}$ are given in (\ref{equdef:Balphanull11})-(\ref{equdef:BS21}).
\end{theorem}

We have the following approximation for the resulting wave $\Uk$:
\begin{theorem}\label{THM2:2HR}
	Let $V_r\DEF\{ z\in\RR^2_+\MID z_2>r \}$. There exist constants $C_{(\ref{Thm2:2HR})},\widetilde{C}_{(\ref{Thm2:2HR})}>0$ such that
	\begin{align}\label{Thm2:2HR}
		&			\NORM{\Uks\!-\!(U^k_{\Seu_p^\star}\!+\!U^k_{\Teu_p^\star}\!+\!U^k_{\mathrm{RHS},p})}_{\Leu^\infty(V_r)}\!\!
					+\!\NORM{\nabla\left[\Uks\!-\!(U^k_{\Seu_p^\star}\!+\!U^k_{\Teu_p^\star}\!+\!U^k_{\mathrm{RHS},p})\right]}_{\Leu^\infty(V_r)}\\
		&\quad		\leq C_{(\ref{Thm2:2HR})}\left(\!\frac{\delta}{|\log(\eps)|^3}\Big(\norm{(\BrmMdkone\BrmMdktwo)^{-1}}\Big)\!+\!\frac{\delta}{|\log(\eps)|^2}\!+\! \delta\eps +\delta\,e^{-\widetilde{C}_{(\ref{Thm2:2HR})}r}+\delta^2\right)\,,\nonumber
	\end{align}
	for $\eps$ small enough, where 
	\begin{align*}
		\Uk_{\Seu_p^\star}(z)
			=\,& 		-e^{i(k_2 z_2-k_1 z_1)}\frac{h}{p}\sum_{j\in\{1,2\}}\int_{\Lambda_j} (\Bmudkstar)_j(y)\intd \sigma_y\,,\\
		\Uk_{\Teu_p^\star}(z)
			=\,&	 	e^{i(k_2 z_2-k_1 z_1)}\frac{1}{p}\sum_{j,\hat{j}\in\{1,2\}}\int_{\Lambda_j}\int_{\del D_{\hat{j}}}(\Bmudkstar)_j(y)\NdelBOonedkc(y,w)\nu_w\!\cdot\!\begin{pmatrix}0\\1\end{pmatrix}\,\intd \sigma_w\intd \sigma_y\,,\\
		\Uk_{\mathrm{RHS},p}(z)
			=\,& 		e^{i\delta (k_2z_2\!-\!k_1z_1)}\left[\sum_{j\in\{1,2\}}\int_{\del D_j} \del_\nu (\udknull-\udknull\circ \opP)(y)\frac{\sin(\delta k_2y_2)\,e^{i\delta k_1y_1}}{\delta k_2 p}\intd \sigma_y \right.\\
			&\mkern-100mu	\left.-\!\!\!\sum_{j,\hat{j}\in\{1,2\}}\int_{\del D_j}\int_{\del D_{\hat{j}}}\del_\nu (u^{\delta k}_0-u^{\delta k}_0\!\!\circ\! \opP)(y) \NdelBOonedkc(y,w)\nu_w\cdot\begin{pmatrix}\frac{i\,k_1}{p\,k_2}\sin(\delta k_2 w_2)\\ \frac{1}{p}\cos(\delta k_2w_2)\end{pmatrix}e^{i\delta k_1 w_1}\intd \sigma_w\intd \sigma_y \right]\nonumber.
	\end{align*}
	Here,  for $y=((t_1-\xi,h)^\TransT,(t_2+\xi,h)^\TransT)\in\Lambda_1\times\Lambda_2$, we have
	\begin{align*}
		\Bmu_\star(y)
			=&  	\begin{bmatrix} 
					\frac{-1}{\sqrt{\eps^2-t_1^2}} \bigg(\leps^2\frac{\pi}{4}
					\Big(\BA_{D}(\BrmYss)(\BrmMdkone\BrmMdktwo)^{-1}(\BrmYss)^{-1}\Big[\fdk_{D_1}, \fdk_{D_2}\Big]^\TransT\Big)_1
					\!\!+\!\frac{1}{2}\leps\fdk_{D_1}\bigg) \\ 
					\frac{-1}{\sqrt{\eps^2-t_2^2}} \bigg(\leps^2\frac{\pi}{4}
					\Big(\BA_{D}(\BrmYss)(\BrmMdkone\BrmMdktwo)^{-1}(\BrmYss)^{-1}\Big[\fdk_{D_1}, \fdk_{D_2}\Big]^\TransT\Big)_2
					\!\!+\!\frac{1}{2}\leps\fdk_{D_2}\bigg)
					\end{bmatrix}\,,
	\end{align*}
	where $f_{D_1}^{\delta k}, f_{D_2}^{\delta k} \in \CC$  are given by
	\begin{align*}
		\fdk_{D_1} =&
			2ia_0\sin(\delta k_2 h)e^{-i\delta k_1\xi}-2a_0\delta\!\!\!\!\!\!\int\limits_{\del D_1\cup\del D_2}\!\!\!\!\!\! \nu_y \!\cdot\! \begin{pmatrix} k_1 e^{-i\delta k_1y_1}\sin(\delta k_2 y_2) \\ ik_2e^{-i\delta k_1y_1}\cos(\delta k_2 y_2)\end{pmatrix}\NdelBOonedkp\left(\!\begin{pmatrix}-\xi\\h\end{pmatrix},y\right)\intd \sigma_y\,,\\
		\fdk_{D_2} =&
			2ia_0\sin(\delta k_2 h)e^{i\delta k_1\xi}-2a_0\delta\!\!\!\!\!\!\int\limits_{\del D_1\cup\del D_2}\!\!\!\!\!\! \nu_y \cdot \begin{pmatrix} k_1 e^{-i\delta k_1y_1}\sin(\delta k_2 y_2) \\ ik_2e^{-i\delta k_1y_1}\cos(\delta k_2 y_2)\end{pmatrix}\NdelBOonedkp\left(\!\begin{pmatrix}\xi\\h\end{pmatrix},y\right)\intd \sigma_y\,.
	\end{align*}
	Here, $\rmYss$ is given as in Theorem \ref{THM1:2HR} and
	\begin{align*}
		\BrmMdkone
				&=		\begin{bmatrix}
							\delta k - \delta k^{\delta,\eps}_{1,+} & 0 \\ 0 & \delta k - \delta k^{\delta,\eps}_{2,+}
						\end{bmatrix}\,,\quad
		\BrmMdktwo
				=		\begin{bmatrix}
							\delta k - \delta k^{\delta,\eps}_{1,-} & 0 \\ 0 & \delta k - \delta k^{\delta,\eps}_{2,-}
						\end{bmatrix}\,, \\
		\BA_{D}&=
			\begin{bmatrix}
					1/|D_1| & 0 \\ 0 & 1/|D_2|
			\end{bmatrix}\,.
	\end{align*}
\end{theorem}

We see from Theorem \ref{THM2:2HR} that the function
\begin{align*}
	\Ukapp(z)
		\DEF& (\Uknull-\UknullcircP)(z)+\Uk_{\Seu_p^\star}(z)+\Uk_{\Teu_p^\star}(z) +\Uk_{\mathrm{RHS},p}(z)\,,
\end{align*}
gives an accurate approximation of $\Uk$ in the far-field. Moreover, it satisfies the Helmholtz equation in $\RR^2_+$ with the boundary condition
\begin{align*}
	\Ukapp(z)+\delta\cIBC\del_{z_2}\Ukapp(z)=0\,,\quad\text{ for } z\in\del\RR^2_+\,,
\end{align*}
and $\Ukapp-\Uknull$ satisfies the outgoing radiation condition. 

%The boundary condition is called 'Impedance Boundary Condition' and for $\cIBC=0$ it yields a Dirichlet boundary condition and for $\cIBC \gg 1$ it approximates a Neumann boundary condition for $\Ukapp$. 

Using Theorem \ref{THM2:2HR} we can express $\cIBC$ as follows.
\begin{theorem}\label{THM3:2HR}
	The constant in the impedance boundary condition is given as
	\begin{align*}
		\cIBC = \frac{1}{2ia_0\delta k_2}\,C_{(\ref{Thm3:HR2})}^{\delta k}+\OO(\delta)\,,
	\end{align*}
	where $C_{(\ref{Thm3:HR2})}^{\delta k}\in\RR$ is defined as
	\begin{align}\label{Thm3:HR2}
		C_{(\ref{Thm3:HR2})}^{\delta k} = 
			&		-\frac{h}{p}\sum_{j\in\{1,2\}}\int_{\Lambda_j} (\Bmudkstar)_j(y)\intd \sigma_y
					+\!\!\!\sum_{j\in\{1,2\}}\int_{\del D_j} \del_\nu (\udknull-\udknull\circ \opP)(y)\frac{\sin(\delta k_2y_2)\,e^{i\delta k_1y_1}}{\delta k_2 p}\intd \sigma_y \nonumber\\
			&\mkern-65mu		+\frac{1}{p}\sum_{j,\hat{j}\in\{1,2\}}\int_{\Lambda_j}\int_{\del D_{\hat{j}}}(\Bmudkstar)_j(y)\NdelBOonedkc(y,w)\nu_w\!\cdot\!\begin{pmatrix}0\\1\end{pmatrix}\,\intd \sigma_w\intd \sigma_y\\
			&\mkern-65mu		-\sum_{j,\hat{j}\in\{1,2\}}\int_{\del D_j}\int_{\del D_{\hat{j}}}\del_\nu (u^{\delta k}_0-u^{\delta k}_0\!\!\circ\! \opP)(y) \NdelBOonedkc(y,w)\nu_w\cdot\begin{pmatrix}\frac{i\,k_1}{p\,k_2}\sin(\delta k_2 w_2)\\ \frac{1}{p}\cos(\delta k_2w_2)\end{pmatrix}e^{i\delta k_1 w_1}\intd \sigma_w\intd \sigma_y \nonumber\,.
	\end{align}
\end{theorem}

%\section{Expansion of the Resulting Wave in Terms of Delta}
\subsection{Proof of the Main Results}

We want to proof Theorem \ref{THM1:1HR} - \ref{THM3:1HR}. First, we express the resulting wave outside the Helmholtz resonators and the resulting wave inside of the Helmholtz resonators through operators acting on the resulting wave, but restricted on the gap. This leads us to a condition with the linear operator $\BAAdnoboldke$, whose solution is the resulting wave on the gap up to a term of order $\delta^2$. We solve this linear system based on the procedure given in \cite{HaiHabib}. We will see that it is solvable for a complex wave vector near $0$ except in five points, two of which are the resonances of our system. Then we recover the resulting wave outside the resonators up to a term of order $\delta^2$. We will see, that we can split the resulting wave into a propagating wave and a evanescent one. The propagating one leads us to the impedance boundary condition constant $\cIBC$.

\subsubsection{Collapsing the Wave-Informations on the Two Gaps}
Let us consider the resulting wave $\udk$ in the microscopic view, recall Subsection \ref{sec:reswaveinmicroscopicview:2HR}. We will keep the microscopic view until Subsection \ref{subsec:IBC:2HR}. We only look on the main strip $Y\DEF\{ y\in\RR^2_+\MID |y_1|<p/2\}$. $D$ is the Helmholtz resonator on that strip and $\Lambda$ the gap on $\del D$. Furthermore, we fix $k_1/k\FED e_1$ and $k_2/k\FED e_2$ and assume that $\delta k\in \kkkk \DEF \{ \hat{k} \in \RR \MID  0\neq |\hat{k}|< k_{{D_1\cup D_2}, \min, \Laplace}/2 \text{ and } |\hat{k}|^2<\inf\{|l-\hat{k}\,e_1|^2\MID l\DEF 2\pi n/p,\; n\in\ZZ\setminus\{0\}\}\}$. Consider that $\udk$ is continuous on the gap $\Lambda_1\cup\Lambda_2$, thus $\udk(z)$ is well-defined for $z\in\Lambda_1\cup\Lambda_2$.
\begin{proposition}\label{prop:udkFormulaOnDj}
   Let $j\in\{1,2\}$ and let $\Neu^k_{D_j}$ be as in Definition \ref{def:NEK}. Let $z\in  D_j$. Then,  
   \begin{align*}
      \udk(z)=-\int_{\Lambda_j}\del_\nu \udk(y)\NDjdk(z,y)\intd \sigma_y\,.
   \end{align*}
   Let $z\in\Lambda_j$. Then, 
   \begin{align}\label{equ:udkFormulaOndelDj}
      \udk(z)=-\int_{\Lambda_j}\del_\nu \udk(y)\NdelDjdk(z,y)\intd \sigma_y\,.
   \end{align}
\end{proposition}

\begin{proposition}\label{prop:udkFormulaOnComplementOfBOmega}
   Let $\NOkp$, $\NOkc$ be as in Definition \ref{def:NOKp}. Let $z\in  Y\setminus\overline{D}$. Then, 
   \begin{align*}
      \udks(z)=\!\int_{\Lambda_1\cup\Lambda_2}\!\!\del_\nu \udk(y)\NBOonedkp(z,y)\intd \sigma_y - \int_{\del D_1 \cup \del D_2}\!\!\! \del_\nu (\udknull\!-\!\udknull\circ \opP)(y)\NBOonedkp(z,y)\intd \sigma_y\,.
   \end{align*}
   Let $z\in\Lambda$. Then, 
   \begin{align}\label{equ:udkFormulaOnDelBOmega}
      \udks(z)=\int_{\Lambda_1\cup\Lambda_2}\del_\nu \udk(y)\NdelBOonedkp(z,y)\intd \sigma_y \!-\! \int_{\del D_1\cup\del D_2} \del_\nu (\udknull-\udknull\circ \opP)(y)\NdelBOonedkp(z,y)\intd \sigma_y\,.
   \end{align}
\end{proposition}

Using that $\udk$ is continuous on the gap we can deduce from the following proposition a necessary condition for $\del_\nu\udk\MID_\Lambda$. Assume we can obtain a solution $\del_\nu\udk$ from that condition, then from Propositions \ref{prop:udkFormulaOnDj} and \ref{prop:udkFormulaOnComplementOfBOmega} we can recover the resulting wave on $Y$.
\begin{proposition}[Gap-Formula]\label{prop:CompressedInfosOn2Gap}
   Let $j,\hat{j}\in\{1,2\}$, $j \neq \hat{j}$ and let $z\in \Lambda_j$ then
   \begin{multline}\label{equ:CompressedInfosOn2Gap}
      \int_{\Lambda_j} \del_\nu \udk(y)\left( \NdelBOonedkp(z,y) +\NdelDjdk(z,y) \right) \intd \sigma_y + \int_{\Lambda_{\hat{j}}} \del_\nu \udk(y) \NdelBOonedkp(z,y) \intd \sigma_y \\ 
     =  \int_{\del D_1\cup\del D_2} \del_\nu (\udknull-\udknull\circ \opP)(y)\NdelBOonedkp(z,y)\intd \sigma_y-(\udknull-\udknull\circ \opP)(z)\,.
   \end{multline}
\end{proposition}

Consider that the right-hand-side in  (\ref{equ:CompressedInfosOn2Gap}) does not depend on $\udk$ and it is computable.

\begin{proof}[Proposition \ref{prop:udkFormulaOnDj}]
   Let us look at  (\ref{equ:udkFormulaOndelDj}) first. Let $z\in\del D_j$ then using Green's formula with $\left( \Laplace + k^2 \right)\udk=0$  we obtain that 
   \begin{align*}
      \udk(z)=&\int_{D_j} \udk(y)\left( \Laplace_y + k^2 \right)\NdelDjdk(z,y)\intd y \\
      =& \int_{\del D_j} \udk(y)\del_{\nu_y}\NdelDjdk(z,y)\intd \sigma_y - \int_{\del D_j} \del_\nu\udk(y)\NdelDjdk(z,y)\intd \sigma_y\,.
   \end{align*}
   Using that $\del_{\nu_y}\NdelDjdk(z,y)=0$ on $\del D_j$ and $\del_\nu\udk(y)=0$ on $\del D_j\setminus\Lambda_j$ we obtain the desired equation.\\
   We get the other equation analogously.
\end{proof}

\begin{proof}[Proposition \ref{prop:udkFormulaOnComplementOfBOmega}]
   Let us look at (\ref{equ:udkFormulaOnDelBOmega}) first.\\
   Let $r>0$ and $U_r\DEF \{ y\in \RR^2\setminus\overline{D_1\cup D_2}\;\MID |y_1|<p /2 \wedge 0<y_2<r \}$, $\del U_{r,0}\DEF \{ y\in \del U_r\,\MID y_2=0 \}$, $\del U_{r,-}\DEF \{ y\in \del U_r\,\MID y_1=-p/2 \}$, $\del U_{r,+}\DEF \{ y\in \del U_r\,\MID y_1=+p/2 \}$ and $\del U_{r,r}\DEF \{ y\in \del U_r\,\MID y_2=r \}$.
   Then
   \begin{align*}
      \udks(z)=\lim_{r\rightarrow\infty}\int_{U_r}\udks(y) (\Laplace_y+k^2)\NdelBOonedkp(z,y)dy,
   \end{align*}
   because of the Dirac measure. Using Green's formula, we have
   {\setlength{\belowdisplayskip}{0pt} \setlength{\belowdisplayshortskip}{0pt}\setlength{\abovedisplayskip}{0pt} \setlength{\abovedisplayshortskip}{0pt}
   \begin{align}\label{equ:PFudks=-INTdeludksBNeu-rhs:1}
   		\udks(z)=\lim_{r\rightarrow\infty} \bigg( \int_{U_r}(\Laplace +k^2)\udks(y)\,\NdelBOonedkp(z,y)\intd y
   \end{align}
   \begin{align} 
   &-\int_{\del D_1\cup\del D_2}\udks(y)\del_{\nu_y}\NdelBOonedkp(z,y)\intd \sigma_y+\int_{\del D_1\cup\del D_2}\del_{\nu}\udks(y)\NdelBOonedkp(z,y)\intd \sigma_y \label{equ:PFudks=-INTdeludksBNeu-rhs:2}\\
   &+\int_{\del U_{r,0}}\udks(y)\del_{\nu_y}\NdelBOonedkp(z,y)\intd \sigma_y-\int_{\del U_{r,0}}\del_{\nu}\udks(y)\NdelBOonedkp(z,y)\intd \sigma_y \label{equ:PFudks=-INTdeludksBNeu-rhs:3}\\
   &+\int_{\del U_{r,-}}\udks(y)\del_{\nu_y}\NdelBOonedkp(z,y)\intd \sigma_y-\int_{\del U_{r,-}}\del_{\nu}\udks(y)\NdelBOonedkp(z,y)\intd \sigma_y \label{equ:PFudks=-INTdeludksBNeu-rhs:4}\\
   &+\int_{\del U_{r,+}}\udks(y)\del_{\nu_y}\NdelBOonedkp(z,y)\intd \sigma_y-\int_{\del U_{r,+}}\del_{\nu}\udks(y)\NdelBOonedkp(z,y)\intd \sigma_y \label{equ:PFudks=-INTdeludksBNeu-rhs:5}\\
   &+\int_{\del U_{r,r}}\udks(y)\del_{\nu_y}\NdelBOonedkp(z,y)\intd \sigma_y-\int_{\del U_{r,r}}\del_{\nu}\udks(y)\NdelBOonedkp(z,y)\intd \sigma_y \label{equ:PFudks=-INTdeludksBNeu-rhs:6} \bigg).
   \end{align}}
The right-hand-side in  (\ref{equ:PFudks=-INTdeludksBNeu-rhs:1}) vanishes because $\udks$ satisfies the homogeneous Helmholtz equation. The left term in  (\ref{equ:PFudks=-INTdeludksBNeu-rhs:2}) vanishes because $\NdelBOonedkp$ has a vanishing normal derivative on $\del \BOm^1$. Both terms in  (\ref{equ:PFudks=-INTdeludksBNeu-rhs:3}) vanish because of the Dirichlet boundary. The terms in  (\ref{equ:PFudks=-INTdeludksBNeu-rhs:4}) and in  (\ref{equ:PFudks=-INTdeludksBNeu-rhs:5}) cancel each other out because of the quasi-periodicity with quasi-momentum $-k_1p$ for $\udks$ and the quasi-periodicity with quasi-momentum $k_1p$ for $\NdelBOonedkp$, together with the explicit expression for the normal on $\del U_{r,-}$, which is $(-1,0)^\TransT$, and the explicit expression for the normal on $\del U_{r,+}$, which is $(1,0)^\TransT$. Thus we are left with 
   \setlength{\belowdisplayskip}{1.5pt} \setlength{\belowdisplayshortskip}{1.5pt}\setlength{\abovedisplayskip}{1.5pt} \setlength{\abovedisplayshortskip}{1.5pt}
   \begin{multline}
      \udks(z,x)=\int_{\del D_1\cup\del D_2}\del_{\nu}\udks(y)\NdelBOonedkp(z,y)\intd \sigma_y \\
      +\lim_{r\rightarrow\infty}\left(\int_{\del U_{r,r}}\udks(y)\del_{\nu_y}\NdelBOonedkp(z,y)\intd \sigma_y-\int_{\del U_{r,r}}\del_{\nu}\udks(y)\NdelBOonedkp(z,y)\intd \sigma_y\right). 
   \end{multline}
   Using that $\udks$ and $\NdelBOonedkp$ satisfy the outgoing radiation condition, we can write $\NdelBOonedkp(z,y)=\frac{1}{ik_2}\del_{y_2}\NdelBOonedkp(z,y)+o(1)$ and $\del_{y_2}\udks(y)=ik_2\udks(y)+o(1)$ for $y_2\rightarrow\infty$. With that we can eliminate the integrals within the limes. 
   
   Finally, using that $\del_{\nu}\udk\MID_{\del D_1\cup\del D_2\setminus\Lambda_1\cup \Lambda_2}=0$ and the definition of $\udks$, we proved  (\ref{equ:udkFormulaOnDelBOmega}).
   
   We get the other equation analogously.
\end{proof}

\begin{proof}[Proposition \ref{prop:CompressedInfosOnGap}]
   Using that $\udk$ is continuous at $\Lambda_j$ we have that $\udk\MIDD_+(z)-\udk\MIDD_-(z)=\udk(z)-\udk(z)=0$, for $z\in\Lambda_j$. Inserting  (\ref{equ:udkFormulaOnDelBOmega}) and  (\ref{equ:udkFormulaOndelDj}) we obtain  (\ref{equ:CompressedInfosOn2Gap}).
\end{proof}

\subsubsection{Expanding the Gap-Formula in Terms of Delta}
We define $\Ofdk:\Lambda_1\cup\Lambda_2\rightarrow\CC$ as the right-hand-side of the Gap-Formula \ref{prop:CompressedInfosOnGap}, that is
\begin{align}\label{equdef:Ofdk}
   \Ofdk(z)\DEF \int_{\del D_1\cup\del D_2} \del_\nu (\udknull-\udknull\circ \opP)(y)\NdelBOonedkp(z,y)\intd \sigma_y-(\udknull-\udknull\circ \opP)(z).
\end{align}
We define $\fdk_{D_1}(\tau)$ as $\Ofdk((\tau-\xi\,,h)^\TransT)$ and $\fdk_{D_2}(\tau)$ as $\Ofdk((\tau+\xi\,,h)^\TransT)$, for $\tau\in (-\eps\,,\eps)$.

Let us define the following operator-spaces and their respective norms:
\begin{definition}
   Recall Definition \ref{def:curlXe}, we define
   \begin{align*}
      \BcurlXe&\DEF \curlXe \times \curlXe, \\
      \NORM{\mu}_{\BcurlXe} &\DEF \NORM{\mu_1}_{\curlXe}+\NORM{\mu_2}_{\curlXe},\quad \text{ for }\mu=(\mu_1\,,\mu_2)\in \BcurlXe\,,\\
      \BcurlYe&\DEF \curlYe \times \curlYe,\\
      \NORM{\mu}_{\BcurlYe}&\DEF\NORM{\mu_1}_{\curlYe}+\NORM{\mu_2}_{\curlYe},\quad \text{ for }\mu=(\mu_1\,,\mu_2)\in \BcurlYe\,.
   \end{align*}
\end{definition}
\begin{definition}
	Let $\mu\in\BcurlXe$ and $\alpha>0$. We say $\mu=\OO_\BcurlXe(\alpha)$ for $\alpha\rightarrow 0$ if 
	$
		\frac{\NORM{\mu}_{\BcurlXe}}{\alpha}
	$
	is bounded as $\alpha\rightarrow 0$.
\end{definition}
With those spaces we can define the following operators:
\begin{definition}\label{def:OperatorsFor2HR}
   The following operators are defined as functions from $\BcurlXe$ to $\BcurlYe$ or from $\curlXe$ to $\curlYe$. Let $n \in \NN\setminus\{ 0\}$, and $\mu=(\mu_1\,,\mu_2)\in\BcurlXe$ then
   {\setlength{\belowdisplayskip}{0pt} \setlength{\belowdisplayshortskip}{0pt}\setlength{\abovedisplayskip}{0pt} \setlength{\abovedisplayshortskip}{0pt}
   \begingroup
   \renewcommand*{\arraystretch}{1.2}
   \begin{align*}
      \BLLe[\mu](\tau)\DEF& 
      		\begin{bmatrix}\LLe & 0 \\ 0 & \LLe\end{bmatrix}\begin{bmatrix}\mu_1 \\ \mu_2\end{bmatrix}(\tau)=
      		\begin{bmatrix}
      			\LLe[\mu_1](\tau) \\ \LLe[\mu_2](\tau)
      		\end{bmatrix}\,,\\
      \BKKe[\mu](\tau)\DEF&
      		\begin{bmatrix}
      			\KKe_1[\mu_1](\tau) \\  \KKe_2[\mu_2](\tau)
      		\end{bmatrix}\DEF
      		\begin{bmatrix}
      			\inteps \frac{\mu_1(t)}{|D_1| }\intd t \\  \inteps \frac{\mu_2(t)}{|D_2| }\intd t
      		\end{bmatrix}\,,
   \end{align*}
   \begin{align*}
      \BRRdke[\mu](\tau)\DEF&
      		\begin{bmatrix}
      			\RRdke_{1,1} & \RRdke_{1,2} \\ \RRdke_{2,1} & \RRdke_{2,2}
      		\end{bmatrix}\begin{bmatrix}\mu_1 \\ \mu_2\end{bmatrix}(\tau)\,,\\
      \RRdke_{0,0}[\mu_1](\tau)\DEF& \inteps \mu_1(t)\bigg[\frac{1}{\pi}\!\log\bigg(\frac{\pi}{p}\bigg)+\frac{1}{\pi}\!\log\bigg( \sinc\bigg|\frac{\pi}{p}(\tau- t)\bigg|\bigg)\nonumber\\
      		&- \frac{1}{2\pi}\log\left(\sinh\left(\frac{\pi}{p}2\,h\right)^2+\sin\left(\frac{\pi}{p}(\tau-t)\right)^2\right)\Bigg]\intd t\,,\\ 
  % \end{align*}
  % \begin{align}
      \RRdke_{1,1}[\mu_1](\tau)\DEF&\RRdke_{0,0}[\mu_1](\tau)\\
      		&\mkern-32mu +\inteps\mu_1(t)\Bigg[\Reu^{\delta k}_{\del D_1} 
      		\begin{pmatrix*} 
      			\begin{pmatrix*} \tau-\xi \\ h \end{pmatrix*}, 
      			\begin{pmatrix*} t-\xi \\ h \end{pmatrix*}
      		\end{pmatrix*} +
      		\RdelOonedkp
      		\begin{pmatrix*} 
      			\begin{pmatrix*} \tau-\xi \\ h \end{pmatrix*}, 
      			\begin{pmatrix*} t-\xi \\ h \end{pmatrix*}
      		\end{pmatrix*} \Bigg]\intd t\,,\nonumber \\
      \RRdke_{1,2}[\mu_2](\tau)\DEF&\inteps \mu_2(t) \NdelBOonedkp
      		\begin{pmatrix*} 
      			\begin{pmatrix*} \tau-\xi \\ h \end{pmatrix*}, 
      			\begin{pmatrix*} t+\xi \\ h \end{pmatrix*}
      		\end{pmatrix*}
       		\intd t\,,\\
  % \end{align}
  % \begin{align}
      \RRdke_{2,1}[\mu_1](\tau)\DEF&\inteps \mu_1(t) \NdelBOonedkp
      		\begin{pmatrix*} 
      			\begin{pmatrix*} \tau+\xi \\ h \end{pmatrix*}, 
      			\begin{pmatrix*} t-\xi \\ h \end{pmatrix*}
      		\end{pmatrix*}
       		\intd t\,,\\
      \RRdke_{2,2}[\mu_2](\tau)\DEF&\RRdke_{0,0}[\mu_2](\tau)\\
      		&\mkern-32mu +\inteps\mu_2(t)\Bigg[\Reu^{\delta k}_{\del D_2} 
      		\begin{pmatrix*} 
      			\begin{pmatrix*} \tau+\xi \\ h \end{pmatrix*}, 
      			\begin{pmatrix*} t+\xi \\ h \end{pmatrix*}
      		\end{pmatrix*} +
      		\RdelOonedkp
      		\begin{pmatrix*} 
      			\begin{pmatrix*} \tau+\xi \\ h \end{pmatrix*}, 
      			\begin{pmatrix*} t+\xi \\ h \end{pmatrix*}
      		\end{pmatrix*} \Bigg]\intd t\,,\nonumber \\
%   \end{align}
%   \begin{align}
      \BcalGG^{k,\eps}_{+,n}[\mu](\tau)\DEF&
      		\begin{bmatrix}
      			\inteps\mu_1(t)\,\Gamma^k_{+,n}
      			\begin{pmatrix*} 
      				\begin{pmatrix*} \tau-\xi \\ h \end{pmatrix*}, 
      				\begin{pmatrix*} t-\xi \\ h \end{pmatrix*}
      			\end{pmatrix*}
      			\intd t  \\
      			\inteps\mu_2(t)\,\Gamma^k_{+,n}
      			\begin{pmatrix*} 
      				\begin{pmatrix*} \tau+\xi \\ h \end{pmatrix*}, 
      				\begin{pmatrix*} t+\xi \\ h \end{pmatrix*}
      			\end{pmatrix*}
      			\intd t
      		\end{bmatrix}\,,\\
      \BAAdnoboldke[\mu](\tau)\DEF& \frac{2}{\pi}\BLLe[\mu](\tau)+\frac{\BKKe[\mu](\tau)}{\delta^2 k^2}+\BRRdnoboldke[\mu](\tau)\,,\\\nonumber
   \end{align*}
   \endgroup}
   where $\Gamma^k_{+,n}$ is given in Lemma \ref{prop:GammaExpansion}.
\end{definition}

For $\tau\in(\eps,\eps)^2$, we also define
\begingroup
\renewcommand*{\arraystretch}{1.2}
\begin{align}\label{equdef:Bfdk}
	\Bfdk(\tau)\DEF 
			\begin{bmatrix}
					\fdk_{D_1}(\tau_1) \\ \fdk_{D_2}(\tau_2)
			\end{bmatrix}\,.
\end{align}
\endgroup

Later, we will show that 
\begin{align*}
   \begin{bmatrix}\del_\nu \udk\MID_{\Lambda_1} \\ \del_\nu \udk\MID_{\Lambda_2} \end{bmatrix} = (\BAAdke)^{-1}[\Bfdk] + \OO(\delta^2).
\end{align*}

\begin{proposition} \label{prop:2GapFormulaINOperators}
   Let $2\eps<p$, let $\tau\in (-\eps,\eps)^2$, then
   \begin{align*}
   		\BAAdke\begin{bmatrix}\del_\nu \udk\MID_{\Lambda_1} \\ \del_\nu \udk\MID_{\Lambda_2} \end{bmatrix}(\tau)+\sum_{n=1}^\infty\delta^n\,2 \,\BcalGG^{k,\eps}_{+,n}\begin{bmatrix}\del_\nu \udk\MID_{\Lambda_1} \\ \del_\nu \udk\MID_{\Lambda_2} \end{bmatrix}(\tau) = \Bfdk(\tau).
   \end{align*}
\end{proposition}

\begin{proof}
 	Let $z\DEF(\tau\pm\xi,h)^\TransT\in\Lambda_j$, where $j, \hat{j} \in\{ 1,2 \}$, $j\neq \hat{j}$. From Proposition \ref{prop:CompressedInfosOn2Gap} we have 
 	\begin{align*}
 		\int_{\Lambda_j} \del_\nu \udk(y)\left( \NdelBOonedkp(z,y) +\NdelDjdk(z,y) \right) \intd \sigma_y + \int_{\Lambda_{\hat{j}}} \del_\nu \udk(y) \NdelBOonedkp(z,y) \intd \sigma_y
 		= \Ofdk(z)\,.
 	\end{align*}
 	Using Lemma \ref{prop:GammaExpansion} we can rearrange the last equation and obtain
 	\begin{multline}
 		\Ofdk(z)=\int_{\Lambda_j} \del_\nu \udk(y)\bigg( 
 				2\Gamma_{+}^0(z,y)+
 				\frac{1}{\pi}\log|z-y|+
 				\frac{1/|D_j|}{\delta^2 k^2}
 				+\RdelBOonedKp(z,y) +
 				\RdelDjdK(z,y)\\+
 				\sum_{n=1}^\infty \delta^n \,2\Gamma_{+,n}^k(z,y)
 		\bigg) \intd \sigma_y+
 		\int_{\Lambda_{\hat{j}}} 
 				\del_\nu \udk(y) 
 				\NdelBOonedkp(z,y)
 		\intd \sigma_y.
 	\end{multline}
 	Using that $\Lambda$ is a line segment parallel to the $x_1$-axis and writing $y=y_1=(t-\xi,h)^\TransT$, on the gap $\Lambda_1$, and writing $y=y_2=(t+\xi,h)^\TransT$, on the gap $\Lambda_2$, we have that $\intd\sigma_y = \intd t$. Using  (\ref{equ:ch2:logsinhsin}) for $\Gamma_{+}^0(z,y)$ and using that the expansion in $\delta$ (Lemma \ref{prop:GammaExpansion}) is uniform, we can interchange the infinite sum and the integration. Let $\mu_j(t)\DEF\del_\nu \udk(y)\MID_{\Lambda_j}\,$, we have that
 	\begin{align}\label{equ:PF2GapFormulaINOperators:1}
 	  	\Ofdk(z)\!&=
 	  			\!\inteps \mu_j(t)\bigg[ 
 	  					\frac{1}{\pi}\log|\tau\!-\!t|+
						\!\frac{1}{2\pi}\! \log\bigg(\!\!\sin\bigg(\!\frac{\pi}{p}(\tau\!-\! t)\bigg)^2 \bigg)\nonumber\\
 						&-\!\frac{1}{2\pi}\log\bigg(\!\sinh\bigg(\frac{\pi}{p}2\,h\bigg)^2\!\!+\!\sin\bigg(\!\frac{\pi}{p}(\tau\!-\!t)\bigg)^2\bigg)+\RdelBOonedkp(z,y_j) \!+\!\RdelDjdK(z,y_j)
 				\bigg] \intd t \nonumber\\
 			&+\RRdke_{j,\hat{j}}[\mu_{\hat{j}}](\tau)+\frac{\KKe[\mu_j](\tau)}{\delta^2 k^2}+
 			\sum_{n=1}^\infty\delta^n \,2\calGG^{k,\eps}_{+,n}[\mu_j](\tau).
 	\end{align}
 	Now consider that for $2\eps<p$, we have
 	\begin{multline}
 	 	\frac{1}{2\pi}\!\log\bigg(\sin\bigg(\frac{\pi}{p}(\tau- t)\bigg)^2 \bigg) =
 	 	\frac{1}{\pi}\!\log\bigg(\sin\bigg|\frac{\pi}{p}(\tau- t)\bigg| \bigg)\\
 	 	= \frac{1}{\pi}\!\log\bigg(\frac{\pi}{p}\bigg)+ 
 	 	\frac{1}{\pi}\!\log|\tau-t| +
 	 	\frac{1}{\pi}\!\log\bigg( \sinc\bigg|\frac{\pi}{p}(\tau- t)\bigg| \bigg).
 	\end{multline}
 	Inserting last equation into  (\ref{equ:PF2GapFormulaINOperators:1}) we obtain that
 	\begin{align*}
 		\frac{2}{\pi}\LLe[\mu_j](\tau) \!+\!
		\RRdke_{j,j}[\mu_j](\tau)\!+\!
		\RRdke_{j,\hat{j}}[\mu_{\hat{j}}](\tau)\!+\!
		\frac{\KKe_j[\mu_j](\tau)}{\delta^2 k^2}\!+\!
 		\sum_{n=1}^\infty\delta^n \,\calGG^{k,\eps}_{+,n}[\mu_j](\tau)
 		=\Ofdk(z). 			
 	\end{align*}
 	With Definition \ref{def:OperatorsFor2HR} we have proven Proposition \ref{prop:2GapFormulaINOperators}.
\end{proof}

From Lemma \ref{prop:LLEisInjective} we readily get that for $0<\eps<2$, the operator $\BLLe: \BcurlXe\rightarrow\BcurlYe$ is linear, bounded and invertible and has the inverse
\begin{align}\label{lemma:BLLEisInjective}
	\BLLeinv[\Bmu](t) = \begin{bmatrix}
			\LLeinv[\mu_1](t_1) \\ \LLeinv[\mu_2](t_2)
	\end{bmatrix}\,.
\end{align}	 

Since $\RdelDjdK$ and $\RdelBOonedKp$ are continuous, $\BRRdke$ is a compact operator. Thus we have that $\tfrac{2}{\pi}\BLLe+\BRRdke$ is a Fredholm operator of index zero. Thus for the operator $\BAAdke$, extending the domain $\kkkk$ to the complex numbers in a disk-shaped form, we will see that $\BAAdke$ is invertible except for a finite amount of values of $\delta k$. Some of those values are the resonances of our physical system. To that end, we will need the following result.

\subsubsection{Characteristic Values of \boldmath$\AAdnoboldke$ and the four Resonance Values}
Let us first look at the characteristic values of 
\begin{align*}
	\BQQdke[\Bmu] \DEF \frac{2}{\pi}\BLLe[\Bmu]+\frac{\BKKe[\Bmu]}{\delta^2 k^2},
\end{align*}
where $\Bmu\in\BcurlXe$ and $\delta k\in\kkkkc\DEF\{k_\ast\in \CC\MID \sqrt{k_\ast \bar{k}_\ast}< k_{{D_1\cup D_2},\min,\Laplace}/2\}$. For $\Bmu, \bm{\lambda} \in \BcurlXe$ we define $(\Bmu\,, \bm{\lambda})_{\eps\oplus\eps}\DEF (\mu_1\,,\lambda_1)_\eps+(\mu_2\,,\lambda_2)_\eps\,$, where $(\cdot\,, \cdot)_\eps$ is the $\Ltwo((-\eps,\eps))$ inner-product. We also define $\boldone\DEF [1,1]^\TransT\in\BcurlXe$, $\boldeone \DEF [1,0]^\TransT\in\BcurlXe$ and $\boldetwo \DEF [0,1]^\TransT\in\BcurlXe$.

\begin{lemma}\label{lemma:BZeroOrderRes}
	$\BQQdke$ has exactly the four characteristic values $\pm k^{\delta,\eps}_{j,0}$ for $j\in\{ 1,2 \}$ with the characteristic functions $\mu^{\delta,\eps}_{j,0}$, where
	\begin{align*}
		k^{\delta,\eps}_{j,0} =&
		\frac{1}{\delta}\left( -\frac{\pi}{2|D_j|} \Big(\LLeinv[1]\,,1\Big)_\eps \right)^{1/2} 
		= \frac{1}{\delta}\left(-\frac{\pi}{2|D_j|\log{(\eps/2)}}  \right)^{1/2} , \\
		\mu^{\eps}_{j,0} =& 
		-\frac{\pi}{2|D_j|}\frac{\LLeinv[1]}{(\delta k^{\delta,\eps}_{j,0})^2}\,,
	\end{align*}
	after imposing $(\mu_1,1)_\eps=1$, $(\mu_2,1)_\eps=1$.
\end{lemma}

Consider that $k^{\delta,\eps}_{j,0}$ is real and positive.

\begin{proof}
	We are looking for $\Bmu=[\mu_1\,,\mu_2]^\TransT\in\BcurlXe$ such that
	\begin{align*}
		\BQQdke[\Bmu]=\boldnull\,.
	\end{align*}
	Since $\BLLeinv[\boldnull]=\boldnull$, the last equation is equivalent to
	\begin{align}\label{equ:PFBZeroOrderRes:2}
		\frac{2}{\pi}\Bmu+\frac{\BLLeinv\BKKe[\Bmu]}{(\delta k)^2}=\boldnull\,.
	\end{align}
	Applying once $(\cdot\,, \boldeone)_{\eps\oplus\eps}$ and once $(\cdot\,, \boldetwo)_{\eps\oplus\eps}$ on both sides, we obtain
	\begin{align}\label{equ:PFBZeroOrderRes:3}
		(\mu_j\,,1)_\eps\left( \frac{2}{\pi}+\frac{1}{|D_j|}\frac{(\LLeinv[1]\,,1)_\eps}{(\delta k)^2} \right)=0\,, \quad \text{for } j\in\{ 1,2 \}.
	\end{align}
	If $(\mu_j\,, 1)_\eps=0$, then $\LLe[\mu_j]=0$, because of the condition $\BQQdke[\Bmu]=0$, and then $\mu_j=0$, since $\LLe$ is invertible and linear. But $0$ cannot be a characteristic function, by definition. Hence $(\mu_j\,, 1)\neq 0$. Thus the second factor in  (\ref{equ:PFBZeroOrderRes:3}) has to be zero. This leads us to
	\begin{align*}
		\delta^2 k^2 =-\frac{\pi}{2|D_j|}\Big(\LLeinv[1]\,,1\Big)_\eps\,, \quad \text{for } j\in\{ 1,2 \}.
	\end{align*}
	Using Lemma \ref{lemma:exactLLeValues}, we can calculate that $(\LLeinv[1]\,,1)_\eps\,= \frac{1}{\log(\eps/2)}$, and obtain the characteristic values.
	
	As for the characteristic functions, we rewrite  (\ref{equ:PFBZeroOrderRes:2}) as
	\begin{align*}
		\frac{\mu_j}{(\mu_j,1)_\eps}=-\frac{\pi}{2|D_j|}\frac{\LLeinv[1]}{(\delta k)^2}\,, \quad \text{for } j\in\{ 1,2 \}.
	\end{align*}
	Imposing the normalization on $\mu$ we have proven our statement.
\end{proof}

%To facilitate future expressions we define
%\begin{align}
%	\cepsj\DEF -\frac{\pi}{|D_j|\log{(\eps/2)}} = -\frac{\pi}{|D_j|}(\LLeinv[1]\,,1)_\eps\,.
%\end{align}
%Next we will look at the characteristic values of $\BAAdke$. Denote $\BwLLdke\DEF \BLLe+\frac{\pi}{2}\BRRdnoboldke$ and $\BSSdke\DEF \frac{\pi}{2}\BLLeinv\BRRdnoboldke$, where we fixed the angles of the incoming wave vector, but let the magnitude be complex. Using that $\BRRdnoboldke$ is in $\cC^{1,\eta}$, for $\eta\in (0,1)$, and $\BLLe$ invertible and using Lemma \ref{lemma:normLinvR}, we can apply the Neumann series, whenever $\eps$ is small enough, and thus we have
%\begin{align}\label{equ:BwLLdkeExpansion}
%	(\BwLLdke)^{-1}=\big( \BcalI + \BSSdke \big)^{-1}\BLLeinv = \sum_{l=0}^{\infty} (-\BSSdke)^{l}\BLLeinv,
%\end{align}
%where $\BcalI$ denotes the identity function in $\BcurlXe$. 

%\todo{From here on we use that $D_1$ has the same shape, scale and height as $D_2$ and that $\Lambda_1$ and $\Lambda_2$ are symmetric with respect to the $x_2$-axis. Thus especially $|D_1|=|D_2|\FED|D|$. We also see that $\RdelDonedK(z,x)=\RdelDtwodK(\tilde{z},\tilde{x})$, for $z,x\in\Lambda_1$ and $\tilde{z}\DEF z+(2\xi,0)^\TransT\in\Lambda_2,\tilde{x}\DEF x+(2\xi,0)^\TransT\in\Lambda_2$ and $\Gdkp(z,x)$ only depends on $|z_1-x_1|$ for $z_2=x_2=h$.
%
% We probably need a different statement and one about Next. But do this later when we really need it. For$\Adkej$ we only need same area}

Let us look at the characteristic values of $\BAAdke$. Denote $\BwLLdke\DEF \BLLe+\frac{\pi}{2}\BRRdnoboldke$ and $\BSSdke\DEF \frac{\pi}{2}\BLLeinv\BRRdnoboldke$, where we fixed the angles of the incoming wave vector, but let the magnitude be complex. Using that $\BRRdnoboldke$ is in $\cC^{1,\eta}$, for $\eta\in (0,1)$, and $\BLLe$ invertible and using Lemma \ref{lemma:normLinvR}, we can apply the Neumann series, whenever $\eps$ is small enough, and thus we have
\begin{align}\label{equ:BwLLdkeExpansion}
	(\BwLLdke)^{-1}=\big( \BcalI + \BSSdke \big)^{-1}\BLLeinv = \sum_{l=0}^{\infty} (-\BSSdke)^{l}\BLLeinv,
\end{align}
where $\BcalI$ denotes the identity function in $\BcurlXe$. 

We then define the $\RR^{2\times 2}$-matrix $\BAdke$ as
\begin{align}\label{equ:SumOfBAdke}
	(\BAdke)_{j,\hat{j}}\DEF -\frac{1}{|D_{\hat{j}}|}\left((\BSSdke)^{-1}[\boldej]\,,\boldejhat \right)_{\eps\oplus\eps}=-\frac{\pi}{2|D_{\hat{j}}|} ( \LLeinv\RRdke_{j,\hat{j}}[1]\,,1)_\eps\,,
\end{align}
for $j, \hat{j} \in \{1,2\}$.

\begin{lemma}\label{lemma:BAAdkeCharvalueAreDK2-BAdke}
	Any characteristic value of $\BAAdnoboldke$ is a characteristic value of the $\RR^{2\times 2}$-matrix $\delta^2 k^2\mathbb{I}_2 - \BAdke$.
\end{lemma}

\begin{proof}
	Suppose $(\delta k)^2$ is a characteristic value of $\BAAdke$. Substituting $\BLLe$ with $\BwLLdke$ in Lemma \ref{lemma:BZeroOrderRes}, we readily see that
	\begin{align}\label{equ:PFBZerok-A:1}
		\frac{2}{\pi}\Bmu+\frac{\BwLLdkeinv\BKKe[\Bmu]}{(\delta k)^2}=\boldnull\,.
	\end{align}
	Applying  $(\cdot\,, \boldeone)_{\eps\oplus\eps}$ on both sides, we obtain
	\begin{align}\label{equ:PFBZerok-A:2}
		\left( 
				\frac{2}{\pi}(\mu_1\,,1)_\eps
				+\frac{(\mu_1\,,1)_\eps}{|D_1|}\frac{(\LLeinv\RRdke_{1,1}[1]\,,1)_\eps}{(\delta k)^2}
				+\frac{(\mu_2\,,1)_\eps}{|D_2|}\frac{(\LLeinv\RRdke_{1,2}[1]\,,1)_\eps}{(\delta k)^2} \right)=0\,.
	\end{align}
	Thus,
	\begin{align*}
		(\delta k)^2(\mu_1\,,1)_\eps=
			-\frac{\pi\,(\mu_1\,,1)_\eps}{2|D_1|}(\LLeinv\RRdke_{1,1}[1]\,,1)_\eps-\frac{\pi\,(\mu_2\,,1)_\eps}{2|D_2|}(\LLeinv\RRdke_{1,2}[1]\,,1)_\eps\,.
	\end{align*}
	Analogously,
	\begin{align*}
		(\delta k)^2(\mu_2\,,1)_\eps=
			-\frac{\pi\,(\mu_1\,,1)_\eps}{2|D_1|}(\LLeinv\RRdke_{2,1}[1]\,,1)_\eps-\frac{\pi\,(\mu_2\,,1)_\eps}{2|D_2|}(\LLeinv\RRdke_{2,2}[1]\,,1)_\eps\,.
	\end{align*}
	Thus,
	\begin{align*}
		(\delta k)^2 \begin{bmatrix}
			(\mu_1\,,1)_\eps \\ (\mu_2\,,1)_\eps 
		\end{bmatrix}= \BAdke \begin{bmatrix}
			(\mu_1\,,1)_\eps \\ (\mu_2\,,1)_\eps 
		\end{bmatrix}\,.
	\end{align*}
	Hence,  if $(\delta k)^2$ is a characteristic value of $\BAAdke$ then it also is a characteristic value of $(\delta k)^2-\BAdke$.
\end{proof}

We define
\begin{align*}
	\leps\DEF \frac{-1}{\log{(\eps/2)}}\,.
\end{align*}

\begin{proposition}\label{prop:exactly4charvalues}
There exist four characteristic values, counting multiplicity, for the operator $\BAAdnoboldke$ function in $\kkkkc$. Moreover, they have the asymptotic
	\begin{align*}
		\delta k&= \pm \delta k^{\delta,\eps}_0+ \delta\OO(\leps),&&\quad\text{for } \eps\rightarrow 0 \,.
	\end{align*}
\end{proposition}
\begin{proof}
	Recall that the operator-valued analytic function $\BQQdke$ is finitely meromorphic and of Fredholm type. Moreover, it has four characteristic values $\pm k_{j,0}^{\delta k,\eps}$ , and has a pole at $0$ with order two in $\kkkkc$. Thus, the multiplicity of $\BQQdke$ is $2$ in $\kkkkc$. Note that for $\delta k \in \kkkkc\setminus\{ 0,\pm \delta k^{\delta,\eps}_{1,0}, \pm \delta k^{\delta \eps}_{2,0}\}$, the operator $\BQQdke$ is invertible, because it is of Fredholm type and because it is injective due to Lemma \ref{lemma:BZeroOrderRes}. With that,
\begin{align*}
	(\BQQdke)^{-1}\BRRdnoboldke =\frac{2}{\pi}\left( \frac{2}{\pi}\BcalI - \frac{\BLLeinv\BKKe}{\delta^2 k^2} \right)^{-1}\BSSdke\,.
\end{align*}	 
Thus, $\NORM{(\BQQdke)^{-1}\BRRdnoboldke}_{\calL(\curlXe,\curlXe)}=\OO(\leps)$ uniformly for $\delta k\in\del\kkkkc$.
By the generalized Rouché’s theorem \cite[Theorem 1.15]{LPTSA}, we can conclude that for $\eps$ sufficiently small, the operator $\BAAdnoboldke$ has the same multiplicity as the operator $\BQQdke$ in $\kkkkc$ , which is $2$. Since $\BAAdnoboldke$ has a pole of order two, we derive that $\BAAdnoboldke$ has four characteristic values counting multiplicity. This completes the proof of the proposition.
\end{proof}

Let us give an asymptotic expression for those characteristic values.

Let $l\geq 1$ be an integer, we define the $\RR^{2\times 2}$-matrix $\BSdke_{(l)}$ as
\begin{align*}
	(\BSdke_{(l)})_{j,\hat{j}}\DEF -\frac{\pi}{2|D_{\hat{j}}|}\left( (-\BSSdke)^{l}\BLLeinv[\boldej]\,, \boldejhat\right)_{\eps\oplus\eps}\,,
\end{align*}
for $j,\hat{j}\in\{1,2\}$.
Because of  (\ref{equ:BwLLdkeExpansion}), we can write
\begin{align*}
	\BAdke= \frac{\pi}{2} \leps\begin{bmatrix} \frac{1}{|D_1|} & 0 \\ 0 & \frac{1}{|D_2|}   \end{bmatrix}	 +\sum_{l=1}^\infty \BSdke_{(l)}\,.
\end{align*}
We want to give a second order analytic expression for \\$(\BSdke_{(1)})_{j,\hat{j}}=\frac{\pi}{2|D_{\hat{j}}|}\left(\frac{\pi}{2} \BLLeinv\BRRdnoboldke\BLLeinv[\boldej]\,, \boldejhat\right)_{\eps\oplus\eps}$. To this end, we define %HACK
\begin{align}
	(\BSdke_{(1,1)})_{j,\hat{j}}=\frac{\pi}{2|D_{\hat{j}}|}\left(\frac{\pi}{2} \BLLeinv\BRRdkeone\BLLeinv[\boldej]\,, \boldejhat\right)_{\eps\oplus\eps}\,,\label{equdef:BSdke11}\\
	(\BSdke_{(1,2)})_{j,\hat{j}}=\frac{\pi}{2|D_{\hat{j}}|}\left(\frac{\pi}{2} \BLLeinv\BRRdketwo\BLLeinv[\boldej]\,, \boldejhat\right)_{\eps\oplus\eps}\,,\label{equdef:BSdke12}\\
	(\BSdke_{(1,3)})_{j,\hat{j}}=\frac{\pi}{2|D_{\hat{j}}|}\left(\frac{\pi}{2} \BLLeinv\BRRdkethree\BLLeinv[\boldej]\,, \boldejhat\right)_{\eps\oplus\eps}\,,\label{equdef:BSdke13}
\end{align}
where
\begin{align}
	\BRRdkeone[\Bmu](\tau)\DEF \label{equdef:BRRdkeone}
		&	\begin{bmatrix} \Balphanull_{1,1}(\mu_1\,, 1)_\eps & \BT_{1,2}(\mu_2\,, 1)_\eps \\ \BT_{2,1}(\mu_1\,, 1)_\eps & \Balphanull_{2,2}(\mu_2\,, 1)_\eps\end{bmatrix}\,,\\
	\BRRdketwo[\Bmu](\tau)\DEF \label{equdef:BRRdketwo}
		&	\delta k \begin{bmatrix} \Balphaone_{1,1}(\mu_1\,, 1)_\eps & \BS_{1,2}(\mu_2\,, 1)_\eps \\ \BS_{2,1}(\mu_1\,, 1)_\eps & \Balphaone_{2,2}(\mu_2\,, 1)_\eps\end{bmatrix}\,,\\
	(\BRRdkethree)_{j,\hat{j}}[\mu_{\hat{j}}](\tau)\DEF \label{equdef:BRRdkethree}
		&	\inteps \mu_{\hat{j}}(t) \left( t\widetilde{\del}_t R^{\delta k}_{j,\hat{j}}(\tau,t)+ \tau\widetilde{\del}_\tau R^{\delta k}_{j,\hat{j}}(\tau,t)+ \delta^2 k^2\widetilde{\del}^2_{\delta k} R^{\delta k}_{j,\hat{j}}(\tau,t) \right)\intd t\,.
\end{align}
Here, $R^{\delta k}_{j,\hat{j}}(\tau,t)$ denotes the kernel of $(\BRRdnoboldke)_{j,\hat{j}}$, see Definition \ref{def:OperatorsFor2HR}, and $\widetilde{\del}$ denotes the derivative part of the remainder in Taylor's theorem in the Peano form and where
{\setlength{\belowdisplayskip}{0pt} \setlength{\belowdisplayshortskip}{0pt}\setlength{\abovedisplayskip}{0pt} \setlength{\abovedisplayshortskip}{0pt}
   \begingroup
   \renewcommand*{\arraystretch}{1.0}
\begin{align}
	\Balphanull_{1,1}\DEF\label{equdef:Balphanull11}
		&	\Reu^{0}_{\del D_1} 
      			\begin{pmatrix*} 
      				\begin{pmatrix*} -\xi \\ h \end{pmatrix*}, 
      				\begin{pmatrix*} -\xi \\ h \end{pmatrix*}
      			\end{pmatrix*} +
      		\Reu^{0}_{\del \BOm^1,+}
      			\begin{pmatrix*} 
      				\begin{pmatrix*} -\xi \\ h \end{pmatrix*}, 
      				\begin{pmatrix*} -\xi \\ h \end{pmatrix*}
      			\end{pmatrix*} \nonumber \\
      	& 	+\frac{1}{\pi}\!\log\bigg(\frac{\pi}{p}\bigg)
      		- \frac{1}{2\pi}\log\left(\sinh\left(\frac{\pi}{p}2\,h\right)^2\right)\,,\\
    \Balphanull_{2,2}\DEF\label{equdef:Balphanull22}
		&	\Reu^{0}_{\del D_2} 
      			\begin{pmatrix*} 
      				\begin{pmatrix*} \xi \\ h \end{pmatrix*}, 
      				\begin{pmatrix*} \xi \\ h \end{pmatrix*}
      			\end{pmatrix*} +
      		\Reu^{0}_{\del \BOm^1,+}
      			\begin{pmatrix*} 
      				\begin{pmatrix*} \xi \\ h \end{pmatrix*}, 
      				\begin{pmatrix*} \xi \\ h \end{pmatrix*}
      			\end{pmatrix*} \nonumber \\
      	& 	+\frac{1}{\pi}\!\log\bigg(\frac{\pi}{p}\bigg)
      		- \frac{1}{2\pi}\log\left(\sinh\left(\frac{\pi}{p}2\,h\right)^2\right)\,,
\end{align}
\begin{align}
	\Balphaone_{1,1}\DEF\label{equdef:Balphaone11}
		&	\del_{\delta k}\Reu^{0}_{\del D_1} 
      			\begin{pmatrix*} 
      				\begin{pmatrix*} -\xi \\ h \end{pmatrix*}, 
      				\begin{pmatrix*} -\xi \\ h \end{pmatrix*}
      			\end{pmatrix*} +
      		\del_{\delta k}\Reu^{0}_{\del \BOm^1,+}
      			\begin{pmatrix*} 
      				\begin{pmatrix*} -\xi \\ h \end{pmatrix*}, 
      				\begin{pmatrix*} -\xi \\ h \end{pmatrix*}
      			\end{pmatrix*} \,,\\
    \Balphaone_{2,2}\DEF\label{equdef:Balphaone22}
      	&	\del_{\delta k}\Reu^{0}_{\del D_2} 
      			\begin{pmatrix*} 
      				\begin{pmatrix*} \xi \\ h \end{pmatrix*}, 
      				\begin{pmatrix*} \xi \\ h \end{pmatrix*}
      			\end{pmatrix*} +
      		\del_{\delta k}\Reu^{0}_{\del \BOm^1,+}
      			\begin{pmatrix*} 
      				\begin{pmatrix*} \xi \\ h \end{pmatrix*}, 
      				\begin{pmatrix*} \xi \\ h \end{pmatrix*}
      			\end{pmatrix*} \,,
\end{align}
\begin{align}
	\BT_{1,2}\DEF\label{equdef:BT12}
		&	\Neu^{0}_{\del \BOm^1,+}
      		\begin{pmatrix*} 
      			\begin{pmatrix*} -\xi \\ h \end{pmatrix*}, 
      			\begin{pmatrix*} +\xi \\ h \end{pmatrix*}
      		\end{pmatrix*} \,,\\
    \BT_{2,1}\DEF\label{equdef:BT21}
		&	\Neu^{0}_{\del \BOm^1,+}
      		\begin{pmatrix*} 
      			\begin{pmatrix*} +\xi \\ h \end{pmatrix*}, 
      			\begin{pmatrix*} -\xi \\ h \end{pmatrix*}
      		\end{pmatrix*} \,,
\end{align}
\begin{align}
	\BS_{1,2}\DEF \label{equdef:BS12}
		&	\del_{\delta k}\Neu^{0}_{\del \BOm^1,+}
      		\begin{pmatrix*} 
      			\begin{pmatrix*} -\xi \\ h \end{pmatrix*}, 
      			\begin{pmatrix*} +\xi \\ h \end{pmatrix*}
      		\end{pmatrix*} \,,\\
    \BS_{2,1}\DEF\label{equdef:BS21}
		&	\del_{\delta k}\Neu^{0}_{\del \BOm^1,+}
      		\begin{pmatrix*} 
      			\begin{pmatrix*} +\xi \\ h \end{pmatrix*}, 
      			\begin{pmatrix*} -\xi \\ h \end{pmatrix*}
      		\end{pmatrix*} \,.
\end{align}
\endgroup}
%\begin{remark}
%	We believe that $\BT_{1,2}=\BT_{2,1}$, which we believe can be proven using the fact that $\NdelBOonedkp(z,x)=\NdelBOonedkc(x,z)$, which can be proven using Green's formula, and the yet unproven claim that for $k=0$, $\NdelBOonedkp(z,x)=\NdelBOonedkc(z,x)$. The same might be applied for $\BS_{1,2}$ and $\BS_{2,1}$.
%\end{remark}
We define
\begin{align*}
	\BBalphanull\DEF
		&	\begin{bmatrix}
				\Balphanull_{1,1}/{|D_1|} & 0 \\ 0 & \Balphanull_{2,2}/{|D_2|} 
			\end{bmatrix}\,,\\
	\BBTnull\DEF
		&	\begin{bmatrix}
				0 & \BTnull_{1,2}/{|D_2|}  \\ \BTnull_{2,1}/{|D_1|} & 0
			\end{bmatrix}\,,\\
	\BBalphaone\DEF
		&	\begin{bmatrix}
				\Balphaone_{1,1}/{|D_1|} & 0  \\ 0 & \Balphaone_{2,2}/{|D_2|} 
			\end{bmatrix}\,,\\
	\BBTone\DEF
		&	\begin{bmatrix}
				0 & \BTone_{1,2}/{|D_2|}  \\ \BTone_{2,1}/{|D_1|} & 0 
			\end{bmatrix}\,.
\end{align*}
We obtain the second order analytic expression with the following lemma:
\begin{lemma}
	For all $l\in\NN$, we have that
	\begin{align*}
		\BSdke_{(l)}=&\OO(\leps^{l+1})\,,
	\end{align*}
	and for $\BSdke_{(1,1)}$, $\BSdke_{(1,2)}$ and $\BSdke_{(1,3)}$, we have
	\begin{align*}
		\BSdke_{(1,1)}=& \frac{\pi^2}{2}\leps^2(\BBalphanull+\BBTnull) \,, \\
		\BSdke_{(1,2)}=& \delta k\,\frac{\pi^2}{2}\leps^2(\BBalphaone+\BBTone) \,, \\
		\BSdke_{(1,3)}=& \OO(\eps\,\leps^2+ \leps^2\, \delta^2 k^2) \,.
	\end{align*}
\end{lemma}
\begin{proof}
	The proof follows from straightforward calculation using Lemma \ref{lemma:normLinvR} and the expressions in  (\ref{equdef:BRRdkeone})-(\ref{equdef:BRRdkethree}) and  (\ref{equdef:BSdke11})-(\ref{equdef:BSdke13}). For $\BSdke_{(1,3)}$,  we can use the argument in (\ref{equdef:SSdkel}).
\end{proof}
Now we can deduce that
\begin{align}\label{equ:BAdkeExp}
	\BAdke= 
		&	\frac{\pi}{2}  \leps
			\begin{bmatrix} 
				\frac{1}{|D_1|} & 0 \\ 0 & \frac{1}{|D_2|}   
			\end{bmatrix}	 
			+\frac{\pi^2}{4}\leps^2
			(\BBalphanull+\BBTnull)
			+\delta k\,\frac{\pi^2}{4}\leps^2
			(\BBalphaone+\BBTone)
			+\OO(\eps\,\leps^2+ \leps^2\, \delta^2 k^2)\,.
\end{align} 

\begin{lemma}\label{lemma:BAsqrtleps3/2Expansion}
	There exists a $\RR^{2\times 2}$-matrix $\BAdkesqrt$ such that $\BAdke = (\BAdkesqrt)^2$ and
	\begin{align*}
		\BAdkesqrt =\sqrt{\frac{\pi}{2} \leps}\left(
				\BA_{\star, D}
				+\frac{\pi\leps}{2}\BA_{\star, (0)}
				+\frac{\delta k \pi\leps}{2}\BA_{\star, (1)}\right)
				+\OO(\leps^{3/2}\, \delta^2 k^2)+\OO(\leps^{5/2})\,,
	\end{align*}
	where
	\begin{align*}
		\BA_{\star, D}
		=&		
					\begin{bmatrix}
						\frac{1}{\sqrt{|D_1|}} & 0 \\
						0 & \frac{1}{\sqrt{|D_2|}}
					\end{bmatrix}\,,\\
		\BA_{\star, (0)}
		=&		\begin{bmatrix}
					\frac{\Balphanull_{1,1}}{2\sqrt{|D_1|}} 
					& \frac{\BTnull_{1,2}\sqrt{|D_1|/|D_2|}}{\sqrt{|D_1|}+\sqrt{|D_2|}} \\
					\frac{\BTnull_{2,1}\sqrt{|D_2|/|D_1|}}{\sqrt{|D_1|}+\sqrt{|D_2|}}
					& \frac{\Balphanull_{2,2}}{2\sqrt{|D_2|}} 
				\end{bmatrix}\,,\\
		\BA_{\star, (1)}
		=&		\begin{bmatrix}
					\frac{\Balphaone_{1,1}}{2\sqrt{|D_1|}} 
					& \frac{\BTone_{1,2}\sqrt{|D_1|/|D_2|}}{\sqrt{|D_1|}+\sqrt{|D_2|}} \\
					\frac{\BTone_{2,1}\sqrt{|D_2|/|D_1|}}{\sqrt{|D_1|}+\sqrt{|D_2|}}
					& \frac{\Balphaone_{2,2}}{2\sqrt{|D_2|}} 
				\end{bmatrix}\,.
	\end{align*}	 
\end{lemma}

\begin{proof}
	We use the following approach:
	\begin{align*}
		\BAdkesqrt=\leps^{1/2}\BAdk_{\star,1} + \leps\BAdk_{\star, 2} + \leps^{3/2}\BAdk_{\star, 3} + \leps^{2}\BAdk_{\star,4}+\OO(\leps^{5/2}).
	\end{align*}
	Thus,
\begin{align*}
	\BAdke = (\BAdkesqrt)^2
	&=					(\leps^{1/2}\BAdk_{\star,1} 
						+ \leps\BAdk_{\star, 2} 
						+ \leps^{3/2}\BAdk_{\star, 3} 
						+ \leps^{2}\BAdk_{\star,4}
						+ \OO(\leps^2))^2 \\
	&\mkern-110mu=		(\BAdk_{\star,1})^2\leps 
						+ (\BAdk_{\star,1} \BAdk_{\star, 2}
						+\BAdk_{\star, 2}\BAdk_{\star,1})\leps^{3/2} \nonumber
						+((\BAdk_{\star, 2})^2+\BAdk_{\star,1}\BAdk_{\star, 3} +\BAdk_{\star, 3} \BAdk_{\star,1})\leps^2 \nonumber\\
	&\mkern-110mu		+(\BAdk_{\star,4}\BAdk_{\star,1}
						+\BAdk_{\star,1}\BAdk_{\star,4}
						+ \BAdk_{\star,2}\BAdk_{\star,3}
						+ \BAdk_{\star,3}\BAdk_{\star,2})
						+ \OO(\leps^{3}).
\end{align*}
Comparing this equation to (\ref{equ:BAdkeExp}), it follows that
\begin{align*}
	(\BAdk_{\star,1})^2 
			=&		\frac{\pi}{2}  
					\begin{bmatrix} 
						\frac{1}{|D_1|} & 0 \\ 0 & \frac{1}{|D_2|}   
					\end{bmatrix}\,, \\
	(\BAdk_{\star,1} \BAdk_{\star, 2}+\BAdk_{\star, 2}\BAdk_{\star,1})
			=& 		0\,,\\
	((\BAdk_{\star, 2})^2
	+\BAdk_{\star,1}\BAdk_{\star, 3} 
	+\BAdk_{\star, 3} \BAdk_{\star,1})
			=&		\frac{\pi^2}{4}\left(
					(\BBalphanull+\BBTnull)+
					\delta k (\BBalphaone+\BBTone)\right)+\OO(\delta^2 k^2)\,,\\
	(\BAdk_{\star,4}\BAdk_{\star,1}
	+\BAdk_{\star,1}\BAdk_{\star,4}
	+ \BAdk_{\star,2}\BAdk_{\star,3}&
	+ \BAdk_{\star,3}\BAdk_{\star,2})
			=		0\,,
\end{align*}
which implies
\begin{align*}
	\BAdk_{\star,1}
			=&		\sqrt{\frac{\pi}{2}}\begin{bmatrix}
						\frac{1}{\sqrt{|D_1|}} & 0 \\
						0 & \frac{1}{\sqrt{|D_2|}}
					\end{bmatrix}\,, \\
	\BAdk_{\star, 2}
			=&		0\,, \\
	\BAdk_{\star, 3}
			=&		\frac{\pi^{\frac{3}{2}}}{2\sqrt{2}}
					\begin{bmatrix}
						\frac{\Balphanull_{1,1}+\delta k\Balphaone_{1,1}}{2\sqrt{|D_1|}} 
						& \frac{(\BTnull_{1,2}+\delta k \BTone_{1,2})\sqrt{|D_1|/|D_2|}}{\sqrt{|D_1|}+\sqrt{|D_2|}} \\
						\frac{(\BTnull_{2,1}+\delta k \BTone_{2,1})\sqrt{|D_2|/|D_1|}}{\sqrt{|D_1|}+\sqrt{|D_2|}}
						& \frac{\Balphanull_{2,2}+\delta k\Balphaone_{2,2}}{2\sqrt{|D_2|}} 
					\end{bmatrix}
					+\OO(\delta^2 k^2)\,,\\
	\BAdk_{\star, 4}
			=&		0\,.		
\end{align*}
This leads us to
\begin{align}\label{equ:squareroota}
\BAdkesqrt
		=&		\sqrt{\frac{\pi}{2} \leps}\left(
					\begin{bmatrix}
						\frac{1}{\sqrt{|D_1|}} & 0 \\
						0 & \frac{1}{\sqrt{|D_2|}}
					\end{bmatrix}
					+
					\frac{\pi\leps}{2}\begin{bmatrix}
						\frac{\Balphanull_{1,1}}{2\sqrt{|D_1|}} 
						& \frac{\BTnull_{1,2}\sqrt{|D_1|/|D_2|}}{\sqrt{|D_1|}+\sqrt{|D_2|}} \\
						\frac{\BTnull_{2,1}\sqrt{|D_2|/|D_1|}}{\sqrt{|D_1|}+\sqrt{|D_2|}}
						& \frac{\Balphanull_{2,2}}{2\sqrt{|D_2|}} 
					\end{bmatrix}
					\right.\nonumber\\
		&			\left.+
					\frac{\delta k\,\pi\leps}{2}\begin{bmatrix}
						\frac{\Balphaone_{1,1}}{2\sqrt{|D_1|}} 
						& \frac{\BTone_{1,2}\sqrt{|D_1|/|D_2|}}{\sqrt{|D_1|}+\sqrt{|D_2|}} \\
						\frac{\BTone_{2,1}\sqrt{|D_2|/|D_1|}}{\sqrt{|D_1|}+\sqrt{|D_2|}}
						& \frac{\Balphaone_{2,2}}{2\sqrt{|D_2|}} 
					\end{bmatrix}
					\right)
					\!+\!
					\OO(\leps^{3/2}\, \delta^2 k^2)\!+\!\OO(\leps^{5/2})\,.
\end{align}
\end{proof}

With $\BAdkesqrt$ we can write $\delta^2 k^2\mathbb{I}_2-\BAdke=(\delta k\mathbb{I}_2-\BAdkesqrt)(\delta k\mathbb{I}_2 + \BAdkesqrt)$, thus it is enough to find the characteristic values of $(\delta k\mathbb{I}_2-\BAdkesqrt)$ and $(\delta k\mathbb{I}_2 + \BAdkesqrt)$ to get the characteristic values of $\delta^2 k^2\mathbb{I}_2-\BAdke$.

We define 
	\begin{align}\label{equ:BAestarstar}
		\BA^\eps_{\star,\star} 
			\DEF 
					\sqrt{\frac{\pi}{2}\leps}\left(
						\BA_{\star, D}
						+\frac{\pi\leps}{2}\BA_{\star, (0)}
					\right)\,.
	\end{align}
Consider that $\BAdkesqrt= \BA^\eps_{\star,\star} + \frac{\delta k (\pi\leps)^{3/2}}{2}\BA_{\star, (1)}+\OO(\leps^{5/2}+ \leps^{3/2}\, \delta^2 k^2)$
%We define
%\begin{align}
%	\tBBTnull\DEF 
%		&		\begin{bmatrix}
%					0
%					& \frac{\BTnull_{1,2}\sqrt{|D_1|/|D_2|}}{\sqrt{|D_1|}+\sqrt{|D_2|}} \\
%					\frac{\BTnull_{2,1}\sqrt{|D_2|/|D_1|}}{\sqrt{|D_1|}+\sqrt{|D_2|}}
%					& 0
%				\end{bmatrix}
%				\FED
%				\begin{bmatrix}
%					0
%					& \beta_{1,2} \\
%					\beta_{2,1}
%					& 0
%				\end{bmatrix}\,.
%\end{align}

Now for $\eps$ small enough, $\BAdkess$ is diagonalizable with $\BAdkess=\BrmYss\BrmMss(\BrmYss)^{-1}$ where $\BrmYss=[(\rmYss)_1\,, (\rmYss)_2]$, with $(\rmYss)_1$ being the normalized eigenvector to the eigenvalue $\delta k^{\delta,\eps}_{\star,1}$ of $\BAdkess$ and $(\rmYss)_2$ being the normalized eigenvector to the eigenvalue $\delta k^{\delta,\eps}_{\star,2}$ of $\BAdkess$, and for $\eps$ small enough, but not zero, $\delta k^{\delta,\eps}_{\star,1}$ and $\delta k^{\delta,\eps}_{\star,2}$ are distinct, since $\BA_{\star, (0)}$ is not zero and thus $\BAdkess$ is not similar to $\mathbb{I}_2$. Then we sort $\delta k^{\delta,\eps}_{\star,1}$ and $\delta k^{\delta,\eps}_{\star,2}$ such that the real part of $\delta k^{\delta,\eps}_{\star,1}$ is greater or equal to $\delta k^{\delta,\eps}_{\star,2}$
%\begin{align}
%	\rmY_1
%		=&		\frac{1}{\sqrt{\beta_{1,2}+\beta_{2,1}}}\begin{bmatrix}
%					-\sqrt{\beta_{1,2}} \\ \sqrt{\beta_{2,1}}
%				\end{bmatrix}
%				\,, \quad
%	\rmY_2
%		=		\frac{1}{\sqrt{\beta_{1,2}+\beta_{2,1}}}\begin{bmatrix}
%					\sqrt{\beta_{1,2}} \\ \sqrt{\beta_{2,1}}
%				\end{bmatrix}\,.
%\end{align}
and 
\begin{align*}
	\BrmMss = 	\begin{bmatrix}
					\delta k^{\delta,\eps}_{\star,1} & 0 
					\\ 0 & \delta k^{\delta,\eps}_{\star,2}
				\end{bmatrix}\,.
\end{align*}
We can explicitly compute those values:
\begin{lemma}
	We have that
	\begin{align*}
		\delta k^{\delta,\eps}_{\star, 1} 
			=&	\sqrt{\pi\leps}
				\left(
					\frac{1}{2}
					\left(
						k^{\eps}_{\star,\mathrm{tr1}}
						+k^{\eps}_{\star,\mathrm{tr2}}
					\right)
					+
					\left[
						\frac{1}{4}
						\left(
							k^{\eps}_{\star,\mathrm{tr1}}
							-k^{\eps}_{\star,\mathrm{tr2}}
						\right)^2
						-k^{\eps}_{\star,\mathrm{det}}
					\right]^{1/2}
				\right)\,,\\
		\delta k^{\delta,\eps}_{\star, 2} 
			=&	\sqrt{\pi\leps}
				\left(
					\frac{1}{2}
					\left(
						k^{\eps}_{\star,\mathrm{tr1}}
						+k^{\eps}_{\star,\mathrm{tr2}}
					\right)
					-
					\left[
						\frac{1}{4}
						\left(
							k^{\eps}_{\star,\mathrm{tr1}}
							-k^{\eps}_{\star,\mathrm{tr2}}
						\right)^2
						-k^{\eps}_{\star,\mathrm{det}}
					\right]^{1/2}
				\right)\,,
	\end{align*}
	where
	\begin{align*}
		k^{\eps}_{\star,\mathrm{tr1}} 
				=&	\frac{1}{\sqrt{|D_1|}}\left( 1+\frac{\Balphanull_{1,1}\pi\leps}{4} \right)\,,\\
		k^{\eps}_{\star,\mathrm{tr2}} 
				=&	\frac{1}{\sqrt{|D_2|}}\left( 1+\frac{\Balphanull_{2,2}\pi\leps}{4} \right)\,,\\
		k^{\eps}_{\star,\mathrm{det}} 
				=&	\frac{\pi^2\leps^2}{4}\frac{\BTnull_{1,2}\BTnull_{2,1}}{(\sqrt{|D_1|}+\sqrt{|D_2|})}\,.
	\end{align*}
\end{lemma}

With this lemma we especially see that $\delta k^{\delta,\eps}_{\star,1}=\OO(\leps^{1/2})$ and $\delta k^{\delta,\eps}_{\star,2}=\OO(\leps^{1/2})$.

\begin{proposition}
	There exists exactly 2 characteristic values for each of the matrix-valued functions $\delta k\mathbb{I}_2-\BAdkesqrt$ and $\delta k\mathbb{I}_2+\BAdkesqrt$. For $j\in\{1,2\}$, these are 
	\begin{align}
		\delta k^{\delta, \eps}_{j,+} =& \delta k^{\delta,\eps}_{\star,j}\bigg(1+\frac{(\pi\leps)^{\tfrac{3}{2}}}{4}\mathrm{e}_j^\TransT(\BrmYss)^{-1}\BA_{\star, (1)}(\rmYss)_j\bigg)+\OO(\leps^{5/2})\,,\label{equ:kdejp}\\
		\delta k^{\delta, \eps}_{j,-} =& \delta k^{\delta,\eps}_{\star,j}\bigg(-\!1+\frac{(\pi\leps)^{\tfrac{3}{2}}}{4}\mathrm{e}_j^\TransT(\BrmYss)^{-1}\BA_{\star, (1)}(\rmYss)_j\bigg)+\OO(\leps^{5/2})\,.\label{equ:kdejm}
	\end{align}
	%and the corresponding characteristic values are given as
	%\begin{align}
	%	\rmY_{j,+} =& \rmY_j + kjfdghksjldfhgsjkdfgh3lkjdfghsdjkfg\\
	%	\rmY_{j,-} =& \rmY_j + kjfdghksjldfhgsjkdfgh3lkjdfghsdjkfg\,.
	%\end{align}
\end{proposition}

\begin{proof}
	\textbf{Step 1: Non-perturbed characteristic values.} 
	Let us find the characteristic values and the corresponding vectors for the matrix-valued function $\delta k - \BAdke_{\star,\star}$. By definition,
	\begin{align*}
		(\BrmYss)^{-1}(\delta k \mathbb{I}_2 - \BAdkess)\BrmYss = \delta k \mathbb{I}_2 - \BrmMss\,.
	\end{align*}
	Thus we see that the characteristic values of $(\delta k \mathbb{I}_2 - \BAdkess)$ are $k^{\delta,\eps}_{\star,1}$ and $k^{\delta,\eps}_{\star,2}$ with the characteristic vectors $(\rmYss)_1$ and $(\rmYss)_2$ and similarly we see that the characteristic values of $(\delta k \mathbb{I}_2 + \BAdkess)$ are $-k^{\delta,\eps}_{\star,1}$ and $-k^{\delta,\eps}_{\star,2}$ with the characteristic vectors $(\rmYss)_1$ and $(\rmYss)_2$.\\
	\textbf{Step 2: Existence of the perturbed characteristic values near the unperturbed ones.}
	We now apply the generalized Rouch\'e’s theorem to obtain the existence of the characteristic values for $\delta k\mathbb{I}_2 \!-\! \BAdkesqrt$. Observe that $(\delta k\mathbb{I}_2\!-\!\BAdkess)^{-1}\!=\! \BrmYss (\delta k\mathbb{I}_2\!-\!\BrmMss)^{-1}(\BrmYss)^{-1}$, where $\NORM{\BrmYss}=\OO(1)$, because of the normalization, thus $\NORM{(\delta k^{\delta,\eps}_{\star,j}\mathbb{I}_2\!-\!\BAdkess)^{-1}}\!=\!\OO(\leps^{-\tfrac{1}{2}})$. We define the domains $W_1\DEF\{k\in\kkkkc\MID |k-k^{\delta, \eps}_{\star,1}|<C_{1}\leps^2\}$ and $W_2\DEF\{k\in\kkkkc\MID |k-k^{\delta, \eps}_{\star,2}|<C_{2}\leps^2\}$ where $C_1,C_2>0$. Since $\BAdkesqrt- \BA^\eps_{\star,\star} = \delta k \OO(\leps^{3/2})+\OO(\leps^{5/2}+ \leps^{3/2}\, \delta^2 k^2)$ and $k^{\delta,\eps}_{\star,1}$ and $k^{\delta,\eps}_{\star,2}$ are pairwise different, we can conclude that for $\eps$ sufficiently small, there exists $C_1$ such that the following inequality holds
	\begin{align*}
		\NORM{(\delta k\mathbb{I}_2-\BAdkess)^{-1}(\BAdkesqrt-\BAdkess)}< 1\quad\text{for }k\in\del W_1\,, 
	\end{align*}
	and the same holds for $C_2$. Then the generalized Rouch\'e’s theorem yields that there exists exactly one characteristic value $k^{\delta, \eps}_{1,+}$ and one $k^{\delta, \eps}_{2,+}$ in the domain $W_1\cup W_2$ for $\delta k \mathbb{I}_2-\BAdkesqrt$, thus $k^{\delta, \eps}_{1,+}=k^{\delta, \eps}_{\star,1}+\OO(\leps^2)$ and $k^{\delta, \eps}_{2,+}=k^{\delta, \eps}_{\star,2}+\OO(\leps^2)$. For $\eps$ small enough, $W_1\cup W_2\subset\kkkkc$ and with Proposition \ref{prop:exactly4charvalues}, $k^{\delta, \eps}_{1,+}$ and one $k^{\delta, \eps}_{2,+}$ are two characteristic values of the four of $\BAAdke$ in $\kkkkc$. We can apply the same argument for $\delta k + \BAdkesqrt$.\\
	\textbf{Step 3: Expansion of the characteristic vectors.}
	Let $\rmY_1$ and $\rmY_2$ be the characteristic vectors to $\delta k - \BAdkesqrt$ to the characteristic values $k^{\delta, \eps}_{1,+}$ and $k^{\delta, \eps}_{2,+}$. We show that $$\rmY_j=(\rmYss)_j+ \OO(\sqrt{\leps})$$ for $j\in\{1,2\}$. Indeed, note that
	\begin{align*}
		\delta k^{\delta,\eps}_{j,+}\mathbb{I}_2-\BAdkesqrt\MID_{k=k^{\delta,\eps}_{j,+}}
			=	\delta k^{\delta,\eps}_{j,+}\mathbb{I}_2-\BAdkess+\OO(\leps^2)\,,
	\end{align*}
	since $\BAdkesqrt- \BA^\eps_{\star,\star} = \OO(\delta k \leps^{3/2})+\OO(\leps^{5/2}+ \leps^{3/2}\, \delta^2 k^2)$ and $k^{\delta, \eps}_{1,+}=k^{\delta, \eps}_{\star,1}+\OO(\leps^2)$ and $$k^{\delta, \eps}_{2,+}=k^{\delta, \eps}_{\star,2}+\OO(\leps^2).$$ Using that $\delta k- \BAdkess=\BrmYss (\delta k \mathbb{I}_2-\BrmMss)(\BrmYss)^{-1}$ we get
	\begin{align*}
		(\delta k^{\delta,\eps}_{j,+} \mathbb{I}_2-\BrmMss)(\BrmYss)^{-1}\rmY_j=\OO(\leps^2)\,.
	\end{align*}
	Using again that $k^{\delta, \eps}_{1,+}=k^{\delta, \eps}_{\star,1}+\OO(\leps^2)$ and $k^{\delta, \eps}_{2,+}=k^{\delta, \eps}_{\star,2}+\OO(\leps^2)$ and the definition of $\BrmMss$ we see with $(\BrmYss)^{-1}=((\BrmYss)^\TransT\,\BrmYss)^{-1}\,(\BrmYss)^{\TransT}$ and $\NORM{((\BrmYss)^\TransT\,\BrmYss)^{-1}}=\OO(1)$ that
	\begin{align*}
		(\rmYss)_1^{\TransT}\rmY_1&=\OO(1)\,, &&(\rmYss)_1^{\TransT}\rmY_2=\OO(\leps^{1/2})\,,\\
		(\rmYss)_2^{\TransT}\rmY_1&=\OO(\leps^{1/2})\,, &&(\rmYss)_2^{\TransT}\rmY_2=\OO(1)\,.
	\end{align*}
	It follows that $\rmY_j$ can be written as 
	\begin{align}\label{equ:rmYjEXPANSION}
		\rmY_j= (\rmYss)_j+\leps^{1/2}\rmY_{j,(1)}+\OO(\leps)\,.
	\end{align}
%	where we impose that $\rmY_{j,(1)}=\rho_j(\rmYss)_{\hat{j}}$ for $j,\hat{j}\in\{1,2\}$, $j\neq\hat{j}$. \\
	\textbf{Step 4: Expansions for the characteristic values}\\
	Using $k^{\delta,\eps}_{j,+}=k^{\delta,\eps}_{\star,j}+\OO(\leps^2)$, we obtain that
	\begin{align*}
		\delta k^{\delta,\eps}_{j,+} -  \BAdkesqrt\MID_{k=k^{\delta,\eps}_{j,+}}
			&=					\delta k^{\delta,\eps}_{j,+} - \BAdkess + \BAdkess - \BAdkesqrt\MID_{k=k^{\delta,\eps}_{j,+}}\\
			&\mkern-110mu=		\BrmYss (\delta k^{\delta,\eps}_{j,+} \mathbb{I}_2\!-\!\BrmMss)(\BrmYss)^{-1}
								\!-\!\delta k^{\delta,\eps}_{j,+} \frac{(\pi\leps)^{\tfrac{3}{2}}}{4}\BA_{\star, (1)}
								\!+\!\OO(\leps^{5/2}+ \leps^{3/2}\, \delta^2 (k^{\delta,\eps}_{j,+})^2 )\\
			&\mkern-110mu=		\BrmYss (\delta k^{\delta,\eps}_{j,+} \mathbb{I}_2-\BrmMss)(\BrmYss)^{-1}
								-\delta k^{\delta,\eps}_{\star,j} \frac{(\pi\leps)^{\tfrac{3}{2}}}{4}\BrmYss(\BrmYss)^{-1}\BA_{\star, (1)}\BrmYss(\BrmYss)^{-1}
								+\OO(\leps^{5/2})\\
			&\mkern-110mu=		\BrmYss \Bigg(
									\delta k^{\delta,\eps}_{j,+} \mathbb{I}_2-\BrmMss
									-\frac{\delta k^{\delta,\eps}_{\star,j} (\pi\leps)^{\tfrac{3}{2}}}{4}(\BrmYss)^{-1}\BA_{\star, (1)}\BrmYss
								\Bigg)(\BrmYss)^{-1}
								+\OO(\leps^{5/2})\,.
	\end{align*}
	Since $(\delta k^{\delta,\eps}_{j,+} -  \BAdkesqrt\MID_{k=k^{\delta,\eps}_{j,+}})\BrmY_j=0$ and $\rmY_j= (\rmYss)_j+\leps^{1/2}\rmY_{j,(1)}+\OO(\leps)$ we have
	\begin{multline}
		\Bigg(\!\delta k^{\delta,\eps}_{j,+} \mathbb{I}_2-\BrmMss
				-\!\frac{\delta k^{\delta,\eps}_{\star,j} (\pi\leps)^{\tfrac{3}{2}}}{4}(\BrmYss)^{-1}\BA_{\star, (1)}\BrmYss\Bigg)\!(\BrmYss)^{-1}
				\left((\rmYss)_j+\leps^{1/2}\rmY_{j,(1)}\!+\!\OO(\leps)\right)\\
				+\OO(\leps^{5/2})=0
	\end{multline}
	We can rewrite this as
	\begin{multline}
		\bigg(\delta k^{\delta,\eps}_{j,+} \mathbb{I}_2-\BrmMss\bigg)(\BrmYss)^{-1}\left((\rmYss)_j+\leps^{1/2}\rmY_{j,(1)}+\OO(\leps)\right)\\
				=\frac{\delta k^{\delta,\eps}_{\star,j} (\pi\leps)^{\tfrac{3}{2}}}{4}(\BrmYss)^{-1}\BA_{\star, (1)}(\rmYss)_j+\OO(\leps^{5/2})\,.
	\end{multline}
	Consider that $(\BrmYss)^{-1}(\rmYss)_1=\mathrm{e}_1\DEF[1 , 0]^\TransT$ and $(\BrmYss)^{-1}(\rmYss)_2=\mathrm{e}_2\DEF[0 , 1]^\TransT$. This results in
	\begin{align*}
		\delta (k^{\delta,\eps}_{j,+}-k^{\delta,\eps}_{\star,j})
			=&		\delta k^{\delta,\eps}_{\star,j} \frac{(\pi\leps)^{\tfrac{3}{2}}}{4}\mathrm{e}_j\cdot(\BrmYss)^{-1}\BA_{\star, (1)}(\rmYss)_j+\OO(\leps^{5/2})\,.%,\\
%		\delta (k^{\delta,\eps}_{j,+}-k^{\delta,\eps}_{\star,\hat{j}})\rho_j\leps^{1/2}\
%			=&		\delta k^{\delta,\eps}_{\star,j} \frac{(\pi\leps)^{\tfrac{3}{2}}}{2}\mathrm{e}_{\hat{j}}\cdot(\BrmYss)^{-1}\BA_{\star, (1)}(\rmYss)_j+\OO(\leps^{5/2})\,.	
	\end{align*}
	Thus
	\begin{align*}
		\delta k^{\delta,\eps}_{j,+} = \delta k^{\delta,\eps}_{\star,j}\bigg(1+\frac{(\pi\leps)^{\tfrac{3}{2}}}{4}\mathrm{e}_j^\TransT(\BrmYss)^{-1}\BA_{\star, (1)}(\rmYss)_j\bigg)+\OO(\leps^{5/2})\,.
	\end{align*}
%	\todo{If we need $\rho_j$ write out the formula for it, else kill above equation. Then you can even completely kill of $(\rmY)_{j,(1)}= \rho_j (\rmYss)_{\hat{j}}$\\}
	By a similar procedure, we can prove (\ref{equ:kdejm}).
\end{proof}

\subsubsection{Inversion of \boldmath$\AAdnoboldke$- Solving the First Order Linear Equation}
We know now that $\BAAdnoboldke$ is invertible for $\delta k\in\kkkkc$, except at the characteristic values $k=k^{\delta,\eps}_{j,+}$ and $k=k^{\delta,\eps}_{j,-}$ and at the pole, $k=0$. Let us examine, how we can express $(\BAAdnoboldke)^{-1}[\Bfdk]$, where $\Bfdk$ is given by  (\ref{equdef:Bfdk}), which uses  (\ref{equdef:Ofdk}).

First consider that we already know that the equation $\BAAdke[\Bmu]=\Bfdk$ is equivalent to
\begin{align}\label{equ:mu=BLinvBK+BLBf}
	\frac{2}{\pi}\Bmu+\frac{\BwLLdkeinv\BKKe[\Bmu]}{(\delta k)^2}=\BwLLdkeinv[\Bfdk]\,.
\end{align}
Similar to the proof of Lemma \ref{lemma:BAAdkeCharvalueAreDK2-BAdke}, we get that
\begin{align}\label{equ:dk2-BA=fl}
	(\delta^2 k^2\mathbb{I}_2-\BAdke)
	\begin{bmatrix}
		(\mu_1\,,1)_{\eps} \\ (\mu_2\,,1)_{\eps}
	\end{bmatrix}
	=
	\frac{\pi}{2}(\delta^2 k^2)
	\begin{bmatrix}
		(\BwLLdkeinv[\Bfdk]\,,\boldeone)_{\eps\oplus\eps} \\ (\BwLLdkeinv[\Bfdk]\,,\boldetwo)_{\eps\oplus\eps}
	\end{bmatrix}\,.
\end{align}
We define
\begin{align*}
	\fdk_{\calL}\DEF
	\begin{bmatrix}
		(\BwLLdkeinv[\Bfdk]\,,\boldeone)_{\eps\oplus\eps} \\ (\BwLLdkeinv[\Bfdk]\,,\boldetwo)_{\eps\oplus\eps}
	\end{bmatrix}\,,\quad
	\fdk_{\boldnull}\DEF
		\begin{bmatrix}
		f_{D_1}^{\delta k}(0) \\ f_{D_2}^{\delta k}(0)
	\end{bmatrix} \,.
\end{align*}
\begin{lemma}\label{lemma:inv(dk^2-Badke)}
	The inverse of $\delta^2 k^2\mathbb{I}_2-\BAdke$ has the following representation:
	\begin{align*}
		(\delta^2 k^2\mathbb{I}_2-\BAdke)^{-1} = (\BrmYss)(\BrmMdkone\BrmMdktwo)^{-1}(\BrmYss)^{-1}
			+ \BMdkerest\,,
	\end{align*}
	where
	\begin{align*}
		\BrmMdkone&=\begin{bmatrix}
			\delta k - \delta k^{\delta,\eps}_{1,+} & 0 \\ 0 & \delta k - \delta k^{\delta,\eps}_{2,+}
		\end{bmatrix}\,,\quad
		\BrmMdktwo=\begin{bmatrix}
			\delta k - \delta k^{\delta,\eps}_{1,-} & 0 \\ 0 & \delta k - \delta k^{\delta,\eps}_{2,-}
		\end{bmatrix}\,,\\
		\BMdkerest &= \begin{bmatrix}
			\frac{\OO(\leps^{1/2})}{\delta k - \delta k^{\delta,\eps}_{1,+}} 
			+\frac{\OO(\leps^{1/2})}{\delta k - \delta k^{\delta,\eps}_{1,-}} 
			& \frac{\OO(\leps^{1/2})}{\delta k - \delta k^{\delta,\eps}_{1,+}} 
			+\frac{\OO(\leps^{1/2})}{\delta k - \delta k^{\delta,\eps}_{2,-}}
			\\ \frac{\OO(\leps^{1/2})}{\delta k - \delta k^{\delta,\eps}_{2,+}} 
			+\frac{\OO(\leps^{1/2})}{\delta k - \delta k^{\delta,\eps}_{1,-}} 
			& \frac{\OO(\leps^{1/2})}{\delta k - \delta k^{\delta,\eps}_{2,+}} 
			+\frac{\OO(\leps^{1/2})}{\delta k - \delta k^{\delta,\eps}_{2,-}}
		\end{bmatrix}\,.
	\end{align*}
\end{lemma}
\begin{proof}
	As we know, 
	\begin{align*}
		\delta^2 k^2\mathbb{I}_2-\BAdke=(\delta k-\BAdkesqrt)(\delta k+\BAdkesqrt).
	\end{align*}
	Since $(\delta k^{\delta, \eps}_{j,+}\mathbb{I}_2-\BAdkesqrt\MID_{k=k^{\delta,\eps}_{j,+}})\rmY_j=0$, thus $\delta k^{\delta, \eps}_{j,+}\rmY_j=\BAdkesqrt\!\MID\!_{k=k^{\delta,\eps}_{j,+}}\rmY_j$, hence $$\BAdkesqrt\rmY_j= \delta k^{\delta, \eps}_{j,+}\rmY_j + \OO(\leps^{1/2}),$$ we have then with  (\ref{equ:rmYjEXPANSION})
	\begin{align*}
		(\delta k-\BAdkesqrt)\rmY_j 
		=	(\delta k-\delta k^{\delta, \eps}_{j,+})(\rmY_j + \OO(\leps^{1/2}))
		=	(\delta k-\delta k^{\delta, \eps}_{j,+})((\rmYss)_{j}+\OO(\leps^{1/2}))\,,
	\end{align*}
	which implies
	\begin{align*}
		(\delta k-\BAdkesqrt)\BrmY =((\BrmYss)+\OO({\leps^{1/2}}))\BrmMdkone,
	\end{align*}
	that is,
	\begin{align*}
		(\delta k-\BAdkesqrt)^{-1}=((\BrmYss)+\OO({\leps^{1/2}}))(\BrmMdkone)^{-1}((\BrmYss)+\OO({\leps^{1/2}}))^{-1}\,,
	\end{align*}
	and analogously we get
	\begin{align*}
		(\delta k+\BAdkesqrt)^{-1}=((\BrmYss)+\OO({\leps^{1/2}}))(\BrmMdktwo)^{-1}((\BrmYss)+\OO({\leps^{1/2}}))^{-1}\,.
	\end{align*}
	This leads us to
	\begin{multline}
		(\delta^2 k^2\mathbb{I}_2-\BAdke)^{-1}
			=	((\BrmYss)+\OO({\leps^{1/2}}))(\BrmMdkone)^{-1}((\BrmYss)+\OO({\leps^{1/2}}))^{-1}\\
				((\BrmYss)+\OO({\leps^{1/2}}))(\BrmMdktwo)^{-1}((\BrmYss)+\OO({\leps^{1/2}}))^{-1}\,.
	\end{multline}
	Using 
	\begin{align*}
		((\BrmYss)+\OO({\leps^{1/2}}))^{-1}((\BrmYss)+\OO({\leps^{1/2}}))
				=&		\mathbb{I}_2+\OO(\leps^{1/2})\,,\\
		((\BrmYss)+\OO({\leps^{1/2}}))^{-1}
				=&		(\BrmYss)^{-1}+\OO(\leps^{1/2})\,,
	\end{align*}
	we proved Lemma \ref{lemma:inv(dk^2-Badke)}.
\end{proof}

\begin{proposition}\label{prop:BLSEsolution}
	Let $\delta k\in\kkkkc\setminus\{ 0, \delta k^{\delta,\eps}_{1,+}, \delta k^{\delta,\eps}_{2,+},\delta k^{\delta,\eps}_{1,-},\delta k^{\delta,\eps}_{2,-} \}$, there exists a unique solution to the equation $\BAAdnoboldke[\Bmu]=\Bfdk$. Moreover, the solution can be written as $\Bmu=\Bmu_\star+\Bmu_\sim$, where
	\begin{align*}
		\Bmu_\star=& \,\leps\frac{\pi^2}{4}
					\Big[\LLeinv[1]\Big(\BA_{D}\left((\BrmYss)(\BrmMdkone\BrmMdktwo)^{-1}(\BrmYss)^{-1}\right)\fdk_{\boldnull}\Big)_j\Big]_{j=1,2}
					+\frac{\pi}{2}\BLLeinv[\fdk_{\boldnull}]\,,\\
		\Bmu_\sim=&  \OO(\delta)\bigg((\OO_{\BcurlXe}(\leps^{5/2})\!+\!\OO_{\BcurlXe}(\eps))
					\Big(\norm{(\BrmMdkone\BrmMdktwo)^{-1}}\!+\!\norm{\BMdkerest}\Big)%\nonumber\\
					\!+\!(\OO_{\BcurlXe}(\leps^{2})\!+\!\OO_{\BcurlXe}(\eps))\bigg)\,.
	\end{align*}
	where 
	\begin{align*}
		\BA_{D}\DEF
			\begin{bmatrix}
					1/|D_1| & 0 \\ 0 & 1/|D_2|
			\end{bmatrix}\,.
	\end{align*}
\end{proposition}

\begin{proof}
	From  (\ref{equ:dk2-BA=fl}), we get that
	\begin{align*}
		(\delta^2 k^2\mathbb{I}_2-\BAdke)
		\begin{bmatrix}
			(\mu_1\,,1)_{\eps} \\ (\mu_2\,,1)_{\eps}
		\end{bmatrix}
		=
		\frac{\pi}{2}(\delta^2 k^2)\fdk_{\calL}\,.
	\end{align*}
	Thus, Lemma \ref{lemma:inv(dk^2-Badke)} gives that
	\begin{align*}
		\begin{bmatrix}
			(\mu_1\,,1)_{\eps} \\ (\mu_2\,,1)_{\eps}
		\end{bmatrix}
		=((\BrmYss)(\BrmMdkone\BrmMdktwo)^{-1}(\BrmYss)^{-1}
			+ \BMdkerest)\frac{\pi}{2}(\delta^2 k^2)\fdk_{\calL}\,.
	\end{align*}
	Note that similarly to Lemma \ref{lemma:EstLinvf}, we have
	\begin{align*}
		\BwLLdkeinv[\Bfdk]
			=&		\begin{bmatrix}
						\fdk_{D_1}(0) \LLeinv[1] + \OO(\delta)(\OO_{\curlXe}(\leps^2)+\OO_{\curlXe}(\eps))\\
						\fdk_{D_2}(0) \LLeinv[1] + \OO(\delta)(\OO_{\curlXe}(\leps^2)+\OO_{\curlXe}(\eps))
					\end{bmatrix}\\
			=&		\BLLeinv[\fdk_{\boldnull}]+
					\OO(\delta)(\OO_{\BcurlXe}(\leps^2)+\OO_{\BcurlXe}(\eps))\,.
	\end{align*}
	%and
	%\begin{align}
	%	\BwLLdkeinv[\boldone] 
	%		=&		2\BLLeinv[\boldone]+
	%				\OO_{\BcurlXe}(\leps^2)+\OO_{\BcurlXe}(\eps)\,.
	%\end{align}
	From  (\ref{equ:mu=BLinvBK+BLBf}), it follows that
	\begin{align*}
		\Bmu =& 		-\frac{\pi}{2}\frac{\BwLLdkeinv\BKKe[\mu]}{(\delta k)^2}+\frac{\pi}{2}\BwLLdkeinv[\Bfdk]\\
			=& 		-\frac{\pi}{2(\delta k)^2}\BwLLdkeinv\left[\BA_{D}((\BrmYss)(\BrmMdkone\BrmMdktwo)^{-1}(\BrmYss)^{-1}+ \BMdkerest)\frac{\pi}{2}(\delta^2 k^2)\fdk_{\calL}\right]\nonumber\\
			&		+\frac{\pi}{2}\BwLLdkeinv[\Bfdk]\\
			=&		-\frac{\pi^2}{4}
					\Big[\LLeinv[1]\Big(\BA_{D}\left((\BrmYss)(\BrmMdkone\BrmMdktwo)^{-1}(\BrmYss)^{-1}+ \BMdkerest\right)\fdk_{\calL}\Big)_j\Big]_{j=1,2}\\
			&		\cdot\Big(1+\OO_{\BcurlXe}(\leps^1)+\OO_{\BcurlXe}(\eps)\Big)
					+\frac{\pi}{2}\left(\BLLeinv[\fdk_{\boldnull}]+
					\OO(\delta)(\OO_{\BcurlXe}(\leps^2)+\OO_{\BcurlXe}(\eps))\right)\nonumber\,.
	\end{align*}
	Consider that
	\begin{align*}
		\fdk_{\calL}=-\leps \fdk_\boldnull+\OO(\delta)(\OO(\leps^2)+\OO(\eps))\,.
	\end{align*}
	This leads us to
	{\setlength{\belowdisplayskip}{0pt} \setlength{\belowdisplayshortskip}{0pt}\setlength{\abovedisplayskip}{0pt} \setlength{\abovedisplayshortskip}{0pt}
	\begin{align*}
		\Bmu
			=&		-\frac{\pi^2}{4}
					\Big[\LLeinv[1]\Big(\BA_{D}\left((\BrmYss)(\BrmMdkone\BrmMdktwo)^{-1}(\BrmYss)^{-1}\right)\fdk_{\calL}\Big)_j\Big]_{j=1,2}			
					+\frac{\pi}{2}\BLLeinv[\fdk_{\boldnull}]\nonumber\\
			&		+\Big((\OO_{\BcurlXe}(\leps^1)+\OO_{\BcurlXe}(\eps)\Big)
					\Big[\LLeinv[1]\Big(\BA_{D}\left((\BrmYss)(\BrmMdkone\BrmMdktwo)^{-1}(\BrmYss)^{-1}+ \BMdkerest\right)\fdk_{\calL}\Big)_j\Big]_{j=1,2}\nonumber\\
			&		+\OO(\delta)(\OO_{\BcurlXe}(\leps^2)+\OO_{\BcurlXe}(\eps))
					+\OO_{\curlXe}(\leps)\Big[\BA_{D}\BMdkerest\fdk_{\calL}\Big]_{j=1,2}\\
%	\end{align*}
%	\begin{align}
			=&		\leps\frac{\pi^2}{4}
					\Big[\LLeinv[1]\Big(\BA_{D}\left((\BrmYss)(\BrmMdkone\BrmMdktwo)^{-1}(\BrmYss)^{-1}\right)\fdk_{\boldnull}\Big)_j\Big]_{j=1,2}
					+\frac{\pi}{2}\BLLeinv[\fdk_{\boldnull}]\nonumber\\
			&		+\OO(\delta)(\OO_{\BcurlXe}(\leps^3)+\OO_{\BcurlXe}(\eps\leps))
					\left(\norm{(\BrmMdkone\BrmMdktwo)^{-1}}+\norm{\BMdkerest}\right)\nonumber\\
			&		+\OO(\delta)(\OO_{\BcurlXe}(\leps^2)+\OO_{\BcurlXe}(\eps))
					+\OO(\delta)\OO_{\curlXe}(\leps^{5/2})\norm{\BMdkerest}\,,	\\
%	\end{align}
%	\begin{align}
			=&		\leps\frac{\pi^2}{4}
					\Big[\LLeinv[1]\Big(\BA_{D}\left((\BrmYss)(\BrmMdkone\BrmMdktwo)^{-1}(\BrmYss)^{-1}\right)\fdk_{\boldnull}\Big)_j\Big]_{j=1,2}
					+\frac{\pi}{2}\BLLeinv[\fdk_{\boldnull}]\nonumber\\
			&		+\OO(\delta)(\OO_{\BcurlXe}(\leps^{5/2})+\OO_{\BcurlXe}(\eps))
					\left(\norm{(\BrmMdkone\BrmMdktwo)^{-1}}+\norm{\BMdkerest}\right)\nonumber\\
			&		+\OO(\delta)(\OO_{\BcurlXe}(\leps^{2})+\OO_{\BcurlXe}(\eps))\,.
	\end{align*}}
	This proves Proposition \ref{prop:BLSEsolution}.
\end{proof}

\subsubsection{Asymptotic Expansion of our Solution to the Physical Problem}
In Proposition \ref{prop:2GapFormulaINOperators} we established
\begin{align}\label{equ:2GapFormulaINOperatorsRepeat:1}
   		\BAAdke\begin{bmatrix}\del_\nu \udk\MID_{\Lambda_1} \\ \del_\nu \udk\MID_{\Lambda_2} \end{bmatrix}(\tau)+\sum_{n=1}^\infty\delta^n\,2 \,\BcalGG^{k,\eps}_{+,n}\begin{bmatrix}\del_\nu \udk\MID_{\Lambda_1} \\ \del_\nu \udk\MID_{\Lambda_2} \end{bmatrix}(\tau) = \Bfdk(\tau).
   \end{align}
On the other hand, we know that for $\delta k\in\kkkkc\setminus\{ 0, \delta k^{\delta,\eps}_{1,+}, \delta k^{\delta,\eps}_{2,+},\delta k^{\delta,\eps}_{1,-},\delta k^{\delta,\eps}_{2,-} \}$ that $\BAAdnoboldke$ is invertible, Proposition \ref{prop:BLSEsolution}. Then for $\delta$ small enough we can use the Neumann series and obtain
{\setlength{\belowdisplayskip}{0pt} \setlength{\belowdisplayshortskip}{0pt}\setlength{\abovedisplayskip}{0pt} \setlength{\abovedisplayshortskip}{0pt}
\begin{align*}
	(\BAAdnoboldke+\sum_{n=1}^\infty 2\,\delta^n \,\BcalGG^{k,\eps}_{+,n})^{-1}
	=&\left(\BcalI +(\BAAdnoboldke)^{-1} \sum_{n=1}^\infty 2\,\delta^n \,\BcalGG^{k,\eps}_{+,n}\right)^{-1}(\BAAdnoboldke)^{-1} 
\end{align*}
\begin{align*}
	=&\sum_{m=0}^\infty\left(-(\BAAdnoboldke)^{-1} \sum_{n=1}^\infty 2\,\delta^n \,\BcalGG^{k,\eps}_{+,n}\right)^m(\BAAdnoboldke)^{-1}\\
	=&\sum_{m=0}^\infty\left( \sum_{n=1}^\infty -2\,\delta^n \,(\BAAdnoboldke)^{-1}\BcalGG^{k,\eps}_{+,n}\right)^m(\BAAdnoboldke)^{-1}\\
	=& (\BAAdnoboldke)^{-1} -2\delta(\BAAdnoboldke)^{-1}\BcalGG^{k,\eps}_{+,1} + \OO(\delta^2) \OO_{\calL(\curlXe,\curlXe)}(1)\,.
\end{align*}}

Thus we have from solving (\ref{equ:2GapFormulaINOperatorsRepeat:1})
\begin{align*}
	\begin{bmatrix}\del_\nu \udk\MID_{\Lambda_1} \\ \del_\nu \udk\MID_{\Lambda_2} \end{bmatrix} 
		= (\BAAdnoboldke)^{-1} [\Bfdk] + \OO(\delta^2)\,,
\end{align*}
where we use the formula for $(\BAAdnoboldke)^{-1}$ and the facts that $\Bfdk=\OO(\delta)$ and $\BcalGG^{k,\eps}_{+,1}$ is linear. With Proposition \ref{prop:BLSEsolution} we split $(\BAAdnoboldke)^{-1} [\Bfdk]$ into $\Bmu^{\delta k}_\star$ and $\Bmu^{\delta k}_\sim$ and thus we have for \\$t\in(-\eps\,,\eps)$ %HACK
\begin{align*}
	\del_\nu \udnoboldk\MID_{\Lambda_j} \left(\!\!\begin{pmatrix} t\pm\xi \\ h \end{pmatrix}\!\!\right)  = (\Bmu^{\delta k}_\star)_j(t) + (\Bmu^{\delta k}_\sim)_j(t) + \OO(\delta^2)\,.
\end{align*}
We also see from Proposition \ref{prop:BLSEsolution} that $\del_\nu \udnoboldk\MID_{\Lambda_j} = \OO(\delta)$.

Now we want to calculate the first order expansion term in $\delta$ for the solution in the far-field. Let $z\in \RR^2$, where $z_2\gg 1$. 

\begin{lemma} \label{lemma:ch2HR:us=urhs+S+T}
We have for $z\in\RR^2_+$, $z_2>1$,
	\begin{multline}\label{equ:ch2HR:us=urhs+S+T}
		\udnoboldks(z)= u^{\delta k}_{\mathrm{RHS},p}(z) + u^{\delta k}_{\mathrm{RHS},e}(z)
			+ (\Seu_{+,p}^{\delta k}+\Seu_{+,e}^{\delta k})[\del_\nu \udnoboldk\MID_{\Lambda_1\cup\Lambda_2}](z) \\
			+ (\Teu^{\delta k}_{+, p}+ \Teu^{\delta k}_{+, e})[\del_\nu \udnoboldk\MID_{\Lambda_1\cup\Lambda_2}](z)+ \OO(\delta^2)\,,
	\end{multline}
	where for $\Bmu\in\BcurlXe$
	\begin{align*}
		\Seu_{+,p}^{\delta k}[\Bmu](z)\DEF & \int_{\Lambda_1\cup\Lambda_2}\Bmu(y)\Gamma^{\delta k}_{+,p}(z,y)\intd \sigma_y\,,\\
		\Seu_{+,e}^{\delta k}[\Bmu](z)\DEF & \int_{\Lambda_1\cup\Lambda_2}\Bmu(y)\Gamma^{\delta k}_{+,e}(z,y)\intd \sigma_y\,, \\
		\Teu_{+,p}^{\delta k}[\Bmu](z)\DEF & -\int_{\Lambda_1\cup\Lambda_2}\int_{\del D_1\cup\del D_2}\Bmu(y)\NdelBOonedkc(y,w)\del_{\nu_w}\Gamma^{\delta k}_{+,p}(z,w)\intd \sigma_w\intd \sigma_y\,,\\
		\Teu_{+,e}^{\delta k}[\Bmu](z)\DEF & -\int_{\Lambda_1\cup\Lambda_2}\int_{\del D_1\cup\del D_2}\Bmu(y)\NdelBOonedkc(y,w)\del_{\nu_w}\Gamma^{\delta k}_{+,e}(z,w)\intd \sigma_w\intd \sigma_y\,, 
	\end{align*}
	and
	\begin{align*}
		u^{\delta k}_{\mathrm{RHS},p}(z)
			=&	\, e^{i\delta (k_2z_2\!-\!k_1z_1)}\left[\int_{\del D_1\cup\del D_2} \del_\nu (\udknull-\udknull\circ \opP)(y)\frac{\sin(\delta k_2y_2)\,e^{i\delta k_1y_1}}{\delta k_2 p}\intd \sigma_y \right.\\
			&\mkern-90mu	\left.-\int\limits_{\del D_1\cup\del D_2}\int\limits_{\del D_1\cup\del D_2}\del_\nu (u^{\delta k}_0-u^{\delta k}_0\!\!\circ\! \opP)(y) \NdelBOonedkc(y,w)
				\nu_w\cdot\begin{pmatrix}\frac{i\,k_1}{p\,k_2}\sin(\delta k_2 w_2)\\ \frac{1}{p}\cos(\delta k_2w_2)\end{pmatrix}e^{i\delta k_1 w_1}
				\intd \sigma_w\intd \sigma_y \!\right]\,,\nonumber\\
		u^{\delta k}_{\mathrm{RHS},e}(z)
			=&	\left[
				-\int_{\del D_1\cup\del D_2} \del_\nu (\udknull-\udknull\circ \opP)(y)\Gamma^{\delta k}_{+,e}(z,y)\intd \sigma_y \right.\\
			&\mkern-90mu	\left.+\int\limits_{\del D_1\cup\del D_2}\int\limits_{\del D_1\cup\del D_2}\del_\nu (u^{\delta k}_0-u^{\delta k}_0\!\!\circ\! \opP)(y) \NdelBOonedkc(y,w)\,\del_{\nu_w} \Gamma^{\delta k}_{+,e}(z,w) \intd \sigma_w\intd \sigma_y \!\right]\nonumber	\,.
	\end{align*}
\end{lemma}
\begin{proof}
	We define $\udnoboldkf(z) \DEF \udnoboldk(z) + \Ofdk(z)\FED \udks(z)-u^{\delta k}_{\mathrm{RHS}}(z)$. Then we have from Proposition \ref{prop:udkFormulaOnComplementOfBOmega} that
	\begin{align*}
		\udnoboldkf(z)
			&=		\int_{\Lambda_1\cup\Lambda_2}\del_\nu \udnoboldk(y)\NBOonedkp(z,y)\intd \sigma_y\\
			&=		\int_{\Lambda_1\cup\Lambda_2}\del_\nu \udnoboldk(y)\Gamma^{\delta k}_+(z,y)\intd \sigma_y+\int_{\Lambda_1\cup\Lambda_2}\del_\nu \udnoboldk(y)\RBOonedkp(z,y)\intd \sigma_y\\
			&\FED	\;\Seu^{\delta k}_{+}[\del_\nu \udnoboldk\MID_{\Lambda_1\cup\Lambda_2}] + \Teu^{\delta k}_{+}[\del_\nu \udnoboldk\MID_{\Lambda_1\cup\Lambda_2}]\,.
	\end{align*}
	Using the splitting $\GKp=\Gamma^\boldk_{+,p}+\Gamma^\boldk_{+,e}$, see Proposition \ref{prop:GKSFormulaFarFieldz2>x2}, we already obtain $\Seu_{+,p}^{\delta k}$ and $\Seu_{+,e}^{\delta k}$. To study the terms $u^{\delta k}_{\mathrm{RHS}}$ and $\Teu^{\delta k}_{+}$, we consider in view of  Proposition \ref{prop:FormulaForR} that
	\begin{align*}
      \ROKp(z,y)=-\int_{\del E}\NdelOKc(y,w)\del_{\nu_y}\GKp(z,w)\intd \sigma_w\,.
   	\end{align*}
	 Hence
	\begin{align*}
		\int_{\Lambda_1\cup\Lambda_2}\del_\nu \udnoboldk(y)\RBOonedkp(z,y)\intd \sigma_y 
			=	-\int\limits_{\Lambda_1\cup\Lambda_2\quad}\mkern-50mu\int\limits^{\quad\del D_1\cup\del D_2}\mkern-20mu\del_\nu \udnoboldk(y)\NdelBOonedkc(y,w)\del_{\nu_w}\Gamma^{\delta k}_{+}(z,w)\intd \sigma_w \intd \sigma_y \,,
	\end{align*}
	and
	\begin{multline}
		\int_{\del D_1\cup\del D_2}\del_\nu (\udknull\!-\!\udknull\!\!\circ\! \opP)(y)\RBOonedkp(z,y)\intd \sigma_y\\
			=	-\int_{\Lambda_1\cup\Lambda_2}\int_{\del D_1\cup\del D_2}\del_\nu (\udknull-\udknull\circ \opP)(y)\NdelBOonedkc(y,w)\del_{\nu_w}\Gamma^{\delta k}_{+}(z,w)\intd \sigma_w \intd \sigma_y \,.
	\end{multline}
	Using again the splitting $\nabla\GKp=\nabla\Gamma^\boldk_{+,p}+\nabla\Gamma^\boldk_{+,e}$ and the explicit formula for $\Gamma^\boldk_{+,p}$ and $\nabla\Gamma^\boldk_{+,p}$ we obtain the formulas stated in Lemma \ref{lemma:ch2HR:us=urhs+S+T}.
\end{proof}

Let us approximate $\Seu_{+}^{\delta k}$ and $\Teu_{+}^{\delta k}$. Let $z_2>1$. We define
\begin{align}
	\Seu_{+,p,0}^{\delta k}[\Bmu](z)
		\DEF & 		-e^{i(\delta k_2 z_2-\delta k_1 z_1)}\frac{1}{p}\int_{\Lambda_1\cup\Lambda_2} y_2 \Bmu(y)\intd \sigma_y \,, \nonumber\\
	\Seu_{+,e,0}^{\delta k}[\Bmu](z)
		\DEF &\,	e^{-i\delta k_1 z_1} \bigg( \int_{\Lambda_1\cup\Lambda_2} \Bmu(y)\Gamma^{0}_{+}(z,y) \intd \sigma_y + \frac{1}{p} \int_{\Lambda_1\cup\Lambda_2} y_2 \Bmu(y) \intd \sigma_y \bigg) \nonumber \\
		=&			-e^{-i\delta k_1 z_1} \int_{\Lambda_1\cup\Lambda_2} \Bmu(y)
			\sum_{\substack{n\in\ZZ\setminus\{0\}\\ l\DEF 2\pi n/p}} \frac{1}{p|l|}e^{il(z_1-y_1)}e^{-|l|z_2}\sinh(|l|y_2) \intd \sigma_y\,, \label{equ:ch2HR:Seu_e0Extended}
\end{align}
and
\begin{align*}
	\Teu_{+,p,0}^{\delta k}[\Bmu](z)\DEF 
		& 	\;e^{i(\delta k_2 z_2-\delta k_1 z_1)}\frac{1}{p}\int_{\Lambda_1\cup\Lambda_2}\int_{\del D_1\cup\del D_2}\Bmu(y)\NdelBOonedkc(y,w)
			\nu_w\!\cdot\!\begin{pmatrix}0\\1\end{pmatrix}\,
			\intd \sigma_w\intd \sigma_y \,,\\
	\Teu_{+,e,0}^{\delta k}[\Bmu](z)\DEF 
		&	\,-e^{-i\delta k_1 z_1} \bigg( \int_{\Lambda_1\cup\Lambda_2}\int_{\del D_1\cup\del D_2}\Bmu(y)\NdelBOonedkc(y,w)\del_{\nu_w}\Gamma^{0}_{+}(z,w)\intd \sigma_w\intd \sigma_y \nonumber\\
		&	+ e^{-i\delta k_2 z_2} \Teu_{+,p,0}^{\delta k}[\Bmu](z)\bigg)\,.
%		=&	e^{-i\delta k_1 z_1} \int_{\Lambda_1\cup\Lambda_2}\int_{\del D_1\cup\del D_2}\mu(y)\NdelBOonedkc(y,w)
%			\!\!\!\!\!\!\sum_{\substack{n\in\ZZ\setminus\{0\}\\ l\DEF 2\pi n/p}} \!\!\!\!\!\!
%					\nu_w\begin{pmatrix}
%						-i\,l\sinh(|l|w_2) \\ |l|\cosh(|l|w_2)
%					\end{pmatrix}
%			\intd \sigma_w\intd \sigma_y \nonumber\,.\label{equ:ch2HR:Teu_e0Extended}
\end{align*}
\begin{lemma}\label{lemma:ch2HR:S-S}
	There exist $C_{(\ref{equ:ch2HR:lemmaS-S:1})}>0$, $C_{(\ref{equ:ch2HR:lemmaS-S:2})}>0$, $C_{(\ref{equ:ch2HR:lemmaS-S:3})}>0$ and $C_{(\ref{equ:ch2HR:lemmaS-S:4})}>0$ such that, for all $z\in\RR^2_+$ with $z_2 > 1$ and all $\Bmu\in\BcurlXe$, it holds that
	\begin{align}\label{equ:ch2HR:lemmaS-S:1}
		\left|\left( \Seu_{+,p}^{\delta k}- \Seu_{+,p,0}^{\delta k}\right)[\Bmu](z)\right|+\frac{1}{\delta}\left|\nabla \left( \Seu_{+,p}^{\delta k}- \Seu_{+,p,0}^{\delta k} \right)[\Bmu](z)\right|\leq& C_{(\ref{equ:ch2HR:lemmaS-S:1})}\NORM{\Bmu}_{\curlXe}\delta^2\,,\\
		\left|\left( \Seu_{+,e}^{\delta k}- \Seu_{+,e,0}^{\delta k}\right)[\Bmu](z)\right|+\frac{1}{\delta}\left|\nabla \left( \Seu_{+,e}^{\delta k}- \Seu_{+,e,0}^{\delta k} \right)[\Bmu](z)\right|\leq& C_{(\ref{equ:ch2HR:lemmaS-S:2})}\NORM{\Bmu}_{\curlXe}\delta\,,\label{equ:ch2HR:lemmaS-S:2}
	\end{align}
	and
	\begin{align}\label{equ:ch2HR:lemmaS-S:3}
		\left|\left( \Teu_{+,p}^{\delta k}- \Teu_{+,p,0}^{\delta k}\right)[\Bmu](z)\right|+\frac{1}{\delta}\left|\nabla \left( \Teu_{+,p}^{\delta k}- \Teu_{+,p,0}^{\delta k} \right)[\Bmu](z)\right|\leq& C_{(\ref{equ:ch2HR:lemmaS-S:3})}\NORM{\Bmu}_{\curlXe}\delta\,,\\
		\left|\left( \Teu_{+,e}^{\delta k}- \Teu_{+,e,0}^{\delta k}\right)[\Bmu](z)\right|+\frac{1}{\delta}\left|\nabla \left( \Teu_{+,e}^{\delta k}- \Teu_{+,e,0}^{\delta k} \right)[\Bmu](z)\right|\leq& C_{(\ref{equ:ch2HR:lemmaS-S:4})}\NORM{\Bmu}_{\curlXe}\delta\,.\label{equ:ch2HR:lemmaS-S:4}
	\end{align}
\end{lemma}
\begin{proof}
	Let us consider $\Seu_{+}^{\delta k}$ first. Using the following splitting for $\Gamma^{\delta k}_+(z,y)$, which is given in Proposition \ref{prop:GKSFormulaFarFieldz2>x2}, we have
	\begin{align*}
		\Gamma^{\delta k}_+(z,y)=\Gamma^{\delta k}_{+,p}(z,x)+\Gamma^{\delta k}_{+,e}(z,x)\,,
	\end{align*}
	where
	\begin{align}
		\Gamma^{\delta k}_{+,p}(z,x)=&-\frac{\sin(\delta k_2x_2)}{\delta k_2p}e^{i\delta (k_2z_2-k_1z_1)}e^{i\delta k_1x_1}\,,\\
      \Gamma^{\delta k}_{+,e}(z,x)
      		=&		e^{-i\delta k_1z_1}\!\!\!\!\sum_{\substack{n\in\ZZ\setminus\{0\}\\ l\DEF 2\pi n/p}}\Bigg( \frac{-e^{il(z_1-x_1)+i\delta k_1x_1}}{p\sqrt{|l-\delta k_1|^2-\delta^2 k^2}}\nonumber\\
      		&		\cdot\sinh\left( \sqrt{|l-\delta k_1|^2-\delta^2 k^2}\,x_2 \Bigg)\!\! \right)e^{-\sqrt{|l-\delta k_1|^2-\delta^2 k^2}\,z_2}\;.\label{equ:ch2HR:PFS-SGammapluse}
	\end{align}
	$\Seu_{+,p,0}^{\delta k}$ is the zeroth order term of the Taylor expansion with respect to $\delta$ of $\Gamma^{\delta k}_{+,p}(z,x)$, but without the $e^{i\delta (k_2z_2-k_1z_1)}$ term, and $\Seu_{+,e,0}^{\delta k}$ is the zeroth order term of the Taylor expansion with respect to $\delta$ of $\Gamma^{\delta k}_{+,e}(z,x)$, but without the $e^{-i\delta k_1z_1}$ term, which is located before the evanescent sum in  (\ref{equ:PFS-SGammapluse}). To see that, write $\Seu_{+,e,0}^{\delta k}$ using
	\begin{align}\label{equ:ch2HR:PFS-SGammaplus0}
		\Gamma^{0}_+(z,x)=-\frac{x_2}{p}-\sum_{\substack{n\in\ZZ\setminus\{0\}\\ l\DEF 2\pi n/p}} \frac{1}{p|l|}e^{il(z_1-x_1)}e^{-|l|z_2}\sinh(|l|x_2)\,,
	\end{align}	
	which is given through (\ref{equ:ch2:k=0GammaSharp}) for $z_2>1$, and rewrite $\int_\Lambda\del_\nu \udnoboldk(y)\Gamma^{0}_+(z,y)\intd \sigma_y$ with it. Then we see that the sum in (\ref{equ:ch2HR:PFS-SGammaplus0}) is exactly the zeroth order term of the Taylor expansion. With that, we obtain the formula in (\ref{equ:ch2HR:Seu_e0Extended}).
	
	Inserting these exact formulas into the expressions in Lemma \ref{lemma:ch2HR:S-S} and using
	\begin{align*}
		\inteps \mu(t) \phi(t) \intd t = 
			\inteps \phi(t) \frac{\sqrt[4]{\eps^2-t^2}}{\sqrt[4]{\eps^2-t^2}} \mu(t)\intd t\leq
			\NORM{\phi}_{\cC^0}\NORM{\mu}_{\curlXe}\sqrt{\pi}\,,
	\end{align*}
	where $\phi\in\cC^0{([-\eps,\eps])}$ and where we used the $\Leu^2$-Cauchy-Schwarz inequality, yields the desired estimates for $\Seu_{+}^{\delta k}$. For $\Teu_{+}^{\delta k}$ it works analogously.
\end{proof}

We define
\begin{align*}
	u^{\delta k}_{\Seu_p^\star}(z)\DEF\,& \Seu_{+,p,0}^{\delta k}[\Bmudkstar](z)\,,\\
		% \quadu ^{\delta k}_{\Seu_p^\sim}(z)\DEF\, \Seu_{+,p,0}^{\delta k}[\mudkt](z)\,,\\
	u^{\delta k}_{\Seu_e^\star}(z)\DEF\,& \Seu_{+,e,0}^{\delta k}[\Bmudkstar](z)\,,\\
		%\quad u^{\delta k}_{\Seu_e^\sim}(z)\DEF\, \Seu_{+,e,0}^{\delta k}[\mudkt](z)\,,\\
	u^{\delta k}_{\Teu_p^\star}(z)\DEF\,& \Teu_{+,p,0}^{\delta k}[\Bmudkstar](z)\,, \\
		%\quad u^{\delta k}_{\Teu_p^\sim}(z)\DEF\, \Teu_{+,p,0}^{\delta k}[\mudkt](z)\,,\\
	u^{\delta k}_{\Teu_e^\star}(z)\DEF\,& \Teu_{+,e,0}^{\delta k}[\Bmudkstar](z)\,.
		%\quad u^{\delta k}_{\Teu_e^\sim}(z)\DEF\, \Teu_{+,e,0}^{\delta k}[\mudkt](z)\,.
\end{align*}
We then have the following proposition:
\begin{proposition}\label{prop:ch2HR:LinftyNormEstimateofu}
	Let $V_r\DEF\{ z\in\RR^2_+\MID z_2>r \}$. There exists a constant $C_{(\ref{equ:ch2HR:LinftyNormEstimateofu})}>0$ such that
	\begin{align}\label{equ:ch2HR:LinftyNormEstimateofu}
		&\NORM{\udks\!-\!(u^{\delta k}_{\Seu_p^\star}\!+\!u^{\delta k}_{\Seu_e^\star}\!+\!u^{\delta k}_{\Teu_p^\star}\!+\!u^{\delta k}_{\Teu_e^\star}\!+\!u^{\delta k}_{\mathrm{RHS},p}\!+\!u^{\delta k}_{\mathrm{RHS},e})}_{\Leu^\infty(V_r)}\nonumber\\
		&\qquad+\frac{1}{\delta}\NORM{\nabla\left[\udks\!-\!(u^{\delta k}_{\Seu_p^\star}\!+\!u^{\delta k}_{\Seu_e^\star}\!+\!u^{\delta k}_{\Teu_p^\star}\!+\!u^{\delta k}_{\Teu_e^\star}\!+\!u^{\delta k}_{\mathrm{RHS},p}\!+\!u^{\delta k}_{\mathrm{RHS},e})\right]}_{\Leu^\infty(V_r)}\\
		&\qquad\leq C_{(\ref{equ:ch2HR:LinftyNormEstimateofu})}\left(\delta\ceps^3\Big(\norm{(\BrmMdkone\BrmMdktwo)^{-1}}\Big)+\delta\ceps^2 + \delta\eps^1 +\delta^2\right)\,,\nonumber
	\end{align}
	for $\eps$ small enough.
\end{proposition}

\begin{proof}
	According to (\ref{equ:ch2HR:us=urhs+S+T}), we have
	\begin{align*}
		\udks-u^{\delta k}_{\mathrm{RHS},p}-u^{\delta k}_{\mathrm{RHS},e}
		= \left(\Seu_{+,e}^{\delta k}+\Seu_{+,p}^{\delta k}+\Teu_{+,e}^{\delta k}+\Teu_{+,p}^{\delta k}\right)[u^{\delta k}|_{\Lambda_1\cup\Lambda_2}]+\OO(\delta^2)\,.
	\end{align*}
	With Lemma \ref{lemma:ch2HR:S-S} and the fact that $u^{\delta k}\MID_{\Lambda_1\cup\Lambda_2}= \Bmudkstar+\Bmudkt+\OO(\delta^2)$, $\NORM{\Bmudkstar}_{\BcurlXe}=\OO(\delta)$ and $\NORM{\Bmudkt}_{\BcurlXe}=\OO(\delta)$, we readily proof Proposition \ref{prop:ch2HR:LinftyNormEstimateofu}.
\end{proof}

\begin{proof}[Theorem \ref{THM2:2HR}]
	We see that $u^{\delta k}_{\Seu_e^\star}$, $u^{\delta k}_{\Teu_e^\star}$ and $u^{\delta k}_{\mathrm{RHS},e}$ are exponentially decaying in $z_2$ and are of order $\OO(\delta)$, thus their $\Leu^\infty(V_r)$-norm is of order $\delta e^{-C\,r}$ for some constant $C>0$. Then with Proposition \ref{prop:ch2HR:LinftyNormEstimateofu} and the change into the macroscopic view with $\Uk(x)=\udk(x/\delta)$, we obtain Theorem \ref{THM2:2HR}.
\end{proof}

\subsubsection{Evaluating the Impedance Boundary Condition}\label{subsec:IBC:2HR}

We switch back to the macroscopic variable $\Uk(x)=\udk(x/\delta)$. We approximate our solution in the far-field with the function
\begin{align}
	\Ukapp(z)
		\DEF& (\udknull-\udknullcircP)(z/\delta)+u^{\delta k}_{\Seu_p^\star}(z/\delta)+u^{\delta k}_{\Teu_p^\star}(z/\delta) +u^{\delta k}_{\mathrm{RHS},p}(z/\delta) \nonumber \\
		=& -2ia_0e^{-ik_1z_1}\sin(k_2z_2)+ e^{i( k_2 z_2 - k_1 z_1)}\,C^{\delta k}_{(\ref{const:ch2HR:Uappconstant})}\,,\label{equdef:ch2HR:Uapp:2}
\end{align}
where
\begin{align}\label{const:ch2HR:Uappconstant}
	C^{\delta k}_{(\ref{const:ch2HR:Uappconstant})}
		\DEF& 	-\frac{h}{p}\int_{\Lambda_1\cup\Lambda_2} \Bmudkstar(y)\intd \sigma_y
				\!+\!\frac{1}{p}\int_{\Lambda_1\cup\Lambda_2}\int_{\del D_1\cup\del D_2}\Bmudkstar(y)\NdelBOonedkc(y,w)\nu_w\!\cdot\!\begin{pmatrix}0\\1\end{pmatrix}\,\intd \sigma_w\intd \sigma_y \nonumber\\
			&\mkern-54mu		+\int_{\del D_1\cup\del D_2} \del_\nu (\udknull-\udknull\circ \opP)(y)\frac{\sin(\delta k_2y_2)\,e^{i\delta k_1y_1}}{\delta k_2 p}\intd \sigma_y \\
			&\mkern-54mu		-\!\!\!\!\int\limits_{\del D_1\cup\del D_2}\int\limits_{\del D_1\cup\del D_2}\del_\nu (u^{\delta k}_0-u^{\delta k}_0\!\!\circ\! \opP)(y) \NdelBOonedkc(y,w)\nu_w\cdot\begin{pmatrix}\frac{i\,k_1}{p\,k_2}\sin(\delta k_2 w_2)\\ \frac{1}{p}\cos(\delta k_2w_2)\end{pmatrix}e^{i\delta k_1 w_1}\intd \sigma_w\intd \sigma_y \nonumber\,.
\end{align}
Consider that $C^{\delta k}_{(\ref{const:ch2HR:Uappconstant})}=\OO(\delta)$ since $\Bmudkstar=\OO(\delta)$ and $(\udknull\!-\!\udknull\!\circ\! \opP)=\OO(\delta)$ and the other factors are of size $\OO(1)$.

We want that $\Ukapp$ satisfies the impedance boundary condition, that is $\Ukapp(z) + \delta \cIBC\del_{z_2}\Ukapp(z)=0$	at $\del\RR^2_+$. By (\ref{equdef:ch2HR:Uapp:2}),  we obtain the condition
\begin{align*}
	e^{-i k_1 z_1}\,C^{\delta k}_{(\ref{const:ch2HR:Uappconstant})}
	+\delta \cIBC (-2ia_0k_2 e^{-ik_1z_1}+ i\,k_2 e^{-i k_1 z_1}\,C^{\delta k}_{(\ref{const:ch2HR:Uappconstant})})=0 \,,\quad \text{for all }z_2\in\RR\,.
\end{align*}
After rearranging the terms, we obtain
\begin{align*}
	\cIBC = \frac{-C^{\delta k}_{(\ref{const:ch2HR:Uappconstant})}}{i\,\delta k_2(\,C^{\delta k}_{(\ref{const:ch2HR:Uappconstant})}-2 a_0)}\,.
\end{align*}
Using that $\frac{1}{1+\OO(\delta)}=1-\OO(\delta)$ we have
\begin{align*}
	\cIBC = \frac{C^{\delta k}_{(\ref{const:ch2HR:Uappconstant})}}{2ia_0\,\delta k_2}+\OO(\delta)\,.
\end{align*}
This proves Theorem \ref{THM3:2HR}.

\subsection{Numerical Illustrations}

In this subsection we compute the impedance boundary condition constant $\cIBC$ with numerical means using Theorem \ref{THM3:2HR}. We use two geometries, both build up upon rectangles, but with different sizes and for each geometry we have different ranges for the wave vector $k$ and the gap length $\eps$. 

%This is because the resonance values $k^{\delta,\eps}_{1,\mathrm{res}}\DEF k^{\delta, \eps}_{+,1}$ and $k^{\delta,\eps}_{2,\mathrm{res}}\DEF k^{\delta, \eps}_{+,2}$ are proportional to the square root of the geometry area, and the same holds for the width of the resonance peak of the $\cIBC$. However, we do not have to consider different values of $\delta$, since according to Theorem \ref{THM3:2HR} all computations are done with the input $\delta k$, thus a different $\delta$ would only scale the first coordinate axis, but would not change the shape of the curve itself.

 We fix $\delta=0.01$.

Again, we implement $\GKS$ using Edwald's Method see \cite{Ewald} or \cite[Chapter 7.3.2]{LPTSA}. We implement the remainders using Lemmas \ref{prop:(I+K)[R]=S[G]}, \ref{prop:FormulaForRdelEK} and \ref{lemma:1=2K[1]}. 

The first geometry has the following set-up. There are two rectangles both with length $0.9$ and height $0.9$, whose gaps are centered at $(-0.5,1)^\TransT$ and $(0.5,1)^\TransT$, respectively. The period is $p=2$ and $h=1$. The amplitude of the incident wave is $a_0=1$. The number of points, with which we approximate the boundary of each rectangle, is $200$. For the wave vector $k$, we take $341$ equidistant points in the interval $[30, 200]$. For $\eps$ we pick 5 values, those are $\{ 0.1, 0.05, 0.03, 0.01, 0.001 \}$. The result can be seen in Figure \ref{fig:cIBCPlotGeometry1_2HR}.

The second geometry has the following set-up. There are two rectangles both with length $0.2$ and height $0.3$, whose gaps are centered at $(-0.5,1)^\TransT$ and $(0.5,1)^\TransT$, respectively. The period is $p=2$ and $h=0.5$. The amplitude of the incident wave is $a_0=1$. The number of points, with which we approximate the boundary of each rectangle, is $200$. For the wave vector $k$, we take $301$ equidistant points in the interval $[100, 400]$. For $\eps$ we pick 5 values, those are $\{ 0.1, 0.05, 0.03, 0.01, 0.001 \}$. Consider that the case $\eps=0.1$ implies that the whole upper boundary of both rectangles are gaps. The result can be seen in Figure \ref{fig:cIBCPlotGeometry2_2HR}.
\begin{figure}[h]
    \centering
    \includegraphics{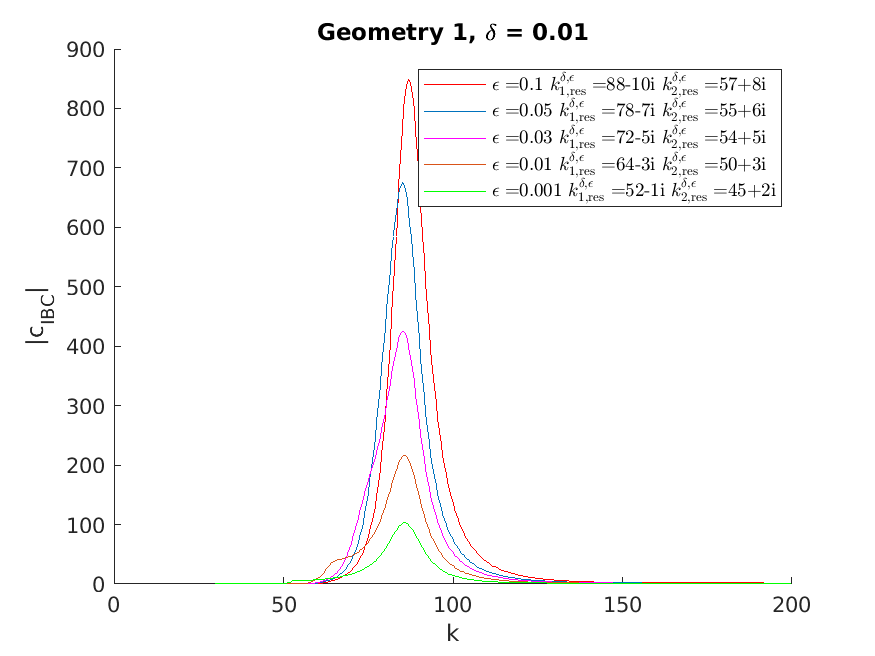}
    \caption[Geometry 1 Plot]{The plot of the absolute value of the variable $\cIBC$ depending on the wave vector $k$ for the first geometry. For every value of $\eps$, the rounded values of the resonance values $k^{\delta,\eps}_{1,\mathrm{res}}\DEF k^{\delta, \eps}_{+,1}$ and $k^{\delta,\eps}_{2,\mathrm{res}}\DEF k^{\delta, \eps}_{+,2}$ are displayed.}
    \label{fig:cIBCPlotGeometry1_2HR}
\end{figure}

Consider that the first geometry is the same geometry as in the one resonator case up to a translated origin and thus Figure \ref{fig:cIBCPlotGeometry1_2HR}, has the same appearance as Figure \ref{fig:cIBCPlotGeometry1_1HR}.

\begin{figure}[h]
    \centering
    \includegraphics{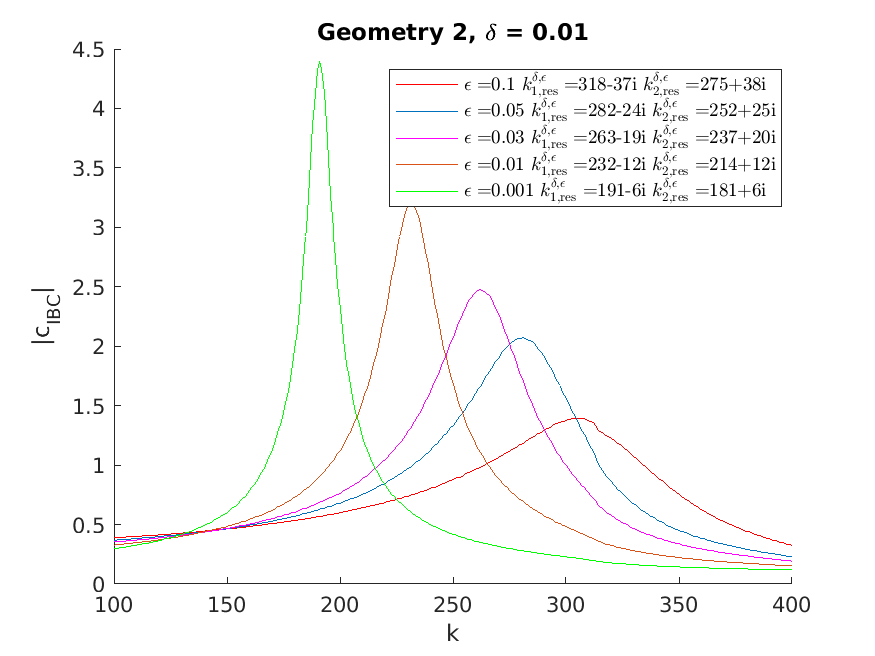}
    \caption[Geometry 2 Plot]{The plot of the absolute value of the variable $\cIBC$ depending on the wave vector $k$ for the second geometry. For every value of $\eps$, the rounded values of the resonance values $k^{\delta,\eps}_{1,\mathrm{res}}\DEF k^{\delta, \eps}_{+,1}$ and $k^{\delta,\eps}_{2,\mathrm{res}}\DEF k^{\delta, \eps}_{+,2}$ are displayed.}
    \label{fig:cIBCPlotGeometry2_2HR}
\end{figure}

\section{Changing a Small Part of the Boundary from Dirichlet to Neumann}\label{ch5}

Let us consider a bounded domain, we can think of it as a cavity, where there is put up a source point, which emits a wave. On the boundary, we have mounted a device, which we can toggle to act like the other part of the boundary or to act in a reflecting manner. 
As shown in the previous sections, such a device can be constructed using arrays of Helmholtz resonators. 
Given a receiving point in the domain, we want to be able to decide, which of the two options for the device give the higher signal at a given receiving point.

After establishing a mathematical set-up for the above described environment, we give the first order expansion term for the difference of the signal between the two option in terms of the size of the device. To this end, we establish the invertiblity of an operator, which  emerges from Green's formula, and compute then the inverse of that operator. Most of the results in this section are inspired by \cite{NarrowEsc}, where layer potential techniques were first introduced for solving the narrow escape problem of a Brownian particle through a small boundary absorbing part. It is worth emphasizing that, in the narrow escape problem, the small part of the boundary  is absorbing while the remaining part is reflecting. Because of such a difference, the derivation of an asymptotic formula for the Green's function here is technically more involved than in \cite{NarrowEsc}. We refer the reader to \cite{lecturenotes,bruno1,bruno2,laptev,mur-garnier} for the analysis of the mixed boundary value problem and the evaluation of the associated eigenvalues and eigenfunctions. 

\subsection{Preliminaries}

\subsubsection{Statement of the Problem}

Let $\Om$ be an open, bounded, and simply connected subset of $\RR^2$ with a $\cC^2$-boundary. Let $\del\Om$ be partitioned in two open intervals $\del\Om_N$ and $\del\Om_D$ such that $\del\Om_N$ is a line segment with length $2\eps$, where $\eps>0$, and with center $(0,0)^\TransT\in\RR^2$. 
For simplicity, we assume that $\Om$ is rotated so that $\del\Om_N$ is parallel to the first coordinate axis, all points on $\del\Om_N$ have height $0$, and the normal on $\del\Om_N$ is $(0,1)^\TransT$. Then we fix two points, one is the source point $x_S\in\Om$ and the other the receiving point $x_R\in\Om$. We are looking for an asymptotic expansion of the following function in terms of $\eps$ and an analytic expression for the first order term. This leading order term gives the topological derivative of the Green's function of the cavity with respect to  changes in the boundary conditions. In other terms, it describes the nucleation of a Neumann boundary condition.

Let $\uxSke :\Om\setminus\{x_S\}\rightarrow\CC$ be the solution to
\begin{align} \label{pde:uz}
	\left\{ 
	\begin{aligned}
		 \left( \Laplace_y + k^2  \right) \uxSke(y) &= \delta_0(x_S-y) \quad &&\text{in} \; &&\Om\,, \\
		 \uxSke(y) &= 0 \quad &&\text{on} \; &&\del\Om_D \,,\\
		 \del_{\nu_y} \uxSke(y) &= 0 \; &&\text{on} \; &&\del\Om_N \,,
	\end{aligned}
	\right.
\end{align}
where we assume that $k\in (0,\infty)$ is not an eigenvalue to $-\Laplace$ with the above boundary conditions, and thus $\uxSke$ is uniquely solvable.

\begin{figure}
	\centering
    \includegraphics[width=0.49\textwidth]{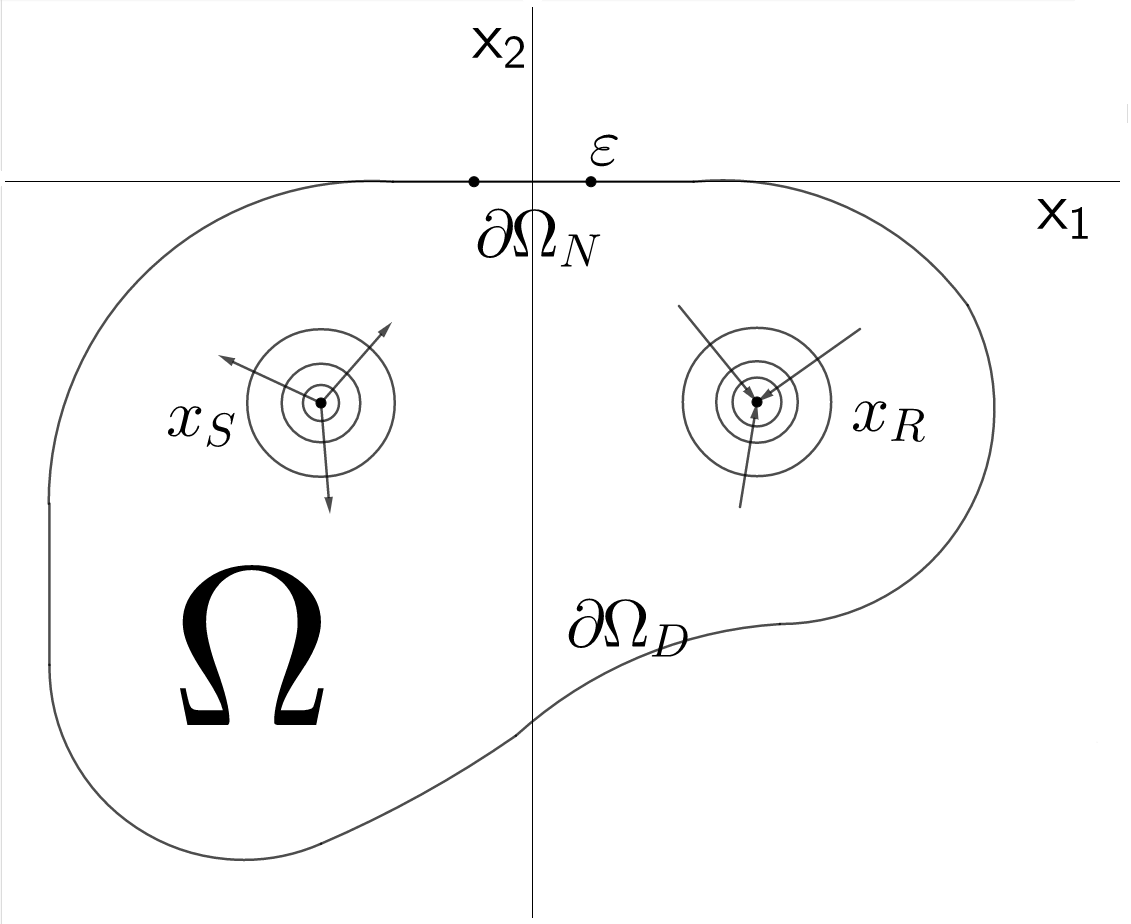}
	\caption[Geometry 2 Plot]{This picture depicts $\Om\subset\RR^2$ with the two disjoint boundary components $\del\Om_D$ and $\del\Om_N$. We have a source point point at $x_S$ and a receiving point at $x_R$.}
    \label{fig:Ch3Dom}
\end{figure}

Next, we need the function, which satisfies the above partial differential equation but has the Dirichlet condition on the whole boundary. It is often denoted as the Dirichlet function $\GOk(z,\cdot):\Om\setminus\{z\}\rightarrow \CC$ and satisfies
\begin{align} \label{pde:Gz}
	\left\{ 
	\begin{aligned}
		 \left( \Laplace_y + k^2  \right) \GOk(z,y) &= \delta_0(z-y) \quad &&\text{in} \; &&\Om\,, \\
		 \GOk(z,y) &= 0 \quad &&\text{on} \; &&\del\Om\,,
	\end{aligned}
	\right.
\end{align}
for $k\in (0,\infty)$ not an eigenvalue to $-\Laplace$ with the above boundary condition.

We will see that we can express $\uxSke$ as 
\begin{align*}
	\uxSke(x_R)=\GOk(x_S,x_R)+\OO(\eps)\,.
\end{align*}

\subsubsection{The Dirichlet Function}

We have the following formula for the Dirichlet function:
\begin{proposition}\label{prop:DiriFctonOmeg}
	Let $z\in\Om$, $x\in\Om\setminus\{z\}$ and $k$ not be an eigenvalue to $-\Laplace$ with the boundary condition given in PDE (\ref{pde:Gz}), then we have
	\begin{align*}
		\GOk(z,x)=\Gamma^k(z,x)+\RGOk(z,x)\,,
	\end{align*}
	where $\RGOk$ is the solution to
	\begin{align*}
		\left\{ 
		\begin{aligned}
		 	\left( \Laplace_x + k^2  \right) \RGOk(z,x) &= 0 \quad && \text{in} \; &&\Om\,, \\
		 	\RGOk(z,x) &= -\Gamma^k(z,x) \quad && \text{on} \; &&\del\Om\,.
		\end{aligned}
		\right.
	\end{align*}
\end{proposition}

The uniqueness and existence of $\RGOk$ in Proposition \ref{prop:DiriFctonOmeg} is a standard result.
%We will also need the Dirichlet function with the first variable being on $\del \Om$, that is $\GdelOk(z,x)$ which solves
%\begin{align}
%	\left\{ 
%	\begin{aligned}
%		\left( \Laplace_x + k^2  \right) \GdelOk(z,x) &= \delta_0(z-x) \quad && \text{in} \; &&\Om\,, \\
%	 	\GdelOk(z,x) &= 0 \quad && \text{on} \; &&\del\Om\,,
%	\end{aligned}
%	\right.
%\end{align}
%for $z\in\del\Om$. 
%We are especially interested for a formula for $\del_{\nu_x}\GdelOk(z,x)$
%for $z\in\del\Om_N$, with a smooth enough remainder. To this end, we use that we can write $\GdelOk(z,x)=2\Gamma^k(z,x)+\RGdelOk(z,x)$. \ki{I suppress here that $\RGdelOk$ is singular at the boundary. I need here a regularity result for $\RGdelOk$ so that the following sentence is true.} Now using Equation (\ref{equ:HankelExpansion}) we can extract the first two spatial, singular terms, that is $\frac{1}{2\pi}\log(|z-x|)$ and $\frac{-1}{8\pi}k^2\log(k|z-x|)|z-x|^2$, and obtain a formula with a smooth enough remainder:
We are especially interested in a formula for $\del_{\nu_x}\GOk(z,x)$, for $z\in\Om$, with a smooth enough remainder. To this end, we use $\GOk(z,x)=\Gamma^k(z,x)+\RGOk(z,x)$, then from (\ref{equ:HankelExpansion}) we can extract the first two spatial, singular terms, that is $\frac{1}{2\pi}\log(|z-x|)$ and $\frac{-1}{8\pi}k^2\log(k|z-x|)|z-x|^2$, and obtain a formula with a smooth enough remainder:

\begin{proposition} \label{prop:delG}
	Let $z\in\Om$, $x\in\del\Om$ and $k$ not be an eigenvalue to $-\Laplace$ with the boundary condition given in PDE (\ref{pde:Gz}), then we have
	\begin{align*}
		\del_{\nu_x}\GOk(z,x)=\frac{1}{2\pi}\frac{\nu_{x}\cdot(x-z)}{|z-x|^2}-\frac{1}{8\pi}k^2\,\nu_{x}\cdot(x-z)(2\log(k|z-x|)+1)+\RdelGOk(z,x)\,,
	\end{align*}
	where $\RdelGOk(z,\cdot)\in\rmH^{5/2}(\Om)$.
\end{proposition}

\subsection{Main Results}
We define 
\begin{align*}
      \curlXe&\DEF \left\{ \mu\in L^2((-\eps,\eps))\middle| \int_{-\eps}^\eps \sqrt{\eps^2-t^2}|\mu(t)|^2 \intd t < \infty \right\}, \\
      \NORM{\mu}_{\curlXe} &\DEF \left( \int_{-\eps}^\eps \sqrt{\eps^2-t^2} |\mu(t)|^2\intd t \right)^{1/2},\\
      \curlYe&\DEF \left\{ \mu\in\cC^0([-\eps,\eps])\middle| \mu'\in\curlXe \right\},\\
      \NORM{\mu}_{\curlYe}&\DEF\left( \NORM{\mu}^2_{\curlXe} + \NORM{\mu'}^2_{\curlXe} \right)^{1/2}\,,\\
		\arcYe&\DEF \left\{ \mu\in \curlYe\;\middle|\; \mu(-\eps)=\mu(\eps)=0 \right\}\,.
\end{align*}
We define the operator $\LLe: \curlXe\rightarrow\curlYe$ and  the operator $\JJe:\arcYe\rightarrow\curlXe$,
\begin{align*}
	\LLe[\mu](\tau)\DEF& \inteps\mu(t)\log(|\tau-t|)\intd t\,,\quad \tau \in (-\eps, \eps), \\
	\JJe[\mu](\tau)\DEF& \Hfpinteps \frac{\mu(t)}{(\tau-t)^2}\intd t	\,.
\end{align*}
where the 'H.f.p' denotes a Hadamard-finite-part integral. $\LLe$ is invertible and the inverse is given in Proposition \ref{prop:LLEisInjective}. $\JJe$ is invertible and the inverse is given in Proposition \ref{prop:JJEInverse}. We then have the following result:
\begin{theorem}
	Let $\eps>0$ be small enough, $\alpha, \beta>0$. The operator $-\alpha\JJe+\beta\LLe: \arcYe\rightarrow\curlXe$ is linear and invertible and the inverse is given in Proposition \ref{prop:-aLLe+bJJeInverse}. The exact function $(-\alpha\JJe+\beta\LLe)^{-1}[1]$ is given in Lemma \ref{lemma:-ajje+blleINV[1]}.
\end{theorem}

\begin{theorem} \label{thm:ch4mainresult}
	Let $\eps>0$ be small enough and let $k\in (0,\infty)$ not be an eigenvalue to $-\Laplace$ with the boundary condition given in PDE (\ref{pde:uz}) as well as with the boundary condition given in PDE (\ref{pde:Gz}). The value $\uxSke(x_R)$ is determined through 
	\begin{align} \label{eq:uxs}
		\uxSke(x_R)
			=&		\,\GOk(x_S,x_R)+\Geu^{k,\eps}_{(1)}(x_S,x_R)+\Geu^{k,\eps}_{(2)}(x_S,x_R)\,,
	\end{align}
	where
	\begin{align*} %\label{eq:uxs}
		\Geu^{k,\eps}_{(1)}(x_S,x_R)
			\DEF-\inteps \del_{\nu_y}\GOk\left(x_R,\begin{pmatrix}t\\0\end{pmatrix}\right) \Big(\frac{-1}{2\pi}\JJe+\frac{k^2}{4\pi}\LLe\Big)^{-1}\Bigg[\del_{\nu_x}\GOk\left(x_S,\begin{pmatrix} \tau\\0 \end{pmatrix}\right) \Bigg](t)\intd  t\,,
	\end{align*}
	where $\Geu^{k,\eps}_{(1)}(x_S,x_R)= \OO\Big(\frac{\eps}{|\log(\eps/2)|}\Big)$ and $\Geu^{k,\eps}_{(2)}(x_S,x_R)=\OO\left(\frac{\eps^2}{|\log(\eps/2)|^2}\right)$.
\end{theorem}

\subsection{Proof of the Main Results}

The idea of the proof is inspired by \cite{NarrowEsc} and is as follows. Using Green's formula we readily establish that $\uxSke$ is a small perturbation of $\GOk(x_S,\cdot)$. We see that the difference $\vzke\DEF \uzke-\GOk(z,\cdot)$ satisfies the following two conditions:
\begin{align*}
	\vzke(x)&=\int_{\del\Om_N} \del_{\nu_y}\GOk(x,y) \,\vzke(y)\intd \sigma_y\,,\quad\text{for }x\in\Om\,,\\
	\del_{\nu_x}\vzke(x) &= -\del_{\nu_x}\GOk(z,x)\,,\quad\text{for }x\in\del\Om_N\,,
\end{align*}
where the first one comes from Green's formula and the second one from the partial differential equation for $\uzke$ and $\GOk(z,\cdot)$. Combining both leads us to the condition
\begin{align*}
		-\del_{\nu_x}\GOk\left(z,\begin{pmatrix} \tau\\0 \end{pmatrix}\right) 
			=&		\,-\frac{1}{2\pi}\JJe[\vzke]
					+\frac{k^2}{4\pi}\LLe[\vzke]\\
			&\mkern-30mu 		+\frac{k^2}{8\pi}(2\log(k)+1)\inteps \vzke(t) \intd t
					+\inteps \vzke(t)\del_{x_2}\RdelGOk\left(
					\begin{pmatrix}
					 	\tau \\ 0
					\end{pmatrix}	,
					\begin{pmatrix}
					 	t \\ 0
					\end{pmatrix}		
					\right) \intd t\,.\nonumber
\end{align*} 
The key now is to invert the operator $\frac{-1}{2\pi}\JJe+\frac{k^2}{4\pi}\LLe[\vzke]$ and proving that the integrals over $(-\eps,\eps)$ with integrand $\vzke$ are then of lower order. The proof for invertiblity uses a result given in \cite[Chapter 11]{SV}. For finding the inverse, we use that $\JJe$ is of the form $\del_t\HHe$, where $\HHe$ is the finite Hilbert transform, and that $\LLeinv$ is of the form $\HHed[\del_t]+C$, where $\HHed$ is the inverse of $\HHe$ on $\mathrm{ker}(\HHe)^\perp$. This, together with 
\begin{align*}
		(-\alpha\JJe+\beta\LLe)^{-1}
			=	\LLeinv (\beta\calI-\alpha\JJe\LLeinv)^{-1}\,,
\end{align*}
leads us to the inverse. For the estimates we adapt the technique used in \cite[Lemma 5.4]{LPTSA}. To this end, we have to compute some integrals. To determine them, we use the mathematics tool Mathematica \cite{Mathematica}. 

\subsubsection{Condition on the Gap}
\begin{proposition}\label{prop:v=intdelGv}
	Let $z\in\Om$ and $x\in\Om\setminus\{z\}$ then
	\begin{align*}
		\uzke(x)=\GOk(z,x)+\int_{\del\Om_N} \del_{\nu_y}\GOk(x,y)\left( \uzke(y)-\GOk(z,y) \right)\intd \sigma_y\,.
	\end{align*}
\end{proposition}
\begin{proof}
	With Green's formula we have
	\begin{align*}
		\uzke(x)
			=&		\int_{\Om}\left(\Laplace+k^2\right)\GOk(x,y)\,\uzke(y)\intd y \,\\
			=&		\int_{\Om}\GOk(x,y)\,\left(\Laplace+k^2\right)\uzke(y)\intd y \nonumber\\
			&\quad	+\int_{\del\Om}\del_{\nu_y}\GOk(x,y)\,\uzke(y)\intd \sigma_y
					-\int_{\del\Om}\GOk(x,y)\,\del_{\nu_y}\uzke(y)\intd \sigma_y\,\\
			=&		\,\GOk(x,z) 
					+\int_{\del\Om_N}\del_{\nu_y}\GOk(x,y)\,\uzke(y)\intd \sigma_y 
					-0\,.
	\end{align*}
	We claim that for $z,x\in\Om$, $z\neq x$, we have that $\GOk(z,x)=\GOk(x,z)$. With that claim we conclude then
	\begin{align*}
		\uzke(x)
			=&		\GOk(z,x) 
					+\int_{\del\Om_N}\del_{\nu_y}\GOk(x,y)\,\uzke(y)\intd \sigma_y \,\\
			=&		\GOk(z,x) 
					+\int_{\del\Om_N}\del_{\nu_y}\GOk(x,y)\,\uzke(y)\intd \sigma_y 
					\!-\!\int_{\del\Om_N}\del_{\nu_y}\GOk(x,y)\,\GOk(z,y)\intd \sigma_y\,\\
			=&		\GOk(z,x)
					+\int_{\del\Om_N} \del_{\nu_y}\GOk(x,y)\left(\uzke(y)-\GOk(z,y)\right)\intd \sigma_y\,.
	\end{align*}
	Thus the proof follows. To prove the claim, consider that with Green's formula
	\begin{align*}
		\GOk(z,x) 
			=&		\int_\Om \left(\Laplace +k^2\right)\GOk(x,y)\,\GOk(z,y)\intd y\,,\\
			=&		\int_\Om \GOk(x,y)\,\left(\Laplace +k^2\right)\GOk(z,y)\intd y+0+0\,,\\
			=&		\,\GOk(x,z)\,,
	\end{align*}
	which is exactly what we wanted.
\end{proof}

Let us define $\vzke(x)\DEF \uzke(x)-\GOk(z,x)$. Thus we see that $\vzke$ satisfies the partial differential equation
\begin{align}\label{pde:vzke}
	\left\{ 
	\begin{aligned}
		\left( \Laplace_x + k^2  \right) \vzke(x) &= 0 \quad && \text{in} \; &&\Om\,, \\
	 	\vzke(x) &= 0 \quad && \text{on} \; &&\del\Om_D\,,\\
	 	\del_{\nu_x}\vzke(x) &= -\del_{\nu_x}\GOk(z,x) \quad && \text{on} \; &&\del\Om_N\,,
	\end{aligned}
	\right.
\end{align}
and from Proposition \ref{prop:v=intdelGv} we have for $x\in\Om\setminus\{z\}$
\begin{align}\label{equ:vzke=intdGvzke}
	\vzke(x)=\int_{\del\Om_N} \del_{\nu_y}\GOk(x,y) \,\vzke(y)\intd \sigma_y\,.
\end{align}
Combining (\ref{pde:vzke}) and (\ref{equ:vzke=intdGvzke}) we obtain the following condition for $\vzke$:
\begin{lemma}\label{lemma:-dG=dintdGv}
	Let $z\in\Om$ and $x\in\del\Om_N$ then
	\begin{align*}
		-\del_{\nu_x}\GOk(z,x)=
				\nu_{x}\cdot
				\lim_{
					\substack{
						\widetilde{x}\rightarrow x 
						\\ \widetilde{x}\in\Om}}
				\nabla_{\widetilde{x}}
				\Big[
				\int_{\del\Om_N} \del_{\nu_y}\GOk(\widetilde{x},y) \,\vzke(y)\intd \sigma_y
				\Big]\,,
	\end{align*}
	where $\nu_{x}$ is the outside normal at $x$.
\end{lemma}
Using Proposition \ref{prop:delG}, we have
\begin{multline}
	\nabla_{x}\int_{\del\Om_N} \del_{\nu_y}\GOk({x},y) \,\vzke(y)\intd \sigma_y
		=	\nabla_{x}\int_{\del\Om_N} \Bigg[\frac{1}{2\pi}\frac{y_2-x_2}{|y-x|^2}\\
			-\frac{1}{8\pi}k^2\,(y_2-x_2)(2\log(k|y-x|)+1)
			+\RdelGOk(x,y)\Bigg]\,\vzke(y)\intd \sigma_y\,.
\end{multline}

Using that $\del\Om_N$ is a line segment of length $2\eps$ with center $(0,0)^\TransT$, we further compute that the right-hand-side in the last equation is
\begin{multline}
	\nabla_{x}\inteps\Bigg[\frac{1}{2\pi}\frac{-x_2}{(t-x_1)^2+x_2^2}
			-\frac{1}{8\pi}k^2\,(-x_2)\left(2\log\Big(k\sqrt{(t-x_1)^2+x_2^2}\Big)+1\right)\\
			+\RdelGOk(x,(t,0)^\TransT)\Bigg]\,\vzke((t,0)^\TransT)\intd t\,.
\end{multline}
Pulling the $\nabla$-operator inside the integral, then pulling the limes in Lemma \ref{lemma:-dG=dintdGv} inside the integral, wherever possible, and considering that $\del_{\nu_x}=\del_{x_2}$, we obtain

\begin{multline}
	\lim_{\substack{h\rightarrow 0\\h>0}}\inteps\Bigg[\frac{-1}{2\pi}\frac{((t-x_1)^2-h)\vzke(t)}{((t-x_1)^2+h)^2}+\frac{k^2\,\vzke(t)}{8\pi}\left( 2\log(k|t-x_1|)+1 \right)\\
			+\vzke(t)\del_{x_2}\RdelGOk(x,(t,0)^\TransT)\Bigg]\intd t\,,
\end{multline}
where we identified $\vzke(t)$ with $\vzke((t,0)^\TransT)$. This leads us to
\begin{align}\label{equ:-dG=J+L+I+R:NoOps}
	-\del_{\nu_x}\GOk(z,x) 
		=& 		\,\frac{-1}{2\pi}\lim_{h\rightarrow 0}\inteps\frac{((t-x_1)^2-h)\vzke(t)}{((t-x_1)^2+h)^2}\intd t
				+\frac{k^2}{4\pi}\inteps{\vzke(t)\log{|t-x_1|}\,\intd t}\\
		&		+\frac{k^2}{8\pi}(2\log(k)+1)\inteps \vzke(t) \intd t
				+\inteps \vzke(t)\del_{x_2}\RdelGOk(x,(t,0)^\TransT) \intd t\,.\nonumber
\end{align}

Consider that for $h>0$
\begin{align*}
	\inteps\frac{((t-x_1)^2-h)\;\vzke(t)}{((t-x_1)^2+h)^2}\intd t
		=&		\left[ \vzke(t)\frac{x_1-t}{h+(t-x_1)^2} \right]_{t=-\eps}^\eps %\nonumber\\
				+\inteps\frac{(t-x_1)\;\del_t\vzke(t)}{(t-x_1)^2+h}\intd t\,.
\end{align*}

Using that $\vzke(-\eps)=\vzke(\eps)=0$, because of the Dirichlet boundary, we can readily compute that
\begin{align*}
	\lim_{h\rightarrow 0}\inteps\frac{((t-x_1)^2-h)\vzke(t)}{((t-x_1)^2+h)^2}\intd t
		=&	 	-\pvinteps\frac{\del_t\vzke(t)}{x_1-t}\intd t
		=		\Hfpinteps \frac{\vzke(t)}{(x_1-t)^2}\intd t\,,
\end{align*}
where the last integral is the Hadamard-finite-part integral.
\begin{definition}\label{def:JJe}
	We define the operators $\HHe:\curlXe\rightarrow\curlXe$ and $\JJe:\arcYe\rightarrow\curlXe$ as
   	\begin{align*}
   		\HHe[\mu](\tau)\DEF& \pvinteps \frac{\mu(t)}{\tau-t}\intd t	\,,\\
   		\JJe[\mu](\tau)\DEF& \Hfpinteps \frac{\mu(t)}{(\tau-t)^2}\intd t	\,.
   	\end{align*}
\end{definition}

\begin{remark}
	From the discussion above we especially obtain formulas for the operators $\HHe$ and $\JJe$ for $\mu\in\curlXe$ and $\mu\in\arcYe$, respectively. These are:
	\begin{align*}
		\HHe[\mu]
				=&		\pvinteps \frac{\mu(t)}{\tau-t}\intd t
				=		\lim_{\substack{h\rightarrow 0\\ h>0}}\inteps\frac{(\tau-t)\;\mu(t)}{(\tau-t)^2+h}\intd t \,,\\
		\JJe[\mu]
				=&		\Hfpinteps \frac{\mu(t)}{(\tau-t)^2}\intd t 
				=  		-\HHe[\del_\tau \mu]
				=		\lim_{\substack{h\rightarrow 0\\h>0}}\inteps\frac{((\tau-t)^2-h)\mu(t)}{((\tau-t)^2+h)^2}\intd t\,.
	\end{align*}
\end{remark}

With Definition \ref{def:JJe} and  (\ref{equ:-dG=J+L+I+R:NoOps}) we then obtain the following proposition:
\begin{proposition}\label{prop:formulaforvzke:InOperators}
	Let $z\in\Om$ and $\tau\in(-\eps,\eps)$, then
	\begin{align*}
		-\del_{\nu_x}\GOk\left(z,\begin{pmatrix} \tau\\0 \end{pmatrix}\right) 
			=&		\,-\frac{1}{2\pi}\JJe[\vzke]
					+\frac{k^2}{4\pi}\LLe[\vzke]\\
			&\mkern-30mu 		+\frac{k^2}{8\pi}(2\log(k)+1)\inteps \vzke(t) \intd t
					+\inteps \vzke(t)\del_{x_2}\RdelGOk\left(
					\begin{pmatrix}
					 	\tau \\ 0
					\end{pmatrix}	,
					\begin{pmatrix}
					 	t \\ 0
					\end{pmatrix}		
					\right) \intd t\,,\nonumber
	\end{align*}
	where $\LLe$ is defined in Definition \ref{def:OperatorsFor1HR} and $\nu_x$ denotes the outward normal at $(\tau,0)^\TransT$.
\end{proposition}

\subsubsection{Hypersingular Operator Analysis}

We know that $\LLe: \curlXe\rightarrow\curlYe$ is an isomorphism, where the inverse is given in Proposition \ref{prop:LLEisInjective}. From \cite[Chapter 11.5]{SV} we get that $\JJe:\arcYe\rightarrow\curlXe$ is an isomorphism. Moreover, we have the following formula for the inverse.
\begin{proposition}\label{prop:JJEInverse}
	Let $0<\eps<2$. The operator $\JJe: \arcYe\rightarrow\curlXe$ is linear and invertible and has the inverse
	\begin{align*}
		\JJeinv[\eta](t) = -\frac{1}{\pi^2\sqrt{\eps^2-t^2}} \left(
							\pvinteps\frac{\sqrt{\eps^2-\tau^2}\,\int_{-\eps}^{\tau}\eta(\tilde{\tau})\intd\tilde{\tau}}{t-\tau}\intd \tau
							+t\,C_{\Jeu,1}
							+\,C_{\Jeu,2}
					\right)\,,
	\end{align*}	 
	where
	\begin{align*}
	 	C_{\Jeu,1}
			=&	-\inteps\frac{1}{\sqrt{\eps^2-\tau^2}}\left(\int_{-\eps}^{\tau}\eta(\tilde{\tau})\intd\tilde{\tau}\right)\intd \tau\,,\\
		C_{\Jeu,2}
			=&	-\inteps\frac{\tau}{\sqrt{\eps^2-\tau^2}}\left(\int_{-\eps}^{\tau}\eta(\tilde{\tau})\intd\tilde{\tau}\right)\intd \tau\,.
	\end{align*}
	are constants depending on $\eta$ and they are linear in $\eta$.
\end{proposition}

\begin{proof}
	The proof for invertibility is given in \cite[Chapter 11.5]{SV}. Thus for every $\mu\in\arcYe$ there exists exactly one $\eta\in\curlXe$ such that $\JJe[\mu]=\eta$. Using the fact that the Hadamard-finite-part integral can be expressed as $\JJe[\mu]=\del_{\tau}\HHe[ \mu]$, and that $\HHe$ is isomorphic up to a one dimensional kernel of the form $\mathrm{ker}(\HHe)=\mathrm{span}\{(\eps^2-t^2)^{-1/2}\}$, and the inverse on $\mathrm{ker}(\HHe)^{\perp}$, which we call $\HHed$, is of the following form (see for instance in \cite{Mushkheli}; it is also used in \cite[Chapter 5.2.3]{LPTSA})
	\begin{align*}
		\HHed[\eta](t)=-\frac{1}{\pi^2\sqrt{\eps^2-t^2}}\pvinteps\frac{\sqrt{\eps^2-\tau^2}\,\eta(\tau)}{t-\tau}\intd \tau\,,
	\end{align*}
	we can rewrite $\JJe[\mu]=\eta$ as $\HHe[\mu]=\int_{-\eps}^{\tau}\eta+C_{\int}$
	%\begin{align}
	%	-\del_t \mu(t) = \HHed[\eta](t)+\frac{C_{\eta}}{\sqrt{\eps^2-t^2}}\,.
	%\end{align}
	%Now using that $\JJe[\mu]=\eta \Rightarrow \JJe[\int_{t_0}^{t}\mu+C_{\int}]=\int_{t_0}^{t}\eta+C_{\int}$, where we impose that $\int_{t_0}^{t}\mu+C_{\int}\in\arcYe$, we deduce
	%\begin{align}
	%	\mu(t) 
	%		=&	 	-\HHed\left[\int_{t_0}^{t}\eta +C_{\int}\right](t)-\frac{C_{\eta}}{\sqrt{\eps^2-t^2}}\,,\\
	%		=&		-\HHed\left[\int_{\tau_0}^{\tau}\eta \right](t)
	%				-\frac{C_{\eta}}{\sqrt{\eps^2-t^2}}
	%				-C_{\int}\HHed\left[1 \right](t)\,,\\
	%		=&		-\HHed\left[\int_{\tau_0}^{\tau}\eta \right](t)
	%				-\frac{C_{\eta}}{\sqrt{\eps^2-t^2}}
	%				+\frac{C_{\int}\,t}{\pi\sqrt{\eps^2-t^2}}\,,
	%\end{align}
	%where we computed $\HHed[1]=-\frac{t}{\pi\sqrt{\eps^2-t^2}}$, see \cite[Chapter 5.2.3]{LPTSA}.
	and then write
	\begin{align}
		\mu(t) 
			=&	 	\HHed\left[\int_{-\eps}^{\tau}\eta+C_{\int}\right]
					+\frac{C_{\eta}}{\sqrt{\eps^2-t^2}}\,, \nonumber \\
			=&		-\frac{1}{\pi^2\sqrt{\eps^2-t^2}} \left(
							\pvinteps\frac{\sqrt{\eps^2-\tau^2}\,\int_{-\eps}^{\tau}\eta}{t-\tau}\intd \tau
							+\pvinteps\frac{\sqrt{\eps^2-\tau^2}\,C_{\int}}{t-\tau}\intd \tau
							-\pi^2\,C_\eta
					\right)\,,\\
			=&		-\frac{1}{\pi^2\sqrt{\eps^2-t^2}} \left(
							\pvinteps\frac{\sqrt{\eps^2-\tau^2}\,\int_{-\eps}^{\tau}\eta}{t-\tau}\intd \tau
							+\pi\,t\,C_{\int}
							-\pi^2\,C_\eta
					\right)\,, \label{equ:PFJJeInv:1}
%			=&		\HHed\left[\int_{-\eps}^{\tau}\eta\right]
%					+\frac{C_{\eta}}{\sqrt{\eps^2-t^2}}
%					+C_{\int}\HHed\left[1 \right](t)\,,\\
%			=&		\HHed\left[\int_{-\eps}^{\tau}\eta\right]
%					+\frac{C_{\eta}}{\sqrt{\eps^2-t^2}}
%					-\frac{C_{\int}\,t}{\pi\sqrt{\eps^2-t^2}}\,, \nonumber
	\end{align}
	where we computed $\pvinteps\frac{\sqrt{\eps^2-\tau^2}}{t-\tau}\intd\tau=\pi\,t$, see \cite[Chapter 5.2.3]{LPTSA}.\\ 
	Let us find explicit expressions for the constants $C_{\eta}$ and $C_{\int}$. Consider that the part in between the brackets in  (\ref{equ:PFJJeInv:1}) has to be zero for the values $t=\eps$ and $t=-\eps$, so that we can satisfy the condition $\mu\in\arcYe$. This leads us to the system of equations
	\begin{align*}
		\inteps \frac{\sqrt{\eps+\tau}}{\sqrt{\eps-\tau}}\left(\int_{-\eps}^{\tau}\eta\right)\intd \tau
		+\pi\eps\,C_{\int}-\pi^2\,C_\eta
		&=0\,,\\
		\inteps \frac{\sqrt{\eps-\tau}}{-\sqrt{\eps+\tau}}\left(\int_{-\eps}^{\tau}\eta\right)\intd \tau
		-\pi\eps\,C_{\int}-\pi^2\,C_\eta
		&=0\,.
	\end{align*}	 
	After solving this, we obtain
	\begin{align*}
		C_{\int}
			=&	-\frac{1}{2\eps\pi}\pvinteps\frac{2\eps}{\sqrt{\eps^2-\tau^2}}\left(\int_{-\eps}^{\tau}\eta\right)\intd \tau\,,\\
		C_\eta
			=&	\frac{1}{2\pi^2}\pvinteps\frac{2\tau}{\sqrt{\eps^2-\tau^2}}\left(\int_{-\eps}^{\tau}\eta\right)\intd \tau\,.
	\end{align*}
	This proves Proposition \ref{prop:JJEInverse}.
\end{proof}

Let us consider the operator $-\frac{1}{2\pi}\JJe+\frac{k^2}{4\pi}\LLe$.

\begin{proposition}\label{prop:-aLLe+bJJeInverse}
	Let $\eps$ be small enough, $\alpha, \beta>0$. The operator $-\alpha\JJe+\beta\LLe: \arcYe\rightarrow\curlXe$ is linear and invertible and for $\eta\in\curlXe$, the inverse is given by
	\begin{align}\label{equ:-aLLe+bJJeInverse:1}
		(-\alpha\JJe+\beta\LLe)^{-1}[\eta](t) 
			=	\LLeinv[\mu_\iota](t)\,,
	\end{align}	 
	where
	\begin{align}\label{equ:-aLLe+bJJeInverse:2}
	 	\mu_\iota(\tau)
	 		=&	C_{(\ref{equ:-aLLe+bJJeInverse:2}),1} \,e^{\frac{\sqrt{\beta} \,\tau}{\sqrt{\alpha}}} 
				+ C_{(\ref{equ:-aLLe+bJJeInverse:2}),2}\, e^{-\frac{\sqrt{\beta} \,\tau}{\sqrt{\alpha}}} \nonumber\\
		&		+ \frac{1}{2 \sqrt{\alpha\beta}} e^{-\frac{\sqrt{\beta} \,\tau}{\sqrt{\alpha}}} \int_{-\eps}^{\tau} \eta(s) \,e^{\frac{\sqrt{\beta} \,s}{\sqrt{\alpha}}} \intd s 
				- \frac{1}{2 \sqrt{\alpha\beta}}e^{\frac{\sqrt{\beta} \,\tau}{\sqrt{\alpha}}} \int_{-\eps}^{\tau} \eta(s)\, e^{-\frac{\sqrt{\beta} \,s}{\sqrt{\alpha}}} \intd s\,,
	\end{align}
	where $C_{(\ref{equ:-aLLe+bJJeInverse:2}),1}\in\CC$ and $C_{(\ref{equ:-aLLe+bJJeInverse:2}),2}\in\CC$ are given through solving the system of equations
	\begin{align}
		\inteps
			\frac{
				\sqrt{\eps+\tau}\,\del_\tau\mu_\iota(\tau)
			}{
				\sqrt{\eps-\tau}
			}\intd \tau
			-\frac{\pi}{\log(\eps/2)}\,C_{\LLe[\mu_\iota]}
		=&\,0\,,\label{equ:-aLLe+bJJeInverse:SOE:1}\\
		\inteps
			\frac{
				\sqrt{\eps-\tau}\,\del_\tau\mu_\iota(\tau)
			}{
				-\sqrt{\eps+\tau}
			}\intd \tau
			-\frac{\pi}{\log(\eps/2)}\,C_{\LLe[\mu_\iota]}
		=&\,0\,.\label{equ:-aLLe+bJJeInverse:SOE:2}
	\end{align}
\end{proposition}

\begin{proof}
	From \cite[Chapter 11.1]{SV} we have that $-\alpha\JJe+\beta\LLe$ is a Fredholm operator with index 0. Thus we only have to show that it is injective. \\
	{Let us show that $-\alpha\JJe+\beta\LLe$ is injective.} To this end, consider that with Fubini's theorem we have 
	\begin{align}\label{equ:LLe=IntHHe}
		\LLe[\mu](\tau)=\inteps\log(|\tau-t|)\mu(t)\intd t=\int\limits_{-\eps}^{\tau}\HHe[\mu](\tilde{\tau})\intd \tilde{\tau}-C_{(\ref{equ:LLe=IntHHe})}\,,
	\end{align}
	where $C_{(\ref{equ:LLe=IntHHe})}=-\inteps\log(|\eps+t|)\mu(t)\intd t$. Then we get that 
	\begin{align*}
		(-\alpha\JJe+\beta\LLe)[\mu](\tau)
				= -\alpha\,\del_\tau \HHe[\mu](\tau)
				+\beta\int\limits_{-\eps}^{\tau}\HHe[\mu](\tilde{\tau})\intd \tilde{\tau}-\beta C_{(\ref{equ:LLe=IntHHe})}\,.
	\end{align*}
	Now, let $\mu\in\mathrm{ker}(-\alpha\JJe+\beta\LLe)$, it follows that $\mu$ satisfies
	\begin{align}\label{equ:-a/bdelH+intH=C5}
		-\frac{\alpha}{\beta}\,\del_\tau \HHe[\mu](\tau)
		+\int\limits_{-\eps}^{\tau}\HHe[\mu](\tilde{\tau})\intd \tilde{\tau}
		= C_{(\ref{equ:LLe=IntHHe})}\,.
	\end{align} 
	Deriving both sides, we obtain
	\begin{align*}
		-\frac{\alpha}{\beta}\,\del^2_\tau \HHe[\mu](\tau)
		+\HHe[\mu](\tau)
		= 0\,.
	\end{align*}
	With the substitution $\mathcal{M}\DEF\HHe[\mu]$ we obtain a second-order linear ordinary differential equation with the solution
	\begin{align}\label{equ:calM}
		\mathcal{M}(\tau)=C_{(\ref{equ:calM}),1}\, \exp\left({\frac{\sqrt{\beta} \,\tau}{\sqrt{\alpha}}}\right) + C_{(\ref{equ:calM}),2}\,\exp\left({-\frac{\sqrt{\beta} \,\tau}{\sqrt{\alpha}}}\right)\,.
	\end{align}
	Inserting $\mathcal{M}$ into (\ref{equ:-a/bdelH+intH=C5}), we obtain 
	\begin{align}\label{equ:calMconst}
		-C_{(\ref{equ:calM}),1}\,e^{-\eps\frac{\sqrt{\alpha}}{\sqrt{\beta}}}
		+C_{(\ref{equ:calM}),2}\,e^{\eps\frac{\sqrt{\alpha}}{\sqrt{\beta}}}
		=\frac{\sqrt{\beta}}{\sqrt{\alpha}} C_{(\ref{equ:LLe=IntHHe})}\,.
	\end{align}
	
	%we see that $C_{(\ref{equ:calM}),1}=C_{(\ref{equ:calM}),2}$ and that then $C_{(\ref{equ:calM}),1}= \sqrt{\beta}C_{(\ref{equ:LLe=IntHHe})}/(2\sqrt{\alpha}\sinh({\eps\sqrt{\beta}/\sqrt{\alpha}}))$.
	Equation $\mathcal{M}=\HHe[\mu]$ has the general solution
	\begin{align}\label{equ:muintermsofM}
		\mu(t)= -\frac{1}{\pi^2\sqrt{\eps^2-t^2}}\left(
					\pvinteps
					\frac{
						\sqrt{\eps^2-\tau^2}\,\mathcal{M}(\tau)
					}{
						t-\tau
					}\intd \tau
					-\pi^2\,C_{\mathcal{M}}
				\right)\,,
	\end{align}
	compare the proof of Proposition \ref{prop:JJEInverse}. We insert this expression into $C_{(\ref{equ:LLe=IntHHe})}$ and obtain
	\begin{align*}
		C_{(\ref{equ:LLe=IntHHe})}
			=&	C_{(\ref{equ:calM}),1}\,\frac{1}{\pi^2}\inteps \log(|\eps+t|)\frac{1}{\sqrt{\eps^2-t^2}}\left( \pvinteps \frac{\sqrt{\eps^2-\tau^2}}{t-\tau}\exp\bigg(\frac{\sqrt{\beta}}{\sqrt{\alpha}}\tau\bigg)\intd\tau \right)\intd t \nonumber\\
			&	+C_{(\ref{equ:calM}),2}\,\frac{1}{\pi^2}\inteps \log(|\eps+t|)\frac{1}{\sqrt{\eps^2-t^2}}\left( \pvinteps \frac{\sqrt{\eps^2-\tau^2}}{t-\tau}\exp\bigg(-\frac{\sqrt{\beta}}{\sqrt{\alpha}}\tau\bigg)\intd\tau \right)\intd t \nonumber\\
			&	-C_{\calM}\,\inteps \log(|\eps+t|)\frac{1}{\sqrt{\eps^2-t^2}}\intd t\,.
	\end{align*}
	
	Consider that $\mu\in\arcYe$, this implies that the expression inside the brackets in (\ref{equ:muintermsofM}) has to be zero for $t=\eps$ and $t=-\eps$. This leads us to the system of equations
	\begin{align}
		C_{(\ref{equ:calM}),1}\inteps\label{equ:jjelleSoE:1}
		\frac{
			\sqrt{\eps+\tau}\, e^{\frac{\sqrt{\beta} \,\tau}{\sqrt{\alpha}}}
		}{
			\sqrt{\eps-\tau}
		}\intd \tau
		+C_{(\ref{equ:calM}),2}\inteps
		\frac{
			\sqrt{\eps+\tau}\,e^{-\frac{\sqrt{\beta} \,\tau}{\sqrt{\alpha}}}
		}{
			\sqrt{\eps-\tau}
		}\intd \tau
		-\pi^2\,C_{\mathcal{M}}
				=0\,,\\
%%%%%%%
		C_{(\ref{equ:calM}),1}\inteps\label{equ:jjelleSoE:2}
		\frac{
			\sqrt{\eps-\tau}\, e^{\frac{\sqrt{\beta} \,\tau}{\sqrt{\alpha}}}
		}{
			-\sqrt{\eps+\tau}
		}\intd \tau 
		+C_{(\ref{equ:calM}),2}\inteps
		\frac{
			\sqrt{\eps-\tau}\,e^{-\frac{\sqrt{\beta} \,\tau}{\sqrt{\alpha}}}
		}{
			-\sqrt{\eps+\tau}
		}\intd \tau
		-\pi^2\,C_{\mathcal{M}}
				=0\,.
	\end{align}
	Using the power series of the exponential function, we have
	\begin{align*}
		%\inteps\frac{
		%	\sqrt{\eps+\tau}\, e^{\frac{\sqrt{\beta} \,\tau}{\sqrt{\alpha}}}
		%}{
		%	\sqrt{\eps-\tau}
		%}\intd \tau
		%&=	\eps\pi+\frac{\sqrt{\beta}}{\sqrt{\alpha}}\frac{\eps^2\pi}{2}+\OO(\eps^3)\,,\\
		\inteps\frac{
			\sqrt{\eps+\tau}\, e^{\pm\frac{\sqrt{\beta} \,\tau}{\sqrt{\alpha}}}
		}{
			\sqrt{\eps-\tau}
		}\intd \tau
		&=	\eps\pi\pm\frac{\sqrt{\beta}}{\sqrt{\alpha}}\frac{\eps^2\pi}{2}+\OO(\eps^3)\,,
	\end{align*}
	We can compute that
	\begin{align*}
		\inteps \log(|\eps+t|)\frac{1}{\sqrt{\eps^2-t^2}}\intd t = \pi\log\left(\frac{\eps}{2}\right)\,.
	\end{align*}
	With the mathematics tool Mathematica \cite{Mathematica} we can further compute for $c\in\RR$
	\begin{align*}
		\inteps \log(|\eps+t|)\frac{1}{\sqrt{\eps^2-t^2}}\left( \pvinteps \frac{\sqrt{\eps^2-\tau^2}}{t-\tau}\exp(c\tau)\intd\tau \right)\intd t
		=	\eps\pi^2-c\frac{\eps^2\pi^2}{4}+\OO(\eps^3)\,.
	\end{align*}
	Using (\ref{equ:calMconst}), (\ref{equ:jjelleSoE:1}), (\ref{equ:jjelleSoE:2}), we get a $3\times 3$ system of equations, whose only solution is $C_{(\ref{equ:calM}),1}=C_{(\ref{equ:calM}),2}=C_{\mathcal{M}}=0$, for $\eps$ small enough. We conclude $\mu=0$ and that $-\alpha\JJe+\beta\LLe$ is an injective Fredholm operator of index 0, hence it is invertible.
	
	{Let us find the inverse of $-\alpha\JJe+\beta\LLe$.}

	Now that we know that $-\alpha\JJe+\beta\LLe$ is invertible, we can reformulate the inverse of the operator as
	\begin{align*}
		(-\alpha\JJe+\beta\LLe)^{-1}
			\!=\!	((-\alpha\JJe\LLeinv+\beta\calI)\LLe)^{-1}
			\!=\!	\LLeinv (\beta\calI-\alpha\JJe\LLeinv)^{-1}\,.
	\end{align*}
	where $\beta\calI-\alpha\JJe\LLeinv$ is an operator from $\LLe(\arcYe)$ to $\curlXe$, and it is invertible, because $-\alpha\JJe+\beta\LLe$ and $\LLe$ are, and where $\calI$ denotes the identity operator on $\curlXe$. Consider that 
	\begin{align*}
		\JJe\LLeinv[\eta]
			=&	\del_\tau\HHe\left[ \HHed[\del_\tau\eta]+ \frac{C_{\calL}[\eta]}{\pi\log(\eps/2)\sqrt{\eps^2-t^2}}\right]\,\\
			=&  \del_\tau^2 \eta+0\,,	
	\end{align*}
	where $\HHed$ is discussed in the proof of Proposition \ref{prop:JJEInverse}. Now the general form of the solution to $(\beta \calI-\alpha\del_\tau^2)[\mu_\iota]=\eta$ is
	\begin{align}\label{equ:calMetaextra}
		\mu_\iota(t)
		=&
				C_{(\ref{equ:calMetaextra}),1} \,e^{\frac{\sqrt{\beta} \,t}{\sqrt{\alpha}}} 
				+ C_{(\ref{equ:calMetaextra}),2}\, e^{-\frac{\sqrt{\beta} \,t}{\sqrt{\alpha}}} \nonumber\\
		&		+ \frac{1}{2 \sqrt{\alpha\beta}} e^{-\frac{\sqrt{\beta} \,t}{\sqrt{\alpha}}} \int_{-\eps}^t \eta(s) \,e^{\frac{\sqrt{\beta} \,s}{\sqrt{\alpha}}} \intd s 
				- \frac{1}{2 \sqrt{\alpha\beta}}e^{\frac{\sqrt{\beta} \,t}{\sqrt{\alpha}}} \int_{-\eps}^t \eta(s)\, e^{-\frac{\sqrt{\beta} \,s}{\sqrt{\alpha}}} \intd s\,.
	\end{align}
	Then the solution of $(\beta \LLe-\alpha\JJe)[\mu]=\eta$ is given through $\mu=\LLeinv[\mu_\iota]$, where the constant $C_{(\ref{equ:calMetaextra}),1}$ and $C_{(\ref{equ:calMetaextra}),2}$ are chosen such that $\LLeinv[\mu_\iota]\in\arcYe$, which results in solving a $2\times 2$ matrix.
\end{proof}

\begin{lemma}\label{lemma:-ajje+blleINV[1]}
	Let $\eps$ be small enough, and $\alpha, \beta>0$. We have that
	\begin{align}\label{equ:lemma:-ajje+blleINV[1]}
		(-\alpha\JJe+\beta\LLe)^{-1}[1](t)=\frac{1}{\beta}\LLeinv[1]-C_{(\ref{equ:lemma:-ajje+blleINV[1]})}\LLeinv\left[\cosh\left(  \sbsa \,t\right)\right]\,,
	\end{align}
	where
	\begin{align*}
		C_{(\ref{equ:lemma:-ajje+blleINV[1]})}
			=&	\frac{1}{\beta}\frac{1}{\bigg( 1-\frac{\beta}{2\alpha}\eps^2(\log(\eps/2)-\tfrac{1}{2}))+\OO(\eps^3) \bigg)}\,,
	\end{align*}
\end{lemma}

Using the power series of $\cosh$, the difference between $C_{(\ref{equ:lemma:-ajje+blleINV[1]})}\LLeinv\left[\cosh\left(  \sbsa \,t\right)\right]$ and $\frac{1}{\beta}\LLeinv[1]$ yields a term in $\OO(\eps/|\log(\eps)|)$.

\begin{proof}
	Using the notation in Proposition \ref{prop:-aLLe+bJJeInverse}, we have $\eta=1$, thus
	{\setlength{\belowdisplayskip}{0pt} \setlength{\belowdisplayshortskip}{0pt}\setlength{\abovedisplayskip}{0pt} \setlength{\abovedisplayshortskip}{0pt}
	\begin{align}
		\mu_\iota(t)
		=&		C_{(\ref{equ:PFlemma:-ajje+blleINV[1]:1}),1} \,e^{\frac{\sqrt{\beta} \,t}{\sqrt{\alpha}}} 
				+ C_{(\ref{equ:PFlemma:-ajje+blleINV[1]:1}),2}\, e^{-\frac{\sqrt{\beta} \,t}{\sqrt{\alpha}}} \nonumber\\
		&		+ \frac{1}{2 \sqrt{\alpha\beta}} e^{-\frac{\sqrt{\beta} \,t}{\sqrt{\alpha}}} \int_{-\eps}^{t} e^{\frac{\sqrt{\beta} \,s}{\sqrt{\alpha}}} \intd s 
				- \frac{1}{2 \sqrt{\alpha\beta}}e^{\frac{\sqrt{\beta} \,t}{\sqrt{\alpha}}} \int_{-\eps}^{t} e^{-\frac{\sqrt{\beta} \,s}{\sqrt{\alpha}}} \intd s\,\label{equ:PFlemma:-ajje+blleINV[1]:1}\\
		=&	C_{(\ref{equ:PFlemma:-ajje+blleINV[1]:2}),1} \,e^{\frac{\sqrt{\beta} \,t}{\sqrt{\alpha}}} 
			+ C_{(\ref{equ:PFlemma:-ajje+blleINV[1]:2}),2}\, e^{-\frac{\sqrt{\beta} \,t}{\sqrt{\alpha}}} 
			+\frac{1}{\beta}\,.\label{equ:PFlemma:-ajje+blleINV[1]:2}
	\end{align}}
	Thus
	\begin{align*}
		\del_t\mu_\iota(t)=\frac{\sqrt{\beta}}{\sqrt{\alpha}}C_{(\ref{equ:PFlemma:-ajje+blleINV[1]:2}),1} \,e^{\frac{\sqrt{\beta} \,t}{\sqrt{\alpha}}} 
			- \frac{\sqrt{\beta} }{\sqrt{\alpha}}C_{(\ref{equ:PFlemma:-ajje+blleINV[1]:2}),2}\, e^{-\frac{\sqrt{\beta} \,t}{\sqrt{\alpha}}}\,.
	\end{align*}
	Let us solve the system of equations (\ref{equ:-aLLe+bJJeInverse:SOE:1}), (\ref{equ:-aLLe+bJJeInverse:SOE:2}). We readily see that
	\begin{align*}
		\inteps
			\frac{
				2\eps\,\del_\tau\mu_\iota(\tau)
			}{
				\sqrt{\eps^2-\tau^2}
			}\intd \tau
		=&\,0\,.
	\end{align*}
	Using the mathematics tool Mathematica \cite{Mathematica}, we obtain that $\inteps \frac{\exp(c\,t)}{\sqrt{\eps^2-t^2}}\intd t = \pi\rmI_0(\eps\,c)$, for $c\in\RR$, where $\rmI_n$ is the modified Bessel function of the first kind.
	This leads us to
	\begin{align*}
		0\,
			=&	 C_{(\ref{equ:PFlemma:-ajje+blleINV[1]:2}),1} \,\rmI_0(\eps\,\sbsa)
			-  C_{(\ref{equ:PFlemma:-ajje+blleINV[1]:2}),2}\, \rmI_0(-\eps\,\sbsa)\,.
	\end{align*}
	Since $\rmI_0$ is even we have $C_{(\ref{equ:PFlemma:-ajje+blleINV[1]:2}),1}=C_{(\ref{equ:PFlemma:-ajje+blleINV[1]:2}),2}$. Thus
	\begin{align}\label{equ:PFlemma:-ajje+blleINV[1]:3}
		\del_t\mu_\iota(t)=\sbsa C_{(\ref{equ:PFlemma:-ajje+blleINV[1]:3})} \,\sinh\left( \sbsa t \right)\,.
	\end{align}
	Now we solve
	\begin{align*}
		\inteps
			\frac{
				\sqrt{\eps+\tau}\,\del_\tau\mu_\iota(\tau)
			}{
				\sqrt{\eps-\tau}
			}\intd \tau
			-\frac{\pi}{\log(\eps/2)}\,C_{\LLe[\mu_\iota]}
		=&\,0\,.
	\end{align*}
	Then, we obtain
	\begin{align*}
		\inteps
			\frac{
				\sqrt{\eps+\tau}\,\del_\tau\mu_\iota(\tau)
			}{
				\sqrt{\eps-\tau}
			}\intd \tau
			= \sbsa C_{(\ref{equ:PFlemma:-ajje+blleINV[1]:3})}\eps\,\pi \rmI_1\left(\eps\sbsa\right)\,.
	\end{align*}
	and 
	\begin{align*}
		C_{\LLe[\mu_\iota]}
			=&		\mu_\iota(0)-\LLe\bigg[ -\frac{1}{\pi^2\sqrt{\eps^2-t^2}}\pvinteps\frac{\sqrt{\eps^2-\tau^2}\;\del_{\tau}\mu_\iota(\tau)}{t-\tau}\intd \tau \bigg](0)\,,\\
			=&		C_{(\ref{equ:PFlemma:-ajje+blleINV[1]:3})}
					+\frac{1}{\beta}
					+\frac{1}{\pi^2}\inteps \frac{\log{|t|}}{\sqrt{\eps^2-t^2}}\pvinteps\frac{\sqrt{\eps^2-\tau^2}\;\del_{\tau}\mu_\iota(\tau)}{t-\tau}\intd \tau\intd t\,,\\
			=&		C_{(\ref{equ:PFlemma:-ajje+blleINV[1]:3})}\Bigg(
							1+\sbsa\frac{1}{\pi^2}\inteps \frac{\log{|t|}}{\sqrt{\eps^2-t^2}}\pvinteps\frac{\sqrt{\eps^2-\tau^2}\;\sinh\left(\sbsa\tau\right)}{t-\tau}\intd \tau\intd t
					\Bigg)
					+\frac{1}{\beta}\,.
	\end{align*}
	Using the power series for the sinus hyperbolicus, we have
	{\setlength{\belowdisplayskip}{0pt} \setlength{\belowdisplayshortskip}{0pt}\setlength{\abovedisplayskip}{0pt} \setlength{\abovedisplayshortskip}{0pt}
	\begin{align*}
		\inteps \frac{\log{|t|}}{\sqrt{\eps^2-t^2}}\pvinteps\frac{\sqrt{\eps^2-\tau^2}\;\sinh\left(\sbsa\tau\right)}{t-\tau}\intd \tau\intd t
		=&		\sbsa\frac{\eps^2\pi^2}{4}+\left(\sbsa\right)^{3}\frac{\eps^4\pi^2}{64}+\OO(\eps^6)\,,
	\end{align*}
	and
	\begin{align*}
		\rmI_1\left(\eps\sbsa\right)
		=&		\sbsa\frac{\eps}{2}+\left(\sbsa\right)^{3}\frac{\eps^3}{16}+\OO(\eps^5)\,.
	\end{align*}}
	We infer that
	\begin{align*}
		C_{(\ref{equ:PFlemma:-ajje+blleINV[1]:3})}\sbsa\pi\,\eps \Bigg( \sbsa\frac{\eps}{2}+\OO(\eps^3) \Bigg)
		\!-\! \frac{\pi}{\log(\eps/2)}C_{(\ref{equ:PFlemma:-ajje+blleINV[1]:3})}\Big(
				1+\frac{\beta}{\alpha} \frac{\eps^2}{4}+\OO(\eps^4)
			\Big)
		\!-\! \frac{\pi}{\log(\eps/2)}\frac{1}{\beta}
		=\,0\,.
	\end{align*}
	This leads us to
	\begin{align*}
		C_{(\ref{equ:PFlemma:-ajje+blleINV[1]:3})}\bigg(
			- \frac{\pi}{\log(\eps/2)}
			+\frac{\beta}{\alpha}\frac{\pi}{2}\eps^2\Big(
				1-\frac{1}{2\log(\eps/2)}
			\Big)
			+\OO(\eps^3)
		\bigg)
		=\frac{\pi}{\log(\eps/2)}\frac{1}{\beta}\,.
	\end{align*}
	Thus
	\begin{align*}
		C_{(\ref{equ:PFlemma:-ajje+blleINV[1]:3})}\Big(
			1+\OO(\eps)
		\Big)
		=-\frac{1}{\beta}\,.
	\end{align*}
	Hence, we proved Lemma \ref{lemma:-ajje+blleINV[1]}.
\end{proof}

\begin{lemma}\label{lemma:normLJinvR}
	Let $\eps$ be small enough, and $\alpha, \beta>0$. Let $\calR$ be the integral operator defined from $\arcYe$ into $\curlXe$ by
	\begin{align*}
		\calR[\mu](\tau) \,=\, \inteps R(\tau ,t)\, \mu(t)\intd t\,,
	\end{align*}
	where $R(\tau,t)$ is of class $\cC^{0,1/2}$ in $\tau$ and $t$.
	For $\eps$ small enough, there exists a positive constant $C_{(\ref{equ:normLJinvR})}$, independent of $\eps$, such that for all $\mu\in\arcYe$, we have
	\begin{align}\label{equ:normLJinvR}
		\NORM{(-\alpha\JJe+\beta\LLe)^{-1}\calR[\mu]}_{\curlXe}\leq \frac{\eps \,C_{(\ref{equ:normLJinvR})}}{|\log(\eps)|}\NORM{R}_{\cC^{0,1/2}}\NORM{\mu}_{\curlXe}\,.
	\end{align}
\end{lemma}

\begin{proof}
	We define $\eta\DEF \calR[\mu]$ and use then the notation in Proposition \ref{prop:-aLLe+bJJeInverse}. It follows with Fubini's Theorem that 
	\begin{align}\label{equ:PFLemma:normLJinvR:2}
		\mu_\iota(t)
	 		=&	C_{(\ref{equ:PFLemma:normLJinvR:2}),1} \,e^{\frac{\sqrt{\beta} \,t}{\sqrt{\alpha}}} 
				+ C_{(\ref{equ:PFLemma:normLJinvR:2}),2}\, e^{-\frac{\sqrt{\beta} \,t}{\sqrt{\alpha}}} \nonumber\\
			&	+ \frac{1}{2 \sqrt{\alpha\beta}} e^{-\frac{\sqrt{\beta} \,t}{\sqrt{\alpha}}} \int_{-\eps}^{t} \eta(s) \,e^{\frac{\sqrt{\beta} \,s}{\sqrt{\alpha}}} \intd s 
				- \frac{1}{2 \sqrt{\alpha\beta}}e^{\frac{\sqrt{\beta} \,t}{\sqrt{\alpha}}} \int_{-\eps}^{t} \eta(s)\, e^{-\frac{\sqrt{\beta} \,s}{\sqrt{\alpha}}} \intd s\, \nonumber \\
%%%%%%%%%%%%%%%%%%%%%%%%%%%
			=&	C_{(\ref{equ:PFLemma:normLJinvR:2}),1} \,e^{\frac{\sqrt{\beta} \,t}{\sqrt{\alpha}}} 
				+ C_{(\ref{equ:PFLemma:normLJinvR:2}),2}\, e^{-\frac{\sqrt{\beta} \,t}{\sqrt{\alpha}}} 
				+ \frac{1}{2 \sqrt{\alpha\beta}} 
					\calR_\star[\mu](t)\,,
	\end{align}
	where
	\begin{align*}
		\calR_\star[\mu](t) \DEF 
			&	\inteps \int\limits_{-\eps}^{t} R(s ,q) \,e^{\frac{\sqrt{\beta} \,(s-t)}{\sqrt{\alpha}}} \,\mu(q)\,\intd s\,\intd q
				- \inteps \int\limits_{-\eps}^{t} R(s ,q)\, e^{-\frac{\sqrt{\beta} \,(s-t)}{\sqrt{\alpha}}} \,\mu(q)\,\intd s\, \intd q\,\\
			=&	\inteps \left(\int\limits_{-\eps}^{t} 2\,R(s ,q)\,\sinh\left({\frac{\sqrt{\beta} \,(s-t)}{\sqrt{\alpha}}}\right) \intd s\right) \, \mu(q)\intd q\,\\
			\FED&	\inteps R_\star(t,q) \, \mu(q)\intd q\,.
	\end{align*}
	Let us examine those constants. They are given through solving the following system:
	\begin{align}
		\inteps
			\frac{
				\sqrt{\eps+\tau}\,\del_\tau\mu_\iota(\tau)
			}{
				\sqrt{\eps-\tau}
			}\intd \tau
			&-\frac{\pi}{\log(\eps/2)}\,C_{\LLe[\mu_\iota]}
		=\,0\,,\label{equ:PFLemma:normLJinvR:SOE:1}\\
		\inteps
			\frac{
				2\eps\,\del_\tau\mu_\iota(\tau)
			}{
				\sqrt{\eps^2-\tau^2}
			}\intd \tau&
		=\,0\,.\label{equ:PFLemma:normLJinvR:SOE:2}
	\end{align}
	We compute using $\LLe[\LLeinv]=\calI$ and $\LLe[1/\sqrt{\eps^2-t^2}](0)=-\pi\log(\eps/2)$ that
	{\setlength{\belowdisplayskip}{0pt} \setlength{\belowdisplayshortskip}{0pt}\setlength{\abovedisplayskip}{0pt} \setlength{\abovedisplayshortskip}{0pt}
	\begin{align*}
		C_{\LLe[\mu_\iota]}
			=&		\,\mu_\iota(0)+\LLe\bigg[ \frac{1}{\pi^2\sqrt{\eps^2-t^2}}\pvinteps\frac{\sqrt{\eps^2-\tau^2}\;\del_{\tau}\mu_\iota(\tau)}{t-\tau}\intd \tau \bigg](0)\,,\\
			=&		C_{(\ref{equ:PFLemma:normLJinvR:2}),1}
					+ C_{(\ref{equ:PFLemma:normLJinvR:2}),2} 
					%+ \frac{1}{2 \sqrt{\alpha\beta}} \calR_\star[\mu](0) 
					+ \frac{1}{2\sqrt{\alpha\beta}}{C_{\LLe[\calR_\star[\mu]]}}\\
			&		+ \frac{1}{\pi^2}\LLe\Bigg[
						\pvinteps \frac{\sqrt{\eps^2-\tau^2}}{\sqrt{\eps^2-t^2}}\Bigg(
							\sbsa C_{(\ref{equ:PFLemma:normLJinvR:2}),1}\frac{e^{\sbsa\,\tau}}{t-\tau}
							-\sbsa C_{(\ref{equ:PFLemma:normLJinvR:2}),2}\frac{e^{-\sbsa\,\tau}}{t-\tau}
						\Bigg)\intd \tau
					\Bigg](0)\,,\nonumber\\
%	\end{align*}
%	\begin{align}
			=&		C_{(\ref{equ:PFLemma:normLJinvR:2}),1}
					+ C_{(\ref{equ:PFLemma:normLJinvR:2}),2} 
					%+ \frac{1}{2 \sqrt{\alpha\beta}} 
					%\calR_\star[\mu](0) 
					+\frac{1}{2\sqrt{\alpha\beta}}{C_{\LLe[\calR_\star[\mu]]}}  \\
			&		+  	\sbsa C_{(\ref{equ:PFLemma:normLJinvR:2}),1}\left(\sbsa\frac{\eps^2}{4}+\OO(\eps^4)\right)
							-\sbsa C_{(\ref{equ:PFLemma:normLJinvR:2}),2}\left(-\sbsa\frac{\eps^2}{4}+\OO(\eps^4)\right)\nonumber \,,\\
%%%%%%%%%%%%%%%%%%%%%%
			=&		C_{(\ref{equ:PFLemma:normLJinvR:2}),1} \Bigg( 1+\frac{\beta}{4\alpha}\eps^2+\OO(\eps^4) \Bigg)
					+C_{(\ref{equ:PFLemma:normLJinvR:2}),2} \Bigg( 1+\frac{\beta}{4\alpha}\eps^2+\OO(\eps^4) \Bigg)\nonumber\\
			&		%+\frac{1}{2 \sqrt{\alpha\beta}} \calR_\star[\mu](0) 
					+ \frac{1}{2\sqrt{\alpha\beta}}{C_{\LLe[\calR_\star[\mu]]}}\,.
	\end{align*}
	and for $c\in\RR$
	\begin{align*}
		\inteps \frac{
				\sqrt{\eps+\tau}\,\exp(c\,\tau)
			}{
				\sqrt{\eps-\tau}
			}\intd \tau
		= \pi\eps+\pi\,c\,\frac{\eps^2}{2}+\pi\,c^2\,\frac{\eps^3}{2}+\OO(\eps^4)\,,
	\end{align*}
	and
	\begin{align*}
		\inteps
			\frac{
				\exp(c\,\tau)
			}{
				\sqrt{\eps^2-\tau^2}
			}\intd \tau
		= \pi+\pi\,c^2\frac{\eps^2}{2}+\OO(\eps^4)\,.
	\end{align*}}
	This leads us to the $2\times 2$ system of equations
	\begin{multline}
		\begin{pmatrix}
			-\frac{\pi}{\log(\eps/2)} +\OO(\eps)
			& -\frac{\pi}{\log(\eps/2)} +\OO(\eps)\\
			\sbsa\pi +\OO(\eps^2)
			& -\sbsa\pi +\OO(\eps^2)
		\end{pmatrix}
		\begin{pmatrix}
			C_{(\ref{equ:PFLemma:normLJinvR:2}),1}\\
			C_{(\ref{equ:PFLemma:normLJinvR:2}),2}
		\end{pmatrix}
		\\=
		\frac{-1}{2\sqrt{\alpha\beta}}
		\begin{pmatrix}
			\inteps\frac{\sqrt{\eps+\tau}}{\sqrt{\eps-\tau}}\del_\tau\calR_\star[\mu](\tau)\intd \tau
			%+\calR_\star[\mu](0) 
			+{C_{\LLe[\calR_\star[\mu]]}}
			\\
			\inteps \frac{\del_\tau \calR_\star[\mu](\tau)}{\sqrt{\eps^2-\tau^2}}\intd \tau
		\end{pmatrix}\,.\label{equ:PFLemma:normLJinvR:SOE:MAT-LSE}
	\end{multline}
	Let us give an expansion in $\eps$ for the right-hand-side. According to see Equation (\ref{equ:intleqXenorm}), we have that $\inteps\mu(q)\intd q\leq \sqrt{\pi}\NORM{\mu}_{\curlXe}$.
%	\begin{align}
%		\calR_\star[\mu](0) 
%			=& \inteps \left(\int\limits_{-\eps}^{0} 2\,R(s ,q)\,\sinh\left({\frac{\sqrt{\beta} \,s}{\sqrt{\alpha}}}\right) \intd s\right) \, \mu(q)\intd q\,\\
%			\leq&	\NORM{R}_{\infty}\left|-\eps^2\sbsa+\OO(\eps^4)\right|\inteps|\mu(q)|\intd q\,\\
%			\leq&	\left|\eps^2\sbsa+\OO(\eps^4)\right|\NORM{R}_{\infty}\pi\NORM{\mu}_{\curlXe}\,.
%	\end{align}
	Hence, we establish that
	\begin{align*}
		\del_t\calR_\star[\mu](t) 
			=&		-\inteps \left(\int\limits_{-\eps}^{t} 2\,R(s ,q)\,\cosh\left({\frac{\sqrt{\beta} \,(s-t)}{\sqrt{\alpha}}}\right) \intd s\right) \, \mu(q)\intd q\,\\
			\leq&	\NORM{R}_{\infty}\frac{2\sqrt{\alpha}}{\sqrt{\beta}}\left|\sinh\left(\sbsa (\eps+t)\right)\right|\sqrt{\pi}\NORM{\mu}_{\curlXe}\,, 
	\end{align*}
	%We define $\calR_{\del,\star}[\mu](t)\DEF \del_t\calR_\star[\mu](t)$, 
	Then we have
	\begin{align*}
		\inteps\frac{\sqrt{\eps+\tau}}{\sqrt{\eps-\tau}}\del_\tau\calR_\star[\mu](\tau)\intd \tau
			\leq&		\NORM{R}_{\infty}\frac{2\sqrt{\alpha}}{\sqrt{\beta}}\inteps|\mu(q)|\intd q\inteps\frac{\sqrt{\eps+\tau}}{\sqrt{\eps-\tau}}\sinh\left(\sbsa (\eps+\tau)\right)\intd \tau\,\\
			\leq&	\frac{2\sqrt{\alpha}}{\sqrt{\beta}}\NORM{R}_{\infty}\pi\NORM{\mu}_{\curlXe}\left( \sbsa\frac{3\eps^2\pi}{2}+\OO(\eps^4) \right)\\
			=&		\pi\NORM{R}_{\infty}\NORM{\mu}_{\curlXe}\left( 3\eps^2\pi+\OO(\eps^4) \right)\,.
	\end{align*}
	Next we have that
	\begin{align*}
		\inteps \frac{\del_\tau \calR_\star[\mu](\tau)}{\sqrt{\eps^2-\tau^2}}\intd \tau
			\leq&	\pi\frac{2\sqrt{\alpha}}{\sqrt{\beta}}\NORM{R}_{\infty}\NORM{\mu}_{\curlXe}\inteps  \frac{\sinh\big(\sbsa (\eps+t)\big)}{\sqrt{\eps^2-\tau^2}}\intd \tau\,\\
			\leq&	\pi^2\frac{2\sqrt{\alpha}}{\sqrt{\beta}}\NORM{R}_{\infty}\NORM{\mu}_{\curlXe}\left( \eps\sbsa +\OO(\eps^3)  \right)\,.
	\end{align*}
	Finally,
	{\setlength{\belowdisplayskip}{0pt} \setlength{\belowdisplayshortskip}{0pt}\setlength{\abovedisplayskip}{0pt} \setlength{\abovedisplayshortskip}{0pt}
	\begin{align}
		{C_{\LLe[\calR_\star[\mu]]}}
			=&		0 + \LLe\left[\frac{1/\pi^2}{\sqrt{\eps^2-t^2}}\pvinteps\frac{\sqrt{\eps^2-\tau^2}\;\del_\tau\calR_\star[\mu](\tau)}{t-\tau}\intd \tau\right](-\eps)\, \nonumber\\
			=&		\inteps\frac{\log(\eps+t)/\pi^2}{\sqrt{\eps^2-t^2}}\inteps\sqrt{\eps^2-\tau^2}\, \frac{\del_\tau\calR_\star[\mu](\tau)-\del_\tau\calR_\star[\mu](t)}{t-\tau} \intd \tau\intd t\nonumber\\
			&		+\inteps\frac{\log(\eps+t)/\pi^2}{\sqrt{\eps^2-t^2}}\pvinteps\sqrt{\eps^2-\tau^2}\frac{\del_\tau\calR_\star[\mu](t)}{t-\tau} \intd \tau\intd t\, \nonumber \\
			\leq&	\inteps\frac{\log(\eps+t)/\pi^2}{\sqrt{\eps^2-t^2}}C_{(\ref{equ:PFLemma:normLJinvR:SOE:3})}\pi\eps\sqrt{\eps}\,\intd t \left|\del_t \calR_\star[\mu]\right|_{\cC^{0,1/2}}\label{equ:PFLemma:normLJinvR:SOE:3}\\
			&		+\inteps\frac{\log(\eps+t)/\pi^2}{\sqrt{\eps^2-t^2}}\,t\,\pi \intd t\NORM{\del_\tau\calR_\star[\mu]}_{\cC^0}\,\nonumber\\
	%\end{align}
	%\begin{align}
			=&		C_{(\ref{equ:PFLemma:normLJinvR:SOE:3})}\,{\log(\eps/2)}\,\eps\sqrt{\eps}\,\left|\del_t \calR_\star[\mu]\right|_{\cC^{0,1/2}}
					+\eps \NORM{\del_\tau\calR_\star[\mu]}_{\cC^0}\,,\label{equ:PFLemma:normLJinvR:SOE:3+}
	\end{align}}
	where we used that $\inteps \frac{\sqrt{\eps^2-\tau^2}}{\sqrt{|t-\tau|}}\intd\tau\leq C_{(\ref{equ:PFLemma:normLJinvR:SOE:3})}\,\pi\,\eps\sqrt{\eps}$. We readily see that
	\begin{align}
		\left|\del_t \calR_\star[\mu]\right|_{\cC^{0,1/2}} 
			\leq	C_{(\ref{equ:PFLemma:normLJinvR:SOE:4})}\left|R\right|_{\cC^{0,1/2}} \NORM{\mu}_{\curlXe}\,,\label{equ:PFLemma:normLJinvR:SOE:4}\\
		\NORM{\del_\tau\calR_\star[\mu]}_{\cC^0} 
			\leq	\eps\,C_{(\ref{equ:PFLemma:normLJinvR:SOE:5})}\NORM{R}_{\cC^0}\NORM{\mu}_{\curlXe}\,,\label{equ:PFLemma:normLJinvR:SOE:5}
	\end{align}
	where $C_{(\ref{equ:PFLemma:normLJinvR:SOE:4})} = \OO(1)$, and $C_{(\ref{equ:PFLemma:normLJinvR:SOE:5})}=\OO(1)$, for $\eps\rightarrow 0$. Now we can solve  (\ref{equ:PFLemma:normLJinvR:SOE:MAT-LSE}), and obtain that 
\begin{align}
	C_{(\ref{equ:PFLemma:normLJinvR:2}),1}
		\leq& C_{(\ref{equ:PFLemma:normLJinvR:SOE:6})}\,\eps \, \NORM{R}_{\cC^{0,1/2}}\NORM{\mu}_{\curlXe}\,,\label{equ:PFLemma:normLJinvR:SOE:6}\\
	C_{(\ref{equ:PFLemma:normLJinvR:2}),2}
		\leq& C_{(\ref{equ:PFLemma:normLJinvR:SOE:7})}\,\eps \, \NORM{R}_{\cC^{0,1/2}}\NORM{\mu}_{\curlXe}\,,\label{equ:PFLemma:normLJinvR:SOE:7}
\end{align}
for $\eps$ small enough, where $C_{(\ref{equ:PFLemma:normLJinvR:SOE:6})}=\OO(1)$ and $C_{(\ref{equ:PFLemma:normLJinvR:SOE:7})}=\OO(1)$, for $\eps\rightarrow 0$.\\
	We have examined the constant. Now we estimate $\NORM{\LLeinv[\mu_\iota]}_{\curlXe}$. We have
	\begin{multline}
		\LLeinv[\mu_\iota](t)
			=	\,C_{(\ref{equ:PFLemma:normLJinvR:2}),1} \,\NORM{\LLeinv\Big[e^{\frac{\sqrt{\beta} \,t}{\sqrt{\alpha}}} \Big]}_{\curlXe}
				+ C_{(\ref{equ:PFLemma:normLJinvR:2}),2}\, \NORM{\LLeinv\Big[e^{-\frac{\sqrt{\beta} \,t}{\sqrt{\alpha}}} \Big]}_{\curlXe}\\
				+ \frac{1}{2 \sqrt{\alpha\beta}} 
					\NORM{\LLeinv[\calR_\star[\mu]]}_{\curlXe}\,.
	\end{multline}
	Consider that for $c\in\RR$, we have that
	\begin{align}\label{equ:PFLemma:normLJinvR:SOE:8}
		\NORM{\LLeinv\Big[e^{c\,t}\Big]}_{\curlXe}
			=&		\NORM{\LLeinv\Bigg[\inteps e^{c\,t}\frac{1}{\pi\sqrt{\eps^2-s^2}}\intd s\Bigg]}_{\curlXe}\, \\
			\leq&	\frac{C_{(\ref{equ:PFLemma:normLJinvR:SOE:8+})}}{|\log(\eps)|}\NORM{\frac{1}{\pi\sqrt{\eps^2-s^2}}}_{\curlXe}\label{equ:PFLemma:normLJinvR:SOE:8+}\,\\
			=&		\frac{C_{(\ref{equ:PFLemma:normLJinvR:SOE:8++})}}{|\log(\eps)|}\,.\label{equ:PFLemma:normLJinvR:SOE:8++}
	\end{align}
	where we used Lemma \ref{lemma:normLinvR}. Then
	\begin{align*}
		\LLeinv[\calR_\star[\mu]]
			=&		-\frac{1}{\pi^2\sqrt{\eps^2-t^2}}\pvinteps\frac{\sqrt{\eps^2-\tau^2}\;\del_{\tau}\calR_\star[\mu](\tau)}{t-\tau}\intd \tau+\frac{C_{\LLe[\calR_\star[\eta]]}}{\pi\log(\eps/2)\sqrt{\eps^2-t^2}}\,.
	\end{align*}
	In  (\ref{equ:PFLemma:normLJinvR:SOE:3+}) we already found out that $C_{\LLe[\calR_\star[\eta]]}= \OO(\frac{\eps}{-\log(\eps/2)})$. For the other term, consider that
	\begin{align*}
		-\frac{1}{\pi^2\sqrt{\eps^2-t^2}}\pvinteps\frac{\sqrt{\eps^2-\tau^2}\;\del_{\tau}\calR_\star[\mu](\tau)}{t-\tau}\intd \tau\\
			&\mkern-250mu =		\frac{-1/\pi^2}{\sqrt{\eps^2-t^2}}\inteps\sqrt{\eps^2-\tau^2}\, \frac{\del_\tau\calR_\star[\mu](\tau)-\del_\tau\calR_\star[\mu](t)}{t-\tau} \intd \tau\\
			&\mkern-250mu		+\frac{-1/\pi^2}{\sqrt{\eps^2-t^2}}\pvinteps\sqrt{\eps^2-\tau^2}\frac{\del_\tau\calR_\star[\mu](t)}{t-\tau} \intd \tau\,,
	\end{align*}
	thus
	{\setlength{\belowdisplayskip}{0pt} \setlength{\belowdisplayshortskip}{0pt}\setlength{\abovedisplayskip}{0pt} \setlength{\abovedisplayshortskip}{0pt}
	\begin{align*}
		\NORM{\frac{-1}{\pi^2\sqrt{\eps^2-t^2}}\pvinteps\frac{\sqrt{\eps^2-\tau^2}\;\del_{\tau}\calR_\star[\mu](\tau)}{t-\tau}\intd \tau}_{\curlXe}^2\\
			&\mkern-300mu \leq		\inteps\frac{1/\pi^4}{\sqrt{\eps^2-t^2}}\left|\inteps\sqrt{\eps^2-\tau^2}\, \frac{\del_\tau\calR_\star[\mu](\tau)-\del_\tau\calR_\star[\mu](t)}{t-\tau} \intd \tau\right|^2\intd t\nonumber\\
			&\mkern-300mu		+\inteps\frac{1/\pi^4}{\sqrt{\eps^2-t^2}}\left|\pvinteps\sqrt{\eps^2-\tau^2}\frac{\del_\tau\calR_\star[\mu](t)}{t-\tau} \intd \tau\right|^2\intd t\,
\end{align*}%HACK
\begin{align*}
			&\mkern-0mu \leq		\inteps\frac{1/\pi^4}{\sqrt{\eps^2-t^2}}\left|\inteps\frac{\sqrt{\eps^2-\tau^2}}{\sqrt{|t-\tau|}} \intd \tau\right|^2\intd t\NORM{\del_t \calR_\star[\mu]}_{\cC^{0,1/2}}^2\nonumber\\
			&\mkern-0mu		+\inteps\frac{1/\pi^4}{\sqrt{\eps^2-t^2}}\left|\pvinteps\frac{\sqrt{\eps^2-\tau^2}}{{t-\tau}} \intd \tau\right|^2\intd t\NORM{\del_t \calR_\star[\mu]}_{\cC^{0}}^2\,\\
%%%%%%%%%%%%%%%
			&\mkern-0mu\leq	C_{(\ref{equ:PFLemma:normLJinvR:SOE:3})}^2\,\frac{\eps^3}{\pi}\left|\del_t \calR_\star[\mu]\right|_{\cC^{0,1/2}}^2
					+ \frac{\eps^2}{2\pi}\NORM{\del_t \calR_\star[\mu]}_{\cC^{0}}^2\,.
	\end{align*}}
	We conclude that
	\begin{align*}
		\NORM{\LLeinv[\mu_\iota]}_{\curlXe}
			\!\!\!=& 	\NORM{R}_{\cC^{0,1/2}}\bigg(\!\OO\Big(\frac{\eps}{|\log(\eps)|}\Big)
				\!+\!\OO\Big(\frac{\eps}{|\log(\eps)|}\Big)
				\!+\!\OO({\eps^{3/2}})
				\!+\!\OO\Big(\frac{\eps}{(\log(\eps/2))^2}\Big)\!\bigg)\,\\
			=&	\NORM{R}_{\cC^{0,1/2}}\OO\Big( \frac{\eps}{|\log(\eps)|} \Big)\,.
	\end{align*}
	This proves Lemma \ref{lemma:normLJinvR}.
\end{proof}

\subsubsection{Solution to the Main Problem}
We know from Proposition \ref{prop:formulaforvzke:InOperators}, that
\begin{align*}
	-\del_{\nu_x}\GOk\left(z,\begin{pmatrix} \tau\\0 \end{pmatrix}\right) 
		=&		\,\Big(-\frac{1}{2\pi}\JJe
				+\frac{k^2}{4\pi}\LLe\Big)[\vzke]\\
		&\mkern-30mu 
				+\inteps \vzke(t)\left(\frac{k^2}{8\pi}(2\log(k)+1)+\del_{x_2}\RdelGOk\left(
				\begin{pmatrix}
				 	\tau \\ 0
				\end{pmatrix}	,
				\begin{pmatrix}
				 	t \\ 0
				\end{pmatrix}		
				\right)\right) \intd t\,.\nonumber
\end{align*}
In the last subsection we saw that the operator $(-\alpha\JJe+\beta\LLe)$, with $\beta=\frac{k^2}{4\pi}$ and $\alpha=\frac{1}{2\pi}$, is invertible, thus we have
\begin{align*}
	\vzke(t)	
		=&		\Big(-\frac{1}{2\pi}\JJe+\frac{k^2}{4\pi}\LLe\Big)^{-1}\Bigg[-\del_{\nu_x}\GOk\left(z,\begin{pmatrix} \tau\\0 \end{pmatrix}\right) \Bigg](t)	\\
		&\mkern-70mu 
				-\Big(-\frac{1}{2\pi}\JJe+\frac{k^2}{4\pi}\LLe\Big)^{-1}\left[\inteps \vzke(t)\left(\frac{k^2}{8\pi}(2\log(k)+1)+\del_{x_2}\RdelGOk\left(
				\begin{pmatrix}
				 	\tau \\ 0
				\end{pmatrix}	,
				\begin{pmatrix}
				 	t \\ 0
				\end{pmatrix}		
				\right)\right) \intd t\right]\,.\nonumber
\end{align*}

Then Proposition \ref{prop:v=intdelGv} yields
\begin{align*}
		\uzke(x)
			=&		\,\GOk(z,x)+\!\int_{\del\Om_N} \del_{\nu_y}\GOk(x,y) \vzke(y) \intd \sigma_y\,\\
			=&		\,\GOk(z,x)
					-\!\int_{\del\Om_N} \del_{\nu_y}\GOk(x,y) \Big(-\frac{1}{2\pi}\JJe+\frac{k^2}{4\pi}\LLe\Big)^{-1}\Bigg[\del_{\nu_x}\GOk\left(z,\begin{pmatrix} \tau\\0 \end{pmatrix}\right) \Bigg](y)\intd \sigma_y\nonumber\\
			&		+\!\int_{\del\Om_N} \del_{\nu_y}\GOk(x,y) \Big(\frac{-1}{2\pi}\JJe+\frac{k^2}{4\pi}\LLe\Big)^{-1}\Bigg[\inteps \vzke(t)\bigg(\frac{k^2}{8\pi}(2\log(k)+1)\nonumber\\
			&		+\del_{x_2}\RdelGOk\left(
				\begin{pmatrix}
				 	\tau \\ 0
				\end{pmatrix}	,
				\begin{pmatrix}
				 	t \\ 0
				\end{pmatrix}		
				\right)\bigg) \intd t\Bigg](y)\intd \sigma_y
\end{align*}
from where we obtain $\Geu^{k,\eps}_{(1)}$ and $\Geu^{k,\eps}_{(2)}$ in Theorem \ref{thm:ch4mainresult}.

With Lemma \ref{lemma:normLJinvR}, that is $\NORM{(-\alpha\JJe+\beta\LLe)^{-1}\calR[\mu]}_{\curlXe}\leq\frac{\eps \,C_{(\ref{equ:normLJinvR})}}{|\log(\eps)|}\NORM{R}_{\cC^{0,1/2}}\NORM{\mu}_{\curlXe}$, and with the simple reformulation 
\begin{align*}
	-\del_{\nu_x}\GOk\left(z,\begin{pmatrix} \tau\\0 \end{pmatrix}\right)
		=&	\inteps -\del_{\nu_x}\GOk\left(z,\begin{pmatrix} \tau\\0 \end{pmatrix}\right)\frac{1}{\pi\sqrt{\eps^2-t^2}}\intd t\,,
\end{align*}
we can infer
\begin{align}\label{equ:NORMvzke}
	\NORM{\vzke}_{\curlXe}			
			\leq&		\frac{\eps\,C_{(\ref{equ:NORMvzke}),1}}{|\log(\eps)|}\NORM{\del_{\nu_x}\GOk\left(z,\begin{pmatrix} \tau\\0 \end{pmatrix}\right) }_{\cC^{0,1/2}}	
						+\frac{\eps\,C_{(\ref{equ:NORMvzke}),2}}{|\log(\eps)|}\NORM{\vzke}_{\curlXe}\,.
\end{align}
Consider that $\inteps \mu(t)\intd t\leq\sqrt{\pi}\NORM{\mu}_\curlXe$, see Equation (\ref{equ:intleqXenorm}). Hence we conclude for $z,x\in\Omega, z\neq x$ that
\begin{align} 
	|\Geu^{k,\eps}_{(1)}(z,x)|
			\leq \frac{\eps\,C_{(\ref{equ:Geuoneesti})}}{|\log(\eps)|}\,,\label{equ:Geuoneesti}\\
	|\Geu^{k,\eps}_{(2)}(z,x)|
			\leq \frac{\eps^2\,C_{(\ref{equ:Geutwoesti})}}{|\log(\eps)|^2}\,. \label{equ:Geutwoesti}	
\end{align}

%we can infer 
%\begin{align}\label{equ:NORMvzke}
%	\NORM{\vzke}_{\curlXe}			\leq&		\NORM{\Big(-\frac{1}{2\pi}\JJe+\frac{k^2}{4\pi}\LLe\Big)^{-1}\Bigg[-\del_{\nu_x}\GOk\left(z,\begin{pmatrix} \tau\\0 \end{pmatrix}\right) \Bigg]}_{\curlXe}	
%				+\frac{\eps\,C_{(\ref{equ:NORMvzke})}}{|\log(\eps)|}\NORM{\vzke}_{\curlXe}\,.
%\end{align}
%\begin{align}\label{equ:NORMvzke:2}
%	\NORM{\vzke}_{\curlXe}	
%		\leq&		\frac{\eps\,C_{(\ref{equ:NORMvzke:2})}}{|\log(\eps)|}\NORM{\del_{\nu_x}\GOk\left(z,\begin{pmatrix} \tau\\0 \end{pmatrix}\right) }_{\cC^{0,1/2}}						+\OO(\eps^2)\,.
%\end{align}
\section{Nucleation of the Neumann Boundary Condition} \label{ch6f}

In this section, based on Theorem \ref{thm:ch4mainresult}, we derive a simple procedure to maximize the norm of the Green's function.
The main idea is to nucleate the Neumann boundary conditions in order to increase the transmission between the point source and the receiver. By considering a disk shaped cavity, we illustrate by some numerical experiments  the applicability of the proposed approach.

\subsection{The Disk Case}

%In this subsection we will numerically implement the result in previous section. Given a domain $\Omega$, a source at $x_S\in\Omega$ and a receiver at $x_R\in\Omega$, we will rotate the small opening $\partial\Omega_N$ on $\partial\Omega$ and regard $u_{x_S}^{k,\epsilon}$ in Theorem \ref{thm:ch4mainresult} as a function of $\theta$, i.e. the position angle of $\partial\Omega_N$. Empirically we believe that if we switch the place with positive $\frac{\partial}{\partial\theta}u_{x_S}^{k,\epsilon}$ from Dirichlet boundary to Neumann boundary,  $u_{x_S}^{k,\epsilon}(x_R)$ will be magnified compared to a pure Dirichlet case. We will plot $\frac{\partial}{\partial\theta}u_{x_S}^{k,\epsilon}$ and give boundary configuration with different $x_S$, $x_R$ locations.
%
%

Let $\Omega$ be the unit disk and let the source and the receiver be respectively $x_S=(x,y)$ and $x_R=(\xi,\eta)$. 
Suppose that the opening $\partial\Omega_N$ is an arc centered at $(1,0)$ with length $2\epsilon$. 

Denote by $\rho=\sqrt{x^2+y^2}$  the distance between $x_S$ and the origin, and by $\tilde{\rho}=\sqrt{\xi^2+\eta^2}$ the distance between $x_R$ and the origin. Define $\theta = \arctan(y/x)$, and $\tilde{\theta} = \arctan(\eta/\xi)$. It is well known that the Green's function  in the unit disk is given by 
\begin{equation}\label{eq:num1}
G_\Omega^k(x_S,x_R) = -\frac{i}{4}H_0^{(1)}(k|x_S-x_R|) + \frac{i}{4}\sum\limits_{n=-\infty}^{+\infty}A_n(k,\rho)J_n(k\tilde\rho)e^{in(\theta-\tilde{\theta})}.
\end{equation}
Recall that the cylindrical wave expansion of the free-boundary Green's function is 
\begin{equation}\label{eq:num2}
-\frac{i}{4}H_0^{(1)}(k|x_S-x_R|) = -\frac{i}{4}\sum\limits_{n=-\infty}^{+\infty}J_n(k\rho_<)H_n^{(1)}(k\rho_>)e^{in(\theta-\tilde{\theta})},
\end{equation}
where $\rho_< = \min(\rho,\tilde{\rho})$, $\rho_>=\max(\rho,\tilde{\rho})$. Substituting \eqref{eq:num2} into \eqref{eq:num1} yields
\begin{equation}\label{eq:num3}
G_\Omega^k(x_S,x_R) = \frac{i}{4}\sum\limits_{n=-\infty}^\infty(A_nJ_n(k\rho)-J_n(k\rho_<)H_n^{(1)}(k\rho_>))e^{in(\theta-\tilde{\theta})}.
\end{equation}
Imposing the Dirichlet boundary condition on \eqref{eq:num3} gives
$$
A_n(k,\rho) = \frac{J_n(k\rho)H_n^{(1)}(k)}{J_n(k)},
$$
where $J_n$ is the Bessel function of first kind and order $n$. 

Hence the Green's function is 
\begin{equation}\label{eq:greendisk}
G_\Omega^k(x_S,x_R) = -\frac{i}{4}H_0^{(1)}(k|x_S-x_R|)+\frac{i}{4}\sum\limits_{n=-\infty}^\infty\frac{J_n(k\rho)H_n^{(1)}(k)}{J_n(k)}J_n(k\tilde{\rho}),
\end{equation}
and its normal derivative on $\partial\Omega_N$ is
\begin{equation}\label{eq:greenderivative}
\begin{aligned}
\frac{\partial G_\Omega^k}{\partial\tilde\rho}\left(x_S,\begin{pmatrix}t \\ 1 \end{pmatrix}\right) =& -\frac{i}{4}kH_0^{(1)'}(k\sqrt{\rho^2+1-2\rho\cos(\theta-\theta_0)})\frac{1-\rho\cos(\theta-\theta_0)}{\sqrt{\rho^2+1-2\rho\cos(\theta-\theta_0)}} \\
& + \frac{i}{4}\sum\limits_{n=-\infty}^\infty\frac{J_n(k\rho)H_n^{(1)}(k)}{J_n(k)}kJ_n'(k) ,
\end{aligned}
\end{equation}
with $\arctan\theta_0 = 1/t$. 

%By the properties of Hankel function, we could use the following recurrence property to calculate derivative to $H_0^{(1)}$:
%$$
%2\frac{\mathrm{d}H_\alpha^{(1)}}{\mathrm{d}x} = \frac{1}{2}(H_{\alpha-1}^{(1)}(x) - H_{\alpha+1}^{(1)}(x)).
%$$

%We would like to examine the value of $u_{x_S}^{k,\epsilon}$ given fixed source, receiver position and different opening position. 
%
%In simulation, we will fix the opening to be the segment between $(-\epsilon,1)$ and $(\epsilon,1)$, but rotate $x_S$ and $x_R$. That is, $\theta$ and $\tilde{\theta}$ will vary while $\rho$, $\tilde{\rho}$, $\theta-\tilde{\theta}$ remain unchanged. From now on we will focus on the estimation to $\frac{\partial u_{x_S}^{k,\epsilon}}{\partial\theta}$.
%
%From \eqref{eq:uxs}, $u_{x_S}^{k,\epsilon}$ consists of two parts. The first part $G_\Omega^k(x_S,x_R)$ does not change because $\rho$, $\tilde{\rho}$ and $\theta-\tilde{\theta}$ remain unchanged, therefore we only need to care about the second part. 

Define

$$
I_1:=\int_{-\epsilon}^{\epsilon}\partial_{\nu_y}G_\Omega^k\left(x_R,\begin{pmatrix}
t \\ 1
\end{pmatrix}\right)(-\frac{1}{2\pi}\mathcal{J}^\epsilon + \frac{k^2}{4\pi}\mathcal{L}^\epsilon)^{-1}\left[\partial_{\nu_x}G_\Omega^k\left(x_S,\begin{pmatrix}
s \\ 1
\end{pmatrix}\right)\right](t)\mathrm{d}t.
$$
%We will try to find out the main part of $I_1$ in the following subsection.
%
%\subsubsection{Estimation of $I_1$}

By Proposition \ref{prop:-aLLe+bJJeInverse}, we have
$$
(-\alpha\mathcal{J}^\epsilon + \beta\mathcal{L}^\epsilon)^{-1}[f](t) = (\mathcal{L}^\epsilon)^{-1}[u_l](t),
$$
where
$$
u_l(t) = C_1e^{\sqrt{\frac{\beta}{\alpha}}t} + C_2e^{-\sqrt{\frac{\beta}{\alpha}}t} + \frac{1}{2\sqrt{\alpha\beta}}e^{-\sqrt\frac{\beta}{\alpha}t}\int_{-\epsilon}^tf(s)e^{\sqrt\frac{\beta}{\alpha}s}\mathrm{d}s - \frac{1}{2\sqrt{\alpha\beta}}e^{\sqrt\frac{\beta}{\alpha}t}\int_{-\epsilon}^tf(s)e^{-\sqrt\frac{\beta}{\alpha}s}\mathrm{d}s,
$$
and $C_1$, $C_2$ are constants determined by \eqref{equ:-aLLe+bJJeInverse:SOE:1} and \eqref{equ:-aLLe+bJJeInverse:SOE:2}. Therefore, 
$$
u_l'(\tau) = \sqrt\frac{\beta}{\alpha}\left(C_1e^{\sqrt{\frac{\beta}{\alpha}}\tau} - C_2e^{-\sqrt{\frac{\beta}{\alpha}}\tau}-\frac{1}{2\sqrt{\alpha\beta}}\int_{-\epsilon}^\tau f(s)2\cosh(\sqrt\frac{\beta}{\alpha}(s-\tau))\mathrm{d}s\right).
$$
By Taylor expansion, 
\begin{equation}\label{eq:ulderivative}
\begin{aligned}
&u_l'(\tau) - u_l'(0) \\
=& \sqrt{\frac{\beta}{\alpha}}(C_1e^{\sqrt{\frac{\beta}{\alpha}}\tau} - C_2e^{-\sqrt{\frac{\beta}{\alpha}}\tau} - (C_1-C_2)) - \frac{1}{2\alpha}\int_0^\tau f(s)2\cosh(\sqrt\frac{\beta}{\alpha}(s-\tau))\mathrm{d}s \\
&-\frac{1}{2\alpha}\int_{-\epsilon}^0f(s)2\left(\cosh(\sqrt{\frac{\beta}{\alpha}}(s-\tau))-\cosh(\sqrt{\frac{\beta}{\alpha}}s)\right)\mathrm{d}s \\
=& \sqrt{\frac{\beta}{\alpha}}\left(C_1\sqrt{\frac{\beta}{\alpha}} + C_2\sqrt{\frac{\beta}{\alpha}}\right)\tau + \frac{1}{2}\left(\frac{\beta}{\alpha}\right)^{3/2}(C_1-C_2)\tau^2 \\
&- \frac{1}{2\alpha}\left(2f(0)\tau + f'(0)\tau^2 + \frac{\beta}{\alpha}f(0)\tau^2\epsilon - \frac{\beta}{\alpha}f(0)\tau\epsilon^2 + \OO(\tau^3)\right) \\
=& \left(\frac{\beta}{\alpha}(C_1+C_2)\tau - \frac{1}{\alpha}f(0)+\frac{\beta}{2\alpha^2}f(0)\epsilon^2\right)\tau \\
&+ \left(\frac{1}{2}\left({\frac{\beta}{\alpha}}\right)^{3/2}(C_1-C_2) - \frac{1}{2\alpha}f'(0) - \frac{\beta}{2\alpha^2}f(0)\epsilon\right)\tau^2 + \OO(\eps, \tau, 3)\,,
\end{aligned}
\end{equation}
where $\OO(\eps, \tau, 3)$ denotes the infinite sum of terms of the form $C\,\eps^n\,\tau^m$, where $n+m\geq 3$. Denote by
\begin{equation}\label{eq:cul1}
C_{u_l}^1 := \frac{\beta}{\alpha}(C_1+C_2) - \frac{1}{\alpha}f(0) + \frac{\beta}{2\alpha^2}f(0)\epsilon^2
\end{equation}
and by
\begin{equation}\label{eq:cul2}
C_{u_l}^2 := \frac{1}{2}\left({\frac{\beta}{\alpha}}\right)^{3/2}(C_1-C_2) - \frac{1}{2\alpha}f'(0) - \frac{\beta}{2\alpha^2}f(0)\epsilon.
\end{equation}

For $u_l'(0)$ we have
\begin{equation}\label{eq:ul0derivative}
\begin{aligned}
u_l'(0) =& \sqrt{\frac{\beta}{\alpha}}\left(C_1-C_2-\frac{1}{2\sqrt{\alpha\beta}}\int_{-\epsilon}^0 f(s)2\cosh(\sqrt{\frac{\beta}{\alpha}}s)\mathrm{d}s\right) \\
=& \sqrt{\frac{\beta}{\alpha}}(C_1-C_2) \\
&- \frac{1}{\sqrt{\alpha\beta}}\left(f(0)\sqrt{\frac{\alpha}{\beta}}\sinh(\sqrt\frac{\beta}{\alpha}\epsilon) + f'(0)\frac{\alpha}{\beta}\left(-1+\cosh(\sqrt{\frac{\beta}{\alpha}}\epsilon)-\epsilon\sqrt{\frac{\beta}{\alpha}}\sinh(\sqrt{\frac{\beta}{\alpha}}\epsilon)\right)+ O(\epsilon^3)\right) \\
=& \sqrt\frac{\beta}{\alpha}(C_1-C_2) - \frac{1}{\alpha}f(0)\epsilon + \frac{1}{2\alpha}f'(0)\epsilon^2+ O(\epsilon^3),
\end{aligned}
\end{equation}
while for $u_l(0)$ it holds that
\begin{equation}\label{eq:ul0}
\begin{aligned}
u_l(0) &= C_1 + C_2 + \frac{1}{2\sqrt{\alpha\beta}}\int_{-\epsilon}^0 f(s)(e^{\sqrt\frac{\beta}{\alpha}s} - e^{-\sqrt\frac{\beta}{\alpha}s})\mathrm{d}s \\
&= C_1 + C_2 + \frac{1}{\alpha}f(0)(-\frac{1}{2}\epsilon^2) + O(\epsilon^3)
\end{aligned}
\end{equation}
Recall the formula for $(\mathcal{L}^\epsilon)^{-1}$ and $C_{\mathcal{L}^\epsilon}[u_l]$ in Proposition \ref{prop:LLEisInjective}. Since $C_{\mathcal{L}^\epsilon}[u_l]$ is a constant, we can simply evaluate it at $x=0$. We find that
\begin{align*}
C_{\mathcal{L}^\epsilon}[u_l] &= u_l(0) - \mathcal{L}^\epsilon[-\frac{1}{\pi^2\sqrt{\epsilon^2-t^2}}\pvintepsnl\frac{\sqrt{\epsilon^2-\tau^2}u_l'(\tau)}{t-\tau}\mathrm{d}\tau](0) \\
&= u_l(0) + \int_{-\epsilon}^\epsilon \pvintepsnl \frac{1}{\pi^2\sqrt{\epsilon^2-t^2}}\frac{\sqrt{\epsilon^2-\tau^2}u_l'(\tau)}{t-\tau}\log|t|\mathrm{d}\tau\mathrm{d}t	.
\end{align*}
From \eqref{eq:ulderivative}, we obtain
\begin{align*}
&\int_{-\epsilon}^\epsilon \pvintepsnl \frac{1}{\pi^2\sqrt{\epsilon^2-t^2}}\frac{\sqrt{\epsilon^2-\tau^2}u_l'(\tau)}{t-\tau}\log|t|\mathrm{d}\tau\mathrm{d}t \\
=&\int_{-\epsilon}^\epsilon \pvintepsnl \frac{1}{\pi^2\sqrt{\epsilon^2-t^2}}\frac{\sqrt{\epsilon^2-\tau^2}(u_l'(0) + C_{u_l}^1\tau + C_{u_l}^2\tau^2 + O(\epsilon^3)) }{t-\tau}\log|t|\mathrm{d}\tau\mathrm{d}t  \\
=&u_l'(0)\int_{-\epsilon}^\epsilon \frac{1}{\pi^2\sqrt{\epsilon^2-t^2}}\pi t\log|t|\mathrm{d}t + C_{u_l}^1\int_{-\epsilon}^\epsilon\frac{1}{\pi^2\sqrt{\epsilon^2-t^2}}(-\frac{\pi\epsilon^2}{2} + \pi t^2)\log|t|\mathrm{d}t \\
& + C_{u_l}^2\int_{-\epsilon}^\epsilon\frac{1}{\pi^2\sqrt{\epsilon^2-t^2}}(-\frac{1}{2}\epsilon^2\pi t + \pi t^3)\log|t|\mathrm{d}t + O(\epsilon^4) \\
=&  \frac{C_{u_l}^1}{\pi}\left(-\frac{\epsilon^2}{2}(-\pi\log(\frac{2}{\epsilon}))-\frac{1}{4}\epsilon^2\pi(-1+2\log(\frac{2}{\epsilon}))\right) + 0 + O(\epsilon^4) \\
=& \frac{C_{u_l}^1}{4}\epsilon^2 + O(\epsilon^4).
\end{align*}
Here, we have used the fact that 
$$
\pvintepsnl\frac{\tau\sqrt{\epsilon^2-\tau^2}}{t-\tau}\mathrm{d}\tau = -\frac{\pi\epsilon^2}{2} + \pi t^2,
$$
$$
\int_{-\epsilon}^\epsilon\frac{\log|t|}{\sqrt{\epsilon^2-t^2}}\mathrm{d}t = -\pi\log\left(\frac{2}{\epsilon}\right),
$$
and
$$
\int_{-\epsilon}^\epsilon \frac{t^2\log|t|}{\sqrt{\epsilon^2-t^2}}\mathrm{d}t = -\frac{1}{4}\epsilon^2\pi\left(-1+2\log\left(\frac{2}{\epsilon}\right)\right).
$$
Therefore, the following estimation holds:
\begin{equation}\label{eq:clepsilonest}
C_{\mathcal{L}^\epsilon}[u_l] = u_l(0) + \frac{C_{u_l}^1}{4}\epsilon^2 + O(\epsilon^4).
\end{equation}
We now compute $C_1$ and $C_2$ from \eqref{equ:-aLLe+bJJeInverse:SOE:1} and \eqref{equ:-aLLe+bJJeInverse:SOE:2}. Applying \eqref{eq:ulderivative} to \eqref{equ:-aLLe+bJJeInverse:SOE:1}, we get
$$
\int_{-\epsilon}^\epsilon\sqrt\frac{\epsilon+\tau}{\epsilon-\tau}(u_l'(0) + C_{u_l}^1\tau + C_{u_l}^2\tau^2 + O(\tau^3))\mathrm{d}\tau - \frac{\pi}{\log(\epsilon/2)}C_{\mathcal{L}^\epsilon}[u_l] = 0,
$$
which implies that
$$
u_l'(0)\pi\epsilon + C_{u_l}^1\frac{\pi\epsilon^2}{2} + C_{u_l}^2\frac{\pi\epsilon^3}{2}  - \frac{\pi}{\log(\epsilon/2)}(u_l(0) + \frac{C_{u_l}^1}{4}\epsilon^2) +O(\epsilon^4)= 0.
$$
Combining the last equation together with \eqref{eq:ul0} and \eqref{eq:ul0derivative}, it follows that
\begin{equation}
\begin{aligned}
&\left(\sqrt\frac{\beta}{\alpha}(C_1-C_2) - \frac{1}{\alpha}f(0)\epsilon + O(\epsilon^2)\right)\pi\epsilon + C_{u_l}^1\frac{\pi\epsilon^2}{2} + C_{u_l}^2\frac{\pi\epsilon^3}{2} + O(\epsilon^4) \\
&- \frac{\pi}{\log(\epsilon/2)}\left(C_1 + C_2 - \frac{1}{2\alpha}f(0)\epsilon^2 + O(\epsilon^4) + \frac{C_{u_l}^1}{4}\epsilon^2\right) = 0.
\end{aligned}
\end{equation}
Thus, 
\begin{equation}
\begin{aligned}
\sqrt\frac{\beta}{\alpha}\epsilon(C_1 - C_2) &- \frac{1}{\log(\epsilon/2)}(C_1 + C_2) \\
&= 
\frac{1}{\alpha}f(0)\epsilon^2 -\frac{C_{u_l}^1}{2}\epsilon^2 - \frac{C_{u_l}^2}{2}\epsilon^3- \frac{1}{\log(\epsilon/2)}\left(\frac{1}{2\alpha}f(0)\epsilon^2 - \frac{C_{u_l}^1}{4}\epsilon^2\right) + O(\epsilon^4).
\end{aligned}
\end{equation}
Similarly, from  \eqref{equ:-aLLe+bJJeInverse:SOE:2} we obtain
\begin{equation}
\begin{aligned}
-\sqrt\frac{\beta}{\alpha}\epsilon(C_1 - C_2) &- \frac{1}{\log(\epsilon/2)}(C_1 + C_2) \\
&= 
-\frac{1}{\alpha}f(0)\epsilon^2 - \frac{C_{u_l}^1}{2}\epsilon^2 + \frac{C_{u_l}^2}{2}\epsilon^3- \frac{1}{\log(\epsilon/2)}\left(\frac{1}{2\alpha}f(0)\epsilon^2 - \frac{C_{u_l}^1}{4}\epsilon^2\right) + O(\epsilon^4).
\end{aligned}
\end{equation}
Therefore,
\begin{equation}\label{eq:c1plusc2}
C_1 + C_2 =\left(\frac{1}{2\alpha}f(0)-  \frac{C_{u_l}^1}{4}\right)\epsilon^2 + \frac{C_{u_l}^1}{2}\log(\epsilon/2)\epsilon^2 + O(\epsilon^3),
\end{equation}
and
\begin{equation}\label{eq:c1minusc2}
C_1 - C_2 = \frac{1}{\sqrt{\alpha\beta}}f(0)\epsilon - \frac{1}{2}\sqrt{\frac{\alpha}{\beta}}C_{u_l}^2\epsilon^2+ O(\epsilon^2).
\end{equation}

Combining together \eqref{eq:cul1}  with \eqref{eq:c1plusc2} gives
\begin{equation}\label{eq:c1plusc2final}
\begin{aligned}
C_1 + C_2 &= \frac{\frac{-1}{2\alpha}f(0)\log(\epsilon/2)\epsilon^2 + \frac{3}{4\alpha}f(0)\epsilon^2}{1-\frac{\beta}{2\alpha}\log(\epsilon/2)\epsilon^2+\frac{\beta}{4\alpha}\epsilon^2} + O(\epsilon^3\log\epsilon) \\
&= -\frac{1}{2\alpha}f(0)\log(\epsilon/2)\epsilon^2 + \frac{3}{4\alpha}f(0)\epsilon^2 + O(\epsilon^2\log\epsilon)
\end{aligned}
\end{equation}
and
\begin{equation}\label{eq:cul1final}
C_{u_l}^1 = -\frac{1}{\alpha}f(0) + \frac{\beta}{\alpha}\left(-\frac{1}{2\alpha}f(0)\log(\epsilon/2)\epsilon^2 + \frac{3}{4\alpha}f(0)\epsilon^2\right) + O(\epsilon^2\log\epsilon),
\end{equation}
while combining \eqref{eq:cul2} and \eqref{eq:c1minusc2} together leads to
\begin{equation}\label{eq:c1minusc2final}
C_1 - C_2 = \frac{1}{\sqrt{\alpha\beta}}f(0)\epsilon + \frac{1}{4\sqrt{\alpha\beta}}f'(0)\epsilon^2 + O(\epsilon^3),
\end{equation}
and
\begin{equation}\label{eq:cu2final}
C_{u_l}^2 = -\frac{1}{2\alpha}f'(0) + \frac{1}{8\alpha^2}f'(0)\epsilon^2 + O(\epsilon^3).
\end{equation}
Now, we are ready to estimate $(\LLe)^{-1}[u_l](t)$. From Proposition \ref{prop:LLEisInjective}, we know that
\begin{equation}
\begin{aligned}
(\LLe)^{-1}[u_l](t) &= -\frac{1}{\pi^2\sqrt{\epsilon^2-t^2}}\pvintepsnl\frac{\sqrt{\epsilon^2-\tau^2}u_l'(\tau)}{t-\tau}\mathrm{d}\tau + \frac{C_{\mathcal{L}^\epsilon}[u_l]}{\pi\log(\epsilon/2)\sqrt{\epsilon^2-t^2}} \\
&= -\frac{1}{\pi^2\sqrt{\epsilon^2-t^2}}\int_{-\epsilon}^\epsilon\frac{\sqrt{\epsilon^2-\tau^2}(u_l'(0)+C_{u_l}^1\tau + C_{u_l}^2\tau^2+ O(\tau^3))}{t-\tau}\mathrm{d}\tau \\
& \qquad + \frac{C_{\mathcal{L}^\epsilon}[u_l]}{\pi\log(\epsilon/2)\sqrt{\epsilon^2-t^2}} \\
&= -\frac{1}{\pi\sqrt{\epsilon^2-t^2}}u_l'(0)t - \frac{C_{u_l}^1}{\pi\sqrt{\epsilon^2-t^2}}(-\frac{1}{2}\epsilon^2 + \frac{1}{2}t^2) \\
&\qquad -\frac{C_{u_l}^2}{\pi\sqrt{\epsilon^2-t^2}}(-\frac{1}{2}\epsilon^2\pi t + \pi t^3)+ O(\epsilon^4).
\end{aligned}
\end{equation}
Plugging  $\alpha = 1/(2\pi)$ and $\beta = k^2/(4\pi)$ into \eqref{eq:ul0derivative}, we get
\begin{equation}
\begin{aligned}
&\left(-\frac{1}{2\pi}\mathcal{J}^\epsilon + \frac{k^2}{4\pi}\mathcal{L}^\epsilon\right)^{-1}[f](t) \\
=& -\frac{1}{\sqrt{\epsilon^2-t^2}}\left( -f(0)t^2 - f'(0)\pi t^3 +f(0)\epsilon^2 + \frac{1}{2}f'(0)\epsilon^2\pi t + \frac{3}{2}f'(0)\epsilon^2 t + O(\epsilon^3)\right).
\end{aligned}
\end{equation}
Therefore,
\begin{equation}
\begin{aligned}
&\int_{-\epsilon}^{\epsilon}\partial_{\nu_y}G_\Omega^k(x_R,\begin{pmatrix}
t \\ 1
\end{pmatrix})(-\frac{1}{2\pi}\mathcal{J}^\epsilon + \frac{k^2}{4\pi}\mathcal{L}^\epsilon)^{-1}[f](t)\mathrm{d}t \\
=&\left( \partial_{\nu_y}G_\Omega^k\left(x_R,\begin{pmatrix}
0 \\ 1
\end{pmatrix}\right)+O(\epsilon)\right)\\
&\int_{-\epsilon}^{\epsilon}-\frac{1}{\sqrt{\epsilon^2-t^2}}\left( -f(0)t^2 - f'(0)\pi t^3 +f(0)\epsilon^2 + \frac{1}{2}f'(0)\epsilon^2\pi t + \frac{3}{2}f'(0)\epsilon^2 t + O(\epsilon^3)\right)\mathrm{d}t \\
=&\left( \partial_{\nu_y}G_\Omega^k\left(x_R,\begin{pmatrix}
0 \\ 1
\end{pmatrix}\right)+O(\epsilon^1)\right)\left(f(0)\frac{\epsilon^2\pi}{2}-f(0)\epsilon^2\pi+O(\epsilon^3)\right) \\
=& -f(0)\frac{\pi}{2}\epsilon^2\partial_{\nu_y}G_\Omega^k\left(x_R,\begin{pmatrix}
0 \\ 1
\end{pmatrix}\right) + O(\epsilon^3).
\end{aligned}
\end{equation}
Recall in our setting that $f(s):=\partial_{\nu_x}G_\Omega^k\left(x_S,\begin{pmatrix}
s \\ 1
\end{pmatrix}\right)$. Thus,
\begin{equation}
\begin{aligned}
&\int_{-\epsilon}^{\epsilon}\partial_{\nu_y}G_\Omega^k\left(x_R,\begin{pmatrix}
t \\ 1
\end{pmatrix}\right)\left(-\frac{1}{2\pi}\mathcal{J}^\epsilon + \frac{k^2}{4\pi}\mathcal{L}^\epsilon\right)^{-1}\left[\partial_{\nu_x}G_\Omega^k\left(x_S,\begin{pmatrix}
s \\ 1
\end{pmatrix}\right)\right](t)\mathrm{d}t \\
&=-\frac{\pi}{2}\epsilon^2\partial_{\nu_x}G_\Omega^k\left(x_S,\begin{pmatrix}
0 \\ 1
\end{pmatrix}\right)\partial_{\nu_y}G_\Omega^k\left(x_R,\begin{pmatrix}
0 \\ 1
\end{pmatrix}\right) + O(\epsilon^3). 
\end{aligned}
\end{equation}

\subsection{Numerical Illustrations}

Now, two numerical experiments are presented in order to verify the applicability of the proposed methodology. In each one, the topological derivative is evaluated to detect the parts of the boundary where a Neumann boundary condition should be nucleated. 

Denote by
\begin{equation}\label{eq:deffs}
F_S(\theta):=\frac{\partial G_\Omega^k}{\partial\tilde\rho}\left(x_S,\begin{pmatrix}0 \\ 1\end{pmatrix}\right),
\end{equation}
and  by
\begin{equation}
\label{eq:deffr}
F_R(\theta):=\frac{\partial G_\Omega^k}{\partial\tilde\rho}\left(x_R,\begin{pmatrix}0 \\ 1 \end{pmatrix}\right).
\end{equation}

\begin{figure}[!h]
	\centering
	\includegraphics[scale=0.43]{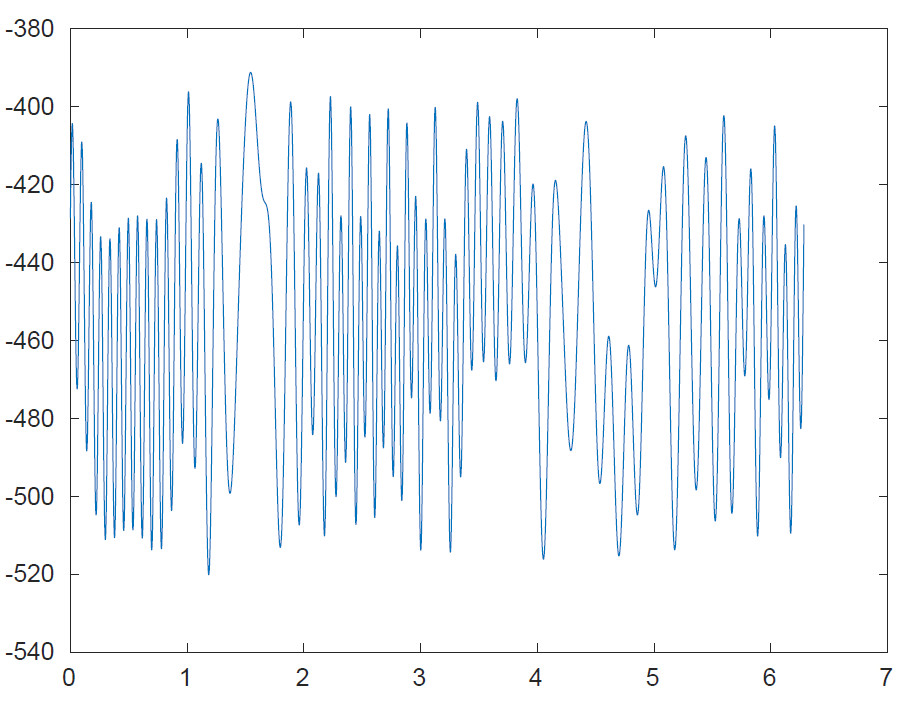}
	\caption{Plot of $y(\theta)$ as a function of $\theta$.}
	\label{fig:ploty}
\end{figure}

\begin{figure}[!h]
	\centering
	\includegraphics[scale=0.45]{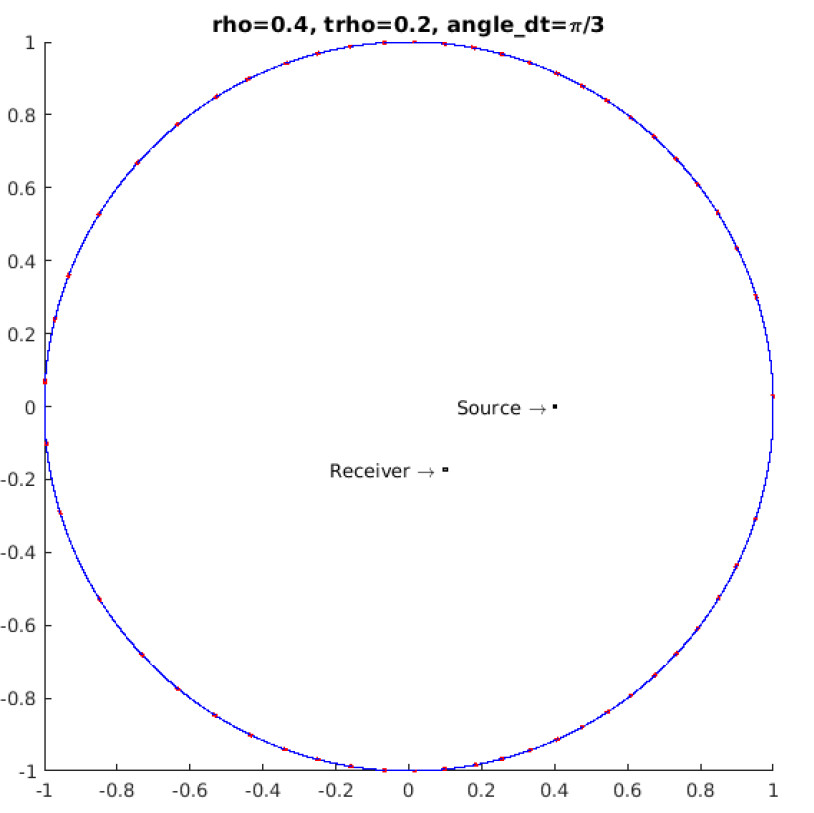}
	\caption{Nucleation of the Neumann boundary condition.}
	\label{fig:smartwall}
\end{figure}

Our goal is to maximize $|u_{x_S}^{k,\epsilon}|^2$, i.e., the norm of Green function at the receiver. We plot 
$y(\theta):= \Re (\pi \partial_{\nu_x}G_\Omega^k\left(x_S,\begin{pmatrix}
0 \\ 1
\end{pmatrix}\right)\partial_{\nu_y}G_\Omega^k\left(x_R,\begin{pmatrix}
0 \\ 1
\end{pmatrix}\right))$ as a function of $\theta$.  

Set the wave number to be $k=200$, the distance between the source and the origin to be $\rho=0.4$, the distance between the receiver and the origin to be $\tilde\rho = 0.2$, the angle difference from the receiver to the source to be $\pi/3$. We divide the whole boundary into $N=10000$ parts, and set $N/1000=10$ parts left and right of each local maximal point of $y(\theta)$ to be a Neumann part. Figure \ref{fig:ploty} presents $y(\theta)$. Figure \ref{fig:smartwall} gives the corresponding configuration of the boundary, where the red part corresponds to a Neumann boundary condition and the blue part corresponds to a  Dirichlet boundary condition. The implementation shows that the norm of the Green function $G_\Omega^k$ with Dirichlet boundary at $x_R$ is 0.067875, while the norm of the Green function $u_{x_S}^{k,\epsilon}$ with modified boundary at $x_R$ is $32.403610$.

\section{Conclusion}\label{chapter:7}
%What was the Problem
%How did we solve it
%What else can be done

In this paper, we have established a mathematical theory of micro-scaled periodically arranged Helmholtz resonators and derived expansions of the scattered fields at the subwavelength resonances in terms of the size of the gap opening. We have highlighted the mechanism of the Neumann and Dirichlet functions to exploit the intrinsic properties of the wave behaviour near and away from the gaps. 
 With this knowledge we were able to answer both question; how can we model an array of Helmholtz resonators and how can we enhance the signal at a receiving point inside the cavity by switching the boundary conditions from Dirichlet to Neumann on  specific parts of the cavity boundary.  

Our approach opens many new avenues for mathematical imaging and focusing of waves in complex media. Whereas the results in Sections \ref{ch1HR} and \ref{ch:2HR} are important for industrial objectives, the results in Sections \ref{ch5} and \ref{ch6f} can lead to enhanced communication between devices, like cell phones by improving the transmission between a source and a receiver through specific eigenmodes of the cavity. However, many challenging problems are still to be solved. For instance, how to optimize some specific cavity eigenmodes or how to design broadband metasurfaces which allow for broadband shaping and controlling of waves in complex media. These challenging problems would be the subject of forthcoming works.

%\appendix

%\input{AppendixA}
%\input{AppendixB}
%\backmatter

\bibliographystyle{plain}
\bibliography{refs_final}

%\includepdf[pages={-}]{DOCFull.pdf}

\end{document}